\documentclass[a4paper,reqno,11pt]{amsart}

\usepackage[margin=2.5cm]{geometry}
\usepackage[foot]{amsaddr}
\usepackage{amsfonts}      
\usepackage{amsmath}       
\usepackage{amssymb}       
\usepackage{amsthm}        
\usepackage{bm}            
\usepackage{comment}
\usepackage{dsfont}
\usepackage{enumitem}
\usepackage{graphicx}      
\usepackage{mathtools}     
\usepackage{makecell}      
\usepackage{xcolor}        
\usepackage[hidelinks]{hyperref} 
\hypersetup{
	colorlinks,
	linkcolor={red!50!black},
	citecolor={blue!50!black},
	urlcolor={blue!80!black}
}
\usepackage{cleveref}      
\usepackage[colorinlistoftodos]{todonotes} 
\usepackage{subcaption}
\usepackage{tikz}
\usetikzlibrary{positioning,arrows.meta,calc}
\usepackage{upgreek}
\usepackage{xpatch} 
\xpatchcmd{\proof} 
  {\itshape}
  {\bfseries}
  {}
  {}

\usepackage{glossaries}
\expandafter\def\csname ver@etex.sty\endcsname{3000/12/31}

\usepackage{autonum}

\usepackage[square,numbers]{natbib}

\numberwithin{equation}{section}
\newtheorem{theorem}{Theorem}[section]
\newtheorem{lemma}[theorem]{Lemma}
\newtheorem{corollary}[theorem]{Corollary}
\newtheorem{proposition}[theorem]{Proposition}

\theoremstyle{definition}
\newtheorem{definition}[theorem]{Definition}

\theoremstyle{definition}
\newtheorem{remark}[theorem]{Remark}

\newcommand*{\N}{\mathbb{N}}
\newcommand*{\E}{\mathbb{E}}
\newcommand*{\R}{\mathbb{R}}

\newcommand*{\Ascr}{\mathcal A}
\newcommand*{\Bscr}{\mathcal B}
\newcommand*{\Cscr}{\mathcal C}
\newcommand*{\Dscr}{\mathcal D}

\newcommand*{\Fscr}{\mathcal F}
\newcommand*{\Gscr}{\mathcal G}

\newcommand*{\Nscr}{\mathcal N}

\newcommand*{\Pscr}{\mathcal P}

\renewcommand{\d}{{\mathrm{\,d}}}
\DeclareMathAlphabet{\mathbbmsl}{U}{bbm}{m}{sl}

\makeatletter
\renewcommand\subsubsection{%
  \@startsection{subsubsection}{3}{\z@}%
    {.5\linespacing\@plus.7\linespacing}%
    {-.5em}%
    {\normalfont\bfseries}%
}
\newcommand{\subsectionnotoc}[1]{%
  \refstepcounter{subsection}%
  \bigskip
  \noindent{\normalfont\bfseries \thesubsection.\ #1.\ }\nobreak
}

\makeatother

\title[{Non-uniqueness of nonlinear Markov processes associated with PDEs}]{{Non-uniqueness of nonlinear Markov processes in the sense of McKean associated with parabolic PDEs
}}

\author{Ehsan Abedi$^\dagger$}
\author{Florian Bechtold$^\ddag$}
\author{Marco Rehmeier$^\dagger$}

\address{$^\dagger$Institute of Mathematics\\
         TU Berlin\\
         10623 Berlin\\
         Germany.}
\address{$^\ddag$Faculty of Mathematics \\
         Bielefeld University \\
         33501 Bielefeld \\
         Germany.}

\email{ehsan.abedi@tu-berlin.de}
\email{fbechtold@math.uni-bielefeld.de}
\email{rehmeier@tu-berlin.de}

\thanks{}


\begin{document}

\begin{abstract}
We derive a general scheme to construct infinitely many probabilistic counterparts for solutions to nonlinear PDEs by recasting the latter as different nonlinear Fokker--Planck equations and by constructing, for each of these equations, a solution to the associated McKean--Vlasov SDE with one-dimensional time marginal densities given by the PDE solution. We utilize this scheme to prove that nonlinear Markov processes in the sense of McKean as introduced by Rehmeier--R\"ockner (J.\,Theor.\,Probab. 38, 60 (2025)) are \emph{not} uniquely determined by their one-dimensional time marginals. This is in sharp contrast to the case of classical Markov processes, which \emph{are} uniquely determined by their one-dimensional time marginals. We demonstrate our results by constructing a continuum of nonlinear Markov processes with one-dimensional time marginal densities given by the Barenblatt solutions to the porous medium and $p$-Laplace equations, as well as by the fundamental solution to the heat equation. This includes a novel martingale representation for the $p$-Laplace Barenblatt solutions. We also prove that a nonlinear Markov process is uniquely determined by its \emph{two-}dimensional time marginals. Moreover, for the porous medium equation, we show that the different McKean--Vlasov SDEs we investigate are consistent with corresponding gradient flow interpretations of the equation in the sense of Otto calculus. 
\end{abstract}

\maketitle

\tableofcontents

\section{Introduction}\label{sec:introduction}
To explain our motivation for this paper, we first recall the following relation between linear PDEs, stochastic differential equations (SDEs), and Markov processes. Let $\{u^z\}_{z\in \R^d}$ be a family of probability density-valued solutions to a linear PDE of type
\begin{equation}\label{intro:lin-PDE}
    \partial_t u (t,x) = \Delta \big( a(x) u (t,x) \big) -  \nabla \cdot\big( b(x) u (t,x)\big),\qquad (t,x) \in (0,\infty)\times \R^d,
\end{equation}
where $a:  \R^d \to [0,\infty)$ and $b: \R^d \to \R^d$, with initial datum $u^z(0,\cdot) = \delta_z$. Then, under broad assumptions, there exists a natural probabilistic counterpart for $\{u^z\}_{z\in \R^d}$, namely a family of stochastic processes $\{X^z\}_{z\in \R^d}$ consisting of the unique solutions to the  SDE
$$dX_t = b(X_t) \, dt + \sqrt{2a(X_t)} \, dW_t,\qquad t >0,$$
where $W = (W_t)_{t \geq 0}$ is a standard $d$-dimensional Brownian motion, with $X^z_0 = z$, and the law of $X^z_t$ is given by $u^z(t,x)dx$ for all $t>0$. Moreover, the path laws of these processes form a Markov process, which is uniquely determined by its one-dimensional time marginals $\{u^z(t,x)dx\}_{t>0, z \in \R^d}$.
The central example is the heat equation ($a \equiv \frac 1 2, b \equiv 0$) and its fundamental solutions, i.e.,  $u^z$ is the classical heat kernel centered at $z$. In this case, $X^z = W  +z$, and the corresponding Markov process consists of Wiener measures translated by $z$. 
These relations, which establish a deep connection between linear PDEs and stochastic analysis by allowing to study PDE problems via probabilistic methods and vice versa, have been fundamental for at least 70 years, for instance in potential theory, see \cite{BliedtnerHansen1986,
      BlumenthalGetoor1968,
      Doob2001,
      Dynkin1965,
      Freidlin1996,
      FukushimaOshimaTakeda2011,
      Liggett2010,
      RogersWilliams2000,
      Sharpe1988,
      Stroock2014,
      StroockVaradh2007}.

\subsubsection*{Nonlinear PDEs and nonlinear Markov processes.} The situation changes drastically when the PDE \eqref{intro:lin-PDE} is replaced by a \emph{nonlinear} PDE of type
\begin{equation}\label{eq:general-PDE_intro1}
    \partial_t u =  L u,
\end{equation}
for a nonlinear differential operator $L$, for example, $Lu = \Delta u^m$ (porous medium equation), or $Lu = \nabla \cdot (|\nabla u|^{p-2}\nabla u)$ ($p$-Laplace equation).
Suppose \eqref{eq:general-PDE_intro1} admits a probability density-valued solution, which also solves a nonlinear Fokker--Planck equation (FPE) of type
\begin{equation}\label{intro:nl-PDE}
     \partial_t u = \Delta \big( a(x,u) u \big) -  \nabla \cdot\big( b(x,u) u \big),\qquad (t,x) \in (0,\infty)\times \R^d,
\end{equation}
where $a: \R^d \times \Pscr_* \to [0,\infty) $ and $b: \R^d \times \Pscr_* \to \R^d $ for some $\Pscr_* \subseteq \Pscr_{\textup{ac}}$ (the latter denotes the set of absolutely continuous probability measures on $\R^d$).
Then, for $u$, under a mild integrability condition  \cite{Trevisan16,BogachevRoecknerShaposhnikov2021,BarbuRoeckner2020nonlinearsuperpositionprinciple}, there still exists a weak solution $X$ to the corresponding SDE, which in this case is a McKean--Vlasov SDE (MV-SDE)
\begin{equation}\label{eq:main_MV-SDE_intro}
        \begin{dcases}
            dX_t = b(X_t,u(t,\cdot)) \, dt +  \sqrt{2a(X_t,u(t,\cdot))}  \, dW_t, \\
            \mathcal{L}(X_t) = u(t,x)dx,\quad t>0
        \end{dcases}
    \end{equation}
(here and throughout, $\mathcal{L}(Y)$ denotes the distribution of a random variable $Y$). But even if \eqref{intro:nl-PDE} has a \emph{unique} solution $u^z$ with initial datum $\delta_z$ for each $z\in \R^d$ and if \eqref{eq:main_MV-SDE_intro} has \emph{unique} solutions $X^z$ with marginal densities $u^z$, the path laws of these $X^z$ do \emph{not} form a Markov process.
This deficiency was resolved in \cite{RehmeierRoeckner2025NonlinearMarkov}, where the authors, guided by ideas from McKean \cite{McKean1966}, introduced the more general notion of \emph{nonlinear Markov processes} (see Definition \ref{def:NMC}). They proved that in the aforementioned well-posed situation, the path laws indeed form a nonlinear Markov process, and that this is also true under substantially more general assumptions (crucially, no uniqueness for \eqref{intro:nl-PDE} is required). Based on \cite{RehmeierRoeckner2025NonlinearMarkov}, nonlinear Markovian probabilistic counterparts were constructed for the explicit Barenblatt solutions of the porous medium, $p$-Laplace, and Leibenson equations, as well as for solutions to Burgers, $2D$ vorticity Navier--Stokes and Euler equations, see \cite{RehmeierRoeckner2025NonlinearMarkov, RehmeierRomito2025,BarbuRehmeierRockner2024,BarbuRocknerZhang2025, R.BGR25}.

However, a natural question remained open, namely, whether, in analogy with the linear theory, these nonlinear Markov processes are uniquely determined by their prescribed PDE solutions, i.e., by their one-dimensional time marginals. This was our initial motivation for the present paper, and as one of our main results, we prove---in contrast to the linear case---the answer to be resoundingly negative.
This sparked a surprising (at least to us) number of further problems and results on probabilistic counterparts of nonlinear PDEs, as well as on MV-SDEs and nonlinear Markov processes. Some of these results are presented in the present paper.

\subsubsection*{List of main results} Let us present the central contributions of our paper. More detailed discussions are given afterwards. 

\begin{enumerate}[label=\textbf{(\roman*)}, font=\normalfont, itemsep=1em ,leftmargin=*]
    \item (\Cref{sec:scheme}) We derive a scheme to construct infinitely many nonlinear Markov processes with one-dimensional time marginal densities given by a prescribed family of solutions to a nonlinear PDE \eqref{intro:nl-PDE} (see \Cref{cor:new}, \Cref{thm:Markov-construction}, and \Cref{crl:nl_markov}). Applying our scheme to some classical nonlinear PDEs yields
    \begin{theorem}\label{intro:thm1}
        Nonlinear Markov processes are not uniquely determined by their one-dimensional time marginals.
    \end{theorem}
   \noindent  More precisely, in our scheme, we construct unique solutions to entirely different, often degenerate or singular MV-SDEs with common prescribed one-dimensional time marginals, and verify their nonlinear Markov property.

   \item This naturally raises the question whether nonlinear Markov processes are uniquely determined by their higher-dimensional time marginals. We provide the following optimal (with regard to \Cref{intro:thm1}) answer:
    \begin{theorem}[see \Cref{thm:NMP_2D}]\label{intro:thm-2D}
        Nonlinear Markov processes are uniquely determined by their two-dimensional time marginals.
    \end{theorem} 
    \end{enumerate}

    We investigate these results specifically for the Barenblatt solutions to the porous medium (PME) and $p$-Laplace equation, and show that our scheme also gives rise to a rich structure even for the heat equation. In this regard, our main results are the following.

  \begin{enumerate}[label=\textbf{(\roman*)}, font=\normalfont, itemsep=1em ,leftmargin=*]
  \setcounter{enumi}{2}
  \item (\Cref{sec:porous_medium}). We construct infinitely many nonlinear Markov processes with one-dimensional time marginal densities given by the fundamental Barenblatt solutions to the PME 
    \begin{equation}\label{intro:PME}
        \partial_t u = \Delta u^m, \quad m>1,
    \end{equation}
    each consisting of solutions to a respective MV-SDE. Among the latter is a zero noise-equation \eqref{eq:PMEdrift-DDSDE}, an additive noise SDE \eqref{eq:PME-addnoise-MVSDE} and a drift-free Stratonovich SDE \eqref{eq:PME-pure-diff-Stratono} (the latter is technically delicate and only treated formally). 

    \item (\Cref{sec:pLaplace}). Similarly, we solve distinct MV-SDEs and construct corresponding nonlinear Markov processes with one-dimensional time marginal densities given by the fundamental Barenblatt solutions to  the $p$-Laplace equation 
    \begin{equation}\label{intro:pL}
         \partial_t u = \nabla \cdot \big( |\nabla u|^{p-2}\nabla u\big),\quad p >2.
    \end{equation}
    \begin{itemize}[leftmargin=*]
         \item[$\circ$] Notably, one of these nonlinear Markov processes consists of \emph{martingales}:
        \begin{proposition}[see Proposition \ref{proposition:pure_diffusion_plaplace}]\label{prop:intro-martingale}
            There exists a family of martingales with $p$-Laplace Barenblatt one-dimensional time marginals that forms a nonlinear Markov process and is given by the unique solutions to the drift-free MV-SDE \eqref{p_lap_pure_diff_sde}.
        \end{proposition}

        \item[$\circ$] We also show that a previously constructed such process in \cite{BarbuRehmeierRockner2024}, called $p$-Brownian motion, can alternatively be viewed (at least formally) as solutions to a drift-free SDE with a ``fully anticipating'' stochastic integral (see Section \ref{subsec:pLaplace_gen_diffusion}).
    \end{itemize}

    \item (\Cref{sec:heat_equation}). We construct infinitely many nonlinear Markov processes with
    heat kernel one-dimensional marginals by interpolating between standard Brownian motion and a pure-drift process solving an ODE (\Cref{prop:HE-infin-many}). We show that these processes depend continuously on the interpolation parameter $\beta \in [0,\infty)$, the diffusion coefficient of the respective MV-SDE (\Cref{lemlem}).
\end{enumerate}
    
Our results raise the problem of identifying a \emph{canonical} nonlinear Markov process associated with given one-dimensional time marginals.
While it is possible to single out the pure-drift processes constructed in this paper for the PME and $p$-Laplace equation from an optimization perspective---as ODE solutions with minimal vector fields \cite[Chapter 8]{AGS2008GFs} or as minimizers of an energy functional among all processes with the same one-dimensional marginals \cite{Abedi2025paths,Abedi2025processes}---and while such processes can be alternatively constructed via optimal transport techniques \cite{Lisini2007,Taghvaei2016}, it remains unclear how to identify a process from a probabilistic perspective that plays a role analogous to Brownian motion.
We suggest one such approach for the PME:

\begin{enumerate}[label=\textbf{(\roman*)}, font=\normalfont, itemsep=1em,leftmargin=*]
  \setcounter{enumi}{5}
  \item (\Cref{sec:SPDE}). We show that all McKean--Vlasov interpretations for the PME provided in Section \ref{sec:porous_medium} can be identified with a corresponding formal gradient flow structure---and, thus, a geometry. We do so by means of associated generalized Dean--Kawaski equations.
\end{enumerate}

\subsubsection*{Details on main results} We now elaborate on our results in more detail.
\begin{enumerate}[label=\textbf{(\roman*)}, font=\normalfont, leftmargin=15pt, itemindent=25pt,itemsep=1em]
\item \textbf{(\Cref{sec:scheme}; Scheme).} The scheme we derive starts from a family of probability density-valued solutions $\{u^z\}_{z\in \R^d}$ to \eqref{intro:nl-PDE} with $u^z(0,x)dx = \delta_z$ and consists of three steps:
\vspace{2pt}
\begin{enumerate}[label=(\arabic*), itemsep=2pt, leftmargin=45pt]
    \item Identify coefficients $(a,b) : \R^d \times \Pscr_* \to \R \times \R^d$ with a domain $\Pscr_* \subseteq \Pscr_{\textup{ac}}$ such that each $u^z$ solves the FPE \eqref{intro:nl-PDE} (in the sense of \Cref{def:FPE-sol}).
    
    \item Construct solutions $X^z$ to the corresponding MV-SDE \eqref{eq:main_MV-SDE_intro} with $\mathcal{L}(X^z_t) = u^z(t,x)dx$.
    
    \item Show that $X^z$ is the unique solution to this MV-SDE with $\mathcal{L}(X^z_t) = u^z(t,x)dx$, and that the family of solution path laws $P^z $ of $X^z$ is a nonlinear Markov process (more precisely, a nonlinear Markov core; see \Cref{def:NMC}).
\end{enumerate}
\vspace{2pt}
Regarding Step (1), such an associated FPE need not always exist, but it does for the examples studied here (and, in fact, for large classes of PDEs, including $2D$ vorticity Navier--Stokes and Euler equations, Burgers equation, porous medium equations with nonlinear transport-type drift, the Leibenson equation, and even PDEs involving a fractional Laplace operator, see \cite{RehmeierRoeckner2025NonlinearMarkov, RehmeierRomito2025,BarbuRocknerZhang2025, R.BGR25}). The key insight here is that \emph{if} such an associated FPE can be derived, then there typically exist infinitely many other FPEs also solved by each $u^z$. In the proposition below, we present a procedure to pass from one Fokker--Planck interpretation to another.
Hence, the nonuniqueness expressed by \Cref{intro:thm1} originates from this first step. 

In Step (2), having solutions $u^z$ to the given FPE \eqref{intro:nl-PDE}, we construct via the celebrated Ambrosio--Figalli--Trevisan superposition principle (see \Cref{thm:SP-pr}), weak solutions $X^z$ to the associated MV-SDE \eqref{eq:main_MV-SDE_intro} 
such that $X^z_t$ has distribution $u^z(t,x)dx$. No further degree of nonuniqueness is induced by Step (2), since this MV-SDE is, at least in our framework, uniquely determined from \eqref{intro:nl-PDE}. Therefore, combining this with our procedure yields the following result, which follows from \Cref{lem:many-PDEs_new} and \Cref{cor:new}: 
    \begin{proposition}\label{prop:many-PDEs_new_intro_summ}
        Let $(u(t,\cdot))_{t>0} \subseteq  \Pscr_* $ be a solution to the FPE \eqref{intro:nl-PDE} (in the sense of \Cref{def:FPE-sol}) with initial condition $\zeta \in \Pscr$ and coefficients $(a, b): \R^d \times \Pscr_* \to \R \times \R^d$ for some $\Pscr_* \subseteq \Pscr_{\textup{ac}} $ satisfying the integrability condition \eqref{eq:integrability_our_paper}. Let $f: \R^d \times \Dscr(f) \to \R $ with $\Dscr(f) \subseteq \Pscr_{\textup{ac}}$, be such that $(u(t,\cdot))_{t >0} \subseteq\Dscr(f)$, and 
        $f(x,u(t,\cdot)) + a(x,u(t,\cdot))u(t,x) \geq 0 \,  u(t,x)dxdt\text{-a.e.,}$
        and that \eqref{eq:assumption_f_strong} holds.
        Then $(u(t,\cdot))_{t> 0}$ also solves
        \begin{equation}\label{eq:f_lemma_intro_FPE}
            \partial_t u = \Delta \Big( \big(a(x,u) + \frac{f(x,u)}{u} \big)  u \Big) - \nabla \cdot \Big( \big( b (x,  u ) + \frac{\nabla f(x,u)}{u} \big) u\Big),
    \end{equation}
    where the new coefficients are defined on $\hat{\Pscr}_* \coloneqq \Pscr_* \cap \{ u \in \Dscr(f): \, f(\cdot,u) \in W^{1,1}_{\textup{loc}}(\R^d)\}$, with the convention $\frac {f(\cdot,u)} {u} := 0, \,\frac{\nabla f(\cdot, u)}{u} := 0$ on $\{u = 0\}$, and there exists a weak solution $(X_t)_{t\geq 0}$ to the corresponding MV-SDE 
        \begin{equation}
        \begin{dcases}
            dX_t = \left( b(X_t,u(t,\cdot)) + \frac{\nabla f(X_t,u(t,\cdot))}{u(t,X_t)} \right) dt + \sqrt{2}\left( a(X_t,u(t,\cdot)) + \frac{f(X_t,u(t,\cdot))}{u(t,X_t)} \right)^{\frac{1}{2}} dW_t, \\
            \mathcal{L}(X_t) = u(t,x)dx,\,  t>0, \,\,\, \mathcal{L}(X_0) = \zeta.
        \end{dcases}
    \end{equation}
\end{proposition}

We remark that the above proposition extends to time-dependent coefficients $a(t,x,u)$, $b(t,x,u)$, and function $f(t,x,u)$ without additional assumptions; however, the present formulation suffices for our purposes.
This proposition shows that essentially equivalent nonlinear formulations of a PDE lead to different MV-SDEs, which can all be solved with the same prescribed one-dimensional time marginal densities $u(t,\cdot)$. Thus, if $\{u^z\}_{z\in \R^d}$ is a family of solutions to \eqref{intro:nl-PDE}, then \Cref{prop:many-PDEs_new_intro_summ} implies the existence of solution families $\{X^z\}_{z\in \R^d}$ to distinct MV-SDEs of type above with one-dimensional time marginals $\{u^z(t,\cdot)\}_{z\in \R^d, t >0}$.

In Step (3), we suppose Step (2) has been implemented for a family of solutions $\{u^z\}_{z\in \R^d}$ to \eqref{intro:nl-PDE}, yielding a family of solutions $\{X^z\}_{z\in \R^d}$ to \eqref{eq:main_MV-SDE_intro} with $\mathcal{L}(X^z_t) = u^z(t,x)dx$ and $X^z_0 = z$. Then we prove that these solutions are unique and their path laws form a \emph{nonlinear Markov core}. The latter is a new natural notion, introduced in \Cref{def:NMC}, which we motivate in the paragraph preceding this definition. Step (3) then culminates in Theorem \ref{thm:Markov-construction} and is a refinement of the general nonlinear Markov construction in \cite[Theorem 3.8]{RehmeierRoeckner2025NonlinearMarkov}. Also from this step no nonuniqueness arises, as we show that the nonlinear Markov process is uniquely determined by its marginals \emph{and the MV-SDE}.
We emphasize that Theorem \ref{thm:Markov-construction} proves weak uniqueness of solutions to the MV-SDE with Dirac initial condition, even though we assume only uniqueness for a class of linear FPEs with regular initial data. 

The conclusion of our scheme is that equivalent FPE-reformulations of a PDE lead to entirely distinct MV-SDEs and, thereby, distinct nonlinear Markov processes with common one-dimensional time marginals. \Cref{fig:pde-nfpe-mvsde-nmp} displays the structure of our scheme.
\end{enumerate}

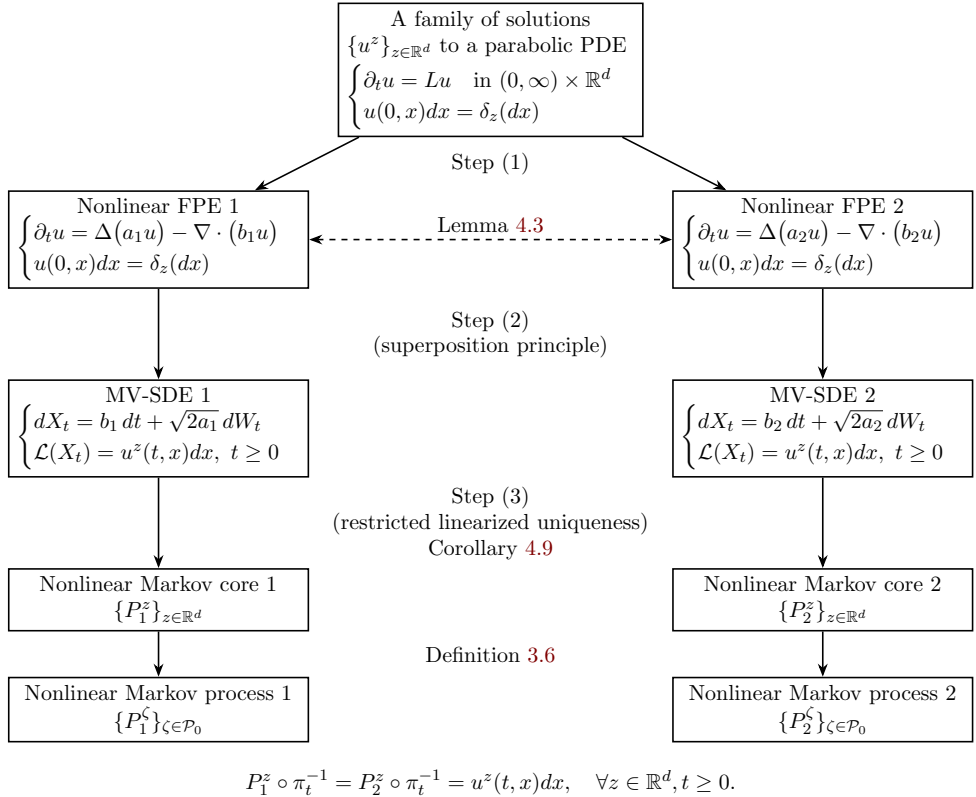
\begin{figure}
\centering
\scalebox{0.75}{
\begin{tikzpicture}[
  node distance=16mm and 5mm, 
  >=Stealth,
  box/.style={draw, thick, minimum width=30mm, minimum height=9mm, align=center}
]

\node[box] (pde) {A family of solutions \\ $\{u^z\}_{z \in \R^d} $ to a parabolic PDE \vspace{2pt} \\ $\begin{cases}
        \partial_t u   =  L u \quad \textrm{in }  (0, \infty) \times \mathbb{R}^d\\
        u(0,x)dx = \delta_z(dx)
    \end{cases}$ };

\node[box, left=of pde, yshift=-30mm, text width=50mm] (fpe1) {Nonlinear FPE 1 \\ $\begin{cases} \partial_t u = \Delta \big( a_1  u \big) - \nabla \cdot \big( b_1 u\big) \\ u(0,x)dx = \delta_z(dx) \end{cases} $};
\node[box, right=of pde, yshift=-30mm, text width=50mm] (fpe2) {Nonlinear FPE 2 \\ $\begin{cases} \partial_t u = \Delta \big( a_2  u \big) - \nabla \cdot \big( b_2 u\big) \\ u(0,x)dx = \delta_z(dx) \end{cases} $};

\node[box, below=of fpe1, yshift=0mm, text width=50mm] (mvsde1) {MV-SDE 1 \\ $\begin{cases}
    dX_t = b_1 \, dt + \sqrt{2 a_1} \, dW_t \\
    \mathcal{L}(X_t) = u^z(t,x)dx,\,\, t \geq 0
    \end{cases}$};
\node[box, below=of fpe2, yshift=0mm, text width=50mm] (mvsde2) {MV-SDE 2 \\ $\begin{cases}
    dX_t = b_2 \, dt + \sqrt{2 a_2} \, dW_t \\
    \mathcal{L}(X_t) = u^z(t,x)dx,\,\, t\geq0
    \end{cases}$};

\node[box, below=of mvsde1, yshift=0mm, text width=50mm] (nmc1) {Nonlinear Markov core 1 \\ $\{P_1^z\}_{z \in \R^d} $}; 
\node[box, below=of mvsde2, yshift=0mm, text width=50mm] (nmc2) {Nonlinear Markov core 2 \\ $\{ P_2^z\}_{z \in \R^d}  $};

\node[box, below=of nmc1, yshift=8mm, text width=50mm] (nmp1) {Nonlinear Markov process 1 \\ $\{P_1^\zeta\}_{\zeta \in \Pscr_0} $};
\node[box, below=of nmc2, yshift=8mm, text width=50mm] (nmp2) {Nonlinear Markov process 2 \\ $\{ P_2^\zeta\}_{\zeta \in \Pscr_0} $};

\draw[->, thick] (pde) -- (fpe1);
\draw[->, thick] (pde) -- (fpe2);

\draw[->, thick] ([xshift=0pt]fpe1.south) -- ([xshift=0pt]mvsde1.north); 

\draw[->, thick] ([xshift=0pt]fpe2.south) -- ([xshift=0pt]mvsde2.north); 

\draw[->, thick] (mvsde1) -- (nmc1);
\draw[->, thick] (mvsde2) -- (nmc2);

\draw[->, thick] (nmc1) -- (nmp1);
\draw[->, thick] (nmc2) -- (nmp2);

\node[draw=none, align=center, yshift=-0mm] 
at ($ ($(pde.south)!0.5!(fpe1.north)$)!0.5!($(pde.south)!0.5!(fpe2.north)$) $) 
{Step (1)};

\draw[<->, dashed, thick] (fpe1.east) -- (fpe2.west)
node[midway, above] {\Cref{lem:many-PDEs_new}};

\node[draw=none, align=center] at ($ ($(fpe1.south)!0.5!(mvsde1.north)$)!0.5!($(fpe2.south)!0.5!(mvsde2.north)$) $) {Step (2)\\
(superposition principle)};

\node[draw=none, align=center] at ($ ($(mvsde1.south)!0.5!(nmc1.north)$)!0.5!($(mvsde2.south)!0.5!(nmc2.north)$) $) {Step (3) \\
(restricted linearized uniqueness)
\\
\Cref{crl:nl_markov}};

\node[draw=none, align=center] at ($ ($(nmc1.south)!0.5!(nmp1.north)$)!0.5!($(nmc2.south)!0.5!(nmp2.north)$) $) {\Cref{def:NMC}};

\node[draw=none, align=center, yshift=-7mm] 
at ($ (nmp1.south)!0.5!(nmp2.south) $) 
{$P^z_1 \circ \pi_t^{-1} = P^z_2 \circ \pi_t^{-1} = u^z(t,x) dx, \quad \forall z \in \R^d, t \geq 0$.};
\end{tikzpicture}}
\captionsetup{font=footnotesize}
\caption{From a PDE to multiple nonlinear Fokker--Planck formulations, to MV-SDEs, and finally to nonlinear Markov processes, illustrating non-uniqueness in the scheme of \Cref{sec:scheme}.}
\label{fig:pde-nfpe-mvsde-nmp}
\end{figure}

\begin{enumerate}[label=\textbf{(\roman*)}, font=\normalfont, leftmargin=15pt, itemindent=25pt,itemsep=1em]
\setcounter{enumi}{2}
\item \textbf{(\Cref{sec:porous_medium}; Porous medium equation).} While a nonlinear Markov process with one-dimensional time marginals given by the Barenblatt solutions \eqref{eq:barenblatt_solution_porous} to the PME \eqref{intro:PME} was already constructed in \cite[Sect. 4.2]{RehmeierRoeckner2025NonlinearMarkov} and it was proven in \cite{R.BGR25} to consist of \emph{strong} solutions to the pure-noise SDE \eqref{eq:DDSDE-PME}, it was left open whether these solutions are unique. We close this gap by providing an affirmative answer via Proposition \ref{prop1}.

Next, by recasting the PME as the nonlinear first-order FPE
$$ \partial_t u = - \nabla \cdot \big( \frac{-\nabla u^m}{u} u\big)$$
and applying our scheme from \Cref{sec:scheme}, we construct in \Cref{prop2} a nonlinear Markov process with Barenblatt one-dimensional time marginals, consisting of the unique solutions to the distribution-dependent ODE
\begin{equation}\label{intro:coeff}
    \begin{dcases}
        dX_t = -\frac{\nabla u^m(t,X_t)}{u(t,X_t)}dt, \\
        \mathcal{L}(X_t) = u(t,x)dx,\,\,\, t >0. 
    \end{dcases}
\end{equation}
We then interpolate between these pure-diffusion and pure-drift cases, i.e.,
we prove uniqueness and the nonlinear Markov property of solutions to 
\begin{equation}\label{eq:beta_MVSDE_PME_intro}
    \begin{dcases}
    dX_t = -(1-\beta) \frac{\nabla u^m(t,X_t)}{u(t,X_t)}dt + \sqrt{2\beta} \, u(t,X_t)^{\frac{m-1}{2}}dW_t, \\ 
    \mathcal{L}(X_t) = u(t,x)dx,\,\,\, t >0,
    \end{dcases}
\end{equation}
(see Proposition \ref{prop_porous_beta}), where $\beta \in (0,\infty)$. Notably, for $\beta = \frac{2m}{m+1}$, this equation turns out to be (formally) equivalent to the pure-noise Stratonovich-SDE \eqref{eq:PME-pure-diff-Stratono} (see \Cref{subsect:PME-Stratono} and Appendix \ref{subsec:theta_interpretations}).

As a final result, we recast the PME as
\begin{equation}
    \partial_t u = \Delta u - \nabla \cdot \bigg(\frac{\nabla u- \nabla u^m}{u} u \bigg),
\end{equation}
in order to construct a nonlinear Markov process with Barenblatt one-dimensional time marginals consisting of the unique solutions to the \emph{additive noise} SDE
\begin{equation}
\begin{dcases}
    dX_t = \frac{\nabla u(t,X_t) - \nabla u^m(t,X_t)}{u(t,X_t)}dt + \sqrt{2} dW_t, \\
    \mathcal{L}(X_t) = u(t,x)dx,\quad t>0,
    \end{dcases}
\end{equation}
for $1<m<2$ (see \Cref{prop:PME-addnoise-MVSDE}).

All these results are obtained within the framework developed in Section \ref{sec:scheme}, i.e., by utilizing \Cref{lem:many-PDEs_new} and \Cref{crl:nl_markov}. One key step towards the application of the latter is to verify a restricted uniqueness condition for the linear PDE \eqref{eq4_new}. For instance, we show that the vector field in \eqref{intro:coeff}, upon choosing $u$ as the Barenblatt solution, is Lipschitz continuous on the support of the latter and that this suffices to derive the required uniqueness result.
\item \textbf{(\Cref{sec:pLaplace}; $p$-Laplace equation).}  In \Cref{subsec:pLaplace_diffusion}, we show that Barenblatt solutions \eqref{eq:barenblatt_solution_plaplace} to the $p$-Laplace equation \eqref{intro:pL} also solve a pure-diffusion FPE
\begin{equation}\label{intro1x}
   \partial_t u = \Delta(a(x,u)u), 
\end{equation}
where $a: \R^d \times \Pscr_* \to \R$ is a \emph{nonlocal} functional of $u$ given in \eqref{eq:a_pure_Ito_diffu_pLaplace_nl}, which is 
defined on the set $\Pscr_*$ \eqref{eq:P_star_pure_Ito_diffu_pLaplace} of absolutely continuous probability measures whose densities are radially symmetric and whose radial profiles possess sufficient Sobolev regularity. By exploiting the radial symmetry of Barenblatt solutions $u^z$, we show that they belong to $\Pscr_*$, satisfy $\nabla \big(a(x,u)u \big) = |\nabla u|^{p-2} \nabla u$ in distributional sense, and that the diffusion coefficient $a$ is non-negative along them. Thus, we prove in \Cref{proposition:pure_diffusion_plaplace} that the associated MV-SDE
    \begin{equation}
        \begin{dcases}
            dX_t = \sqrt{2a(X_t,u(t,\cdot))} \, dW_t, \\
            \mathcal{L}(X_t) = u(t,x)dx,\quad t>0,
        \end{dcases}
    \end{equation}
has unique solutions with Barenblatt time marginals, which form a nonlinear Markov process. By construction, these solutions are \emph{martingales}, which leads to \Cref{prop:intro-martingale}.

In \Cref{subsec:pLaplace_drift}, in contrast to \eqref{intro1x},  we interpret the $p$-Laplace equation as the \emph{first-order} nonlinear FPE
\begin{equation}\label{eq:pLaplace_drift_intro}
        \partial_t u   =  - \nabla \cdot \Big( \frac{-|\nabla u|^{p-2} \nabla u}{u}u \Big)
    \end{equation}
to which we relate the distribution-dependent ODE
\begin{equation}
    \begin{dcases}
        dX_t = -\frac{|\nabla u(t,X_t)|^{p-2} \nabla u(t,X_t)}{u(t,X_t)}dt, \\
        \mathcal{L}(X_t) = u(t,x)dx, \,\,\, \forall t>0.
    \end{dcases}
\end{equation}
Similarly to the PME, we construct unique solutions to this ODE with Barenblatt one-dimensional time marginals and prove their nonlinear Markov property. To us, it was interesting to observe that, upon inserting the Barenblatt solution for $u$, the vector field in the previous ODE is very similar to the one in \eqref{intro:coeff}, see \Cref{rem:similar-ODEs}. Again, interpolation between these pure-diffusion and pure-drift cases is possible in order to construct infinitely many nonlinear Markov processes with Barenblatt one-dimensional time marginals, see \Cref{rem:interpolation-pL}. 

In \Cref{subsec:pLaplace_gen_diffusion}, we indicate that, for any $\theta \in [0,1]$, $u^z$ also solves 
\begin{equation}\label{eq:intro-theta}
        \partial_t u =  \Delta \big(a_\theta (x,u) u\big) -  \nabla \cdot \big((1-\theta)\nabla a_\theta(x,u) u\big),
    \end{equation}
with $a_0(x,u) = |\nabla u(x)|^{p-2}$ and, for $\theta \in ]0,1]$, $a_\theta$ is defined in \eqref{eq:a_generalized_diffusion_pLaplace}. We indicate that the associated MV-SDE is, at least formally, equivalent to the pure-diffusion MV-SDE
    \begin{equation}
    \begin{dcases}
        dX_t = \sqrt{2a_\theta(X_t,u(t,\cdot))} \circ^\theta  dW_t, \\
         \mathcal{L}(X_t) = u(t,x)dx,\quad t>0,
    \end{dcases}  
    \end{equation}   
    where $\circ^\theta$ denotes $\theta$-stochastic integration (see Appendix \ref{subsec:theta_Integration}). For $\theta = 0$ (``fully anticipating'' stochastic integral) and $\theta = \frac 1 2$ (Stratonovich integration), one recovers $p$-Brownian motion from \cite{BarbuRehmeierRockner2024} and a pure-noise Stratonovich-SDE, respectively. This particularly addresses the open question of whether $p$-Brownian motion can be interpreted as the solution to a pure-noise SDE.

   Similarly to \Cref{sec:porous_medium}, these results are obtained by applying \Cref{thm:Markov-construction} and \Cref{lem:extrem-unique-equiv} (more precisely, Corollary \ref{crl:nl_markov}) to the Barenblatt solutions and the respective nonlinear FPEs, i.e., \eqref{intro1x}, \eqref{eq:pLaplace_drift_intro} and \eqref{eq:intro-theta}. In the pure-diffusion case \eqref{intro1x}, the proof of restricted uniqueness for the associated linear PDE \eqref{lin p lap diff} follows from a deep uniqueness result in \cite{Belaribi2012}.

   \begin{figure}
    \centering
    \setlength{\tabcolsep}{4pt}
    \renewcommand{\arraystretch}{1.15}
    \begin{tabular}{>{\centering\arraybackslash}m{0.067\textwidth}
                    >{\centering\arraybackslash}m{0.288\textwidth}
                    >{\centering\arraybackslash}m{0.288\textwidth}
                    >{\centering\arraybackslash}m{0.288\textwidth}}          
        \rule{0pt}{5ex} 
        
        & {\small\textbf{$p$-Laplace equation}}
        & {\small\textbf{Porous medium equation}}
        & {\small\textbf{Heat equation}} \\
        {\footnotesize$\beta = 0.0$}
        & \includegraphics[width=\linewidth]{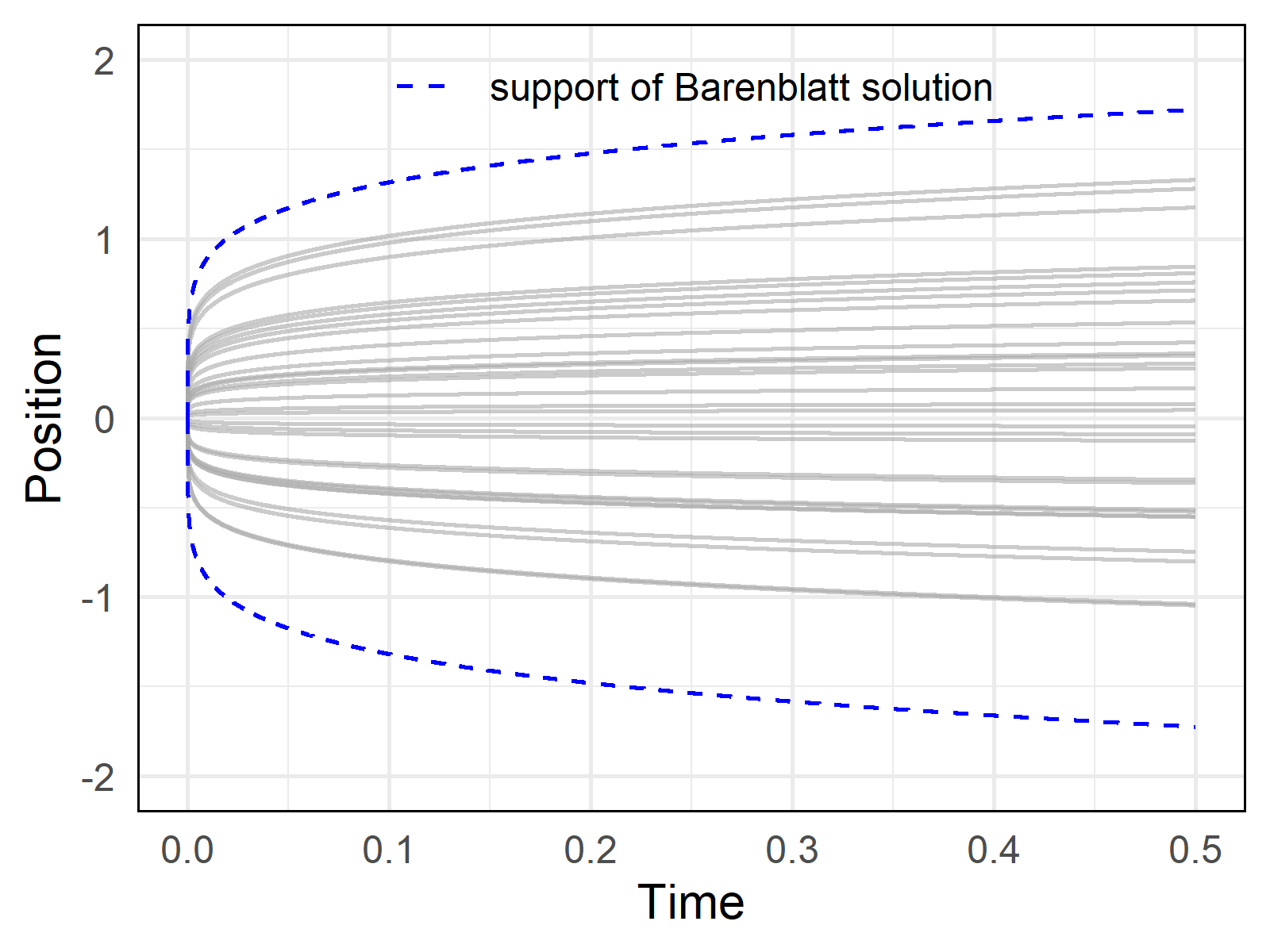}
        & \includegraphics[width=\linewidth]{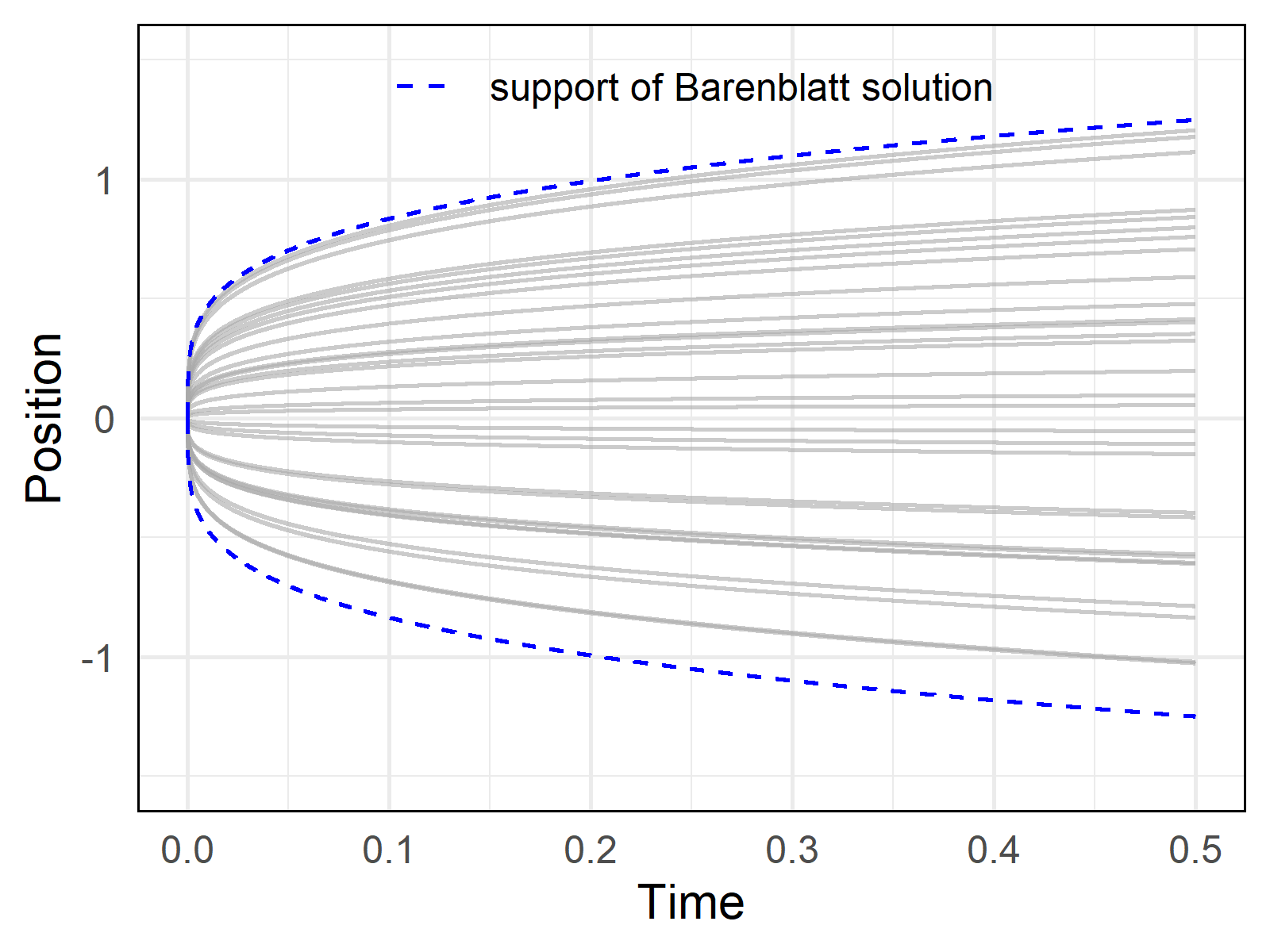}
        & \includegraphics[width=\linewidth]{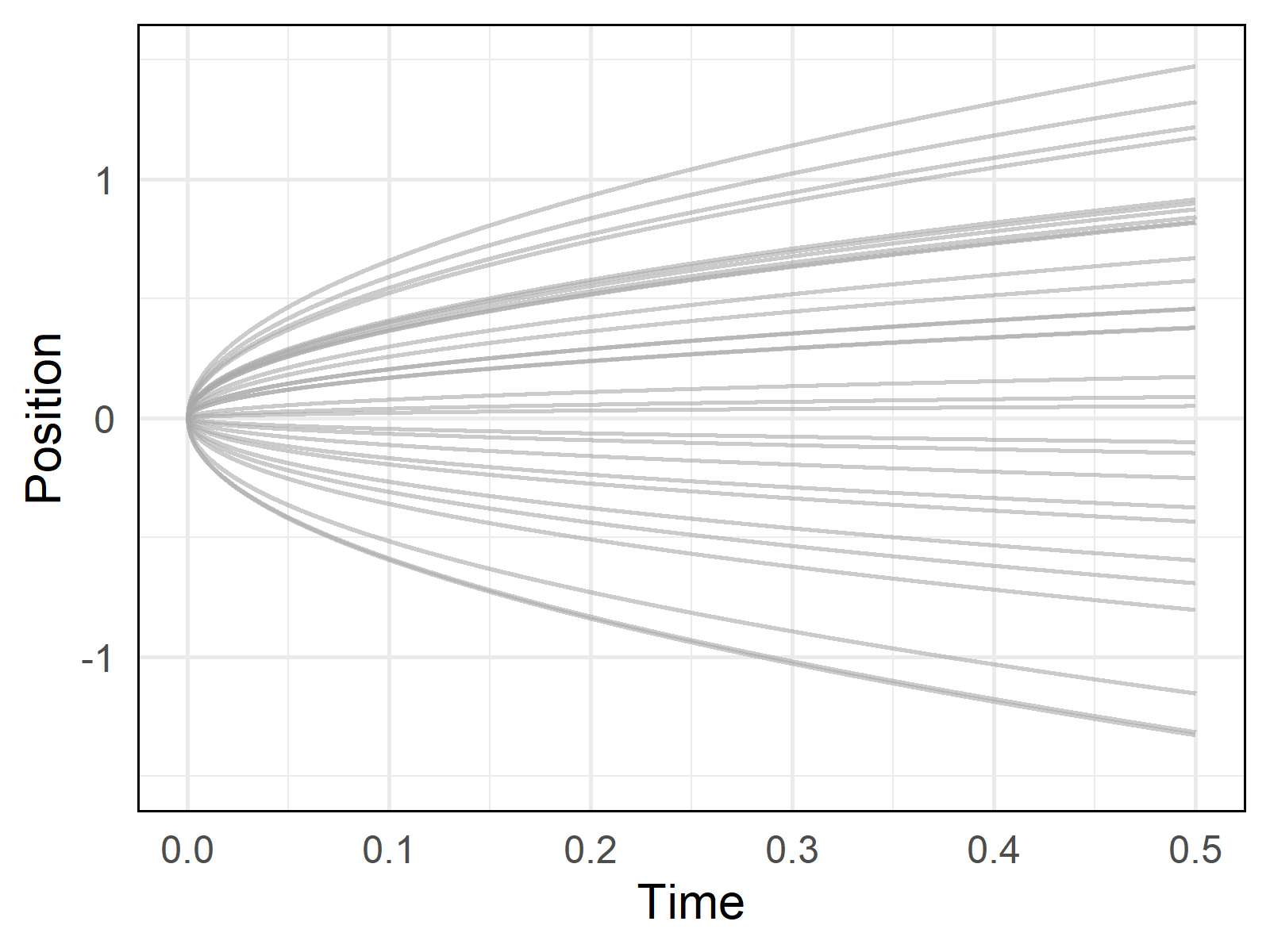} \\

        {\footnotesize$\beta = 0.1$}
        & \includegraphics[width=\linewidth]{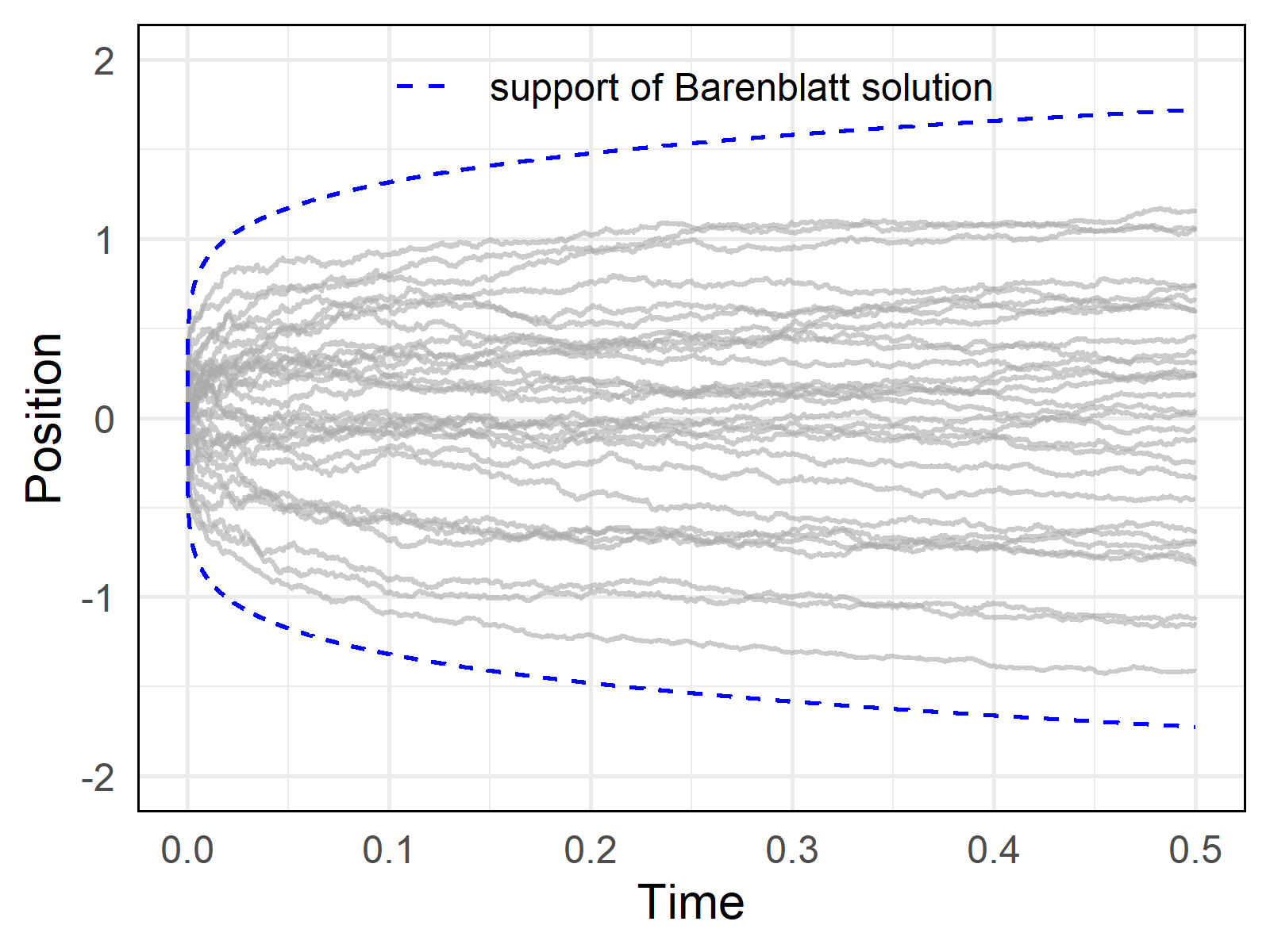}
        & \includegraphics[width=\linewidth]{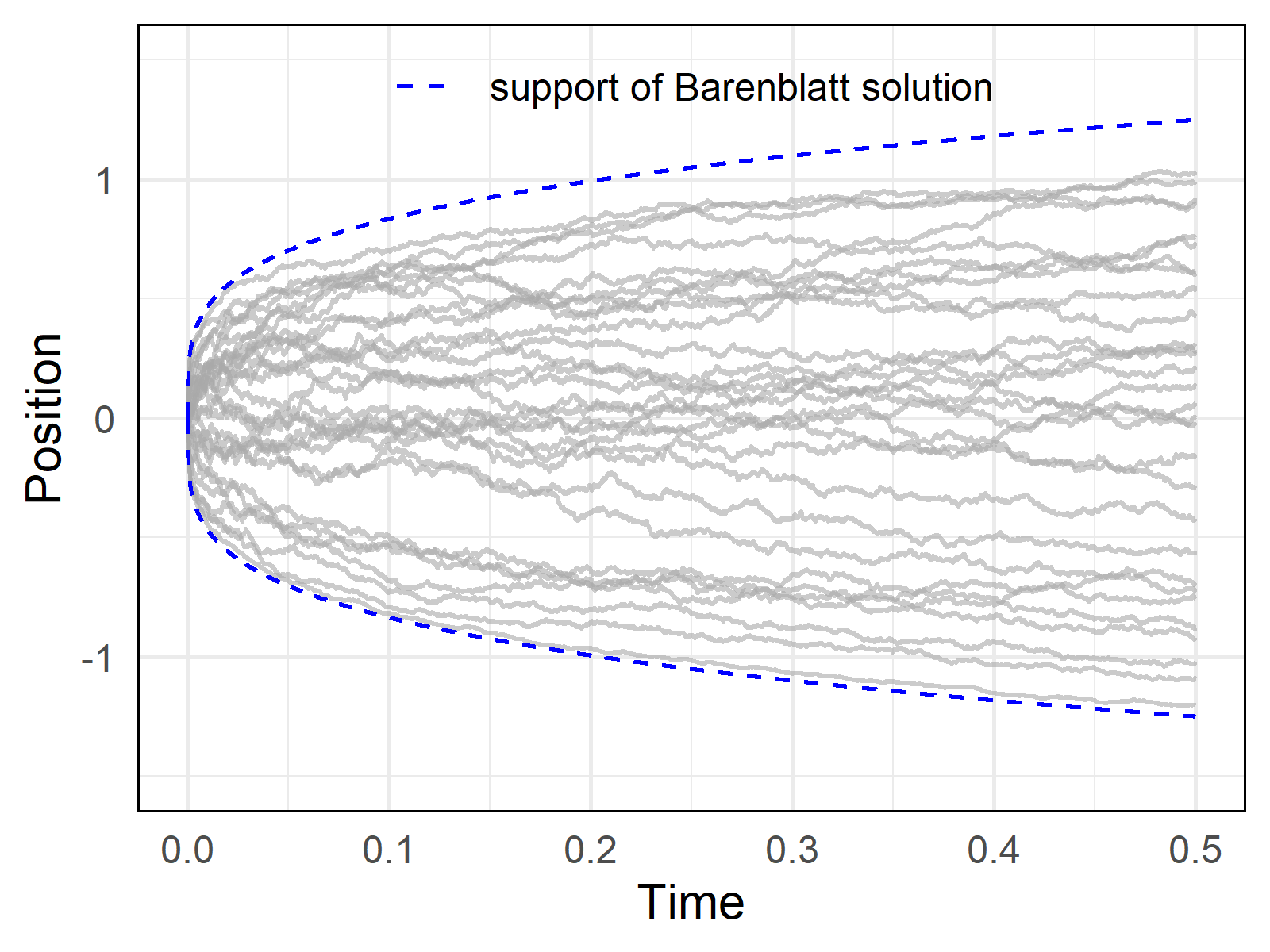}
        & \includegraphics[width=\linewidth]{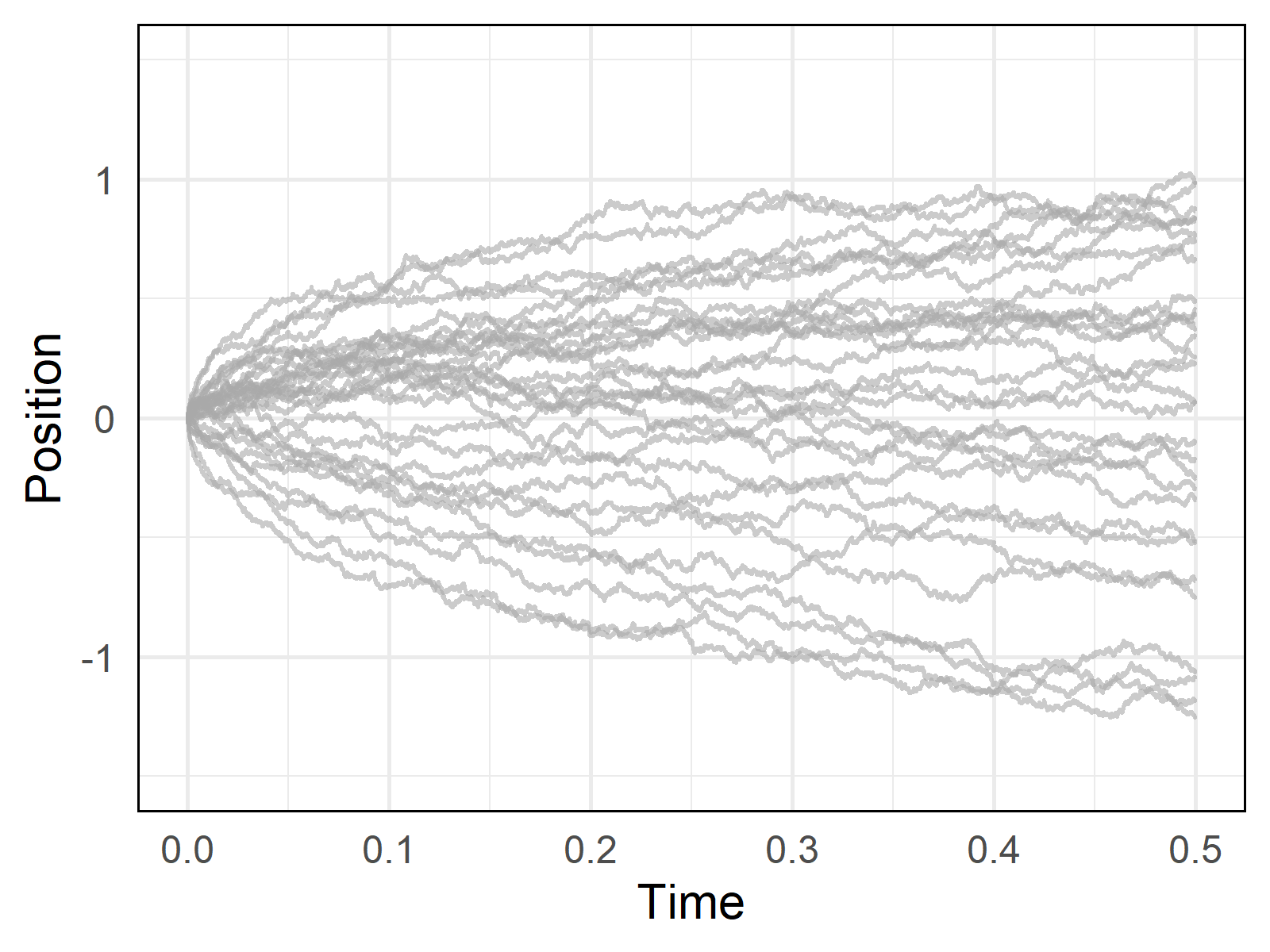} \\

        {\footnotesize$\beta = 1.0$}
        & \includegraphics[width=\linewidth]{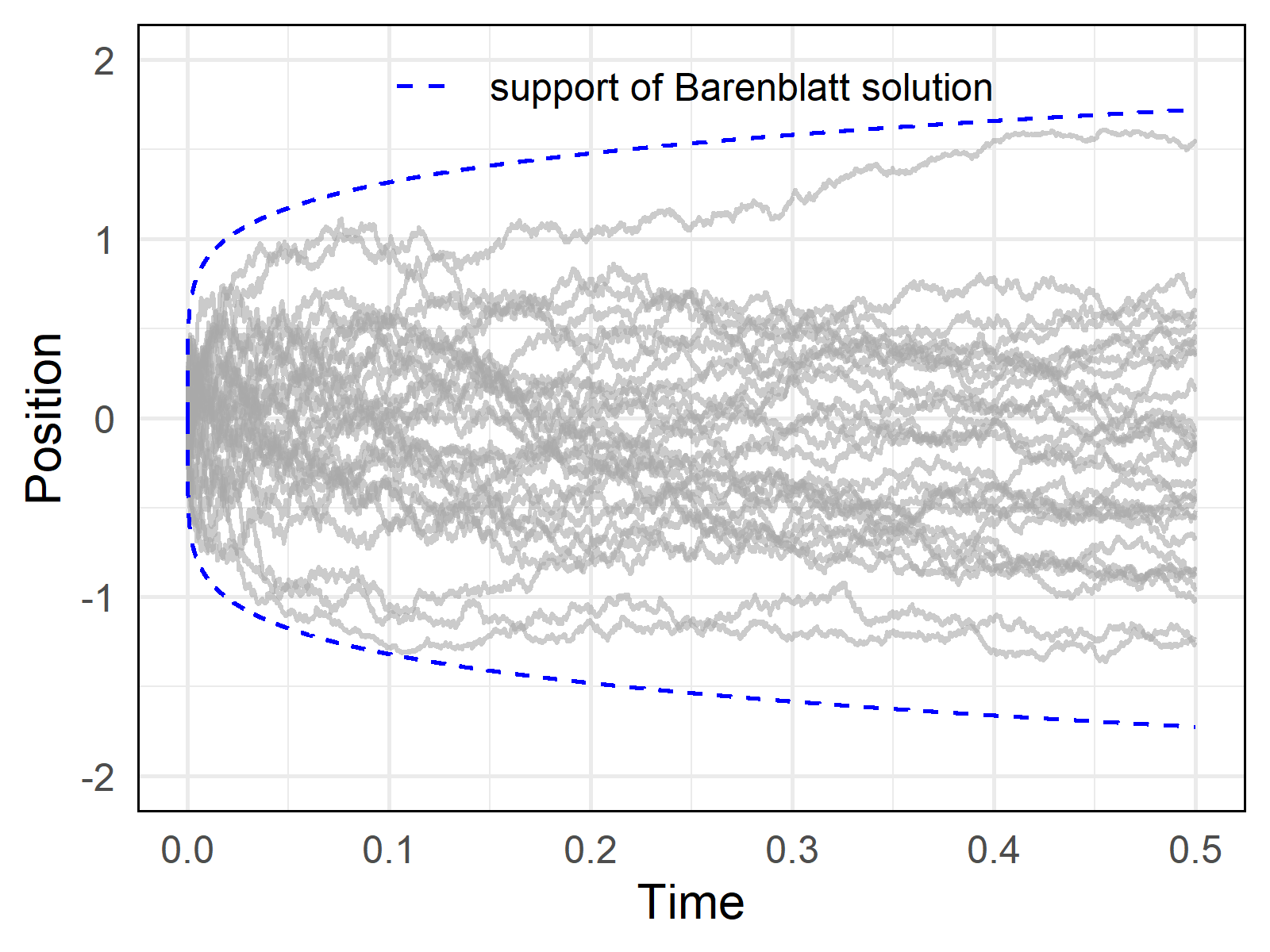}
        & \includegraphics[width=\linewidth]{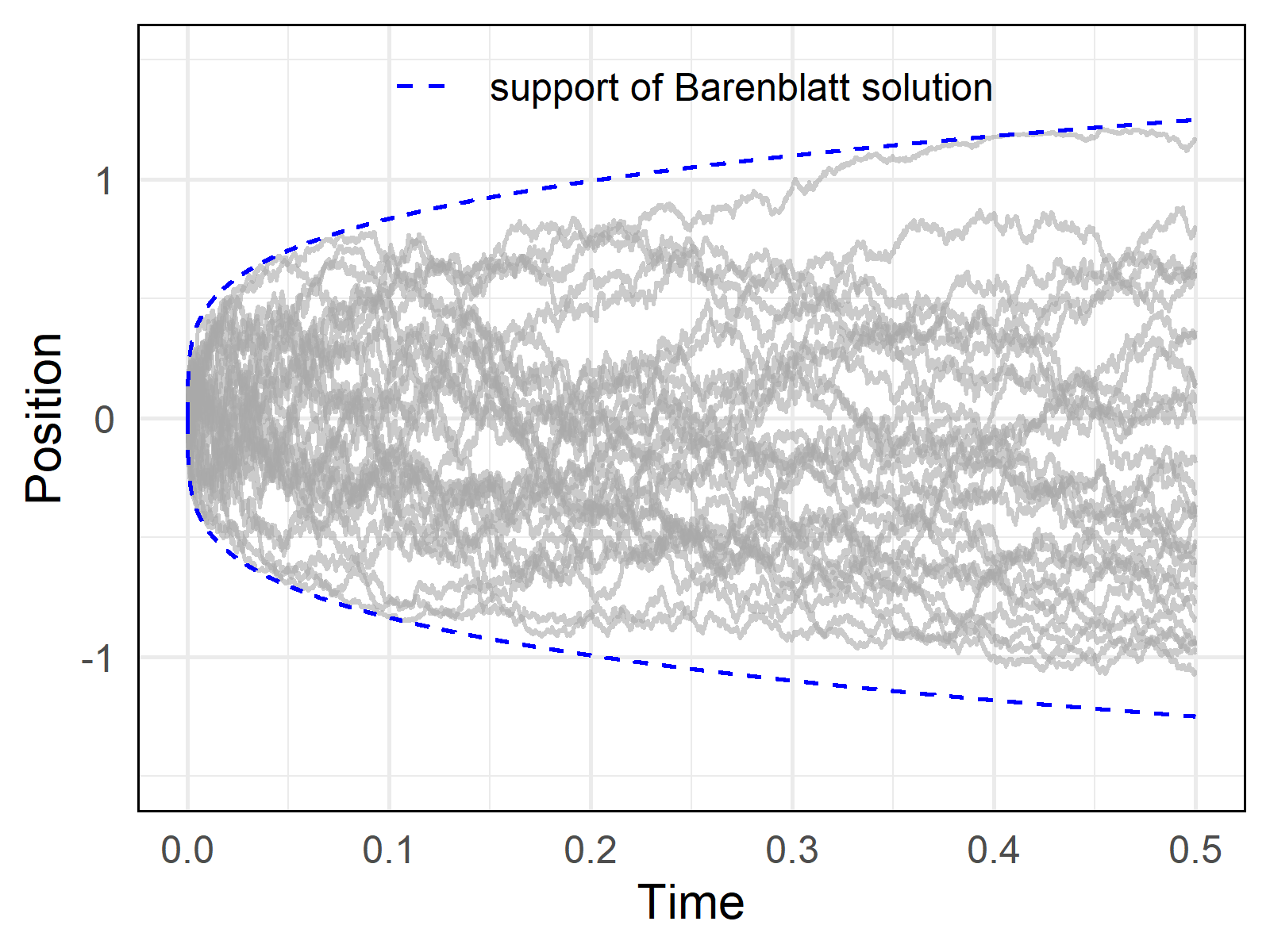}
        & \includegraphics[width=\linewidth]{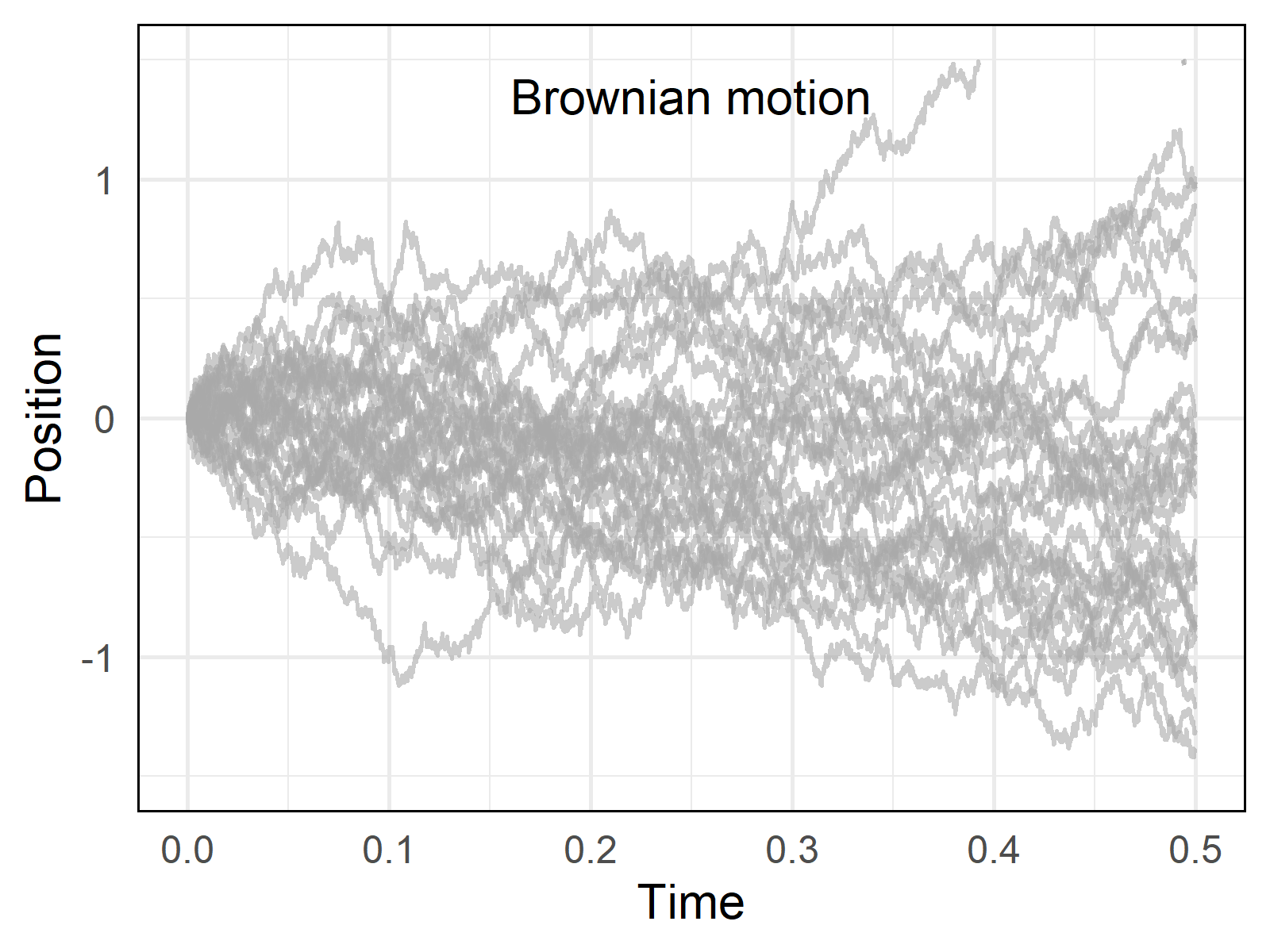} \\

        {\footnotesize$\beta = 1.5$}
        & \includegraphics[width=\linewidth]{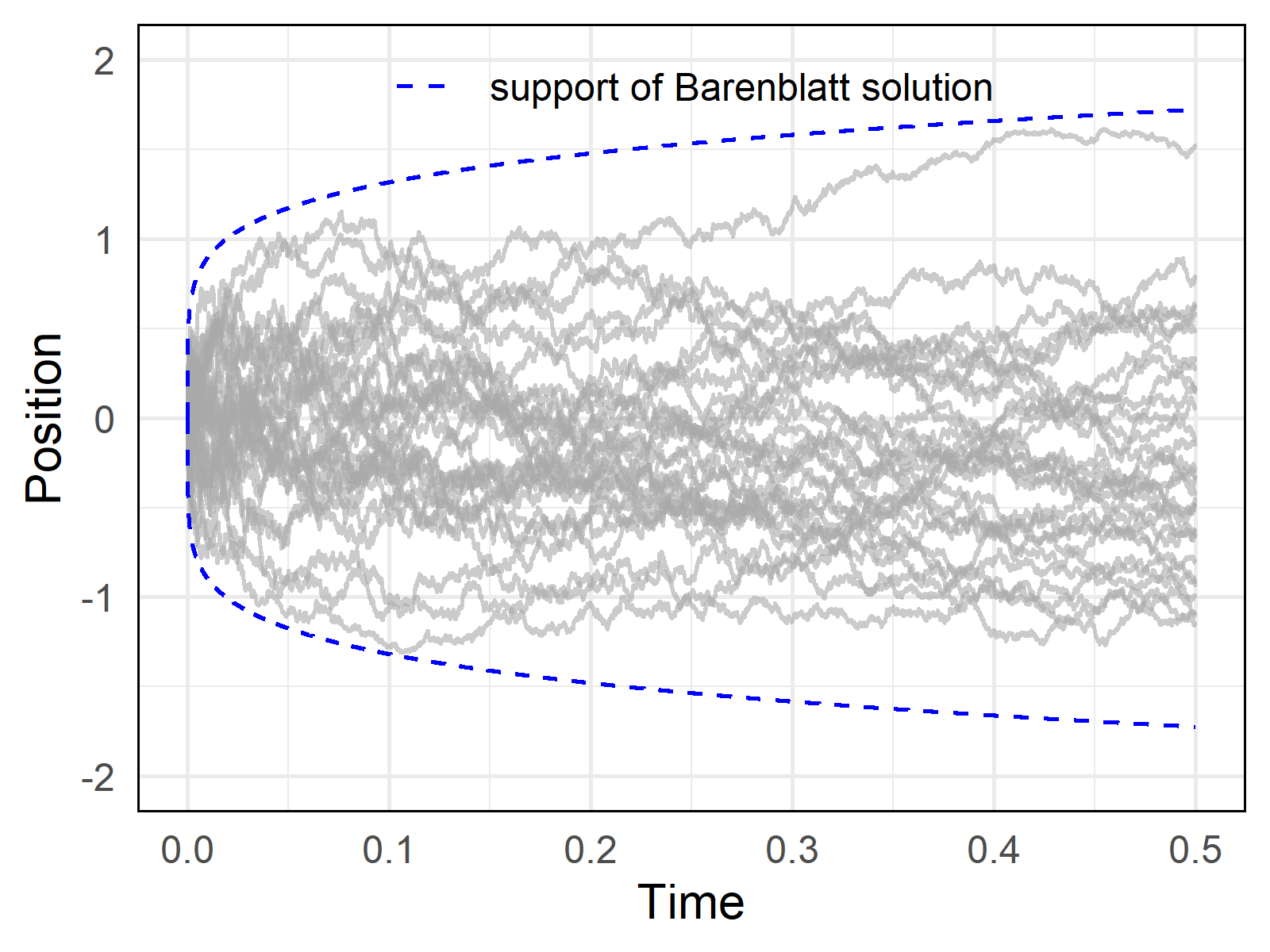}
        & \includegraphics[width=\linewidth]{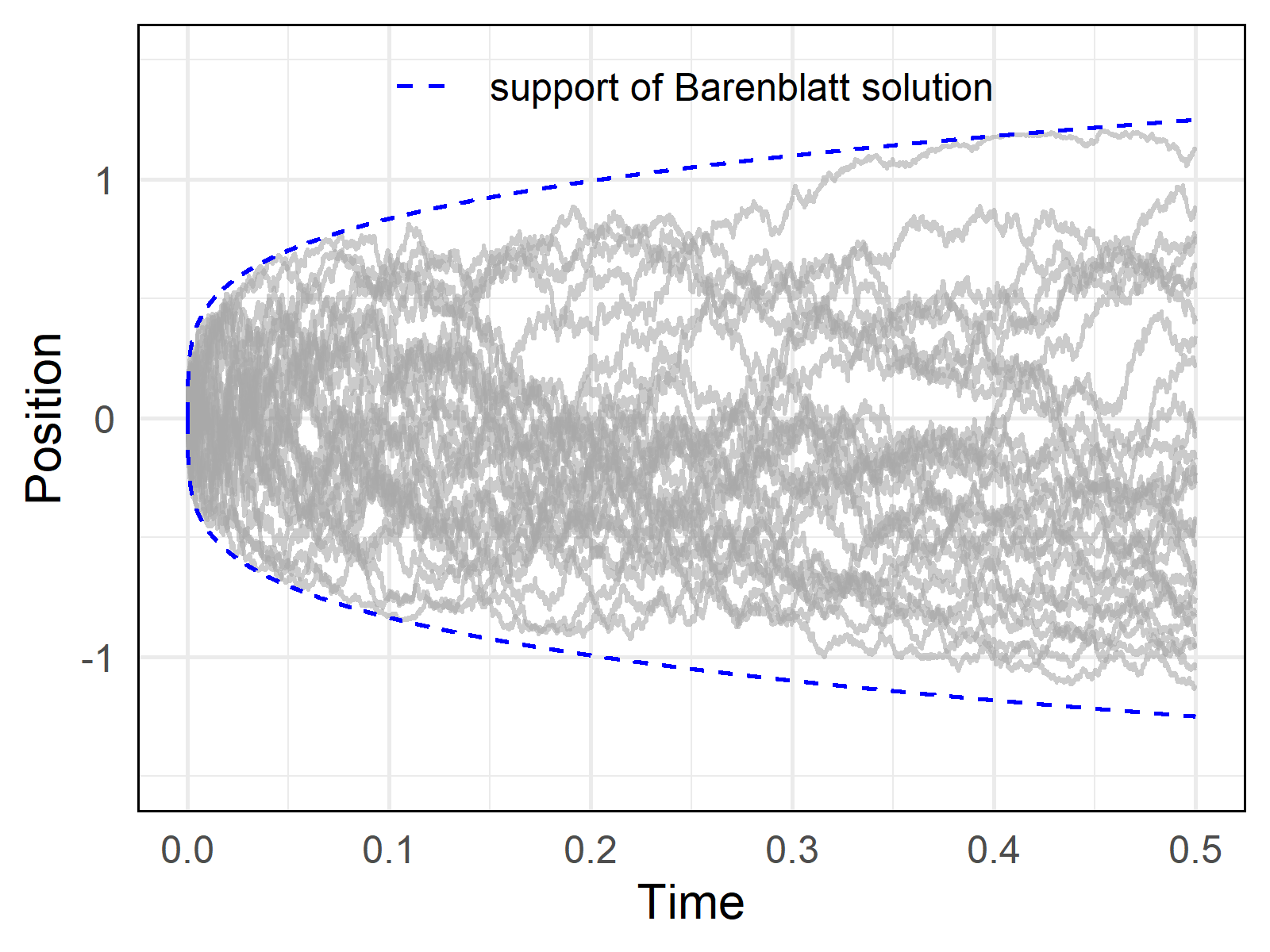}
        & \includegraphics[width=\linewidth]{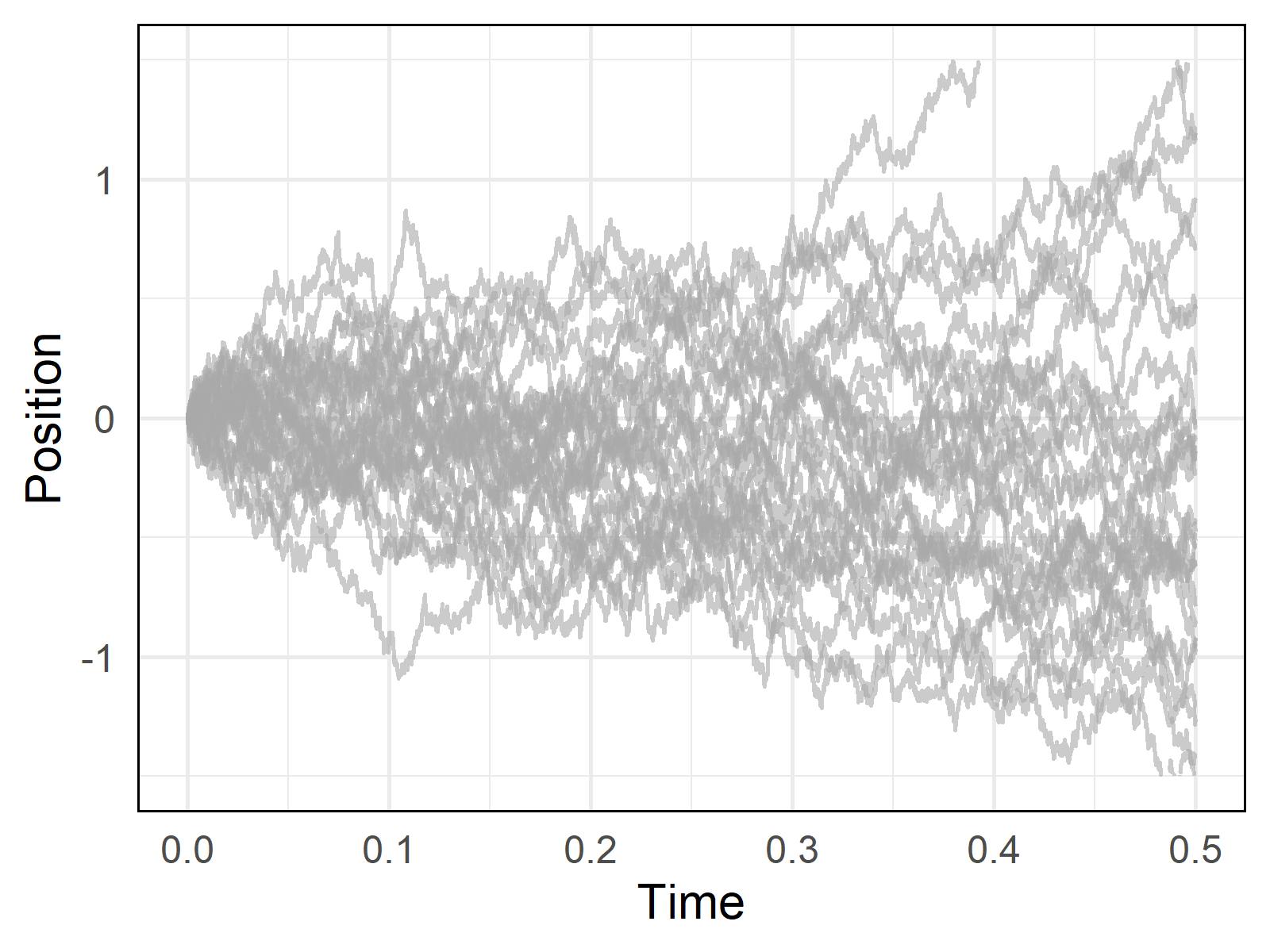} \\
    \end{tabular}

    \captionsetup{font=footnotesize}
    \caption{Sample path simulations of MV-SDEs associated with different nonlinear FP-interpretations of a given PDE, interpolating between the pure-drift ($\beta=0$) and pure-It\^{o}-diffusion ($\beta=1$) cases. The PDEs are the $p$-Laplace equation ($p=4$), the porous medium equation ($m=3$), and the heat equation (all with initial condition $z=0$ on the real line). The dashed blue curves indicate the support of the Barenblatt solutions in the porous medium and $p$-Laplace cases. For the heat equation, $\beta = 1$ recovers standard Brownian motion. For the $p$-Laplace equation, $\beta = 1$ corresponds to the martingale representation of the Barenblatt solution developed in \Cref{subsec:pLaplace_diffusion}. For the porous medium equation, $\beta = 1.5$ corresponds to the pure-Stratonovich-diffusion interpretation discussed in \Cref{subsect:PME-Stratono}. The time interval is $[0,T]$ with $T=0.5$, discretized with step size $\Delta t = 10^{-4}$, and $N=30$ sample paths are simulated via the Euler--Maruyama scheme. In the case of singular coefficients at $t=0$, the position at time $\Delta t$ is sampled from the distribution $u^z(\Delta t,x)\,dx$. For more details and plots, see Appendix \ref{Apndx:num_sim}.}
    \label{fig:simulations-grid}
\end{figure}

\item \textbf{(\Cref{sec:heat_equation}; Heat equation).} In this section, we implement our scheme for the heat equation. First, to shed light on our new notion of nonlinear Markov core, we show by an explicit computation that Brownian motion is a nonlinear Markov core for the standard heat kernel (Proposition \ref{lemaux}). In Section \ref{subsec:heat_equation_drift}, we recast the heat equation as the first-order nonlinear FPE \eqref{prop:Mc-heateq} and explicitly construct the unique nonlinear Markov process associated with the corresponding distribution-dependent ODE by separation of variables and a direct verification of the nonlinear Markov property (Proposition \ref{prop:Mc-heateq}). As a next result, we solve the interpolated MV-SDE
\begin{equation}\label{eq:beta_MVSDE_HE_intro}
    \begin{dcases}
        dX_t =-(1-\beta)\frac{\nabla u(t, X_t)}{2u(t, X_t)} dt + \sqrt{\beta}dW_t, \\
         \mathcal{L}(X_t) = u(t,x)dx,\,\,\, t >0,
    \end{dcases}
\end{equation}
for all $\beta \in [0,\infty)$ and construct the associated nonlinear Markov processes with heat kernel time marginals. This demonstrates that even for solutions to standard linear PDEs, a rich structure of associated nonlinear Markov processes emerges from our scheme. This is underlined further by Section \ref{subsect:further}, where we show that our scheme allows us to obtain new well-posedness results for MV-SDEs with coefficients of superexponential growth (see \Cref{prop9}).  \Cref{fig:simulations-grid} shows numerical simulations of the $\beta$-interpolated MV-SDE associated with the heat equation, PME, and the $p$-Laplace equation. 

\item \textbf{(\Cref{sec:SPDE}; Identifying MV-SDEs with geometries)} In this section, we provide a geometric viewpoint on the previously discussed MV-SDEs related to the PME. Recall that for the heat equation and the PME, the corresponding PDEs model the average behavior of a corresponding particle system (see, e.g., \cite{Sznitman1991}). In this context, the PDE provides a law of large numbers for the particle system. With this viewpoint, it is not surprising that different particle systems (with different mean field limits) yield the same law of large numbers, i.e., the same PDE solution. Beyond the macroscopic PDE scale for such particle systems, one can study mesoscopic fluctuations in a corresponding ``central limit theorem regime.'' This approach naturally leads to the (generalized) Dean--Kawasaki equation \cite{Dean1996}, \cite[Section 1.1]{Cornalba2019}, which we derive for particle systems associated to different McKean--Vlasov interpretations of the PME in Section \ref{sec:SPDE}. This way, we established that different McKean--Vlasov interpretations imply different fluctuations on top of the PME (as expressed by their respective generalized Dean--Kawasaki equations).

In a crucial second step, we recall that besides particle systems, there is the alternative viewpoint of the heat equation and PME as a gradient flow on an infinite-dimensional manifold, pioneered by Otto \cite{Otto2001}. Here, a canonical way of introducing mesoscopic fluctuations (and, thus, deriving an associated generalized Dean--Kawasaki equation) consists in perturbing the PDE by ``Brownian motion on the tangent plane.'' This viewpoint is contingent on a choice of geometry and, in particular, different geometries imply different associated generalized Dean--Kawasaki equations. We explicitly provide geometric interpretations for the PME and derive the associated generalized Dean--Kawasaki equation in \Cref{sec:DK-2}.

Combining both viewpoints, we identify different geometries for the PME with different associated McKean--Vlasov interpretations, if the corresponding generalized Dean--Kawasaki equations --- derived by means of ``Brownian motion on the tangent plane'' and mesoscopic fluctuations of an associated particle system, respectively --- coincide. In this context, the different McKean--Vlasov interpretations provided in Section \ref{sec:porous_medium} are not ad hoc objects but rather manifestations of corresponding geometric viewpoints on the PME. 
\end{enumerate}

\subsubsection*{Organization of the paper}
The rest of the paper is structured as follows.
\textbf{\Cref{sec:FP-prelims}} contains preliminaries on FPEs and MV-SDEs. In \textbf{\Cref{sec:nonunique-Markov}}, we recall the notion of (nonlinear) Markov processes, introduce nonlinear Markov cores, and prove \Cref{intro:thm-2D} above.
In \textbf{\Cref{sec:scheme}}, we set out our scheme, and in \textbf{Sections \ref{sec:heat_equation}, \ref{sec:porous_medium}, \ref{sec:pLaplace},} we apply it to the heat equation, porous medium equation, and $p$-Laplace equation, respectively. \textbf{\Cref{sec:SPDE}} contains the formal discussion on identifying MV-SDEs with geometries on the space of probability measures. 
In the \textbf{Appendix}, we discuss further pure-diffusion interpretations and present supplementary plots as well as details on our numerical simulations.

\subsubsection*{Notation} We set $\R_+ := [0,\infty)$, denote by $A^T$ the transpose of a matrix $A$, and by $1_d$ the $d\times d$-unit matrix. For $x,y\in \R^d$ and $r>0$, $x\cdot y$ is the standard Euclidean inner product and $B_r(x)$ the Euclidean open ball with radius $r$ around $x$. 

\textit{Function spaces.} For topological spaces $U,V$, we write $C(U;V)$ for the space of continuous functions from $U$ to $V$, and abbreviate $\Cscr := C(\R_+;\R^d)$, endowed with the topology of locally uniform convergence, i.e., $\Fscr:= \Bscr(\Cscr) = \sigma(\pi_t, t \geq 0)$, where $\pi_t$ denotes the canonical projection $\pi_t: \Cscr \to \R^d$, $\pi_t(w) = w(t)$. We let $\Fscr_t := \sigma(\pi_s, 0 \leq s \leq t)$ and $\Pi_t: \Cscr \to \Cscr$, $\Pi_t(w) = w(t+\cdot)$. 

For $U \subseteq \R^d$ open and $k,m \in \N_0$, we denote by $C^k(U;\R^m)$ ($C^k_b(U;\R^m)$ and $C^k_c(U;\R^m)$) the usual spaces of $k$-times (bounded and compactly supported, respectively) continuously differentiable $\varphi: U\to \R^m$. For $k=0$ and $m=1$, we write $C(U;\R^m)$ and $C^k(U)$, respectively, and $C^\infty(U;\R^m) = \cap_{k\in \N}C^k(U;\R^m)$ (likewise for $C_b$ and $C_c$).

For $p\in [1,\infty]$, $U \subseteq \R^d$ and $\mu$ a measure on $\Bscr(U)$, let $L^p(U;\mu)$ denote the usual space of $p$-fold $\mu$-integrable functions $\varphi: U\to \R$, and $L^1_{\textup{loc}}(U;\mu)$ the corresponding local spaces; for $\mu = dx$, we write $L^p(U)$. $W^{1,1}(\R^d)$ ($W^{1,1}_{\textup{loc}}(\R^d)$) are the usual spaces of (locally) integrable maps $\varphi: \R^d \to \R$ having a (locally) integrable weak derivative, equipped with the usual norm $||\cdot||_{1,1}$. For $U\subseteq \R^m$, we write $L^1(U;W^{1,1}(\R^d))$ for the space of those $\varphi: U \to W^{1,1}(\R^d)$ with $\int_U ||\varphi(t)||_{1,1} dt < \infty$, and similarly with $L^1$ and $W^{1,1}$ replaced by $L^1_{\textup{loc}}$ and $W^{1,1}_{\textup{loc}}$.

\textit{Measures.} On a topological space $(T,\tau)$, $\Pscr(T)$ is the set of probability measures on the Borel $\sigma$-algebra $\Bscr(T)$ of $T$, endowed with the usual topology of weak convergence of probability measures.; we write $\Pscr = \Pscr(\R^d)$, and $\Pscr_{\textup{ac}}\subseteq \Pscr$ for the set of absolutely continuous measures with respect to Lebesgue measure $dx$. For $z\in \R^d$, $\delta_z$ denotes the usual Dirac measure at $z$, and $\Nscr(z,A)$ is the normal distribution with mean vector $z$ and covariance matrix $A$. For a random variable $X: \Omega \to S$ on a probability space $(\Omega, \Ascr, \mathbb{P})$ with values in a measurable space $S$, we write $\mathcal{L}(X)$ for its distribution, i.e., $\mathcal{L}(X) = \mathbb{P}\circ X^{-1}$ (the latter denotes the image measure of $X$ with respect to $\mathbb{P}$). If $S = \Cscr$, $\mathcal{L}(X) \in \Pscr(\Cscr)$ is the path law of the stochastic process $X$. For $P \in \Pscr(\Cscr)$, we shortly write $P_t := P\circ \pi_t^{-1}$.

Recall that for a Borel probability measure $P$ on a Polish space $T$ and a measurable map $g: \Omega \to S$ to another Polish space $S$, the \emph{disintegration family} of $P$ with respect to $g$ is the $P \circ g^{-1}$-a.s. unique family $\{Q_\omega\}_{\omega \in \Omega} \subseteq \Pscr(T)$ such that $\omega \mapsto Q_\omega(A)$ is measurable and $P(A) = \int_{\Omega} Q_\omega(A)\,d(P\circ g^{-1})$ for all $A\in \Bscr(T)$.

For a stochastic processes $M$ and $N$, the quadratic (co-)variation of $M$ (and $N$) is denoted by $[M]_t$ ($[M,N]_t$).


\subsubsection*{Acknowledgments}
We would like to thank Matthias Erbar for fruitful discussions, in particular on \Cref{sec:SPDE}; Thomas Koch for the inspiration for \Cref{thm:NMP_2D}; and Michael R\"{o}ckner for several insightful discussions on this research program.
F.B. gratefully acknowledges the financial support of the Deutsche Forschungsgemeinschaft (DFG, German Research Foundation) – SFB 1283/2 2021 – 317210226 (Project B7). E.A. acknowledges the partial funding by Deutsche Forschungsgemeinschaft (DFG) – Project-ID 318763901 – SFB 1294.

\section{Preliminaries on Fokker--Planck equations and McKean--Vlasov SDEs}\label{sec:FP-prelims}
We briefly recall the notions of solution to the classes of Fokker--Planck and McKean--Vlasov equations studied in this paper, as well as their relations.

\subsectionnotoc{Nonlinear Fokker--Planck equations} Let $\Pscr_* \subseteq \Pscr$, and $a_{ij}, b_i: \R^d \times \Pscr_* \to \R$, $1\leq i,j \leq d$, be such that $a \coloneqq (a_{ij})_{i,j \leq d}$ takes values in the space of symmetric $d\times d$-matrices. Consider the nonlinear Fokker--Planck equation (FPE)
\begin{equation}\label{eq:FPE}
    \partial_t \mu_t = \partial_{ij}\big( a_{ij}(x,\mu_t)\mu_t \big) - \partial_{i} \big( b_i(x,\mu_t)\mu_t\big)
\end{equation}
(here and throughout, we use Einstein summation convention).
For $\mu \in \Pscr_*$ and $\varphi \in C^2(\R^d)$, we write $L_{a(\mu),b(\mu)}\varphi(x) := a_{ij}(x,\mu)\partial_{ij}\varphi (x) + b_i(x,\mu) \partial_i \varphi (x)$ .
\begin{definition}\label{def:FPE-sol}
    A \emph{distributional solution} (short: \emph{solution}) to \eqref{eq:FPE} with initial condition $\zeta \in \Pscr$ is a weakly continuous curve of probability measures $\mu = (\mu_t)_{t> 0} \subseteq \Pscr$ such that $\mu_t \in \Pscr_*$ for all $t\in (0,\infty)$, $a(x,\mu_t)$ is non-negative definite for all $(t,x) \in (0,\infty)\times \R^d$, $(t,x)\mapsto a_{ij}(x,\mu_t)$ and $ (t,x)\mapsto b_i(x,\mu_t)$, $1\leq i,j \leq d$, belong to $L^1_{\textup{loc}}([0,\infty)\times \R^d; \mu_t dt)$, and for all $t>0$,
    \begin{equation}\label{eq:FPE-sol}
        \int_{\R^d}\varphi \,d\mu_t - \int_{\R^d}\varphi\, d\zeta = \int_0^t \int_{\R^d} L_{a(\mu_s),b(\mu_s)}\varphi\, d\mu_s ds
    \end{equation}
    holds for all $\varphi \in C^2_c(\R^d)$. $\zeta$ is called \emph{initial condition} or \emph{initial datum} of $\mu$.
\end{definition}
Note that $\mu_t \to \zeta$ weakly as $t\to 0$. Indeed, \eqref{eq:FPE-sol} shows $\mu_t \to \zeta$ vaguely, thus weakly, as $\{\mu_t, t >0, \zeta \} \subseteq \Pscr$. Thus, throughout we also refer to $(\mu_t)_{t\geq 0}$ as the solution instead of $(\mu_t)_{t>0}$, with $\mu_0 := \zeta$.

We often require the following stronger integrability condition for solutions $\mu$ to \eqref{eq:FPE}:
\begin{equation}\label{eq:L1-int-SP}
    \int_0^T \int_{\R^d}|a_{ij}(x,\mu_t)| + |b_i(x,\mu_t)| \,d\mu_t(x) dt < \infty, \quad \forall i,j \leq d \text{ and }T>0,
\end{equation}
which holds in particular when $a_{ij}$ and $b_i$ are bounded on $\R^d \times \Pscr_*$.
\\

We also encounter \emph{linear} FPEs with \emph{time-dependent} coefficients, i.e.
\begin{equation}\label{eq:lFPE}
    \partial_t \mu_t = \partial_{ij}\big( a_{ij}(t,x)\mu_t \big) - \partial_{i} \big( b_i(t,x)\mu_t\big)
\end{equation}
for Borel coefficients $a_{ij}, b_i: (0,\infty)\times \R^d \to \R$ such that $a \coloneqq (a_{ij})_{i,j \leq d}$ takes values in the space of symmetric non-negative definite $d\times d$-matrices. We recall their usual notion of distributional solutions (which we could have included as a special case of \Cref{def:FPE-sol}, had we considered time-dependent coefficients therein). Similarly as above, we write 
$L_{a(t),b(t)}\varphi (x) := a_{ij}(t,x)\partial_{ij}\varphi (x) + b_i(t,x)\partial_i \varphi (x)$.
\begin{definition}\label{def:lFPE-sol}
    A \textit{distributional solution} (short: \textit{solution}) to \eqref{eq:lFPE} with initial condition $\zeta \in \Pscr$ is a weakly continuous curve $\nu = (\nu_t)_{t> 0} \subseteq \Pscr$ such that $(t,x) \mapsto a_{ij}(t,x)$ and $(t,x) \mapsto b_i(t,x),$ $1\leq i,j \leq d,$ belong to $ L^1_{\textup{loc}}([0,\infty)\times \R^d ;\nu_t dt)$, and for all $t >0$,
   \begin{equation}\label{eq:lFPE-sol}
        \int_{\R^d}\varphi \,d\nu_t - \int_{\R^d}\varphi\, d\zeta = \int_0^t \int_{\R^d} L_{a(s),b(s)}\varphi\, d\nu_s ds
    \end{equation}
    holds for all $\varphi \in C^2_c(\R^d)$. Again, it follows that $\nu_t \to \zeta$ weakly as $t\to 0$, and we use $(\nu_t)_{t>0}$ and $(\nu_t)_{t\geq 0}$ (with $\nu_0 := \zeta$) interchangeably.
\end{definition}

An important class of equations related to \eqref{eq:FPE} are its corresponding \emph{linearized FPEs}. Precisely, let $ (\mu_t)_{t> 0}$ be any curve in $\Pscr_*$ such that $(t,x) \mapsto a_{ij}(x,\mu_t)$ and $(t,x)\mapsto b_i(x,\mu_t)$ are Borel measurable, and consider
\begin{equation}\label{eq:linFPE}\tag{$\mu-\text{linFPE}$}
    \partial_t \nu_t = \partial_{ij}\big( a_{ij}(x,\mu_t) \nu_t\big) - \partial_i \big(b_i(x,\mu_t)\nu_t\big),
\end{equation}
i.e. the \emph{linear} FPE with coefficients $(t,x)\mapsto a_{ij}(x,\mu_t)$, $(t,x)\mapsto b_i(x,\mu_t)$.
Solutions $\nu$ to \eqref{eq:lFPE-sol} (in the sense of \Cref{def:lFPE-sol}) need not be related to $\mu$ and need not belong to $\Pscr_*$. However, we emphasize the following obvious, yet important observation, which follows immediately from Definitions \ref{def:FPE-sol} and \ref{def:lFPE-sol}.
\begin{lemma}
Let $\mu = (\mu_t)_{t> 0}$ be a solution to \eqref{eq:FPE} with initial condition $\mu_0$. Then $\mu$ is also a solution to \eqref{eq:linFPE} with initial condition $\mu_0$.
\end{lemma}

\vspace{0.5em}
\noindent\textbf{The generalized Nemytskii-case.} An important special case of \eqref{eq:FPE} arises via the \emph{generalized Nemytskii-case}, i.e., when $a$ and $b$ depend on their measure argument pointwise via the measure's density or, as will be important for us, pointwise on a function of the latter. Then $\Pscr_* \subseteq \Pscr_{\textup{ac}}$ and
\begin{equation}\label{eq:Nemytskii-coeff}
    a_{ij}(x,\mu) = A_{ij}\big(x,(\Gamma_{ij} u)(x)\big), \quad b_i(x,\mu) = B_i\big(x,(\Lambda_i u)(x)\big)
\end{equation}
for Borel maps $A_{ij}, B_i: \R^d \times \R \to \R$, 
where $\mu = u(x) dx$, i.e. the density of $\mu$ with respect to Lebesgue measure is $u$, which we identify with $\mu$, and $\Gamma_{ij}, \Lambda_i$ map $u \in \Pscr_*$ to functions $\Gamma_{ij} u, \Lambda_i u: \R^d \to \R$, for instance
$\Gamma_{ij} u = u$
or $\Lambda_i u = |\nabla u|^{\gamma_1}|\partial_i u|^{\gamma_2}$ for $\gamma_1,\gamma_2 \in \R$. $\Gamma_{ij}$ and $\Lambda_i$ need not be linear, but we still write $\Gamma_{ij} u$ and $\Lambda_i u$ instead of $\Gamma_{ij}(u)$ and $\Lambda_i (u)$ to shorten notation. Then, rewriting \eqref{eq:FPE} as an equation for the densities $u(t,\cdot)$ of $\mu_t$ yields the nonlinear PDE
\begin{equation}\label{eq:nlPDE}
    \partial_t u(t,x) = \partial_{ij}\big( A_{ij}\big(x,(\Gamma_{ij} u(t,\cdot))(x)\big)u(t,x)\big) - \partial_{i} \big( B_i\big(x,(\Lambda_i u(t,\cdot))(x)\big)u(t,x)\big).
\end{equation}
Usually, the case $\Gamma_{ij}u = u = \Lambda_i u$ is called \emph{Nemytskii-case} in the literature. For brevity, we henceforth also call \eqref{eq:Nemytskii-coeff} \emph{Nemytskii} instead of \emph{generalized Nemytskii}.
\begin{remark}\label{rem:PDE-FPE-equiv}
\begin{enumerate}[label=\textbf{(\roman*)}]
    \item  In the case of \eqref{eq:Nemytskii-coeff}, a nonnegative Borel function $u: (0,\infty)\times \R^d \to \R_+$ with $\int_{\R^d}u(t,x)dx = 1$ for all $t > 0$ and $u(t,x)dx \to \zeta \in \Pscr$ weakly as $t \to 0$ such that $t \mapsto \mu_t := u(t,x)dx$ is weakly continuous is a weak solution to \eqref{eq:nlPDE} (in usual PDE-sense) if and only if $(\mu_t)_{t > 0}$ is a distributional solution to \eqref{eq:FPE} with initial condition $ \zeta$. 
    \item In \eqref{eq:Nemytskii-coeff}, $a_{ij}$ and $b_i$ are irregular in their measure-variable. Indeed, the map $u(x)dx \mapsto u(x_0)$, mapping an absolutely continuous probability measure to its density (or to a function of the latter) evaluated at a point $x_0\in \R^d$ is not continuous with respect to the standard topologies on $\Pscr$. Hence, the entire literature on nonlinear FPEs with (Wasserstein Lipschitz-)continuous coefficients in their measure-argument does not apply. In particular, proving the existence and uniqueness of solutions to Nemytskii-FPEs with degenerate (e.g., Dirac) initial conditions is often delicate.
    The benefit of this challenging class of FPEs is that many nonlinear PDEs are of type \eqref{eq:nlPDE}, even though this may not be obvious. Hence, solutions to such PDEs can be studied in the framework of FPEs and their associated stochastic differential equations, which we discuss next.
\end{enumerate}
\end{remark}

\subsectionnotoc{Associated McKean--Vlasov SDEs} 
The FPE \eqref{eq:FPE} is closely related to the McKean--Vlasov SDE (MV-SDE)
\begin{equation}\label{eq:DDSDE}
    dX_t = b(X_t,\mathcal{L}(X_t)) dt + \sigma(X_t, \mathcal{L}(X_t)) dW_t,
\end{equation}
where $\sigma: \R^d \times \Pscr_* \to \R^{d\times d}$ is determined via $2a = \sigma\sigma^T$, and $W=(W_t)_{t\geq 0}$ is a standard $d$-dimensional Brownian motion.
\begin{definition}\label{def:MVSDE-sol}
\begin{enumerate}
    \item [(i)]  A \textit{probabilistically weak solution} (short: \textit{solution}) to \eqref{eq:DDSDE} is a triplet consisting of a filtered probability space $(\Omega, \Gscr, (\Gscr_t)_{t\geq 0}, \mathbb{P})$, an $(\Gscr_t)$-Brownian motion $W = (W_t)_{t\geq 0}$ and an $(\Gscr_t)$-adapted stochastic process $X = (X_t)_{t\geq 0}$ such that $\mathcal{L}(X_t) \in \Pscr_*$ for all $t>0$,
    \begin{equation}
        \mathbb{E}\Bigg[\int_0^T |b(X_t,\mathcal{L}(X_t))| + |\sigma \sigma^T(X_t,\mathcal{L}(X_t))|dt\bigg] < \infty, \quad \forall T>0
    \end{equation}
    ($\mathbb{E}[\,\cdot\,]$ denotes expectation with respect to $\mathbb{P}$), and, $\mathbb{P}$-a.s.,
    \begin{equation}
        X_t = X_0 + \int_0^t b(X_s,\mathcal{L}(X_s)) ds + \int_0^t \sigma(X_s,\mathcal{L}(X_s)) dW_s,\quad \forall t \geq 0.
    \end{equation}
    $\mathcal{L}(X_0)$ is called \textit{initial condition} of $X$. Below, we often refer to $X$ as the solution (instead of the entire triplet).
    \item[(ii)] Solutions to \eqref{eq:DDSDE} with initial condition $\zeta \in \Pscr$ are \textit{weakly unique}, if $\mathcal{L}(X) = \mathcal{L}(Y)$ for any weak solutions $X, Y$ with initial condition $\zeta$.
\end{enumerate}
\end{definition}

\begin{remark}
 The path law $\mathcal{L}(X)$ of any weak solution $X$ to \eqref{eq:DDSDE} solves the nonlinear martingale problem with coefficients $a = \frac 1 2 \sigma \sigma^T$ and $b$. Vice versa, for any solution $P \in \Pscr(\Cscr)$ to this martingale problem, there exists a weak solution $X$ to \eqref{eq:DDSDE} with $\mathcal{L}(X) = P$, see for instance \cite{Stroock87,StroockVaradh2007}. We call $P \in \Pscr(\Cscr)$ a \emph{solution path law} to \eqref{eq:DDSDE}, if there exists a weak solution $X$ with $\mathcal{L}(X) = P$. We say that such a solution path law $P$ is \emph{unique} (given some initial condition $\zeta$) if $\mathcal{L}(X) = P$ for any weak solution $X$ with initial condition $\zeta$.
\end{remark}

The intimate relation between \eqref{eq:FPE} and \eqref{eq:DDSDE} is the following: A straightforward application of It\^{o}'s formula implies that for any solution path law $P$ to \eqref{eq:DDSDE}, its curve of one-dimensional time marginals $\mu_t =P_t$, $P_t := P \circ \pi_t^{-1}$, is a solution to \eqref{eq:FPE} satisfying \eqref{eq:L1-int-SP}. The converse is implied by the celebrated Ambrosio--Figalli--Trevisan superposition principle:
\begin{theorem}[Superposition principle, \cite{Trevisan16}]\label{thm:SP-pr}
    Let $(\mu_t)_{t\geq 0}$ be a solution to \eqref{eq:FPE} satisfying \eqref{eq:L1-int-SP}. Then there exists a solution path law $P$ to \eqref{eq:DDSDE} with $P_t = \mu_t$ for all $t \geq 0$.
\end{theorem}
Thus, one way to solve \eqref{eq:DDSDE} is to first solve \eqref{eq:FPE} and to then apply the previous result. This \emph{lifting} of a nonlinear FPE solution to an MV-SDE solution is fundamental for our results.

When the FPE is of Nemytskii-type \eqref{eq:nlPDE}, the associated MV-SDE becomes
\begin{equation}\label{eq:Neymtskii-DDSDE}
    dX_t = B\big(X_t,(\Lambda u(t))(X_t)\big) dt + \Sigma\big(X_t, (\Gamma u(t))(X_t)\big) dW_t, \quad \mathcal{L}(X_t) = u(t,x)dx,\quad t >0,
\end{equation}
where we write $u(t)$ for the function $x\mapsto u(t,x)$ and set
$$B(x,(\Lambda u)(x)) := \big(B_i(x,(\Lambda_i u)(x))\big)_{i \leq d},\quad \Sigma(x,(\Gamma u)(x)) := \big(\Sigma_{ij}(x,(\Gamma_{ij} u)(x))\big)_{i,j \leq d},$$
where $\Sigma = (\Sigma_{ij})_{i,j \leq d}$ is such that $2 A = \Sigma \Sigma^T$ pointwise (with $A = (A_{ij})_{i,j \leq d}$).
We present several explicit such MV-SDEs in Sections \ref{sec:heat_equation},
\ref{sec:porous_medium} and \ref{sec:pLaplace}.

\section{Markov processes and cores}\label{sec:nonunique-Markov}
Here, we first recall the definition of classical Markov processes and their relation to linear FPEs and SDEs. Then, we recall the notion and basic properties of \emph{nonlinear Markov processes} from \cite{McKean1966, RehmeierRoeckner2025NonlinearMarkov}, in particular how to construct a nonlinear Markov process from a family of solutions to a nonlinear FPE (Theorem \ref{thm:NL-MP-construction-old}). Finally, we introduce the natural new notion of \emph{nonlinear Markov core}, which is useful for the results of this paper.

Recall that for $P \in \Pscr (\Cscr)$, we write $P_t \coloneqq P \circ \pi_t^{-1}$.

\subsectionnotoc{Linear Markov processes}
Throughout, we call classical Markov processes \emph{linear Markov processes} to distinguish them from nonlinear Markov processes.

\begin{definition}\label{def:LMP} 
A \emph{(normal, time-homogeneous) linear Markov process} is a family $\{P^z
\}_{z\in \R^d} \subset \Pscr(\Cscr)$  such that 
\begin{enumerate}[label=(\roman*), font=\normalfont]
        \item\label{itm:def_LMP_0} $P_0^z = \delta_z$ for all $z \in \R^d$; 
        \item\label{itm:def_LMP_1} $z\mapsto P^z(G)$ is Borel measurable for all $G \in \Bscr(\Cscr)$;
        \item\label{itm:def_LMP_2}  the \emph{(time-homogeneous) linear Markov property} holds, i.e., 
\begin{equation}\label{eq:linear_Markov_property}
    P^{z} (\pi_{t+s} \in A | \mathcal{F}_s)(\cdot) = P^{\pi_s(\cdot)} (\pi_t\in A) \quad P^z-\text{a.s.},\quad \forall s,t \geq 0, \, A \in \mathcal{B}(\mathbb{R}^d).
\end{equation}
\end{enumerate}
\end{definition}
Observe that the map 
\begin{equation}\label{eq:Pzeta_for_linearMP}
  \Pscr \ni   \zeta \mapsto P^\zeta \coloneqq \int_{\R^d} P^z d \zeta (z)\, \in \Pscr(\Cscr),
\end{equation}
i.e. $P^\zeta(G) = \int_{\R^d} P^z(G)\,d\zeta(z)$ for $G\in \Bscr(\Cscr)$,
is linear, and
\eqref{eq:linear_Markov_property} yields, for $\zeta = P^z_t$,
\begin{equation}\label{eq:shift_identity_LMP}
    P^\zeta = P^z \circ \Pi_t^{-1}.
\end{equation}
The following result is classical and follows from the fact that the family of finite-dimensional time marginals $P^z\circ (\pi_{t_1},\dots,\pi_{t_n})^{-1}$, $n\in \N$, $0\leq t_1 < \dots <t_n$, of a linear Markov process is determined by its one-dimensional time marginals $\{P^z_t\}_{z\in \R^d, t \geq 0}$ and that, of course, $P^z$ is determined by its family of finite-dimensional time marginals.
\begin{proposition}\label{prop:LMP_uniquely_marginal}
    A linear Markov process $\{P^z\}_{z\in \R^d}$ is uniquely determined by its family of one-dimensional time marginals
    $\{\mu_t^{z}\}_{(t,z) \in\mathbb{R}_+\times \mathbb{R}^d}$, where $\mu_t^{z} \coloneqq P_t^z$, i.e., if $\{Q^z\}_{z\in \R^d}$ is another linear Markov process with $Q^z_t = P^z_t$ for all $(t,z) \in\mathbb{R}_+\times \mathbb{R}^d$, then $Q^z = P^z$ for all $z \in \R^d$.
\end{proposition}
As a main result of this paper, we prove that the previous proposition is not true for nonlinear Markov processes (see Theorem \ref{intro:thm1}).

The term \emph{linear} Markov process refers to the linearity of \eqref{eq:Pzeta_for_linearMP} and to the fact that typical examples of such processes are related to \emph{linear} Fokker--Planck equations, as we recall now.

\vspace{1em}
\noindent
\textbf{Relation to SDEs and linear FPEs.}
The following relations between linear Markov processes, linear FPEs, and SDEs are well-known. Assume that for Borel maps $a_{ij},b_i : \R^d \to \R$, $1\leq i,j\leq  d$, with $a = (a_{ij})_{i,j \leq d}$ symmetric non-negative definite, the linear FPE with coefficients $a_{ij}$ and $b_i$ has a unique distributional solution $\mu^z = (\mu^z_t)_{t\geq 0}$ with initial condition $\delta_z$ for all $z\in \R^d$ in the sense of \Cref{def:lFPE-sol} such that \eqref{eq:L1-int-SP} holds with $a_{ij}(x)$ and $b_i(x)$ replacing $a_{ij}(x,\mu_t)$ and $b_i(x,\mu_t)$. Then the SDE with coefficients $b = (b_1,\dots,b_d)$ and $\sigma$ satisfying $2a = \sigma\sigma^T$ has a unique solution path law $P^z$ for all $z \in \R^d$ (see \cite[Thm.6.2.3]{StroockVaradh2007}). Furthermore, $\{P^z\}_{z\in \R^d}$ is a linear Markov process with one-dimensional time marginals $\{\mu^z_t\}_{(t,z)\in \R_+\times \R^d}$, i.e. 
\begin{equation}\label{eq:qrt}
    P^z_t = \mu^z_t,\quad \forall (t,z)\in \R_+ \times \R^d,
\end{equation}
and it is the unique such linear Markov process.

Put differently and in the spirit in which we treat the nonlinear case in this paper, starting from a family of linear FPE-solutions, there exists a unique linear Markov process consisting of solution path laws to the associated SDE with one-dimensional time marginals equal to the given FPE-solution family.

\subsectionnotoc{Nonlinear Markov processes and cores}
We recall the notion of \textit{nonlinear Markov processes} from \cite{McKean1966,RehmeierRoeckner2025NonlinearMarkov}. Since it suffices for our purposes, we restrict to the time-homogeneous case (which was not explicitly presented in \cite{RehmeierRoeckner2025NonlinearMarkov}).

\begin{definition}\label{def:NL-MP} 
    Let $\Pscr_0 \subseteq \Pscr$. A \emph{(time-homogeneous)  nonlinear Markov process} is a family $\{P^\zeta\}_{\zeta \in \Pscr_0} \subset \Pscr(\Cscr)$ such that, setting $\mu^\zeta_t := P^\zeta_t$ for all $(\zeta,t) \in \Pscr_0 \times \R_+$,
    \begin{enumerate}[label=(\roman*), font=\normalfont]
        \item\label{itm:def_NMP_0} $\mu^\zeta_0 = \zeta$ for all $\zeta \in \Pscr_0$;
        \item\label{itm:def_NMP_1} $\mu^\zeta_t \in \Pscr_0$ for all $t > 0, \zeta \in \Pscr_0$;
        \item\label{itm:def_NMP_2} the \emph{(time-homogeneous) nonlinear Markov property} holds, i.e.,
        \begin{equation}\label{eq:nonlinear_Markov_property}
        P^\zeta(\pi_{t+s} \in A | \Fscr_s) ( \cdot) = P^{\mu^\zeta_s,\pi_s(\cdot)}(\pi_t \in A)\quad P^\zeta-\text{a.s.},\quad \forall s,t \geq 0, \, \zeta \in \Pscr_0, \, A \in \Bscr(\R^d),
    \end{equation}
    where for $\zeta \in \Pscr_0$, $\{P^{\zeta,y}\}_{y\in \R^d}$ is the disintegration family of $P^{\zeta}$ with respect to $\pi_0$.
    \end{enumerate}
\end{definition}
By \ref{itm:def_NMP_0}, we have $P^\zeta_0 = \zeta$ for all $ \zeta \in \Pscr_0$, which can be regarded as the \emph{normality} property in analogy with the linear \Cref{def:LMP} \ref{itm:def_LMP_0}. 
$\Pscr_0$ is the class of initial distributions of the nonlinear Markov processes.
By \ref{itm:def_NMP_1}, $\Pscr_0$ is invariant under the flow of one-dimensional time marginals.
The term \emph{nonlinear} reflects that the map
$
    \zeta \mapsto P^\zeta,
$
from $\Pscr_0$ to $\Pscr(\Cscr)$, need \emph{not} be linear on $\Pscr_0$ (even if $\Pscr_0 = \Pscr$), in contrast to the linear case \eqref{eq:Pzeta_for_linearMP}. 

Let us recall some properties of nonlinear Markov processes, in particular the fact that they generalize the notion of linear Markov processes.

\begin{proposition}[{\cite[Section 2]{RehmeierRoeckner2025NonlinearMarkov}}]\label{lem:basic-props-NLMP}
\begin{enumerate}[leftmargin=20pt, label=(\roman*), font=\normalfont]
    \item Let $\{P^z\}_{z\in \R^d}$ be a normal linear Markov process, and set 
    $$P^\zeta := \int_{\R^d} P^z \,d\zeta(z) 
    $$ for all $\zeta \in \Pscr$. Then $\{P^\zeta\}_{\zeta \in \Pscr}$ is a nonlinear Markov process with $\Pscr_0 = \Pscr$.
    Likewise, $\{P^\zeta\}_{\zeta \in \bar{\Pscr}_0}$ is a nonlinear Markov process, where $\bar{\Pscr}_0 \coloneqq \{P^z_t\,:\, z\in \R^d, t \geq 0\}$. The latter is the minimal nonlinear Markov process containing $\{P^z\}_{z\in \R^d}$, i.e. for any nonlinear Markov process $\{P^\zeta\}_{\zeta \in \Pscr_0}$ with $\{\delta_z , z\in \R^d\} \subseteq \Pscr_0$ and $P^{\delta_z} = P^z$ one has $\bar{\Pscr}_0 \subseteq \Pscr_0$.
    \item Unlike for linear Markov processes, the one-dimensional time marginals of a nonlinear Markov process in general do not satisfy the Chapman--Kolmogorov equations
    $$\mu^z_{t+s} = \int_{\R^d} \mu^y_t \,d\mu^z_s(y),\quad \forall z \in \R^d, s,t \geq 0,$$
    even if $\{\delta_z\,|\, z\in \R^d\} \subseteq \Pscr_0$. They do, however, satisfy the \emph{flow property}
    \begin{equation}\label{eq:flow-prop}
        \mu^\zeta_{t+s} = \mu^{\mu^\zeta_s}_t,\quad \forall \zeta \in \Pscr_0, s,t \geq 0.
    \end{equation} 
    In the linear case, the flow property follows from the Chapman--Kolmogorov equations.
    \item Similarly (but more involved) than in the linear case, a nonlinear Markov process is uniquely determined by a family of certain one-dimensional time marginals. More precisely, let $\{P^\zeta\}_{\zeta \in \Pscr_0}$ be a nonlinear Markov process and define $p^{\zeta}_{t}(y,dx) \in \Pscr$ by 
	\begin{equation}\label{eq:rcp-marginals}
		p^{\zeta}_{t}(y,dx) := P^{\zeta,y}\circ \pi_t^{-1}(dx),
	\end{equation}
	which is uniquely determined for $\zeta$-a.e. $y\in \R^d$. Here, $P^{\zeta, y}$ is as in Definition \ref{def:NL-MP}. Then for all $n \in \N_0$, $f: (\R^d)^{\otimes n} \to \R$ bounded measurable and $0 \leq t_1< \dots <t_n$, we have
	\begin{multline}\label{eq:Markov-fdd}
		\E_{\,\zeta} [f(\pi_{t_1},\dots,\pi_{t_n})] = \\ \notag \int_{\mathbb{R}^d}\bigg(\dots \int_{\mathbb{R}^d}\bigg(\int_{\mathbb{R}^d} f(x_1,\dots,x_n)\, p^{\mu^\zeta_{t_{n-1}}}_{t_n-t_{n-1}}(x_{n-1},dx_n)\bigg) p^{\mu^\zeta_{t_{n-2}}}_{t_{n-1}-t_{n-2}}(x_{n-2},dx_{n-1})\dots\bigg)\mu^{\zeta}_{t_1}(dx_1),
	\end{multline}
    where $\mathbb{E}_{\,\zeta}[\cdot]$ denotes expectation with respect to $P^\zeta$. 
\end{enumerate}
\end{proposition}
We stress that the kernels $p^\zeta_t(y,\cdot)$ are \emph{not} equal to the one-dimensional time marginals $\mu^\zeta_t$ of the nonlinear Markov process. Consequently, the previous formula does \emph{not} imply that $\{\mu^\zeta_t\}_{t\geq 0, \zeta \in \Pscr_0}$ uniquely determines the nonlinear Markov process $\{P^\zeta\}_{\zeta \in \Pscr_0}$, in contrast to the linear case, see \Cref{prop:LMP_uniquely_marginal}. 

One main reason to study nonlinear Markov processes is that they naturally arise in the context of nonlinear FPEs and MV-SDEs, as the following result from \cite{RehmeierRoeckner2025NonlinearMarkov} shows.
\begin{theorem}[{\cite[Theorem 3.8]{RehmeierRoeckner2025NonlinearMarkov}}]\label{thm:NL-MP-construction-old}
    Let $\Pscr_0 \subseteq \Pscr$, and for any $\zeta \in \Pscr_0$ let $\mu^\zeta = (\mu^\zeta_t)_{t\geq 0}$ be a solution to \eqref{eq:FPE} satisfying \eqref{eq:L1-int-SP} such that
    \begin{enumerate}[label=(\roman*), font=\normalfont]
        \item $\{\mu^\zeta_t\}_{t\geq 0, \zeta \in \Pscr_0}$ satisfies the flow property \eqref{eq:flow-prop};
        \item for each $\zeta \in \Pscr_0$, $\mu^\zeta$ is an extreme point in the convex set of solutions to the linearized FPE $(\mu^\zeta$-linFPE) with initial condition $\zeta$ satisfying \eqref{eq:L1-int-SP}.
    \end{enumerate}
    Then there exists a nonlinear Markov process $\{P^\zeta\}_{\zeta \in \Pscr_0}$ such that each $P^\zeta$ is the unique solution path law to the associated MV-SDE \eqref{eq:DDSDE} with $P^\zeta_t = \mu^\zeta_t$, $t\geq 0$. In particular, this nonlinear Markov process is uniquely determined by \eqref{eq:DDSDE} and $\{\mu^\zeta_t\}_{t\geq 0, \zeta \in \Pscr_0}$.
\end{theorem}

\vspace{0.5em}
\noindent
\textbf{Motivation for a new notion: Nonlinear Markov cores.}
    For nonlinear Markov processes, typical choices for $\Pscr_0$  are $\Pscr_0 = \Pscr$, $\Pscr_{\textup{ac}}$ or $ \Pscr_{\textup{ac}} \cap L^\infty$ (see also the examples in \cite[Section 4]{RehmeierRoeckner2025NonlinearMarkov}). $\Pscr_0 = \{\delta_z, z \in \R^d\}$ is not reasonable, since by \Cref{def:NL-MP}, this imposes the pathological constraint that \emph{all} one-dimensional time marginals $\mu^\zeta_t$ of the nonlinear Markov process are Dirac measures. Thus, in contrast to the linear case, a nonlinear Markov process typically consists not only of path measures with Dirac initial conditions.

    However, in the present paper we usually encounter the following situation. For each $z\in \R^d$, there exists a solution $u^z = (u^z(t,\cdot))_{t > 0}$ to a nonlinear PDE with initial condition $\delta_z$ (e.g., in the case of the $p$-Laplace equation, $u^z$ is the explicit Barenblatt solution with center of mass $z$) and we would like to construct a (in fact, several) nonlinear Markov process(es) with one-dimensional time marginal densities $(u^z(t,\cdot))_{t > 0}$. Since the PDE is nonlinear, solutions corresponding to general initial data $\zeta \in \Pscr$ cannot be obtained by convex combinations of $u^z$, $z\in \R^d$. Thus, one would like to set $\Pscr_0 = \{\delta_z, z \in \R^d\}$, although, as mentioned above, this is not feasible.
    
    This motivates the following new notion. The idea is to associate with the given data $\{\mu^z_t\}_{t\geq 0,\; z \in \R^d}$, a \emph{minimal} nonlinear Markov process, in the sense that its set of initial data $\Pscr_0$ is as small as possible while necessarily accommodating $\{\mu^\zeta_t\}_{t\geq 0, z\in \R^d}$ (see also \Cref{lem:basic-props-NLMP} (i)).

\begin{definition}\label{def:NMC} 
     Let $\{\mu^z_t\}_{t\geq 0, z \in \R^d} \subset \Pscr$ be such that $\mu^z_0 = \delta_z$ for all $z \in \R^d$ and $(t,z) \mapsto \mu^z_t$ is injective from $\R_+\times \R^d$ to $ \Pscr$. We call a family $\{P^z\}_{z\in \R^d} \subset \Pscr(\Cscr)$ a \emph{(time-homogeneous) nonlinear Markov core for $\{\mu^z_t\}_{t\geq 0, z \in \R^d}$}
     if
        \begin{enumerate}[label=(\roman*), font=\normalfont]
        \item \label{itm:def_NMC_1} $P^z_t = \mu^z_t$ for all $(t,z)\in  \R_+ \times \R^d$;
        \item \label{itm:def_NMC_2}  the family $\{P^\zeta\}_{\zeta \in \mathcal{P}_0}\subset \Pscr(\Cscr)$, defined by $\Pscr_0 \coloneqq \{\mu^z_t\}_{t\geq 0, z \in \R^d}$ and 
        \begin{equation}
            P^{\zeta} \coloneqq P^z \circ \Pi_t^{-1} \quad \text{for } \zeta = \mu^z_t,
        \end{equation}
     is a nonlinear Markov process in the sense of \Cref{def:NL-MP}. We call $\{P^\zeta\}_{\zeta \in \mathcal{P}_0}$ the nonlinear Markov process \emph{induced} by $\{P^z\}_{z\in \R^d}$.
    \end{enumerate}
\end{definition}
The injectivity condition is, in particular, needed to ensure that $P^\zeta$ in (ii) is well-defined. As a first example of a nonlinear Markov core, we consider the linear special case: 

\begin{lemma}\label{lemma:LMP_NMC}
    Let $\{P^z\}_{z\in \mathbb{R}^d}$ be a normal linear Markov process such that $(t,z) \mapsto P^z_t =: \mu^z_t$ is injective from $\R_+\times \R^d$ to $\Pscr$. Then $\{P^z\}_{z\in \R^d}$ is a nonlinear Markov core for $\{\mu^z_t\}_{t\geq 0, z \in \R^d}$ and its induced nonlinear Markov process is $\{P^\zeta\}_{\zeta\in \Pscr_0}$, where $\Pscr_0 := \{\mu^z_t\}_{t \geq 0, z\in \R^d}$ and $P^\zeta \coloneqq \int_{\mathbb{R}^d} P^z d \zeta (z)$.
\end{lemma}

\begin{proof}
    This follows from \eqref{eq:shift_identity_LMP} (which itself follows from the classical Markov property) and the second statement in \Cref{lem:basic-props-NLMP} (i).
\end{proof}

One of our main results is that a nonlinear Markov process is \emph{not} uniquely determined by its \emph{one}-dimensional time marginals. We prove, however, that the \emph{two}-dimensional time marginals, $P^\zeta_{s,t} := P^\zeta \circ (\pi_s,\pi_t)^{-1}$, $0 \leq s \leq t$, characterize a nonlinear Markov process, in the next theorem:
\begin{theorem}\label{thm:NMP_2D}
    Let $\Pscr_0 \subseteq \Pscr$, and $\{\mathbb{P}^\zeta\}_{\zeta \in \Pscr_0}$, $\{\mathbb{Q}^\zeta\}_{\zeta \in \Pscr_0}$ be nonlinear Markov processes such that
    $$\mathbb{P}^\zeta_{s,t} = \mathbb{Q}^\zeta_{s,t},\quad \forall \zeta \in \Pscr_0, \, 0\leq s \leq t.$$
    Then $\mathbb{P}^\zeta = \mathbb{Q}^\zeta$ for all $\zeta \in \Pscr_0$.
\end{theorem}
In fact, the following proof shows that the assumption of the theorem can be weakened by considering only $s=0$ and arbitrary $t>0$.
\begin{proof}
The assumption in particular implies that the one-dimensional time marginals of both nonlinear Markov processes coincide, and we set
$$\mu^\zeta_t := \mathbb{P}^\zeta \circ \pi_t^{-1} = \mathbb{Q}^\zeta \circ \pi_t^{-1},\quad \forall \zeta \in \Pscr_0, \, t \geq 0.$$
Denote by $\{P^{\zeta,y}\}_{y\in \R^d}$ and $\{Q^{\zeta,y}\}_{y\in \R^d}$ the disintegration family of $\mathbb{P}^\zeta$ and $\mathbb{Q}^\zeta$ with respect to $\pi_0$, respectively (each uniquely determined $\zeta$-a.s.), and consider (as in Proposition \ref{lem:basic-props-NLMP} (iii)) for $t\geq 0$ the probability kernels
$$(y,A) \mapsto p^\zeta_t(y,A) := P^{\zeta,y} \circ \pi_t^{-1}(A),\quad (y,A) \in \R^d \times \Bscr(\R^d),$$
and likewise $q^\zeta_t(y,A) := Q^{\zeta,y} \circ \pi_t^{-1}(A)$. 
For each $(\zeta, t) \in \Pscr_0\times \R_+$, $p^\zeta_t$ and $q^{\zeta}_t$ are the $\zeta$-a.s. unique probability kernels $p$ and $q$ such that
$$\mathbb{P}^\zeta \circ (\pi_{0},\pi_{t})^{-1} (A\times B) = \int_A p(z,B)d\zeta(z)$$
	and
		$$\mathbb{Q}^\zeta \circ (\pi_{0},\pi_{t})^{-1} (A\times B) = \int_A q(z,B)d\zeta(z)$$
	for all $A,B \in \Bscr(\R^d)$,	respectively (the exceptional set depends on $\zeta$ and $t$, but this is not harmful for the remainder of the proof). Indeed, the existence and $\zeta$-a.s. uniqueness of such kernels $p$ and $q$ is classical, see for instance \cite[Theorem 8.5]{K21}. We have
\begin{align}
\mathbb{P}^{\zeta} \circ (\pi_{0},\pi_{t})^{-1} (A\times B) = \E_{\zeta}[\mathbb{P}_{\zeta}[\pi_{t}\in B|\pi_{0}] \mathds{1}_A(\pi_{0})] = \int_A p^\zeta_{t}(y,B)\,d\zeta(y),
\end{align}
where for the second equality we used $\mathbb{P}^{\zeta} \circ \pi_0^{-1} = \zeta$ and the nonlinear Markov property. Of course, the applies to $\mathbb{Q}^\zeta$ and $q^\zeta_t$ instead of $\mathbb{P}^\zeta$ and $p^\zeta_t$. Hence, our assumption entails that for each $t\geq 0$ and $\zeta \in \Pscr_0$
$$p^\zeta_t(y,dx) = q^\zeta_t(y,dx)\quad \zeta-\text{a.e. } y\in \R^d.$$ But then the formula for the finite-dimensional time marginals of a nonlinear Markov process from Proposition \ref{lem:basic-props-NLMP} (iii) implies
$$\mathbb{P}^\zeta \circ(\pi_{t_1},\dots,\pi_{t_n})^{-1} = \mathbb{Q}^\zeta \circ (\pi_{t_1},\dots,\pi_{t_n})^{-1},\quad \forall n \in \N, 0\leq t_1 \leq \dots \leq t_n.$$
This implies the claim.
\end{proof}

Finally, we mention that the previous result as well as the notion of nonlinear Markov cores can obviously be extended to the time-inhomogeneous setting, but since all our examples and results concern time-homogeneous cases, we refrain from doing so.

\section{From one PDE to many MV--SDEs and nonlinear Markov processes}\label{sec:scheme}
We come to the core matter of our paper, described by the following three-step scheme.

\vspace{1em}

\noindent
\textbf{Scheme.} We start from a parabolic PDE
\begin{equation}\label{eq:general-PDE}
    \partial_t u =  L u,
\end{equation}
where $L$ is some (possibly nonlinear) differential operator, for instance $Lu = \Delta u$ (heat equation), $Lu = \Delta u^m$ (porous medium equation), or $Lu = \nabla \cdot (|\nabla u|^{p-2}\nabla u)$ ($p$-Laplace equation). Suppose there exists a family of solutions $\{u^z\}_{z\in \R^d}$, $u^z \coloneqq (u^z(t,\cdot))_{t>0}$, to \eqref{eq:general-PDE} such that for all $z\in \R^d$ and $t>0$, $u^z (t,\cdot)$ is a probability density, and $u^z(t,x)dx \rightarrow {\delta}_z$ weakly as $t \downarrow 0$. 
We proceed along the following three steps:
\begin{enumerate}\label{itm:Construction}
    \item Recast \eqref{eq:general-PDE} as a FPE \eqref{eq:FPE} with Nemytskii-coefficients \eqref{eq:Nemytskii-coeff} solved by
    $$\mu^z \coloneqq (\mu^z_t)_{t\geq 0}, \,\mu^z_t \coloneqq u^z (t,x) dx, \,\mu^z_0 = \delta_z, \quad z\in \R^d,$$ in the sense of \Cref{def:FPE-sol}.
    \item Construct solution path laws $\{P^z\}_{z\in \R^d}$ to the corresponding MV-SDE \eqref{eq:Neymtskii-DDSDE} with $P^z_t = \mu^z_t$ via \Cref{thm:SP-pr} (superposition principle).
    \item Prove $\{P^z\}_{z\in \R^d}$ is a nonlinear Markov core for $\{u^z\}_{z\in \R^d}$, inducing a nonlinear Markov process $\{P^\zeta\}_{\zeta \in \Pscr_0}$ in the sense of \Cref{def:NMC}, where $\Pscr_0 := \{\mu^z_t \mid z \in \R^d, t \geq 0\}$.
\end{enumerate}
This way, starting from a solution family $\{u^z\}_{z\in \R^d}$ to the PDE \eqref{eq:general-PDE}, we construct a nonlinear Markov process with one-dimensional time marginals densities $u^z(t,\cdot)$, consisting of solution path laws to an associated MV-SDE.
\begin{remark}
    \begin{enumerate}
        \item [(i)] It is irrelevant in which sense (e.g., strong or weak) $u^z$ solves \eqref{eq:general-PDE} as long as it solves the associated FPE from Step (1) in the sense of \Cref{def:FPE-sol}. Thus we do not explicitly state a definition of solutions for \eqref{eq:general-PDE}.
        \item[(ii)] Step (1) induces a severe degree of nonuniqueness: As we will demonstrate in \Cref{subsect:step1}, a solution to \eqref{eq:general-PDE} typically solves an uncountable number of associated nonlinear FPEs. 
        \item[(iii)] On the other hand, the identification of the MV-SDE associated with a \emph{given} FPE in Step (2) is (at least in our framework) unique. Also, the corresponding nonlinear Markov core and process from Step (3) will be uniquely determined by $\{u^z\}_{z\in \R^d}$ and the MV-SDE.
    \end{enumerate}
\end{remark}

\subsection{Step (1): From one PDE to many nonlinear FPEs}\label{subsect:step1}

    Consider the parabolic PDE \eqref{eq:general-PDE} and assume that we have a solution given by a weakly continuous curve of probability densities $ (u(t,\cdot))_{t>0} \subset \Pscr_{\textup{ac}}$.
    We would like to identify an FPE solved by $(u(t,\cdot))_{t>0}$. 
    Our ansatz is to find coefficients $(a_{ij},b_{i})$ of Nemytskii-type as in \eqref{eq:Nemytskii-coeff}, expressed in terms of $(A_{ij}, B_i,\Gamma_{ij}, \Lambda_i)$, such that
    $$ Lu (x) \overset{!}{=} \partial_{ij}\big( a_{ij}(x, u) u(x)\big) - \partial_{i} \big( b_{i}(x, u) u(x)\big) $$
    which is, at this point, only formal. We restrict our ansatz to the case where the diffusion matrix is diagonal, i.e., $a_{ij} = \delta_{ij} a$, as this is sufficient for our applications. To turn this formal ansatz into a rigorous approach, we consider the following
    \vspace{0.5em}

    \noindent
    \textbf{Problem.}
    For a weakly continuous curve of probability densities $(u(t,\cdot))_{t>0} \subset \Pscr_{\textup{ac}}$, find coefficients $(a,b) : \R^d \times \Pscr_* \to \R \times \R^d$ and a domain $\Pscr_* \subseteq \Pscr_{\textup{ac}}$ such that
    $(u(t,\cdot))_{t>0}$ solves the FPE
        \begin{equation}\label{eq:Main_FPE_our_paper}
            \partial_t u = \Delta \big( a(x,u)  u \big) - \nabla \cdot \big( b (x,  u ) u\big) 
        \end{equation}
    in distributional sense, i.e.,
    \begin{enumerate}[label=(\roman*), font=\normalfont]
        \item\label{itm:sol1} $(u(t,\cdot))_{t >0} \subseteq \Pscr_*$;
        \item\label{itm:sol2} $a(x,u(t,\cdot)) \geq 0$ $u(t,x)$ $dxdt$-a.e. $(t,x) \in (0,\infty)\times \R^d$;
        \item\label{itm:sol3} $(t,x)\mapsto a(x,u(t,\cdot))$ and $ (t,x)\mapsto b(x,u(t,\cdot))$ belong to $L^1_{\textup{loc}}(\R_+\times \R^d;u(t,x)dxdt)$;
        \item\label{itm:sol4} $(u(t,\cdot))_{t>0}$ satisfies the weak formulation of \eqref{eq:Main_FPE_our_paper}.
    \end{enumerate}
    \vspace{0.5em}

    \noindent
     This motivates the definition of the following set. 

    \begin{definition}[Set of admissible coefficients]
        Given a weakly continuous curve of probability densities $(u(t,\cdot))_{t>0} \subset \Pscr_{\textup{ac}}$, we define the following (possibly empty) set 
        {
        \setlength{\jot}{-1pt}
        \begin{align}
            M_{(u(t,\cdot))}   \coloneqq \Big\{ \big((a,b),\Pscr_*\big): \quad\,\,\, \Pscr_* \subseteq \Pscr_{\textup{ac}} \text{ and } (a, b): \R^d \times \Pscr_* \to \R \times \R^d \text{ such that } \,\,\,\\
            (u(t,\cdot))_{t >0} \text{ solves the FPE \eqref{eq:Main_FPE_our_paper} i.e. \ref{itm:sol1}-\ref{itm:sol4} above hold} \Big\}.
        \end{align}
        }
    \end{definition}

    It is clear that this set, provided it is non-empty, is \emph{convex} in $a$ and $b$. More precisely, if $\big((a,b),\Pscr_*\big) \in M_{(u(t,\cdot))}$ and $\big((\hat{a},\hat{b}),\hat{\Pscr}_*\big) \in M_{(u(t,\cdot))}$, then for any $\beta \in (0,1)$
    $$
    \Big((\beta a+(1-\beta)\hat{a},\, \beta b+(1-\beta)\hat{b}), \, \Pscr_* \cap \hat{\Pscr}_* \Big) \in M_{(u(t,\cdot))}.
    $$
     When $\hat{a} = 0$, one can even choose $\beta \in (0,\infty)$. 
    
    It is not our objective here to fully characterize this set of admissible coefficients.
    Instead, we address the following question.
    \emph{Given an admissible tuple, can we construct further ones by a general procedure?} The next lemma provides an affirmative answer. More precisely, starting from  $\big((a,b),\Pscr_*\big)\in M_{(u(t,\cdot))}$, we show that if there exists $f:\R^d\times\Dscr(f)\to\R$ satisfying suitable integrability and regularity assumptions, then, for a suitable domain $\hat\Pscr_* \subseteq \Pscr_*$,
    $$
    \Big((a+\frac{f}{u},b+\frac{\nabla f}{u}),\hat{\Pscr}_*\Big) \in M_{(u(t,\cdot))}.
    $$
    Not all elements in $M_{(u(t,\cdot))}$ are of this type, as is easily seen by adding a weakly divergence-free vector field to $b$.

    \begin{lemma}\label{lem:many-PDEs_new}
        Let $(u(t,\cdot))_{t>0} \subseteq  \Pscr_* $ be a distributional solution to the FPE \eqref{eq:Main_FPE_our_paper} with coefficients $(a, b): \R^d \times \Pscr_* \to \R \times \R^d$ for some $\Pscr_* \subseteq \Pscr_{\textup{ac}} $. Let $f: \R^d \times \Dscr(f) \to \R $ with $\Dscr(f) \subseteq \Pscr_{\textup{ac}}$ such that $(u(t,\cdot))_{t >0} \subseteq\Dscr(f)$,
        \begin{equation}\label{eq:f_lemma_assumption_2}
            f(x,u(t,\cdot)) + a(x,u(t,\cdot))u(t,x) \geq 0 \quad u(t,x)dxdt\text{-a.e.,}
        \end{equation}
        and \begin{equation}\label{eq:f_lemma_assumption_3}
            \big[(t,x)\mapsto f(x,u(t,\cdot))\big] \in L^1_{\textup{loc}}(\R_+;W^{1,1}_{\textup{loc}}(\R^d)).
        \end{equation}
        Then $(u(t,\cdot))_{t> 0}$ is a distributional solution to the FPE with new coefficients $(\hat{a},\hat{b}): \R^d \times \hat{\Pscr}_*   \to \R \times \R^d$ given by 
        \begin{equation}\label{eq:f_lemma_new_coeff}
             \hat{a} (x,u) \coloneqq  a (x,u) + \frac{f(x,u)}{u(x)},\quad \text{and}\quad \hat{b} (x,u) \coloneqq  b (x,u) + \frac{\nabla f(x,u)}{u(x)},
        \end{equation}    
        on $\{u > 0\}$, and 
        $\frac {f(\cdot,u)} {u} := 0, \,\frac{\nabla f(\cdot, u)}{u} := 0$ on $\{u = 0\}$, defined on
        \begin{equation}\label{eq:f_lemma_new_domain}
            \hat{\Pscr}_* \coloneqq \Pscr_* \cap \{ u \in \Dscr(f): \, f(\cdot,u) \in W^{1,1}_{\textup{loc}}(\R^d)\}.
        \end{equation}  
    \end{lemma}

    \begin{remark}\label{rmk:f_lemma_zero_set}
        Note that for fixed $t>0$, $\{x: u(t,x) = 0\}$ clearly does not need to be a $dx$-zero set. However, since in the weak formulation of the FPE, the new coefficients defined in \eqref{eq:f_lemma_new_coeff} are never integrated with respect to $dx$, but only with respect to $u(t,x)dx$, the definition of $\frac {f(\cdot,u)} {u}$ and $\frac{\nabla f(\cdot, u)}{u}$ on $\{(t,x): u(t,x) = 0\}$ is irrelevant, since the latter is clearly a $u(t,x)dxdt$-zero set.
    \end{remark}

    \begin{proof}[{Proof of \Cref{lem:many-PDEs_new}}]
        By assumption \eqref{eq:f_lemma_assumption_3}, we have $f(\cdot,u(t,\cdot)) \in W^{1,1}_{\textup{loc}}(\R^d)$ $dt$-a.s., hence $u(t,\cdot) \in \hat{\Pscr}_*$ $dt$-a.s.
        Assumption \eqref{eq:f_lemma_assumption_2} ensures that $\hat{a}$ is non-negative along the solution.  
        Assumption \eqref{eq:f_lemma_assumption_3}
        guarantees that $(t,x)\mapsto \hat{a}(x,u(t,\cdot))$ and $ (t,x)\mapsto \hat{b}(x,u(t,\cdot))$ belong to $L^1_{\textup{loc}}(\R_+\times \R^d;u(t,x)dxdt)$. 
        It remains to verify the distributional formulation of $u$ for the FPE given by the new coefficients. This holds, since for all $\varphi \in C^2_c(\R^d)$ and $t >s>0$, we have
        \begin{align}
            &\int_s^t\int_{\R^d}\bigg(\frac {f(x,u(r,\cdot))} {u(r,x)}\Delta \varphi(x) + \frac{\nabla f(x,u(r,\cdot))}{u(r,x)} \cdot \nabla \varphi(x)\bigg) u(r,x) \,dx dr \\& = \int_s^t \int_{\R^d}\bigg( f(x,u(r,\cdot)) \Delta \varphi(x) + \nabla f(x,u(r,\cdot)) \cdot \nabla \varphi(x)\bigg) dx dr
            \\&= \int_s^t \int_{\R^d}\bigg( f(x,u(r,\cdot)) \Delta \varphi (x)- f(x,u(r,\cdot)) \Delta \varphi(x) \bigg) dx dr = 0,
        \end{align}
        where in the last step, we applied integration by parts for a.e. fixed $r\in (s,t)$.
    \end{proof}
    
   We now discuss two ``extreme'' special cases, in one of which the diffusion coefficient is eliminated, and in the other the drift coefficient is eliminated.

\subsubsection{Pure-drift formulation}\label{subsubsec:pure_drift}
    Consider the setting of \Cref{lem:many-PDEs_new}. We start from a distributional solution $(u(t,\cdot))_{t>0}$ to the FPE with coefficients $(a,b)$ and domain $\Pscr_*$, and our aim is to find a new coefficient $\hat{b}$ with a suitable domain $\hat{\Pscr}_*$ such that $(u(t,\cdot))_{t>0}$ is also a distributional solution to the first-order FPE
    \begin{equation}\label{eq:pure_drift_FPE_explanation}
    \partial_t u
    =
    -\nabla\cdot\big(\hat b(x,u)\,u\big).
    \end{equation}
    In view of \Cref{lem:many-PDEs_new}, the correct choice for $f$ is
    \begin{equation}
        f(x,u) := -a(x,u)u(x)
    \end{equation}
    with $\Dscr(f) \coloneqq \Pscr_* $, which leads to 
    \begin{equation}
        \hat a(x,u) = 0, \qquad \hat b(x,u)
    =
    b(x,u)
    -
    \frac{\nabla (a(x,u)u(x) )}{u(x)},
    \end{equation}
    defined on $ \hat{\Pscr}_* \coloneqq 
    \{ u\in \Pscr_* : \, 
    a(\cdot,u)u(\cdot) \in W^{1,1}_{\mathrm{loc}}(\R^d)
    \}$, with the same convention on $\{u = 0\}$ as in the previous lemma.
    Therefore, if $(t,x)\mapsto a(x,u(t,\cdot))u(t,x)$  belongs to $L^1_{\mathrm{loc}}\big(\R_+;W^{1,1}_{\mathrm{loc}}(\R^d)\big)$, \Cref{lem:many-PDEs_new} applies and yields that $(u(t,\cdot))_{t>0}$ is a distributional solution to \eqref{eq:pure_drift_FPE_explanation}. 
    Rewriting the FPE in the continuity equation form \eqref{eq:pure_drift_FPE_explanation} is useful from the optimal transport and gradient flow perspectives \cite{AGS2008GFs}, and has also been studied from a computational point of view \cite{MaoutsaReichManfred2020}.
    We refer to this as the \emph{pure-drift} interpretation.

\subsubsection{Pure-diffusion formulation}\label{subsubsec:pure_diff}
    Here, our goal is to find $\hat{a}$ with a suitable domain $\hat{\Pscr}_*$ such that $(u(t,\cdot))_{t>0}$ is a distributional solution to
    \begin{equation}\label{eq:pure_Ito_diffusion_FPE_explanation}
    \partial_t u
    =
    \Delta \big(\hat a(x,u)\,u\big).
    \end{equation} 
    This, however, is \emph{not} always possible.
    In view of \Cref{lem:many-PDEs_new}, the ansatz is to find $f$ such that
    \begin{equation}\label{eq:ansatz_f_pure_Ito_diffusion_FPE}
        \nabla f(x,u) = - b (x,u) u(x)
    \end{equation}
    weakly on $\{u>0\}$, with domain $ \mathcal{D}(f) \coloneqq \{u \in \Pscr_* : \, b(\cdot,u) u(\cdot) \in L^1_{\textup{loc}}(\R^d)\}$. This equation need not have a solution, since $b(\cdot,u)u(\cdot)$ need not be of gradient form.
    
    In dimension $d=1$, however, \eqref{eq:ansatz_f_pure_Ito_diffusion_FPE} can be solved by direct integration, provided that $b(\cdot,u)u(\cdot) \in L^1(\R)$.
    This leads to
    \begin{equation}
     \hat{a} (x,u) \coloneqq  a (x,u) + \frac{1}{u(x)} \int_{x}^{+\infty} b(y,u) u (y) d y ,
    \end{equation}
    defined on $\hat{\Pscr}_* \coloneqq 
    \{ u\in \Pscr_* : \, b(\cdot,u)u(\cdot) \in L^1(\R)
    \}$, with the same convention on $\{u = 0\}$ as before. Consequently, if all assumptions of \Cref{lem:many-PDEs_new} are satisfied, $(u(t,\cdot))_{t>0}$ is a distributional solution to \eqref{eq:pure_Ito_diffusion_FPE_explanation}. 
    
    For $d \geq 2$, a sufficient condition for the existence of a solution $f$ to \eqref{eq:ansatz_f_pure_Ito_diffusion_FPE} is that the vector field on the right-hand side be \emph{curl-free}, that is, $\textup{curl} (b(\cdot,u)u) = 0$ in the weak sense.
    This condition is satisfied in the case of radial symmetry, which appears in our applications. Indeed, suppose $u(x) = U (|x|)$ and $b (x,u) = \Psi (|x|,U) \tfrac{x}{|x|}$ for radial profiles $U: \R_+ \to \R_+$ and $\Psi(\cdot,U): \R_+ \to \R$ such that $\Psi (\cdot, U) U(\cdot) \in L^1(\R_+)$. 
    Then the vector field $b (\cdot,u) u$ is radial and of gradient form; in particular, its curl is zero in the weak sense. Accordingly, one can see that $f (x,u) \coloneqq \int_{|x|}^{+\infty} \Psi(\rho,U) U (\rho) d \rho $ is a candidate solution to \eqref{eq:ansatz_f_pure_Ito_diffusion_FPE}. This leads to
    \begin{equation}
         \hat{a} (x,u) \coloneqq  a (x,u) + \frac{1}{u(x)} \int_{|x|}^{+\infty} \Psi(\rho,U) U (\rho) d \rho ,
    \end{equation}
    defined on
    \begin{align}
       \hat{\Pscr}_* \coloneqq \big\{u \in \Pscr_*: \,\,\, \exists\, U:\R_+\to \R_+ \text{ such that } u(x) = U(|x|) \text{ and } \Psi (\cdot,U)U(\cdot) \in L^1(\R_+) \big\},
   \end{align}
   with the same convention on $\{u = 0\}$ as before.
   One then needs to verify the assumptions of \Cref{lem:many-PDEs_new}; then, it follows that $(u(t,\cdot))_{t>0}$ is a distributional solution to \eqref{eq:pure_Ito_diffusion_FPE_explanation}.
    We call this the \emph{pure-diffusion} interpretation.

\subsection{Step (2): From one nonlinear FPE to one MV-SDE}\label{subsect:step-2} 
    Suppose that for a weakly continuous solution $(u(t,\cdot))_{t>0} \subset \Pscr_{\textup{ac}}$ of a parabolic PDE  \eqref{eq:general-PDE} with initial condition $\zeta \in \Pscr$, we have already identified a nonlinear FPE of type \eqref{eq:Main_FPE_our_paper} solved by $(u(t,\cdot))_{t>0}$ in the sense of \Cref{def:FPE-sol}.
    The second step consists of applying the superposition principle (\Cref{thm:SP-pr}) in order to construct a weak solution to the corresponding MV-SDE  with one-dimensional time marginal densities $(u(t,\cdot))_{t>0}$. To this end, we assume additionally the integrability condition
        \begin{equation}\label{eq:integrability_our_paper}
        \int_0^T \int_{\R^d} \Big( |a(x, u(t,\cdot))| + |b(x, u(t,\cdot))|\Big) u(t,x) dx dt <+\infty, \qquad \forall T>0.
        \end{equation}
        Then, \Cref{thm:SP-pr} yields the existence of a weak solution $(X_t)_{t\geq 0}$ to the MV-SDE associated with \eqref{eq:general-PDE}, i.e.,
    \begin{equation}\label{eq:main_MV-SDE}
        \begin{dcases}
            dX_t = b(X_t,u(t,\cdot)) \, dt + \sqrt{2} \big( a(X_t,u(t,\cdot)) \big)^{\frac{1}{2}} \, dW_t, \\
            \mathcal{L}(X_t) = u(t,x)dx,\quad t>0,
        \end{dcases}
    \end{equation}
    with $\mathcal{L}(X_0) = \zeta$. As seen in Step (1), a curve $(u(t,\cdot))_{t>0}$ 
     may solve different FPEs. The following corollary shows that the
     equivalent reformulations of FPE \eqref{eq:Main_FPE_our_paper} in \Cref{lem:many-PDEs_new} lead to non-equivalent MV-SDEs and corresponding distinct solutions with the same one-dimensional time marginals $(u(t,\cdot))_{t>0}$.

    \begin{corollary}\label{cor:new}
        Let $(u(t,\cdot))_{t>0} \subseteq  \Pscr_* $ be a distributional solution to the FPE \eqref{eq:Main_FPE_our_paper} in the sense of \Cref{def:FPE-sol} with initial condition $\zeta \in \Pscr$ and coefficients $(a, b): \R^d \times \Pscr_* \to \R \times \R^d$ for some $\Pscr_* \subseteq \Pscr_{\textup{ac}} $, and assume further \eqref{eq:integrability_our_paper}. Let $f: \R^d \times \Dscr(f) \to \R $ with $\Dscr(f) \subseteq \Pscr_{\textup{ac}}$ be as in \Cref{lem:many-PDEs_new},
        and assume further
        \begin{equation}\label{eq:assumption_f_strong}
            \big[(t,x)\mapsto f(x,u(t,\cdot))\big] \in L^1_{\textup{loc}}(\R_+;W^{1,1}(\R^d)).
        \end{equation}
        Then there exists a weak solution $(X_t)_{t\geq 0}$ to 
        \begin{equation}
        \begin{dcases}
            dX_t = \left( b(X_t,u(t,\cdot)) + \frac{\nabla f(X_t,u(t,\cdot))}{u(t,X_t)} \right) dt + \sqrt{2}\left( a(X_t,u(t,\cdot)) + \frac{f(X_t,u(t,\cdot))}{u(t,X_t)} \right)^{\frac{1}{2}} dW_t, \\
            \mathcal{L}(X_t) = u(t,x)dx,\quad t>0,
        \end{dcases}
    \end{equation}
    with $\mathcal{L}(X_0) = \zeta$, with the same convention for the coefficients on $\{u=0\}$ as in \Cref{lem:many-PDEs_new}.
    \end{corollary}

    \begin{proof}
        By \Cref{lem:many-PDEs_new}, $(u(t,\cdot))_{t>0}$ solves the FPE with new coefficients $a + \frac{f}{u}$ and $b+ \frac{\nabla f}{u}$. The assumption \eqref{eq:assumption_f_strong} guarantees that the new coefficients also satisfy \eqref{eq:integrability_our_paper}.
        Thus, \Cref{thm:SP-pr} applies and yields the existence of a weak solution $(X_t)_{t\geq 0}$ to the MV-SDE above. 
    \end{proof}

    Clearly, the two MV-SDEs above are not equivalent, even though their corresponding FPEs are (in the sense that they are solved by the same curve of densities). In particular, by choosing $f$ as in \Cref{subsubsec:pure_drift,subsubsec:pure_diff}, one solves in \Cref{cor:new} \emph{pure-drift} and \emph{pure-diffusion} MV-SDEs, respectively.

    The bottomline of Steps (1) and (2) is the following. Starting from a solution $(u(t,\cdot))_{t>0}$ to \eqref{eq:general-PDE}, if we identify an FPE \eqref{eq:Main_FPE_our_paper} solved by $(u(t,\cdot))_{t>0}$, then by \Cref{lem:many-PDEs_new} we typically obtain (uncountably) many equivalent nonlinear FPEs also solved by $(u(t,\cdot))_{t>0}$. Each of these FPEs comes with its associated MV-SDE, for which we construct a weak solution whose one-dimensional time marginal densities are $(u(t,\cdot))_{t>0}$. Although these stochastic processes share the same one-dimensional time marginals, they are substantially different, and the equations they solve typically exhibit very different properties in terms of regularity and 
    singularities.

    In particular, for a \textit{family} of solutions to \eqref{eq:general-PDE}, 
    $\{u^z\}_{z\in \R^d}$, each with initial condition $\delta_z$, we construct uncountably many MV-SDEs and, for each of them, 
    a family of solutions $\{X^z\}_{z\in \R^d}$ with 
    $\mathcal{L}(X^z_t) = u^z(t,x)dx$ for all $t>0$ and $\mathcal{L}(X^z_0) = \delta_z$. 
    In the next step, we investigate whether each of these solution families is uniquely determined by $(u^z(t,\cdot))_{t>0}$ together with its 
    MV-SDE and, moreover, whether each of these families constitutes a nonlinear 
    Markov core.

\subsection{Step (3): From one MV-SDE to one nonlinear Markov core}\label{subsect:step-3}
One of the main results of this paper is the following theorem, which generalizes \Cref{thm:NL-MP-construction-old} in a suitable way in order to implement Step (3) of our scheme. Given a measure-valued curve $\mu \coloneqq (\mu_t)_{t\geq 0}$ and $s \geq 0$, we write $\mu_{s+} \coloneqq (\mu_{s+t})_{t\geq 0}$.

\begin{theorem}\label{thm:Markov-construction}
    For each $z \in \R^d$, let $\mu^z = (\mu^z_t)_{t\geq 0}$ be a solution to the FPE \eqref{eq:FPE} with initial datum $\mu^z_0 = \delta_z$, satisfying \eqref{eq:L1-int-SP}, such that
    \begin{enumerate}[label=(\roman*), font=\normalfont]
        \item\label{itm:extm_1} $(t,z) \mapsto \mu^z_t$ is injective from $\R_+\times \R^d$ to  $\Pscr$;

        \item\label{itm:extm_2} for each $z \in \R^d$ and $s \in \R_+ \setminus \{0\}$, $\mu^z_{s+}$ is an extreme point in the set of solutions to ($\mu^z_{s+}$-linFPE) (in the sense of \Cref{def:lFPE-sol}) with initial condition $\mu^z_s$.
    \end{enumerate}
    Then, for each $z \in \R^d$, there exists a unique solution path law $P^z$ to the MV-SDE \eqref{eq:DDSDE} with one-dimensional time marginals $P^z_t = \mu^z_t$, $t \geq 0$. $\{P^z\}_{z\in \R^d}$ is a nonlinear Markov core for $\{\mu^z\}_{z\in \R^d}$, which is uniquely determined by $\{\mu^z\}_{z \in \R^d}$ and \eqref{eq:DDSDE}.
\end{theorem}

\begin{remark}  We emphasize the following two points:
\begin{itemize}
    \item[(i)] No regularity assumptions are imposed on the coefficients of \eqref{eq:FPE}. Hence, the theorem applies in the Nemytskii case \eqref{eq:nlPDE}. 
    \item[(ii)] The extremality condition \ref{itm:extm_2} (equivalently: the restricted linearized uniqueness condition from \Cref{lem:extrem-unique-equiv} below) is required only for \emph{strictly} positive $s$. In our applications, $\mu^z$ are explicitly given function-valued solutions to a nonlinear parabolic PDE with initial condition $\delta_z$. Typically, such solutions are quite regular for strictly positive times. Hence, we only need to verify the uniqueness condition for the linearized equation from \Cref{lem:extrem-unique-equiv} for regular initial data and thus avoid singularities in the linearized coefficients.
\end{itemize}
\end{remark}

\begin{proof}[Proof of \Cref{thm:Markov-construction}]
Let $\Pscr_0 := \{\mu^z_s \,:\, s \geq 0, z \in \R^d\}$. Assumption \ref{itm:extm_1} implies that for $\zeta \in \Pscr_0$, there is a unique pair $(s,z)$ such that $\zeta = \mu^z_s$. Hence for $s\geq 0$ and $\zeta \in \Pscr_0$ with $\zeta = \mu^z_s$, $\mu^\zeta_t := \mu^z_{t+s}$ is well-defined for $t\geq 0$. Then, for $\zeta = \mu^z_s$,
$$\mu^\zeta_{r+t} = \mu^z_{r+t+s} = \mu^{\mu^z_{r+s}}_t = \mu^{\mu^\zeta_r}_t,\quad \forall r,t \geq 0,$$
i.e. $\{\mu^\zeta\}_{\zeta \in \Pscr_0}$, $\mu^\zeta := (\mu_t^\zeta)_{t\geq 0}$, satisfies the flow property. Moreover, setting $\mathfrak{P}_0 := \{\mu^z_s\,:\, z\in \R^d, s >0\}$ (so that $\Pscr_0 \setminus \mathfrak{P}_0 = \{\delta_z\,:\, z \in \R^d\}$), clearly
$$\mu^\zeta_t \in \mathfrak{P}_0 \text{ for all }(\zeta, t) \in \Pscr_0 \times [0,\infty) \text{ such that either }\zeta \in \mathfrak{P}_0 \text{ or } t>0.$$ Also, it is clear that $\mu^\zeta$ is a solution to \eqref{eq:FPE} with initial condition $\zeta$, and that \eqref{eq:L1-int-SP} is satisfied with $\mu_t$ replaced by $\mu_t^\zeta$ for each $\zeta \in \Pscr_0$. By assumption \ref{itm:extm_2} and since $\mu^z_{s+} = \mu^{\mu^z_s}$, i.e. $\{\mu^z_{s+}\}_{z\in \R^d, s >0} = \{\mu^\zeta\}_{\zeta \in \mathfrak{P}_0}$, it follows that we can apply \cite[Cor.3.10]{RehmeierRoeckner2025NonlinearMarkov} (in the time-homogeneous setting). We thus obtain a nonlinear Markov process $\{P^\zeta\}_{\zeta \in \Pscr_0}$, consisting of solution path laws to the corresponding MV-SDE \eqref{eq:DDSDE} such that, for $\zeta = \mu^z_s$,
\begin{equation}\label{eq:marginals-proof}
P^\zeta_t = \mu^\zeta_t = \mu^{\mu^z_s}_t = \mu^z_{t+s},\quad \forall t \geq 0.
\end{equation}
Moreover, for each $\zeta \in \mathfrak{P}_0$, $P^\zeta$ is uniquely determined by \eqref{eq:DDSDE} and its one-dimensional time marginals \eqref{eq:marginals-proof}.
\vspace{0.3em}

\begin{itemize}[itemsep=0.3em, leftmargin=1.5em] 
    	\item[] \textit{Claim: $P^\zeta$ is uniquely determined by \eqref{eq:DDSDE} and \eqref{eq:marginals-proof} for \emph{every} $\zeta \in \Pscr_0$.}
        
        \item[] \textit{Proof of Claim:}
        Since $\Pscr_0 \setminus \mathfrak{P}_0= \{\delta_z\,:\, z \in \R^d\},$ it suffices to consider $\zeta = \delta_z$ for arbitrary $z\in \R^d$. Suppose $P^1,P^2$ are solution path laws to \eqref{eq:DDSDE} with $P^i_t = \mu^z_t$ for all $t\geq 0$ and $i \in \{1,2\}$. We prove the claim by showing
        \begin{equation}\label{eq:proof-findim-equal}
            P^1_{t_1,\dots,t_n} = P^2_{t_1,\dots,t_n},\quad \forall 0\leq t_1 < \dots < t_n < \infty, n \in \N,
        \end{equation}
        where $P_{t_1,\dots,t_n} := P \circ (\pi_{t_1},\dots,\pi_{t_n})^{-1}$ denotes the finite-dimensional time marginal of a path law $P$ at times $0\leq t_1 < \dots < t_n$. Note first that for $i \in \{1,2\}$ and any $s>0$, $P^i \circ (\Pi_s)^{-1}$ is a solution path law to \eqref{eq:DDSDE} with initial condition $P^i_s = \mu^z_s \in \mathfrak{P}_0$ and one-dimensional time marginals $(P^i \circ (\Pi_s)^{-1})_t = \mu^z_{t+s}$, $t\geq 0$. Thus, if $t_1>0$, then $P^i_{t_1,\dots,t_n} = (P^i \circ \Pi_{t_1}^{-1})_{0,t_2-t_1,\dots,t_n-t_1}$, and we already know from the previous part of the proof that the right-hand sides of the last equality coincide for $i=1$ and $i=2$. If $t_1 = 0$, since $P^i_0 = \delta_z$, we find
        $$P^i_{t_1,\dots,t_n} = \delta_z \otimes P^i_{t_2,\dots,t_n},$$
        and we conclude as in the first case, using $t_2>0$. Thus \eqref{eq:proof-findim-equal} follows, and the claim is proven. 
\end{itemize}

The proof of the claim also shows that if $\zeta = \mu^z_s$, then $P^\zeta = P^z \circ (\Pi_s)^{-1}$. Hence $\{P^z\}_{z\in \R^d}$ is indeed a nonlinear Markov core for $\{\mu^z\}_{z\in \R^d}$, which concludes the proof of the theorem.
\end{proof}

In applications, the following lemma from \cite{RehmeierRoeckner2025NonlinearMarkov} is very useful for verifying the extremality condition \ref{itm:extm_2} from \Cref{thm:Markov-construction}.
\begin{lemma}[{\cite[Lemma 3.5]{RehmeierRoeckner2025NonlinearMarkov}}]\label{lem:extrem-unique-equiv}
  Let $(z,s) \in \R^d \times \R_+\setminus \{0\}$. In the setting of \Cref{thm:Markov-construction}, $\mu^z_{s+}$ is an extreme point in the set of all solutions to ($\mu^z_{s+}$-linFPE) with initial condition $\mu^z_s$ if and only if $\mu^z_{s+}$ is the unique solution to ($\mu^z_{s+}$-linFPE) with initial condition $\mu^z_s$ in 
    \begin{equation}\label{eq:set1}
        \big\{(\nu_t)_{t\geq 0} \in C(\R_+;\Pscr): \nu_t \leq C \mu^z_{s+t}\,\forall t \geq 0 \text{ for some }C\geq 1 \text{ independent of }t\big\},
    \end{equation}
    where the inequality is understood set-wise, i.e. $\nu_t(A) \leq C \mu^z_{s+t}(A)$ for all $A \in \Bscr(\R^d)$.
\end{lemma}

Finally, we apply \Cref{thm:Markov-construction} and \Cref{lem:extrem-unique-equiv} in the framework of Steps (1) and (2), i.e., when $(\mu^z_t)_{t>0} = (u^z(t,x)dx)_{t>0}$ is a solution to the FPE \eqref{eq:Main_FPE_our_paper} with initial condition $\delta_z$. This results in the following corollary.

\begin{corollary}\label{crl:nl_markov}
    Let $\{u^z\}_{z\in\R^d}$ be such that each $u^z\coloneqq (u^z(t,\cdot))_{t>0}$ is a solution to the FPE \eqref{eq:Main_FPE_our_paper} in the sense of \Cref{def:FPE-sol} with initial condition $\delta_z$, satisfying the integrability condition \eqref{eq:integrability_our_paper} with $u^z$ replacing $u$. Assume further injectivity of $(t,z)\mapsto u^z(t,x)dx$ and that, for every $(s,z)\in (0,\infty) \times \R^d$, $u^z_{s+} \coloneqq (u^z({s+t},\cdot))_{t\geq 0} $ is the unique solution to the linearized FPE
    \begin{equation}\label{eq4_new}
        \begin{dcases}
        \partial_t v(t,x) = \Delta\Big( a\big(x, u^z(s+t,\cdot)\big)\, v(t,x)\Big) - \nabla \cdot \Big( b\big(x, u^z(s+t,\cdot)\big) \, v(t,x)\Big), \\
        v(0,\cdot) = u^z(s,\cdot),
        \end{dcases}
    \end{equation}
    in 
     \begin{equation}\label{eq:set2}
     \big\{v \geq 0: [t\mapsto v(t,x)dx] \in C(\R_+;\Pscr),  v(t,x) \leq Cu^z(s+t,x)\,\text{ a.s. for some }C\geq 1 \big\}.
     \end{equation}
     Then there exists a nonlinear Markov core $\{P^z\}_{z\in \R^d}$ for $\{u^z\}_{z\in \R^d}$ in the sense of \Cref{def:NMC}, consisting of the unique solution path laws to \eqref{eq:main_MV-SDE} with $P^z_t = u^z(t,x)dx$ and $P^z_0 = \delta_z$. In particular, the induced nonlinear Markov process $\{P^\zeta\}_{\zeta \in \Pscr_0}$, $\Pscr_0 := \{u^z(t,x)dx| z \in \R^d, t \geq 0\}$, is uniquely determined by $\{u^z\}_{z \in \R^d}$ and \eqref{eq:main_MV-SDE}.
\end{corollary}

Note that the sets in \eqref{eq:set1} and \eqref{eq:set2} indeed coincide when $\mu^z_t = u^z(t,x)dx$, since in this case any element $\nu$ in \eqref{eq:set1} consists of densities $\nu_t = v(t,x)dx$. Hence, restricting to this set of solutions, the linearized FPE ($\mu^z_{s+}$-linFPE) can indeed be reformulated as an equation for densities, as done in \eqref{eq4_new}.

In particular, combined with \Cref{lem:many-PDEs_new}, the previous theorem allows us to construct uncountable families of distinct nonlinear Markov processes with identical one-dimensional time marginal densities $(u^z(t,\cdot))_{t>0}$. Each of these nonlinear Markov processes consists of solution path laws to a MV-SDE and is uniquely determined by the latter and $\{u^z\}_{z\in \R^d}$. 

\section{Heat equation}\label{sec:heat_equation}
As a first application, we consider the classical 
heat equation
\begin{align}\label{eq:heat}
        \partial_t u = \frac{1}{2}\Delta u, \qquad (t,x)\in   (0, +\infty) \times \mathbb{R}^d,
    \end{align}
where $d \in \N$. The family $\{ u^z\}_{z \in \mathbb{R}^d}$ of fundamental solutions to \eqref{eq:heat} with initial condition $u^z(t,\cdot)\,dx \to \delta_z$ weakly as $t \downarrow 0$
is given by the well-known heat kernel
\begin{equation}\label{eq:fundamental_solution_heat}
    u^z (t,x) = \frac{1}{(2\pi t)^{\frac d 2}}\exp\bigg(-\frac {|x-z|^2}{2t}\bigg), \qquad (t,x) \in (0,+\infty) \times \R^d.
\end{equation}
Let us recall their self-similar form
\begin{equation}\label{eq:uz_self_similar_heat}
    u^z(t,x)=t^{-\frac d2}\, g \big(t^{-\frac12}|x-z|\big),
\end{equation}
with $g:\R_+\to\R_+$ given by $g(\xi) \coloneqq (2\pi)^{-\frac d 2} \exp (-\tfrac{\xi^2}{2} )$. 
Moreover, a simple calculation gives
\begin{equation}\label{eq:grad_heat}
    \nabla u^z(t,x)
    = -\frac{x-z}{t}\,u^z(t,x).
\end{equation}

The goal for this section is to study several interpretations of \eqref{eq:heat} as a nonlinear FPE and their associated MV-SDEs, as well as the construction of the corresponding nonlinear Markov cores for $\{u^z\}_{z\in \R^d}$. 
For the sake of completeness, we first recall the usual linear interpretation, for which the associated stochastic process is Brownian motion. We then consider the pure-drift interpretation, as well as interpolations between these two cases and, finally, even more general interpretations beyond such pure-drift and pure-diffusion interpolations.

\smallskip
\noindent
\subsection{Pure-diffusion interpretation}
The heat equation \eqref{eq:heat} is obviously a linear Fokker--Planck with coefficients 
\[
a = \frac{1}{2}, \qquad b =0.
\] 
whose associated SDE is, of course,
\begin{equation}\label{eq:HE-SDE-pure-diff}
        dX_t = dW_t.
\end{equation}
With initial condition $X_0=z$, it has unique solutions $X^z_t=W_t+z$. We set $\mathbb{W}^z:= \mathcal{L}(X^z)$, or put differently, $\mathbb{W}^z = \mathbb{W}^0 \circ T_z^{-1}$, where $\mathbb{W}^0$ is the classical Wiener measure and $T_z: \mathcal{C}\to \mathcal{C}$, $T_z w := z + w$ denotes translation by $z$.

For illustrative reasons, we explicitly verify below that $\{\mathbb{W}^z\}_{z\in \R^d}$ satisfies the nonlinear Markov property. Alternatively, this can be proven by applying \Cref{lemma:LMP_NMC} or \Cref{thm:Markov-construction} (the latter in conjunction with  \Cref{lem:extrem-unique-equiv}).

\begin{proposition}\label{lemaux}
   $\{\mathbb{W}^z\}_{z\in \R^d}$ is a nonlinear Markov core for $\{u^z\}_{z\in \R^d}$.
\end{proposition}
\begin{proof}
We set $\mu_t^z := u^z(t,x)dx$ for $t>0$ and $\mu^z_0 = \delta_z$.
Of course $\mathbb{W}^z_t = \mu^z_t$, and $(t,z)\mapsto \mu^z_t$ is injective. Hence, with regard to Definition \ref{def:NMC}, it remains to show that
$$\{P^{\mu^z_t}\}_{t\geq 0, z\in \R^d},\quad P^{\mu^z_t}:=\mathbb{W}^z \circ \Pi_t^{-1},$$
is a nonlinear Markov process. To this end, we verify Definition \ref{def:NL-MP} by hand, where $\mathcal{P}_0=\{\mu^z_t: t \geq 0, z\in \R^d\}$. Note that for $A\in \mathcal{B}(\R^d)$, by definition,
    \[
    P^{\mu^z_t}_0(A) =\mu^z_t(A)
    \]
    and, similarly, for $s>0$
    \[
    P^{\mu^z_t}_s=\big(\mathbb{W}^z\circ \Pi_t^{-1} \big) \circ \pi_s{-1}=\mu^z_{s+t}\in \mathcal{P}_0,
    \]
    i.e., (i) of \Cref{def:NL-MP} is satisfied.  Regarding (ii), let $s,r,t\geq 0$ and consider the disintegration family $\{P^{\mu^z_{s+r}, y}\}_{y\in \R^d}$ of $P^{\mu^z_{s+r}}$ with respect to $\pi_0$, i.e., the $\mu^z_{s+r}$-a.s. unique family of probability measures on $\Bscr(\Cscr)$ such that for any $B\in \mathcal{B}(\Cscr)$
    \[
    P^{\mu^z_{s+r}}(B)=\int_{\R^d} P^{\mu^z_{s+r}, y}(B) d\mu_{s+r}^z(y).
    \]
    The classical Markov property for $\{\mathbb{W}^z\}_{z\in \R^d}$ implies, for $\mu^z_{s+r}$-a.e. $y\in \R^d$,
    \[\label{aaux}
     P^{\mu^z_{s+r}, y}(B) =\mathbb{W}^y(B)
    \]
   hence, for any $A\in \Bscr(\R^d)$,
   \[\label{bbux}
     P^{\mu^z_{s+r}, \pi_s(\cdot)}(\pi_t\in A)= \mathbb{W}^{\pi_s(\cdot)}(\pi_t\in A),\quad P^{\mu^z_{r}}-\text{a.s.}
    \]
    Next, let $g: \Cscr \to \R_+$ be bounded and $\mathcal{F}_s$-measurable. Then, by the classical Markov property and \eqref{aaux}
    \begin{equation}
        \begin{split}
            \int_\Cscr \mathds{1}_{\{\pi_{t+s}\in A\}}(\omega)g(\omega)dP^{\mu^z_r}(\omega)&= \int_{\R^d}\int_\Cscr \mathds{1}_{\{\pi_{t+s}\in A\}}(\omega)g(\omega)d\mathbb{W}^{y}(\omega)d\mu^z_r(y)\\
            &=\int_{\R^d} \int_\Cscr\mathbb{W}^{\pi_s(\omega)}(\pi_t\in A)g(\omega)d\mathbb{W}^{y}(\omega)d\mu^z_r(y)\\
            &=\int_\Cscr \mathbb{W}^{\pi_s(\omega)}(\pi_t\in A)g(\omega)dP^{\mu^z_r}(\omega),
        \end{split}
    \end{equation}
    from which we infer
    \[ 
    P^{\mu^z_r}(\pi_{t+s}\in A|\mathcal{F}_s)(\cdot)=\mathbb{W}^{\pi_s(\cdot)}(\pi_t\in A),\quad P^{\mu^z_r}\text{-a.s.}
    \]
    which together with \eqref{bbux} shows that  $\{P^{\mu^z_r}\}_{r\geq 0, z\in \R^d}$ is a nonlinear Markov process, and hence $\{\mathbb{W}^z\}_{z\in \R^d}$ is a nonlinear Markov core for $\{u^z\}_{z\in \R^d}$, as claimed.
\end{proof}

\subsection{Pure-drift interpretation}\label{subsec:heat_equation_drift}
Next, we recast the heat equation \eqref{eq:heat} as
\begin{equation}\label{eq:HE-divergenceform}
    \partial_t u = -\nabla \cdot \bigg(\frac{-\nabla u}{2u} \,u\bigg),
\end{equation}
i.e., as a nonlinear first-order FPE with coefficients $(a,b) : \R^d \times \Pscr_* \to \R_+ \times \R^d $ given by
\begin{equation}\label{eq:coefficients_heat-drift}
a(x,u) \coloneqq 0,\qquad  b(x,u) \coloneqq -\frac{\nabla u (x)}{2u(x)},
\end{equation}
with $b(x,u) \coloneqq 0 $ on $\{u=0\}$, defined on $\Pscr_{*} \coloneqq  \Pscr_{\textup{ac}} \cap W^{1,1}_{\textup{loc}}(\R^d)$.
Note that $u^z$ is also a distributional solution to \eqref{eq:HE-divergenceform}. This can be seen, for instance, via \Cref{lem:many-PDEs_new} with the choice $a = \frac{1}{2}, b= 0$, $\Pscr_*=\Pscr $, and $f(u) = - \frac 1 2 u$ on $\mathcal{D}(f) := \Pscr_{\textup{ac}}$. 
With this choice, one obtains new coefficients  \eqref{eq:coefficients_heat-drift} and the domain $\Pscr_{*} = \Pscr_{\textup{ac}} \cap W^{1,1}_{\textup{loc}}(\R^d)$, and, using \eqref{eq:grad_heat}, one can check that $\big[(t,x)\mapsto f(x,u(t,\cdot))\big] \in L^1_{\textup{loc}}(\R_+;W^{1,1}(\R^d))$. Therefore, the assumptions of \Cref{lem:many-PDEs_new} (and \Cref{cor:new}) are all satisfied. 
\\
The MV-SDE corresponding to \eqref{eq:HE-divergenceform} is the distribution-dependent ODE
\begin{equation}\label{eq:HE-ODE}
    \begin{dcases}
    dX_t = -\frac{\nabla u(t,X_t)}{2u(t,X_t)}dt,\\ \mathcal{L}(X_t) = u(t,x)dx \,\,\,\forall t >0.
    \end{dcases}
\end{equation}
Specifically, for $u = u^z $, a direct calculation shows that \eqref{eq:HE-ODE} simplifies to
\begin{equation}\label{eq:HE-ODE_simplification}
    \begin{dcases}
    dX_t = \frac{X_t-z}{2t}dt, \\
    \mathcal{L}(X_t) = u^z(t,x)dx\,\,\,\forall t >0.
     \end{dcases}
\end{equation}
Using separation of variables, an explicit solution to \eqref{eq:HE-ODE_simplification} with $\mathcal{L}(X_0) = \delta_z$ is given by $(X^z_t)_{t\geq 0}$ as follows
\begin{equation}\label{eq:solutions_ODE_heat}
    X^z_t \coloneqq z+\eta\sqrt{t}, 
\end{equation}
where $\eta$ is an $\R^d$-valued random variable  with $\mathcal{L}(\eta)=\mathcal{N}(0,1_d)$.

\begin{proposition}\label{prop:Mc-heateq}
    Let $\mathbb{V}^z:=\mathcal{L}(X^z)$, where $X^z$ is given in \eqref{eq:solutions_ODE_heat} with $\eta$ being an $\R^d$-valued random variable with $\mathcal{L}(\eta)= \mathcal{N}(0,1_d)$. Then $\{\mathbb{V}^z\}_{z\in \R^d}$ is a nonlinear Markov core for $\{u^z\}_{z\in \R^d}$.
\end{proposition}
Again, we verify the nonlinear Markov property directly for illustrative purposes. Alternatively, one may verify the conditions from Theorem \ref{thm:Markov-construction} to obtain the same assertion by applying the latter. Note that our direct calculation below does not prove that $\{\mathbb{V}^z\}_{z\in \R^d}$ is the \emph{unique} nonlinear Markov core for $\{u^z\}_{z\in \R^d}$ consisting of solution path laws to \eqref{eq:HE-ODE_simplification}. For this, one either has to argue that solutions to \eqref{eq:HE-ODE_simplification} are weakly unique or one applies Theorem \ref{thm:Markov-construction}, which includes this uniqueness claim.
\begin{proof}[Proof of \Cref{prop:Mc-heateq}]
Let $\mathcal{P}_0=\{\mu^z_t\}_{t\geq 0, z\in \R^d}$, where $\mu^z_t := u^z(t,x)dx$ for $t>0$ and $\mu^z_0 = \delta_z$. Define
\[
\Phi^z_t:\R^d\to \Cscr, \qquad c\to \gamma^z_{c,t}, \,\, \gamma^z_{c,t}(r):= z+c\sqrt{t+r}
\]
By definition $\mathbb{V}^z=\mathcal{L}(\eta)\circ (\Phi^z_0)^{-1}$, and $\mathbb{V}^z \circ \Pi_t^{-1}= \mathcal{L}(\eta) \circ (\Phi^z_t)^{-1}$. Defining $P^{\mu^z_t}:=\mathbb{V}^z\circ \Pi_t^{-1}$, we have for $A\in \mathcal{B}(\R^d)$
\[
P^{\mu^z_t}_0(A) =\mu^z_t(A),
\]
and, similarly,
\[
P^{\mu^z_t}_s(A)=\mu^z_{s+t}(A),
\]
so the condition (i) of \Cref{def:NL-MP} holds. For $B\in \Bscr(\Cscr)$, 
define the Borel set
$$A(B)^z_{s+r} := A^z_{s+r}\coloneqq\{c\in \R^d \ | \ \exists \omega\in B : \omega=\gamma^z_{c,s+r}\} = (\Phi^z_{s+r})^{-1}(B)$$
Then 
\begin{multline}
   P^{\mu^z_{s+r}}(B)=  \mathcal{L}(\eta)((\Phi^z_{s+r})^{-1}(B))=\mathcal{L}(\eta)(A^z_{s+r}) \\ =\int_{\R^d} \mathds{1}_{A^z_{s+r}}(y)\,d\mu^0_1(y) = \int_{\R^d} \mathds{1}_{A^z_{s+r}}\left(\frac{y-z}{\sqrt{s+r}}\right)d\mu^z_{s+r}(y).
\end{multline}
By definition 
\begin{equation}
    \mathds{1}_{A^z_{s+r}}\left(\frac{y-z}{\sqrt{s+r}}\right)= \delta_{\gamma^z_{\frac{(y-z)}{\sqrt{s+r}},s+r}}(B)  , 
\end{equation}
and therefore, denoting by $\{ P^{\mu^z_{s+r}, y}\}_{y\in \R^d}$ the disintegration family of $P^{\mu^z_{s+r}}$ with respect to $\pi_0$, we conclude $P^{\mu^z_{s+r}, y}=\delta_{\gamma^z_{\frac{(y-z)}{\sqrt{s+r}},s+r}}$ for $\mu^z_{s+r}$-a.e. $y$. Thus,
\[
P^{\mu^z_{s+r}, \pi_s(\cdot)}(\pi_t\in A)=\delta_{\gamma^z_{\frac{(\pi_s(\cdot)-z)}{\sqrt{s+r}},s+r}}(\pi_t\in A)=\mathds{1}_A\left(z+\frac{(\pi_s(\cdot)-z)}{\sqrt{s+r}}\sqrt{s+r+t}\right),\quad P^{\mu^z_r}\text{-a.s.}
\]
It remains to show
\[
P^{\mu_r^z}(\pi_{t+s}\in A|\mathcal{F}_s)=\mathds{1}_A\left(z+\frac{(\pi_s(\cdot)-z)}{\sqrt{s+r}}\sqrt{s+r+t}\right),\quad P^{\mu^z_r}\text{-a.s.}
\]
That is, for any nonnegative bounded $\Fscr_s$-measurable $g$, we need
\[
\int_\mathcal{C}\mathds{1}_A(\pi_{t+s}(\omega))g(\omega)P^{\mu^z_r}d(\omega)=\int_\mathcal{C}\mathds{1}_A\left(z+\frac{(\pi_s(\omega)-z)}{\sqrt{s+r}}\sqrt{s+r+t}\right) g(\omega)P^{\mu^z_r}d(\omega).
\]
By definition, $P^{\mu^z_r}$ is concentrated on 
\[
S^z_r:=\{\omega\in \mathcal{C}| \ \exists c(\omega)\in \R^d: \pi_s(\omega)=z+c(\omega)\sqrt{r+s}\  \text{for all}\ s\geq0 \}.
\]
By definition of $S^z_r$, for every $\omega\in S^z_r$, there exists $c(\omega)\in \R^d$ such that
\[
\pi_s(\omega)=z+c(\omega)\sqrt{s+r}, \qquad \pi_{s+t}(\omega)=z+c(\omega)\sqrt{s+t+r}.
\]
In particular, we have 
\[
\pi_{s+t}(\omega)=z+\frac{\pi_s(\omega)-z}{\sqrt{s+r}}\sqrt{s+r+t}
\]
and, therefore,
\begin{equation}
    \begin{split}
       \int_\mathcal{C}\mathds{1}_A(\pi_{t+s}(\omega))g(\omega)dP^{\mu^z_r}(\omega)&=\int_{S^z_r}\mathds{1}_A(\pi_{t+s}(\omega))g(\omega)dP^{\mu^z_r}(\omega) \\
       &=\int_{S^z_r}\mathds{1}_A(z+\frac{\pi_s(\omega)-z}{\sqrt{s+r}}\sqrt{s+r+t})g(\omega)dP^{\mu^z_r}(\omega) \\
       &=\int_{\mathcal{C}}\mathds{1}_A(z+\frac{\pi_s(\omega)-z}{\sqrt{s+r}}\sqrt{s+r+t})g(\omega)dP^{\mu^z_r}(\omega). \qedhere
    \end{split}
\end{equation}
\end{proof}

By \Cref{lemaux} and \Cref{prop:Mc-heateq}, we constructed nonlinear Markov cores $\{\mathbb{W}^z\}_{z\in \R^d}$ and $\{\mathbb{V}^z\}_{z\in \R^d}$ for the heat kernel $\{u^z\}_{z\in \R^d}$, i.e., in particular
$$\mathbb{W}^z_t = u^z(t,x)dx = \mathbb{V}^z_t,\quad \forall z\in \R^d, t \geq 0.$$
Note that $\mathbb{W}^z$ and $\mathbb{V}^z$ are solution path laws to a pure-diffusion SDE and a distribution-dependent ODE, respectively, namely to \eqref{eq:HE-SDE-pure-diff} and \eqref{eq:HE-ODE}. Clearly, for each $z \in \R^d$, the path laws $\mathbb{W}^z$ and  $\mathbb{V}^z$ from \Cref{lemaux} and \Cref{prop:Mc-heateq} are mutually singular. Indeed, for $\varepsilon >0$, $\mathbb{V}^z$ is concentrated on solution trajectories to \eqref{eq:HE-ODE_simplification}, i.e., on the measurable set $A:=\{w\in \Cscr\,: \, w(t) = c \sqrt{t} + z, \, c \in \R^d\} \subseteq C^{1}((0, \infty); \R^d)$. On the other hand, of course $\mathbb{W}^z\big(
 A) = 0$, since $\mathbb{W}^z(A) >0$ would contradict the well-known fact that Wiener measure is concentrated on non-differentiable paths.

In particular, this shows that nonlinear Markov processes are not uniquely determined by their one-dimensional time marginals. Next, we construct further nonlinear Markov cores for $\{u^z\}_{z\in \R^d}$ by interpolating between the pure-drift and pure-diffusion cases.


\subsection{Interpolation between pure-diffusion and pure-drift cases}\label{subsec:heat_interpolation}
Here we recast \eqref{eq:heat} as
\begin{equation}\label{eq:HE-beta}
    \partial_t u = \frac{\beta}{2}\Delta u - (1-\beta)\nabla \cdot \bigg(\frac{-\nabla u}{2u} \,u\bigg),
\end{equation}
with interpolation parameter $\beta \in (0, \infty)$. This is a nonlinear Nemytskii-type FPE  with coefficients  $(a,b) : \R^d \times \Pscr_* \to \R_+ \times \R^d $ given by
\begin{equation}\label{eq:HE_beta_coeff}
    a (x,u) \coloneqq \frac{\beta}{2}, \qquad b(x,u)\coloneqq(1-\beta)\frac{-\nabla u}{2u}(x)
\end{equation}
and $b(x,u) \coloneqq 0 $ on $\{u=0\}$, defined on $\Pscr_{*} \coloneqq  \Pscr_{\textup{ac}} \cap W^{1,1}_{\textup{loc}}(\R^d)$.
Clearly, $\beta = 1$ and $\beta = 0$ correspond to the previously considered pure-diffusion and pure-drift cases, respectively. The associated MV-SDE for \eqref{eq:HE-beta} is
\begin{equation}\label{eq:SDE-beta}
    \begin{dcases}
    dX_t =-(1-\beta)\frac{\nabla u(t, X_t)}{2u(t, X_t)} dt + \sqrt{\beta}dW_t,\\
    \mathcal{L}(X_t) = u(t,x)dx,\,\,\, t >0.
    \end{dcases}
\end{equation}
For $u = u^z$, this equation simplifies to
\begin{equation}\label{eq:SDE-beta-HE}
    \begin{dcases}
    dX_t = (1-\beta) \frac{X_t-z}{2t} d t + \sqrt{\beta}dW_t, \\ \mathcal{L}(X_t) = u^z(t,x)dx,\,\,\, t >0,
    \end{dcases}
\end{equation}
which, with initial condition $\mathcal{L}(X_0) = \delta_z$, can be solved via variation of constants:
\begin{equation}\label{Xbeta}
    X^{\beta, z}_t \coloneqq z+\sqrt{\beta}\int_0^t\left(\frac{t}{s}\right)^{\frac{1-\beta}{2}}dW_s.
\end{equation}

\begin{proposition}\label{prop:HE-infin-many}
    Let $\beta \in (0,\infty)$. 
    Let $\mathbb{V}^{\beta,z}:=\mathcal{L}(X^{\beta, z})$, where $X^{\beta, z}$ as in \eqref{Xbeta}. Then $\{\mathbb{V}^{\beta,z}\}_{z\in \R^d}$ is a nonlinear Markov core for $\{u^z\}_{z\in \R^d}$, and it is uniquely determined by $\{u^z\}_{z\in \R^d}$ and \eqref{eq:SDE-beta-HE}.
    \label{Xbeta nonlinear core}
\end{proposition}

\begin{proof}
It is immediate that $u^z$ also solves the interpolated FPE \eqref{eq:HE-beta}, since it satisfies both the pure-drift equation \eqref{eq:HE-divergenceform} (as argued there) and the pure-diffusion equation (the heat equation). 
We want to apply Theorem \ref{thm:Markov-construction} together with Lemma \ref{lem:extrem-unique-equiv}.
    Let $z \in \R^d$ and $s>0$. We need to show that the linear FPE
    \begin{equation}\label{eq:linFPE-beta}
    \partial_t \mu_t =  \frac \beta 2 \Delta \mu_t - \frac {1-\beta} {2} \nabla \cdot\bigg( \frac{x-z}{t+s} \mu_t\bigg),\quad (t,x) \in (0,\infty)\times \R^d,\quad \mu_0 = u^z(s,y)dy,
\end{equation}
has a unique distributional solution in 
$$\{(\mu_t)_{t\geq 0} \in C(\R_+;\Pscr)\,:\, \mu_t \leq C u^z(t+s,y)dy, \, \forall t \geq 0, \text{ for some }C\geq 1\text{ independent from }(t,y)\}$$
(this solution is $\mu_t = u^z(t,y)dy$).
Clearly, the coefficients of \eqref{eq:linFPE-beta}, 
$$a(t,x) \coloneqq \frac \beta 2, \quad b(t,x) \coloneqq \frac{1-\beta}{2} \frac{x-z}{t+s},$$
are uniformly in $t\geq 0$ Lipschitz continuous in $x$ and bounded on each $\R_+\times K$ for any compact set $K \subseteq \R^d$ (here we use $s>0$). Thus, since $a$ is a strictly positive constant, the required uniqueness follows by classical theory, for instance from \cite[Theorem 9.4.3]{BogachevKrylovRoecknerShaposhnikov2015}.
 Now Lemma \ref{lem:extrem-unique-equiv} and Theorem \ref{thm:Markov-construction} apply and yield the assertion.
\end{proof}

The next lemma shows that, for each $z \in \R^d$, the family $\{\mathbb{V}^{\beta,z}\}_{\beta >0}$ consists of pairwise mutually singular path laws. In particular, this rules out that these path laws are related by a simple Girsanov transformation.

\begin{lemma}\label{lem:singularlaws}
    Let $z\in \R^d$,  $\beta_1, \beta_2\in (0, \infty)$, and let $\mathbb{V}^{\beta_1,z}$ and $\mathbb{V}^{\beta_2,z}$ be the path laws defined in \Cref{Xbeta nonlinear core}. Then $\mathbb{V}^{\beta_1,z}$ and $\mathbb{V}^{\beta_2,z}$ are mutually singular whenever $\beta_1\neq \beta_2$.
\end{lemma}
\begin{proof}
    Let $\beta_1\neq \beta_2$ and $X^{\beta_i,z}$ be the stochastic processes defined in \eqref{Xbeta} for $i \in \{1,2\}$. Recalling that $X^{\beta_i, z}$ solves the SDE
    \[
    dX^{\beta_i, z}_t= (1-{\beta_i})\frac{X^{\beta_i, z}_t-z}{2t}dt+\sqrt{{\beta_i}}dW_t,
    \]
    we readily infer that their quadratic variation processes $t\mapsto [X^{\beta_i,z}]_t$ equal ${\beta_i}t$, almost surely. Fix $T>0$, set $\varepsilon:=|\beta_1-\beta_2|T/4>0$, and let $(\mathcal{P}^n)_n$ be a sequence of partition of $[0, T]$ with mesh $|\mathcal{P}^n|\to 0$ as $n\to \infty$.
    For $\omega\in C([0, T]; \R^d)$, define
    \[
    \mathcal{Q}^n\omega:=\sum_{[u, v]\in \mathcal{P}^n}(\omega(u)-\omega(v))^2.
    \]
    By the definition of quadratic variation (considering the natural restriction of $\mathbb{V}^{\beta_i,z}$ to $C([0,T];\R^d)$)
    \[
    \lim_{n\to \infty}\mathbb{V}^{\beta_1,z}(\underbrace{\{\omega \in C([0, T]; \R^d): \ |\mathcal{Q}^n\omega-\beta_1T|\geq \varepsilon\}}_{=:A^n})=0.
    \]
    In particular, there exists a subsequence $(n_k)_{k\in \mathbb{N}}$ such that $\mathbb{V}^{\beta_1,z}(A^{n_k})\leq 2^{-k}$. Defining
    $A:=\bigcap_{i=1}^\infty \bigcup_{j=i}^\infty A^{n_j}$,
    Borel-Cantelli lemma yields $\mathbb{V}^{\beta_1,z}(A)=0$. On the other hand, since
    \[
     \lim_{n\to \infty}\mathbb{V}^{\beta_2,z}(\{\omega \in C([0, T]; \R^d): \ |\mathcal{Q}^n\omega-\beta_2T|\geq \varepsilon\})=0,
    \]
   we obtain
    \begin{equation}
        \begin{split}
            \mathbb{V}^{\beta_2,z}(A)&=\mathbb{V}^{\beta_2,z}(\{\omega\in C([0, T]; \R^d):\ |\mathcal{Q}^{n_k}\omega-\beta_1T|\geq \varepsilon \ \text{for infinitely many}\ k\})\\
            &\geq \mathbb{V}^{\beta_2,z}(\{\omega\in C([0, T]; \R^d):\ |\mathcal{Q}^{n_k}\omega-\beta_2T|<\varepsilon  \ \text{for infinitely many}\ k\})=1.
        \end{split}
    \end{equation}
    where we used that by our choice of $\varepsilon>0$, the intervals $[\beta_1T-\varepsilon, \beta_1T+\varepsilon]$ and $[\beta_2T-\varepsilon, \beta_2T+\varepsilon]$ are disjoint.
\end{proof}

Since $\beta = 0$ in \eqref{eq:HE-beta} formally corresponds to the pure-drift case investigated above, it is natural to ask whether $\mathbb{V}^{\beta,z} \to \mathbb{V}^z$ as $\beta\to 0$, where $\mathbb{V}^z$ denotes the path law constructed in \Cref{prop:Mc-heateq}. The answer is positive, as the following lemma shows.
\begin{lemma}\label{lemlem}
    Let $z\in \R^d$, $\beta\in (0, 1)$, $X^{\beta, z}$ be the stochastic process from \eqref{Xbeta}, and let $X^{0, z}_t:=z+\eta \sqrt{t}$, where $\eta$ is an $\R^d$-valued random variable with $\mathcal{L}(\eta) =  \mathcal{N}(0, 1_d)$. Then $X^{\beta,z}\in C^{1/2-}$ almost surely, and $\lim_{\beta\to 0}X^{\beta, z}=X^{0, z}$ in distribution on $[0,T]$ for all $T>0.$
    \label{convergence beta0}
\end{lemma}

\begin{proof}
   Let $T>0$ and $z\in \R^d$. We start by showing that the laws of $\{X^{\beta, z}\}_{\beta\in (0, 1)}$ is tight on $C([0, T]; \R^d)$. For $0\leq s\leq t\leq T$, one has
    \[
    X^{\beta, z}_t-X^{\beta, z}_s=\sqrt{\beta}\int_0^s\left(\left(\frac{t}{r} \right)^{\frac{1-\beta}{2}}-\left(\frac{s}{r} \right)^{\frac{1-\beta}{2}}\right)dW_r+\sqrt{\beta}\int_s^t\left(\frac{t}{r} \right)^{\frac{1-\beta}{2}}dW_r=:I^1+I^2.
    \]
    Note that $I^1$ and $I^2$ are independent Gaussian random variables, thus $X^{\beta,z}_t-X^{\beta,z}_s$ is Gaussian with variance $\mbox{Var}(X^{\beta,z}_t-X^{\beta,z}_s)=\mbox{Var}(I^1)+\mbox{Var}(I^2)$. By It\^{o} isometry
    \[
   \mbox{Var}(I^1)=s^\beta(t^{(1-\beta)/2}-s^{(1-\beta)/2})^2, \qquad  \mbox{Var}(I^2)=t^{1-\beta}(t^\beta-s^\beta).
    \]
Note that we have 
    \begin{equation}
        \begin{split}
            0&\leq s^\beta(t^{(1-\beta)/2}-s^{(1-\beta)/2})^2 +t^{1-\beta}(t^\beta-s^\beta)
            =s^\beta t^{1-\beta}-2s^{(1+\beta)/2}t^{(1-\beta)/2}+s+t-t^{1-\beta}s^\beta\\
            &=(t-s)+2s-2s^{(1+\beta)/2}t^{(1-\beta)/2}
            =(t-s)+2s^{(1+\beta)/2}(s^{(1-\beta)/2}-t^{(1-\beta)/2})
            \leq (t-s).
        \end{split}
    \end{equation}
    Since $X^{\beta, z}_t-X^{\beta, z}_s$ is Gaussian, we infer, for $p\in \mathbb{N}$,
    \[
    \mathbb{E}[|X^{\beta, z}_t-X_s^{\beta, z}|^{2p}]=(p-1)!! \mbox{Var}(X^{\beta,z}_t-X^{\beta,z}_s)^p\leq (p-1)!! |t-s|^p,
    \]
    where $!!$ denotes the usual double factorial.
    Thus, by the Kolmogorov continuity theorem, the laws of $\{X^{\beta, z}\}_{\beta\in (0,1)}$ on $C([0, T]; \R^d)$ are tight and $X^{\beta, z}\in C^{1/2-}$ almost surely, yielding the first claim. By the Prokhorov theorem, there exists a sequence $(\beta^n)_{n\in \mathbb{N}}$  in $(0,1)$ with $\beta^n\to 0$ and a measure $\mathbb{V}^{0,z}\in \Pscr(C([0, T];\R^d))$ such that $\mathbb{V}^{\beta_n,z} \to \mathbb{V}^{0,z}$ weakly. In particular, the corresponding characteristic functions of $k$-time marginals, i.e.
    \[
    \phi_{(X^{\beta^n,z}_{t_1}, \ldots , X^{\beta^n,z}_{t_k})}(\xi)=\exp\left(i\sum_{i=1}^k x_i\xi_i\right)\exp\left(-\frac{1}{2}\xi^T \Sigma_n \xi\right),
    \]
    converge,
    where $\Sigma_n$ denotes the covariance matrix of the Gaussian vector $(X^{\beta^n,z}_{t_1}, \ldots,  X^{\beta^n,z}_{t_k})$. Note that we can calculate its elements explicitly: if $s\leq t$, then
    \[
    \mathbb{E}[(X^{\beta^n, z}_s-z)(X^{\beta^n, z}_t-z)]=t^\frac{1-\beta^n}{2}s^{\frac{1+\beta^n}{2}}.
    \]
    Taking the limit, we obtain
       \[
    \lim_{n\to \infty}\mathbb{E}[(X^{\beta^n, z}_s-z)(X^{\beta^n, z}_t-z)]=t^{1/2}s^{1/2},
    \]
    thus we conclude
    \[
    \lim_{n\to \infty}\phi_{(X^{\beta^n,z}_{t_1}, \ldots, X^{\beta^n,z}_{t_k})}(\xi)=\exp(i\sum_{j=1}^k x_j\xi_j)\exp(-\frac{1}{2}\xi^T \Sigma \xi),
    \]
    where $(\Sigma)_{ij}=(t_i\vee t_j)^{1/2}(t_i\wedge t_j)^{1/2}$. On the other hand, if $X^{0,z}=z+\eta \sqrt{t}$, we can also calculate the characteristic function of the Gaussian vector $(X^{0,z}_{t_1}, \ldots, X^{0,z}_{t_k})$ as
    \[
    \phi_{(X^0_{t_1}, \ldots, X^0_{t_k})}(\xi)=\exp(i\sum_{j=1}^k x_j\xi_j)\exp(-\frac{1}{2}\xi^T \Sigma \xi).
    \]
Since the characteristic functions coincide, we conclude that the distribution of $(X^{0,z}_{t_1}, \ldots, X^{0,z}_{t_k})$ equals $\mathbb{V}^{0,z}\circ (\pi_{t_1}, \ldots, \pi_{t_k})^{-1}$. Since $k\in \mathbb{N}$ was arbitrary, we conclude $\mathbb{V}^{0,z}=\mathcal{L}(X^{0,z})$. Finally, note that the above argument is valid for any sequence $(\beta^n)_n$ converging to zero, yielding the claim.
\end{proof}

\begin{remark}[Heuristics on infinite diffusion asymptotics $\beta \to \infty$]

\label{beta infty}
    Recall that for $s\leq t$
    \[
\mathbb{E}[(X^{\beta, z}_t-z)(X^{\beta, z}_s-z)]=\sqrt{st}{\left(s/t\right)}^{\beta/2}.
\]
Thus, by formally taking the limit $\beta \to \infty$, the expected limiting process $X^{\infty, z}$ as $\beta \to \infty$ is a Gaussian with mean $z$ and covariance
\[
\mathbb{E}[(X^{\infty,z}_t-z)(X^{\infty,z}_s-z)]=t\delta_{st}.
\]
The unique such Gaussian process is
\[X^{\infty, z}_t:=z+\eta_t\sqrt{t},\] where $(\eta_t)_{t\geq 0}$ is a family of i.i.d. $\mathcal{N}(0,1_d)-$distributed $\R^d$-valued random variables. Hence, one might expect $\lim_{\beta \to \infty} X^{\beta, z}=X^{\infty, z}$ in distribution. However, establishing this rigorously appears challenging, as one cannot show tightness of $\mathcal{L}(X^{\beta, z})_{\beta\geq 1}$ on $C([0, T]; \R^d)$. Indeed, note that $X^{\infty, z}$ does not even have c\`adl\`agtrajectories.
\end{remark}

    In conclusion, we have constructed an uncountable number of pairwise mutually singular nonlinear Markov cores $\{\mathbb{V}^{\beta,z}\}_{z\in \R^d}$, $\beta \in (0,\infty),$ for $\{u^z\}_{z\in \R^d}$, each consisting of solution path laws to its associated MV-SDE \eqref{eq:SDE-beta}, and $\beta \mapsto \mathbb{V}^{\beta,z}$ is continuous with respect to the weak topology on $\Pscr(\Cscr)$. This emphasizes that nonlinear Markov processes in the sense of \Cref{def:NL-MP} to the classical heat kernel are far from unique.

\subsection{Further interpretations}\label{subsect:further}
Finally, we provide yet another interpretation of the heat equation as a nonlinear FPE and study its associated MV-SDE. We do so for two reasons: First, by providing an example not covered by the previous interpolation, we stress that the non-uniqueness we discuss here extends far beyond this family. Second, we show that finding interpretations of the heat equation as a nonlinear FPE may lead to corresponding MV-SDEs which are non-trivial to solve by standard methods. Thus, it is possible to use this approach to come up with examples of SDEs that admit weak solutions, while appearing quite pathological at first sight. 

Let $0 < p<1+\frac{1}{d}$ and $c\geq0$ (or $p=1$ and $c\geq -1/2$). Here we show that it is possible to recast the heat equation \eqref{eq:heat} as 
\begin{equation}\label{eq:PDE_heat_further}
    \partial_t u = \Delta \big( a(x,u)  u \big) - \nabla \cdot \big( b (x,  u ) u\big)
\end{equation}
with coefficients $(a,b) : \R^d \times \Pscr_* \to \R_+ \times \R^d $ given by
\begin{equation}\label{eq:coefficients_heat_further}
a(x,u) \coloneqq \frac{1}{2} + c \, u^{p-1},\qquad  b(x,u) \coloneqq c \frac{\nabla u^p (x) }{u(x)},
\end{equation}
with $b(x,u) \coloneqq 0 $ on $\{u=0\}$, defined on $\Pscr_* \coloneqq \{ u \in \Pscr_{\textup{ac}} : u^p \in W^{1,1}_{\textup{loc}}(\R^d) \} $. As will be shown in the next proposition, this interpretation can be obtained from \Cref{lem:many-PDEs_new} starting from $(a = \frac{1}{2}, b =0)$ and the choice $f = c u^{p} $. Observe that \eqref{eq:coefficients_heat_further} generalizes the $\beta$-interpolated interpretation from \Cref{subsec:heat_interpolation}, for which $p=1$ and $c = (\beta-1)/2$, $\beta \in [0,\infty)$.
The associated MV-SDE is
\begin{equation}  \label{nontriv MVSDE_heat}
\begin{dcases}
  dX_t= c \frac{\nabla u^p (t,X_t) }{u(t,X_t)} dt+\sqrt{1+2 \, c \, u (t, X_t)^{p-1}}dW_t,  \\
  \mathcal{L}(X_t) = u(t,x)dx.
  \end{dcases}
\end{equation}
Specifically, for $u = u^z$, using \eqref{eq:grad_heat}, this equation turns into
\begin{equation}  \label{nontriv SDE heat}
\begin{dcases}
  dX_t=-cp\frac{X_t-z}{t}u^z(t, X_t)^{p-1}dt+\sqrt{1+2\, c\, u^z(t, X_t)^{p-1}}dW_t,  \\
  \mathcal{L}(X_t) = u^z(t,x)dx.
  \end{dcases}
\end{equation}

\begin{remark}

   Let us consider the particular case $p<1$. In this case, the diffusion coefficient in \eqref{nontriv SDE heat} is of superexponential growth, since
   \[\label{eq:llabel}
  \sqrt{1+2c(u^z(t, x))^{p-1}}=\sqrt{1+2\frac{c}{(2\pi t)^{d(p-1)/2}}\exp\left(\frac{(1-p)|x-z|^2}{2t}\right)},
   \]
   with $1-p >0$.
   At the same time, the drift is mean-reverting with a superexponential mean reversion level, since 
   \[
  -cp\frac{x-z}{t}u^z(t,x)^{p-1}= -(x-z)\frac{cp}{t}\frac{1}{(2\pi t)^{d(p-1)/2}}\exp\left(\frac{(1-p)|x-z|^2}{2t}\right).
   \]
   Hence, solving \eqref{eq:llabel} directly by probabilistic methods appears challenging. Nevertheless, by our approach, the existence of weak solutions for every initial datum $\delta_z$ is obtained. Heuristically, the superexponential diffusion and the superexponential mean reverting drift are suitably balanced. In fact, we even prove weak uniqueness for \eqref{eq:llabel} and, therefore, that its unique solution path laws form a nonlinear Markov core for $\{u^z\}_{z\in \R^d}$, as the following result shows.
\end{remark} 

\begin{proposition}\label{prop9}
    Let $0 < p<1+\frac{1}{d}$ and $c\geq0$ (or $p=1$ and $c\geq -1/2$). For every $z \in \R^d$, there exists a unique weak solution $X^{p,z}$ to \eqref{nontriv MVSDE_heat} with one-dimensional time marginals given by the heat kernel \eqref{eq:fundamental_solution_heat}, i.e., $\mathcal{L}(X^{p,z}_t) = u^z (t,x) dx$, $t >0$, and $\mathcal{L}(X^{p,z}_0) = \delta_z$.
    \\
    Moreover, the corresponding family of solution path laws $\{\mathbb{V}^{p,z}\}_{z \in \R^d}$, $\mathbb{V}^{p,z} \coloneqq \mathcal{L} (X^{p,z})$,  form a nonlinear Markov core for $\{u^z\}_{z\in \R^d}$, which is uniquely determined by $\{u^z\}_{z\in \R^d}$ and \eqref{nontriv MVSDE_heat}.
\end{proposition}

\begin{proof}
We first verify that $u^z$ also solves \eqref{eq:PDE_heat_further} with \eqref{eq:coefficients_heat_further}. We apply \Cref{lem:many-PDEs_new} starting from $(a = \frac{1}{2}, b =0)$ and the choice $f = c u^{p} $, $\mathcal{D} (f) = \Pscr_{\textup{ac}}$. This gives \eqref{eq:coefficients_heat_further} and the domain $\Pscr_* \coloneqq \{ u \in \Pscr_{\textup{ac}} : u^p \in W^{1,1}_{\textup{loc}}(\R^d) \}$. Second, note that the diffusion coefficient is indeed non-negative in both cases. If $0<p<1+\frac{1}{d}$ and $c\geq 0$, then $ 1+2c\,u^{p-1}\geq 1>0.$
If $p=1$, then $
1+2c\,u^{p-1}=1+2c\geq 0 $, 
provided $c\geq -\tfrac{1}{2}$. Next, we show that $\big[(t,x)\mapsto f(x,u^z(t,\cdot))\big] \in L^1_{\textup{loc}}(\R_+;W^{1,1}(\R^d))$. Using the explicit form of the heat kernel \eqref{eq:fundamental_solution_heat}, we have
\begin{align}
    \int_0^T\int_{\R^d}|f(x,u^z(t,\cdot))|dxdt
    &= |c|\int_0^T\int_{\R^d} u^z(t,x)^p dxdt  \notag\\
    &= |c|\,p^{-\frac d2}(2\pi)^{-\frac{d(p-1)}{2}}
    \int_0^T t^{-\frac{d(p-1)}{2}}dt <+\infty .
\end{align}
Similarly, using \eqref{eq:grad_heat}, we obtain
\begin{align}
    \int_0^T\int_{\R^d}|\nabla f(x,u^z(t,\cdot))|dxdt
    &= |c|p\int_0^T\int_{\R^d} \frac{|x-z|}{t}u^z(t,x)^p dxdt  \notag\\
    &= |c|p(2\pi)^{-\frac{dp}{2}}S_{d-1}
    \int_0^T t^{-1-\frac{dp}{2}}
    \int_0^\infty r^d e^{-\frac{p r^2}{2t}}drdt \notag\\
    &= |c|p(2\pi)^{-\frac{dp}{2}}\frac{S_{d-1}}{2}
    \left(\frac{2}{p}\right)^{\frac{d+1}{2}}
    \Gamma(\frac{d+1}{2})
    \int_0^T t^{-\frac{d(p-1)+1}{2}}dt <+\infty .
\end{align}
Here, $\Gamma$ denotes the Gamma function and $S_{d-1}$ denotes the surface area of the unit sphere in $\R^d$. 
For both cases $p<1+\frac1d$ and $p =1$, the final time integral is finite.
Again, we use Theorem \ref{thm:Markov-construction} and Lemma \ref{lem:extrem-unique-equiv}. To this end, we consider the linear equation ($\mu^z_{s+}$-linFPE) (here, $\mu^z_{s+} = (u^z(s+t,x)dx)_{t\geq 0}$) with coefficients
    $$b^z(t,x) \coloneqq -\frac{cp}{s+t}\frac{x-z}{(2\pi (s+t)^{(p-1)d/2}}\exp\left(\frac{(1-p)|x-z|^2}{2(s+t)}\right)$$
    and
    $$a^z(t,x) \coloneqq \frac 1 2+\frac{c}{(2\pi (t+s))^{(p-1)d/2}}\exp\left(\frac{(1-p)|x-z|^2}{2(t+s)}\right).$$
    These coefficients depend locally Lipschitz continuously on $x$, with local Lipschitz constants independent from $t\geq 0$ (since $s>0$) and, moreover, $a^z\geq \frac 1 2$. Hence the required uniqueness for this FPE in order to apply Theorem \ref{thm:Markov-construction} and Lemma \ref{lem:extrem-unique-equiv} follows from \cite[Theorem 9.4.3]{BogachevKrylovRoecknerShaposhnikov2015}, since clearly, due to $p>0$,
    $$[(t,x)\mapsto a^z(t,x) u^z(s+t,x)], [(t,x)\mapsto b^z(t,x)u^z(s+t,x)] \in L^1((0,T)\times \R^d;dxdt),\quad \forall T>0.$$
\end{proof}

\section{Porous medium equation}\label{sec:porous_medium}
In this section, we consider the porous medium equation (PME)
    \begin{align}\label{eq:porous_medium}
        \partial_t u = \Delta(u^m), \qquad (t,x)\in   (0, +\infty) \times \mathbb{R}^d,
    \end{align}
where $m \in (1,\infty)$ and $d \in \N$.
The family $\{ u^z\}_{z \in \mathbb{R}^d}$ of explicit \emph{Barenblatt solutions} \cite{Barenblatt1952porous,ZeldovichKompaneets1950} to \eqref{eq:porous_medium} with initial condition $u^z(t,\cdot)\,dx \to \delta_z$ weakly as $t \downarrow 0$ is given by
\begin{equation}\label{eq:barenblatt_solution_porous}
u^z(t,x):=t^{-k}\Big(C - q \, t^{-\frac{2k}{d}}
|x-z|^{2}\Big)^{\frac{1}{m-1}}_+, \qquad (t,x) \in (0, +\infty) \times \mathbb{R}^d,
\end{equation}
where $k \coloneqq (m-1+\frac{2}{d})^{-1}$, $q:= \frac{k}{d} \tfrac{m-1}{2m}$, $y_+:=\mathrm{max}(0,y)$ for any real number $y$, and $C=C_{m,d} \in (0,\infty)$ is the unique constant such that $\lVert  u^{z}(t,\cdot) \rVert_{L^1(\mathbb{R}^d)}=1$ for all $t>0$ (a straightforward calculation shows that such a constant exists). One directly calculates that the support of $u^z(t,\cdot)$ is the closed ball centered at $z$ with radius
\begin{equation}\label{eq:support_porous}
    R(t) \coloneqq \left( \tfrac{C}{q}\right)^{\frac{1}{2}} t^{\frac{k}{d}}.
\end{equation}

\smallskip
\noindent  
\textbf{Self-similar form of $u^z$.}
    It is straightforward and useful for our subsequent computation to note that by scaling analysis of the PME, any mass-preserving, self-similar, radially symmetric solution to the PME centered at $z \in \R^d$ is of the form $t^{-k} g (t^{-\frac{k}{d}}|x-z| )$ for some radial profile $g: \R_+ \to \R_+$, where $k$ is given above.
    In particular, for the Barenblatt solution \eqref{eq:barenblatt_solution_porous} we have
    \begin{equation}\label{eq:uz_self_similar_porous}
        u^z(t,x) = t^{-k} g\big(t^{-\frac{k}{d}}|x-z|\big),
    \end{equation}
    with the radial profile $g$ given by
    \begin{equation}\label{eq:g_for_Barenblatt_porous}
        g(\xi) \coloneqq \big(C - q \,  \xi^2\big)^{\frac{1}{m-1}}_+, \quad \xi \in \R_+,
    \end{equation}
    whose derivative can be written as
    \begin{equation}\label{eq:fprime_for_Barenblatt_porous}
        g'(\xi) = - \frac{1}{m} \frac{k}{d}\, \xi\, g(\xi)^{2-m}, \quad \xi \in (0,R(1)).
    \end{equation}

\subsection{Pure-diffusion interpretation}\label{subsec:porous_medium_pure_diff}
A nonlinear Markov process corresponding to the pure-diffusion interpretation of the PME was constructed in \cite[Section 4.2]{RehmeierRoeckner2025NonlinearMarkov}. However, the uniqueness of this process was partially left open. We close this gap in this subsection. 
Consider the pure-diffusion interpretation
\begin{equation}\label{eq:PME-diffusion}
    \partial_t u = \Delta (u^{m-1} u),
\end{equation}
i.e., coefficients  $(a,b) : \R^d \times \Pscr_{\textup{ac}} \to \R_+ \times \R^d $ given by
\begin{equation}
    a(x,u) \coloneqq u^{m-1}(x),\qquad b(x,u) \coloneqq  0.
\end{equation}
The corresponding MV-SDE is
\begin{equation}\label{eq:DDSDE-PME}
    \begin{dcases}
    dX_t = \sqrt{2}u^{\frac{m-1}{2}}(t,X_t) dW_t, \\ \mathcal{L}(X_t) = u(t,x)dx,\,\,\,t >0.
     \end{dcases}
\end{equation}
Observe that for $u=u^z$, by \eqref{eq:uz_self_similar_porous} and \eqref{eq:g_for_Barenblatt_porous}, one obtains
\begin{equation}\label{eq:az_porous_self-similar}
    a(x,u^z(t,\cdot)) = t^{-\alpha}\, h \big(t^{-\frac{k}{d}}|x-z|\big),
\end{equation}
where $\alpha \coloneqq k (m-1) \in (0,1)$ and the radial profile $h:\R_+\to\R_+$ is given by
\begin{equation}\label{eq:h_for_az_porous_self-similar}
    h(\xi)\coloneqq g(\xi)^{m-1}
    =\big(C-q\,\xi^2\big)_+.
\end{equation}

In \cite{RehmeierRoeckner2025NonlinearMarkov}, a nonlinear Markov process $\{\mathbb{P}^\zeta\}_{\zeta \in \Pscr_0}$, $\mathcal{P}_0 = \Pscr$, was constructed, consisting of solution path laws to \eqref{eq:DDSDE-PME} with $\mathbb{P}^\zeta_0 = \zeta$ and $\mathbb{P}^\zeta_t = u^\zeta(t,x)dx$, where $u^\zeta = (u^\zeta(t,\cdot))_{t> 0}$, $u^\zeta(t, \cdot) \in \Pscr_{\textup{ac}}$, is a solution to \eqref{eq:porous_medium} with initial condition $\zeta$ (that such solutions exist for arbitrary initial condition $\zeta\in \Pscr$ was proven in \cite{NLFPK-DDSDE5}), and $u^\zeta = u^z$ as in \eqref{eq:barenblatt_solution_porous} for $\zeta = \delta_z$. However, it was only proven that $\mathbb{P}^\zeta$ is the \emph{unique} solution path law to \eqref{eq:DDSDE-PME} with marginal densities $u^\zeta(t,\cdot)$ if $\zeta \in \Pscr_{\textup{ac}}\cap L^\infty(\R^d)$, so in particular the question of uniqueness of solutions to \eqref{eq:DDSDE-PME} with initial data $\delta_z$ and marginal densities $u^z(t,\cdot)$ was left open. We close this gap by the following result.

\begin{proposition}\label{prop1}
    Let $m>1$. For every $z\in \R^d$,
    there exists a unique weak solution $X^{1,z}$
    to the MV-SDE \eqref{eq:DDSDE-PME} with one-dimensional time marginals given by the Barenblatt solution \eqref{eq:barenblatt_solution_porous}, i.e., $\mathcal{L}(X^{1,z}_t) = u^z (t,x) dx$, $t >0$, and $\mathcal{L}(X^{1,z}_0) = \delta_z$. 
    \\
    Furthermore, the corresponding family of solution path laws $\{\mathbb{V}^{1,z}\}_{z\in \R^d}$, $\mathbb{V}^{1,z} := \mathcal{L}(X^{1,z})$, is a nonlinear Markov core for $\{u^z\}_{z\in \R^d}$. In particular, this nonlinear Markov core is uniquely determined by $\{u^z\}_{z\in \R^d}$ and \eqref{eq:DDSDE-PME}. 
    \\
    Finally, for every $\zeta = u^z(s,x)dx$, $s>0$, the path law $\mathbb{V}^{1,\zeta}$ from the nonlinear Markov process induced by $\{\mathbb{V}^{1,z}\}_{z\in \R^d}$ coincides with the path law $\mathbb{P}^\zeta$ constructed in \cite[Section 4.2 (iv)]{RehmeierRoeckner2025NonlinearMarkov}.
\end{proposition}
\begin{proof}
    All assertions except for the final one follow from \Cref{crl:nl_markov}. In order to apply the latter, we first verify the integrability condition \eqref{eq:integrability_our_paper}. 
    With the self-similar form \eqref{eq:uz_self_similar_porous}, using the change of variables $r\coloneqq |x-z|$ and $\xi\coloneqq t^{-\frac{k}{d}}r$, we compute
    \begin{align}
        \int_{0}^{T} \int_{\mathbb{R}^d} |a(x,u^z (t,\cdot)| u^z (t,x) dx dt
        &= \int_{0}^{T} \int_{\R^d} (u^z(t,x))^m\,dx\,dt \\
        &= S_{d-1} \int_{0}^{T} \int_0^{+\infty} t^{-km}\,
        g\big(t^{-\frac{k}{d}}r\big)^m\, r^{d-1}\,dr\,dt \\
        &= S_{d-1}\left(\int_0^T t^{-\alpha}\,dt\right)
           \left(\int_0^{R(1)} g(\xi)^m\,\xi^{d-1}\,d\xi\right)
           <+\infty, \label{eq:PME_pure_diff_integrability}
    \end{align}
    where $S_{d-1}$ denotes the surface area of the unit sphere in $\mathbb{R}^d$, and $\alpha \coloneqq k (m-1) \in (0,1)$, so that the time integral finite. The spatial integral is also finite because $g$ is bounded and compactly supported in $[0,R(1)]$.
    Second, to apply \Cref{crl:nl_markov}, it only remains to verify that, for all $s>0$ and $z\in \R^d$, $[(t,x)\mapsto u^z(s+t,x)]$ is the unique solution to 
    $$\partial_t v(t,x) = \Delta (u^z(s+t,x)^{m-1}v(t,x)),\quad v(0,x) = u^z(s,x)$$
    in the class
    $$\{v\geq 0: [t\mapsto v(t,x)dx] \in C(\R_+;\Pscr), v(t,x) \leq C u^z(s+t,x) \text{ a.s. for some }C\geq 1\}.$$
    Since 
    $$u^z(s+t,x)^{m-1} =  (t+s)^{-k}\big(C- q (t+s)^{-\frac{2k}{d}} |x-z|^2 \big)_+$$
    is globally Lipschitz continuous in $x$ with Lipschitz constant bounded from above by a constant independent from $t\geq 0$ (here we use $s>0$), solutions to the above linear FPE are even unique in the class of all distributional (measure-valued) solutions. In particular, uniqueness in the class indicated above holds.
    Thus, \Cref{crl:nl_markov} applies.
    
    The final claim follows since it was shown already in \cite{RehmeierRoeckner2025NonlinearMarkov} that solutions to \eqref{eq:DDSDE-PME} are weakly unique for prescribed marginal densities $u^\zeta(t,x)dx$ and initial condition $\zeta$ for every $\zeta \in \Pscr_{\textup{ac}}\cap L^\infty(\R^d)$. Thus, since $u^z(s,x) \in \Pscr_{\textup{ac}} \cap L^\infty(\R^d)$ and $u^\zeta(t,x)$ as constructed in \cite{RehmeierRoeckner2025NonlinearMarkov} for $\zeta = u^z(s,x)dx$ equals $(u^z(s+t,x)dx)_{t\geq 0}$, it follows that $\mathbb{P}^{u^z(s,x)dx}$ from \cite{RehmeierRoeckner2025NonlinearMarkov} and our $\mathbb{V}^{1,u^z(s,x)dx} = \mathbb{V}^{1,z} \circ \Pi_s^{-1}$ solve \eqref{eq:DDSDE-PME} are such solutions to \eqref{eq:DDSDE-PME} and are therefore equal. This concludes the proof.
\end{proof}

\subsection{Pure-drift interpretation}\label{subsec:porous_medium_drift} We now turn to the pure-drift formulation of the PME \eqref{eq:porous_medium},
\begin{equation}\label{eq:PME-drift}
    \partial_t u = - \nabla \cdot \big( \frac{-\nabla u^m}{u} u\big),
\end{equation}
which is a first-order nonlinear FPE with coefficients $(a,b): \R^d \times \Pscr_{*} \to \R_+ \times \R^d$ given by
\begin{equation}\label{eq:coefficients_PME-drift}
    a (x,u) \coloneqq  0, \qquad b(x,u) \coloneqq  -\frac{\nabla u^m(x)}{u(x)},
\end{equation}
and $b(x,u) \coloneqq 0 $ on $\{u=0\}$, defined on $\Pscr_{*} \coloneqq \{u\in \Pscr_{\textup{ac}}: u^m \in W^{1,1}_{\textup{loc}} (\R^d) \}$. The corresponding MV-SDE is the distribution-dependent ODE
\begin{equation}\label{eq:PMEdrift-DDSDE}
    \begin{dcases}
    dX_t = -\frac{\nabla u^m(t,X_t)}{u(t,X_t)}dt, \\
    \mathcal{L}(X_t) = u(t,x)dx,\,\,\, t >0.
    \end{dcases}
\end{equation}
In particular, for $u=u^z$ given by the Barenblatt solution \eqref{eq:barenblatt_solution_porous}, a direct calculation (also included in the proof below) yields
\begin{equation}\label{eq:bz_porous_explicit_formula}
    b(x,u^z(t,\cdot)) =\frac k d \frac{x-z}{t} \mathds{1}_{B_{R(t)}(z)}(x),
\end{equation}
where $R(t)$ is given in \eqref{eq:support_porous}. Therefore, 
\eqref{eq:PMEdrift-DDSDE} with $u = u^z$ turns into
\begin{equation}\label{eq:PME-special}
    \begin{dcases} dX_t =  \frac k d \frac{X_t-z}{t} \mathds{1}_{B_{R(t)}(z)}(X_t) dt, \\ \mathcal{L}(X_t) = u^z(t,x)dx,\,\,\,t>0.
   \end{dcases}
\end{equation}
An explicit solution to \eqref{eq:PME-special} with $\mathcal{L}(X_0) = \delta_z$  is given by $(X^{0,z}_t)_{t\geq 0}$ as follows
\begin{equation}\label{eq:solutions_ODE_PME}
    X^{0,z}_t \coloneqq z+\eta \, t^{\frac{k}{d}}, 
\end{equation}
where $\eta$ is an $\R^d$-valued random variable  with $\mathcal{L}(\eta)=u^{0}(1,x) dx$.

\begin{proposition}\label{prop2}
    Let $m>1$. For every $z\in \R^d$, $X^{0,z}$ from \eqref{eq:solutions_ODE_PME} is the unique weak solution to \eqref{eq:PMEdrift-DDSDE}  with one-dimensional time marginals given by the Barenblatt solution \eqref{eq:barenblatt_solution_porous}, i.e., $\mathcal{L}(X^{0,z}_t) = u^z(t,x)dx$, $t>0$, and $\mathcal{L}(X^{0,z}_0) = \delta_z$.
    \\
    Furthermore, the family of corresponding solution path laws $\{\mathbb{V}^{0,z}\}_{z\in \R^d}$, $\mathbb{V}^{0,z} := \mathcal{L}(X^{0,z})$, is a nonlinear Markov core for $\{u^z\}_{z\in \R^d}$, which is uniquely determined by $\{u^z\}_{z\in \R^d}$ and \eqref{eq:PMEdrift-DDSDE}.
\end{proposition}
\begin{proof}
    We start by verifying the integrability condition  \eqref{eq:integrability_our_paper}. To this end, we first derive the explicit formula \eqref{eq:bz_porous_explicit_formula} for the drift coefficient. By taking the gradient of $u^z(t,x)$ in \eqref{eq:uz_self_similar_porous} at a point $x$ with $0<|x-z|<R(t)$, using \eqref{eq:fprime_for_Barenblatt_porous}, and writing $r \coloneqq |x-z|$, we obtain
    \begin{align}\label{eq:align-PME}
    \nabla (u^z(t,x)^m)
    & = m t^{-km -\frac{k}{d} } (g (t^{-\frac{k}{d}}r))^{m-1} g' (t^{-\frac{k}{d}}r)  \frac{x-z}{r} \\
    & = - \frac{k}{d} \, t^{-km-\frac{2k}{d}}\, g(t^{-\frac{k}{d}}r)\, (x-z) = - \frac{k}{d}  u^z(t,x) \, \frac{x-z}{t},
    \end{align}
    which, after dividing by $u^z(t,x) > 0$, yields \eqref{eq:bz_porous_explicit_formula}. Now, for any $T >0$, using polar coordinates, \eqref{eq:uz_self_similar_porous}, \eqref{eq:fprime_for_Barenblatt_porous}, and the change of variables $\xi \coloneqq t^{-\frac{k}{d}}r$ yields 
    \begin{align}
       \int_{0}^{T} \int_{\mathbb{R}^d} |b(x,u^z (t,\cdot)| u^z (t,x) dx dt & = \int_0^T \int_{\R^d} |\nabla u^z(t,x)^m| dx dt \\
       & = \frac{k}{d} S_{d-1} \int_{0}^{T}t^{-1-k}  \int_{0}^{+\infty} r g \big(t^{-\frac{k}{d}} r\big) r^{d-1} dr d t \\
        & = \frac{k}{d} S_{d-1} \left(\int_{0}^{T}t^{-(1-\frac{k}{d}) } d t \right) \left( \int_{0}^{R(1)} \xi^d  g (\xi ) d\xi   \right) < + \infty, \label{eq:PME_pure_drift_integrability}
     \end{align}
     where $S_{d-1}$ denotes the surface area of the unit sphere in $\R^d$. 
    Note that the time integral is finite because $(1-\frac{k}{d})  \in (0,1)$, and the spatial integral is finite because $g$ is bounded and compactly supported in $[0,R(1)]$.
    
    Next, we verify that $u^z$ is a distributional solution to \eqref{eq:PME-drift}. This can be seen, for instance, via \Cref{lem:many-PDEs_new} with the choice $a = u^{m-1}, b= 0$, $\Pscr_*=\Pscr_{\textup{ac}}$, (i.e., the pure-diffusion interpretation), and $f(u) = - u^m$, $\mathcal{D}(f) := \Pscr_{\textup{ac}}$. 
    With this choice, one obtains the reformulation \eqref{eq:PME-drift}, together with the domain $\Pscr_{*} \coloneqq \{u\in \Pscr_{\textup{ac}}: u^m \in W^{1,1}_{\textup{loc}} (\R^d) \}$. Moreover, by \eqref{eq:PME_pure_drift_integrability} and \eqref{eq:PME_pure_diff_integrability}, 
    we have $[(t,x)\mapsto f(u^z(t,\cdot))] \in L^1([0,T];W^{1,1}(\R^d))$ for all $T>0.$ Therefore, the assumptions of \Cref{lem:many-PDEs_new} (and \Cref{cor:new}) are satisfied. Note that \eqref{eq:solutions_ODE_PME} is indeed an explicit solution of the corresponding MV-SDE. 
    
    To apply \Cref{crl:nl_markov}, it remains to show that, for every $s>0$ and $z\in \R^d$, the linear FPE
    \begin{equation}\label{ui}
          \partial_t v(t,x) = \nabla \cdot \bigg(  \frac{\nabla (u^z)^m(t,x)}{u^z(t,x)}  v(t,x) \bigg),\quad v(0,x) = u^z(s,x),
    \end{equation}
    has a unique solution in the class
     \begin{equation}\label{eq:sol-class}
     \{v\geq 0: [t\mapsto v(t,x)dx] \in C(\R_+;\Pscr), v(t,x) \leq C u^z(s+t,x) \text{ a.s. for some }C\geq 1\}
     \end{equation}
     (then, clearly, this unique solution is $v(t,x) = u^z(s+t,x)$).
     We already calculated
     \begin{equation}
         \frac{\nabla (u^z(s+t,x)^m)}{u^z(t,x)} = \frac k d \frac{x-z}{s+t} \mathds{1}_{B_{R(s+t)}(z)}(x).
     \end{equation}
     While this coefficient is obviously not Lipschitz continuous (due to the cut-off resulting from the indicator function), we can argue as follows: Every solution $v$ from \eqref{eq:sol-class} satisfies $v(t,x) = 0$ $dx$-a.s. on $B_{R(s+t)}(z)^c$ for a.e. $t\geq 0$. Thus, the solutions from \eqref{eq:sol-class} also solve 
     $$\partial_t v(t,x) = \nabla \cdot \big( b^z(t,x) v(t,x)   \big),$$
     with $b^z(t,x) := \frac k d \frac{x-z}{s+t}$. Clearly, $x\mapsto b^z(t,x)$ is Lipschitz continuous with Lipschitz constant bounded from above by a finite number independent from $t\geq 0$ (since $s>0$). Thus, by classical theory, distributional solutions to this equation are unique, which implies that solutions to \eqref{ui} from \eqref{eq:sol-class} are unique. Therefore, \Cref{crl:nl_markov} applies and yields all assertions.
\end{proof}

\subsection{Interpolation between pure-diffusion and pure-drift cases}\label{subsec:PME-beta}
Here, we recast \eqref{eq:porous_medium} as
\begin{equation}\label{eq:PME-beta}
    \partial_t u = \beta \Delta(u^{m-1}u) - (1-\beta)\nabla \cdot \big( \frac{-\nabla u^m}{u} u \big),
\end{equation}
with interpolation parameter $\beta \in (0,\infty)$. This is a nonlinear FPE with coefficients $(a,b): \R^d \times \Pscr_{*} \to \R_+ \times \R^d$ given by
\begin{equation}\label{eq:PME_beta_coeff}
    a(x,u) \coloneqq \beta u^{m-1}(x),\qquad b(x,u) \coloneqq - (1-\beta)\frac{\nabla u^m(x)}{u(x)},
\end{equation}
and $b(x,u) \coloneqq 0 $ on $\{u=0\}$, defined on $\Pscr_{*} \coloneqq \{u \in \Pscr_{\textup{ac}}: u^m \in W^{1,1}_{\textup{loc}} (\R^d) \}$. For $\beta = 0$ and $\beta =1$ one recovers the pure-diffusion and pure-diffusion case, respectively. The corresponding MV-SDE is
\begin{equation}\label{op}
    \begin{dcases}
    dX_t = -(1-\beta) \frac{\nabla u(t,X_t)^m}{u(t,X_t)}dt + \sqrt{2\beta} \, u(t,X_t)^{\frac{m-1}{2}}dW_t, \\
    \mathcal{L}(X_t) = u(t,x)dx,\,\,\, t >0.
    \end{dcases}
\end{equation}
Using \Cref{prop1,prop2}, which correspond to $\beta = 1$ and $\beta = 0$, respectively, we obtain the following result for the interpolation regime.
\begin{proposition}\label{prop_porous_beta}
    Let $m>1$ and $\beta \in (0,\infty)$. For every $z\in \R^d$, there exists a unique weak solution $X^{\beta,z}$ to the MV-SDE \eqref{op} with one-dimensional time marginals given by the Barenblatt solution \eqref{eq:barenblatt_solution_porous}, i.e., $\mathcal{L}(X^{\beta,z}_t) = u^z(t,x)dx$, $t>0$, and $\mathcal{L}(X^{\beta,z}_0) = \delta_z$. 
    \\
    Furthermore, the corresponding family of solution path laws $\{\mathbb{V}^{\beta,z}\}_{z \in \R^d}$, $\mathbb{V}^{\beta,z} := \mathcal{L}(X^{\beta,z})$, forms a nonlinear Markov core for $\{u^z\}_{z\in \R^d}$, which is uniquely determined by $\{u^z\}_{z\in \R^d}$ and \eqref{op}.
\end{proposition}
\begin{proof}
All assertions follow immediately from \Cref{crl:nl_markov}. The latter is applicable here, since all required integrability conditions follow immediately from the fact that these conditions were already verified for $\beta = 0$ and $\beta =1$ in the proofs of \Cref{prop1,prop2}. Since in these proofs it was also argued that the coefficients of the linear FPE
\begin{equation}
\partial_t v(t,x) = \beta\Delta ((u^z)^{m-1}(s+t,x)v(t,x)) - (1-\beta) \nabla \cdot \bigg(\frac{\nabla  (u^z(s+t,x)^m)}{u^z(s+t,x)}v(t,x)\bigg),\,\, v(0,x) = u^z(s,x)
\end{equation}
are Lipschitz continuous in $x$ with Lipschitz constant bounded from above by a constant independent from $t>0$ (here we use $s>0$), the uniqueness of this linear FPE in the class
 $$\{v\geq 0: [t\mapsto v(t,x)dx] \in C(\R_+;\Pscr), v(t,x) \leq C u^z(s+t,x) \text{ a.s. for some }C\geq 1\}$$ follows by standard theory. Hence, \Cref{crl:nl_markov} applies, which concludes the proof.
\end{proof}
\begin{remark}\label{rem-PME}
    By \Cref{prop2,prop_porous_beta}, we have constructed for every $\beta \in [0,\infty)$ a nonlinear Markov core $\{\mathbb{V}^{\beta,z}\}_{z\in \R^d}$ for the PME Barenblatt solutions $\{u^z\}_{z\in \R^d}$. Similarly to the heat equation case, for each $z\in \R^d$, the path laws $\mathbb{V}^{\beta_1,z},\mathbb{V}^{\beta_2,z}$ are distinct whenever $\beta_1 \neq \beta_2$, and similarly to the proof of \Cref{lem:singularlaws}, one can even show their mutual singularity.
\end{remark}

\subsection{Formal pure-Stratonovich-diffusion interpretation}\label{subsect:PME-Stratono}
    We highlight one particular element of the above interpolation family \eqref{eq:PME-beta}-\eqref{eq:PME_beta_coeff}, corresponding to $\beta=\frac{2m}{m+1}$, for which we obtain $a(x,u) = \frac{2m}{m+1} u^{m-1}(x)$ and the gradient relation
    \begin{equation}
        b(x, u)=\frac{1}{2}\nabla a(x, u).
    \end{equation}
If $\sqrt{a(x,u^z(t,\cdot))}$ was smooth, this relation would ensure that \eqref{op} for $\beta = \frac{2m}{m+1}$ was equivalent to the pure-Stratonovich-diffusion MV-SDE
\begin{equation}\label{eq:PME-pure-diff-Stratono}
\begin{dcases}
    dX_t = \sqrt{\tfrac{4m}{m+1}} u(t,X_t)^{\frac{m-1}{2}}\circ dW_t,
    \\
    \mathcal{L}(X_t) = u(t,x)dx,\,\,\, t >0.
\end{dcases}
\end{equation}
However, this smoothness fails, meaning that the rigorous establishment of \eqref{op} for $\beta=\frac{2m}{m+1}$ as a Stratonovich SDE is not straightforward. For a more detailed discussion, see Appendix \ref{subsec:theta_interpretations}.

\subsection{Additive noise interpretation}\label{subsec:PME-Addit}
As a final result for the PME, we construct a nonlinear Markov core with Barenblatt time marginals consisting of solutions to an additive noise MV-SDE. To this end, we recast \eqref{eq:porous_medium} as
\begin{equation}\label{eq:PME-addnoise}
    \partial_t u = \Delta u - \nabla \cdot \bigg(\frac{\nabla u - \nabla u^m }{u} u \bigg),
\end{equation}
i.e. as the nonlinear FPE with coefficients $(a,b): \R^d \times \Pscr_* \to \R_+ \times \R^d$,
\begin{equation}
    a(x,u) := 1, \qquad b(x,u) := \frac{\nabla  u- \nabla u^m }{u}(x),
\end{equation}
with $b(x,u):= 0$ on $\{u = 0\}$; these coefficients are defined on 
\begin{equation}\label{eq:PME-addnoise_domian}
    \Pscr_* := \{u \in \Pscr_{\textup{ac}} \,:\, u, u^m \in W^{1,1}_{\textup{loc}}(\R^d)\}.
\end{equation}
The corresponding MV-SDE is
\begin{equation}\label{eq:PME-addnoise-MVSDE}
    \begin{dcases}
     dX_t = \frac{\nabla u(t,X_t) - \nabla u^m(t,X_t)}{u(t,X_t)}dt + \sqrt{2} dW_t,\\
    \mathcal{L}(X_t) = u(t,x)dx,\quad t>0.
\end{dcases}
\end{equation}

\begin{proposition}\label{prop:PME-addnoise-MVSDE}
    Let $1 < m < 2$. For every $z\in \R^d$, there exists a unique weak solution $X^z$ to the MV-SDE \eqref{eq:PME-addnoise-MVSDE} with one-dimensional time marginals given by the Barenblatt solution \eqref{eq:barenblatt_solution_porous}, i.e., $\mathcal{L}(X^z_t) = u^z(t,x)dx$, $t>0$, and $\mathcal{L}(X^z_0) = \delta_z$.
    \\
    Furthermore, the corresponding family of solution path laws $\{\mathbb{V}^{z}\}_{z\in \R^d}$, $\mathbb{V}^{z}:= \mathcal{L}(X^{z})$, forms a nonlinear Markov core for $\{u^z\}_{z\in \R^d}$, which is uniquely determined by $\{u^z\}_{z\in \R^d}$ and \eqref{eq:PME-addnoise-MVSDE}.
\end{proposition}
\begin{proof}
    We first verify the integrability condition \eqref{eq:integrability_our_paper}, which in this case reads:
    \begin{equation}\label{eq:PME-addnoise-intg}
        \int_0^T \int_{\R^d}\bigg(1 + \frac{|\nabla (u^z(t,x)-(u^z)^m(t,x))|}{u^z(t,x)}\bigg)\,u^z(t,x)dx dt < +\infty, \qquad \forall T>0. 
    \end{equation}
    This indeed holds since $\int_0^T \int_{\R^d} u^z(t,x)dx dt = T$, and $\int_0^T \int_{\R^d} | \nabla (u^z)^m(t,x)| dx dt < + \infty$, which was already shown in the proof of \Cref{prop2}. 
    As for the remaining term,
    using the self-similar form \eqref{eq:uz_self_similar_porous}, writing $r \coloneqq |x-z|$, the change of variable
    \(\xi=t^{-\frac{k}{d}}r\), and \eqref{eq:fprime_for_Barenblatt_porous}, we have  
    \begin{align}
        \int_0^T \int_{\R^d} |\nabla u^z(t,x)|\,dx\,dt
        &= S_{d-1}\int_0^T \int_0^{R(t)}
        t^{-k-\frac{k}{d}}
        \big|g'(t^{-\frac{k}{d}}r)\big|\,r^{d-1}\,dr\,dt \\
        &= S_{d-1} \left( \int_0^T t^{-k-\frac{k}{d}} t^{+k} \,dt  \right) \left(
        \int_0^{R(1)} |g'(\xi)|\,\xi^{d-1}\,d\xi\right) \\
         &= S_{d-1} \left( \int_0^T t^{-\frac{k}{d}}\,dt  \right) \left( \frac{1}{m} \frac{k}{d} 
        \int_0^{R(1)} g(\xi)^{2-m}  \,\xi^{d}\,d\xi\right) < + \infty,
    \end{align}
    where $S_{d-1}$ denotes the surface area of the unit sphere in $\R^d$.
    The time integral is finite since \(\frac{k}{d}\in(0,1)\). For $1<m < 2$, the integrand  $\xi \mapsto g(\xi)^{2-m}\xi^d$ is clearly bounded on $(0,R(1))$, and therefore the spatial integral is also finite. 
    
    Next, we note that $u^z$ is a distributional solution to \eqref{eq:PME-addnoise}. This can be seen, e.g., via \Cref{lem:many-PDEs_new} with $a = u^{m-1}, b= 0$, $\Pscr_*=\Pscr_{\textup{ac}}$, (i.e., the pure-diffusion interpretation), and $f(u) = u - u^m$, $\mathcal{D}(f) \coloneqq \Pscr_{\textup{ac}}$. 
    With this, one obtains precisely the reformulation \eqref{eq:PME-addnoise}, together with the corresponding domain \eqref{eq:PME-addnoise_domian}. Moreover, by \eqref{eq:PME-addnoise-intg} and \eqref{eq:PME_pure_diff_integrability}, we have $[(t,x)\mapsto f(u^z(t,\cdot))] \in L^1([0,T];W^{1,1}(\R^d))$ for all $T>0.$ Therefore, the assumptions of \Cref{lem:many-PDEs_new} are satisfied.

    Finally, in order to apply \Cref{crl:nl_markov}, it remains to prove that, for all $(s,z)\in (0,\infty)\times \R^d$, the linear FPE
    \begin{equation}\label{eqlin}
        \partial_t v(t,x) = \Delta v(t,x) - \nabla \cdot \bigg( \frac{\nabla(u^z-(u^z)^m)}{u^z}(s+t,x) v(t,x)\bigg),\quad v(0,x) = u^z(s,x),
    \end{equation}
    has a unique distributional solution in the sense of \Cref{def:lFPE-sol} in
    \begin{equation}
        \{v\geq 0: [t\mapsto v(t,x)dx] \in C(\R_+;\Pscr), v(t,x) \leq C u^z(s+t,x) \text{ a.s. for some }C\geq 1\}.
    \end{equation}
    By \eqref{eq:PME-addnoise-intg}, for any solution $v$ to this equation  in the above class and any $T > 0$, we have
    \begin{multline}\label{eq:proof_addnoise_int_1}
        \int_0^T \int_{\R^d} \bigg|\frac{\nabla(u^z-(u^z)^m)}{u^z}(s+t,x) \bigg|v(t,x) dx dt \\
        \leq C\int_0^T \int_{\R^d} \bigg|\frac{\nabla(u^z-(u^z)^m)}{u^z}(s+t,x) \bigg|u^z(s+t,x) dx dt < + \infty.
    \end{multline}
    Additionally, for each such $v$ and any compact $K \subseteq \R^d$, we have
    \begin{align}
        \int_0^T \int_K & \bigg|\frac{\nabla(u^z-(u^z)^m)}{u^z}(s+t,x)\bigg|^2  v(t,x)\,dx\,dt \\
        & \leq C\int_0^T \int_K \frac{|\nabla(u^z-(u^z)^m)|^2}{u^z}(s+t,x)\,dx\,dt \\
        &\leq 2C\int_0^T \int_{\R^d} \frac{|\nabla u^z|^2}{u^z}(s+t,x)\,dx\,dt +2C\int_0^T \int_K \frac{|\nabla (u^z)^m|^2}{u^z}(s+t,x)\,dx\,dt \\
        & \leq 2 C S_{d-1}\left(\int_0^T (s+t)^{-\frac{2k}{d}}\,dt\right)
        \left(\frac{1}{m^2}\frac{k^2}{d^2}\int_0^{R(1)} g(\xi)^{3-2m}\,\xi^{d+1}\,d\xi\right) \label{eq:04_14_104}\\
        & \,\, + 2C \Big(\frac{k}{d}\Big)^2 M_K^2 \int_0^T \frac{1}{(s+t)^2}
        \int_K u^z(s+t,x)\,dx\,dt < + \infty.\label{eq:04_14_105}
    \end{align}
    For the last inequality, we estimated the first term using the self-similar form \eqref{eq:uz_self_similar_porous}, the change of variables
    \(\xi=(s+t)^{-\frac{k}{d}}|x-z|\), and \eqref{eq:fprime_for_Barenblatt_porous}. To estimate the second term, we used the explicit formula \eqref{eq:align-PME} and denote $M_K:=\sup_{x\in K}|x-z|<+\infty$.
    In \eqref{eq:04_14_104}, the spatial integral is finite if and only if within the range $1 <m<2$, and the time integral is finite since $s>0$. In \eqref{eq:04_14_105}, the spatial integral can be clearly bounded by 1, and the time integral is also finite since $s>0$. 
    Hence, the required uniqueness follows from \cite[Theorem 9.3.6]{BogachevKrylovRoecknerShaposhnikov2015}. Now, finally, \Cref{crl:nl_markov} applies and yields all assertions.
\end{proof}

\begin{remark}
    As seen in the previous proof, the condition $1 < m <2$ is sharp for conditions of \cite[Theorem 9.3.6]{BogachevKrylovRoecknerShaposhnikov2015} to hold in order to prove the required uniqueness for \eqref{eqlin}. We suppose that this uniqueness can be proven also for $ m\geq 2$ by other means, but handling the drift part $\frac{\nabla u^z}{u^z}$ appears more delicate in this case. 
\end{remark}

\section{\texorpdfstring{$p$}{p}-Laplace equation}\label{sec:pLaplace}
In this section, we consider the parabolic $p$-Laplace equation
\begin{align}\label{eq:pLaplace_eq}
        \partial_t u   =  \nabla \cdot (|\nabla u|^{p-2} \nabla u), \qquad (t,x) \in  (0, +\infty) \times \mathbb{R}^d,
\end{align}
where $p>2$ and $d\in \mathbb{N}$.
The family $\{u^z\}_{z \in \R^d}$ of explicit \emph{Barenblatt solutions} \cite{Barenblatt1952,KaminVazquez1988} to \eqref{eq:pLaplace_eq} with initial condition $u^z(t,\cdot)\,dx \to \delta_z$ weakly as $t \downarrow 0$ is given by
\begin{equation}\label{eq:barenblatt_solution_plaplace}
    u^z(t,x):=t^{-k}\left(C-q \, t^{-\frac{kp}{d(p-1)}}
    |x-z|^{\frac p{p-1}}\right)^{\frac{p-1}{p-2}}_+, \qquad (t,x) \in (0, +\infty) \times \mathbb{R}^d,
\end{equation}
where  $k:=(p-2+\frac{p}{d})^{-1},$  $q:=\tfrac{p-2}{p}(\frac{k}{d} )^{\frac{1}{p-1}}$, $y_+ \coloneqq \max (0,y) $ for any real number $y$, and $C = C_{p,d}\in(0,\infty)$ is the unique constant such that $\lVert  u^{z}(t,\cdot) \rVert_{L^1(\mathbb{R}^d)}=1$ for all $t>0$ (it is a straightforward calculation that such $C$ exists). The support of $u^z(t,\cdot)$ is the closed ball centered at $z$ with radius given by
\begin{equation}\label{eq:support_plaplace}
    R(t) \coloneqq \left( \tfrac{C}{q}\right)^{\frac{p-1}{p}} t^{\frac{k}{d}}, \qquad t \in (0, +\infty). 
\end{equation}

\smallskip
\noindent  
\textbf{Self-similar form of $u^z$.}
    Note that, similarly to the PME studied in the previous section, by scaling analysis of the PDE, any mass-preserving, self-similar, radially symmetric solution to the $p$-Laplace equation centered at $z \in \R^d$ is of the form $t^{-k} g \big(t^{-\frac{k}{d}} |x-z|\big)$ for some radial profile $g : \R_+ \to \R_+$, where $k$ is given as above.
    In particular, in the case of the Barenblatt solution \eqref{eq:barenblatt_solution_plaplace}, we have
    \begin{equation}\label{eq:uz_self_similar_pLaplace}
        u^z(t,x) = t^{-k} g \big(t^{-\frac{k}{d}} |x-z|\big)
    \end{equation}
   with the radial profile
    \begin{equation}\label{eq:g_for_Barenblatt_pLaplace}
    g(\xi) \coloneqq \left(C-q \, \xi^{\frac p{p-1}}\right)^{\frac{p-1}{p-2}}_+, \quad \xi \in \R_+,
    \end{equation}
    whose derivative $g'$ can be expressed as
    \begin{equation}\label{eq:gprime_for_Barenblatt_pLaplace}
        g'(\xi) = - \left( \frac{k}{d} \, \xi \, g(\xi)\right)^{\frac{1}{p-1}},\quad \xi \in (0,R(1)).
    \end{equation}

\subsection{Pure-diffusion interpretation}\label{subsec:pLaplace_diffusion}
Here, we derive a pure-diffusion FPE 
\begin{equation}\label{eq:pLaplace_Ito_diffusion}
    \partial_t u = \Delta \big(a(x,u) u\big),
\end{equation}
solved by the Barenblatt solutions to \eqref{eq:pLaplace_eq}. To do so, we exploit the radial symmetry of the latter. The coefficient $a: \R^d \times \Pscr_* \to \R$ is given by the following \emph{nonlocal} functional of $u$:
\begin{align}
         \label{eq:a_pure_Ito_diffu_pLaplace_nl}
          a(x,u) &\coloneqq -\frac{1}{u(x)} \int_{|x-z|}^{+\infty} |U'(\rho)|^{p-2}U'(\rho) \, d\rho
    \end{align}
   on $\{u>0\}$, and $a(x,u) \coloneqq 0$ on $\{u=0\}$, defined on the set $\Pscr_*$ of absolutely continuous probability measures whose densities are radially symmetric around some point $z$ and whose radial profiles $U$ possess sufficient Sobolev regularity for the above integral to make sense, more precisely,
   \begin{align}
       \Pscr_* \coloneqq \Big\{u \in \Pscr_{{\textup{ac}}}: \,\,\, \exists z \in \R^d, \, U:\R_+\to \R_+ \text{ such that } u(x) = U(|x-z|), \\
       U \in W^{1,1}_{\textup{loc}}(\R_+), \, |U'|^{p-1} \in L^1(\R_+) \Big\}. \label{eq:P_star_pure_Ito_diffu_pLaplace}
   \end{align}
   Indeed, to see where \eqref{eq:a_pure_Ito_diffu_pLaplace_nl} comes from, one can verify that on the set $\Pscr_*$, 
    \begin{equation}\label{eq:check_ansatz}
        \nabla \big(a(x,u)u \big)
        =
        |\nabla u|^{p-2} \nabla u
    \end{equation}
    holds a.e. on the interior of $\{u>0\}$.
    When $a$ is nonnegative along a solution (as will be shown to be the case for the Barenblatt solution), the corresponding MV-SDE is
    \begin{equation}\label{p_lap_pure_diff_sde}
        \begin{dcases}
            dX_t = \sqrt{2a(X_t,u(t,\cdot))} \, dW_t, \\
            \mathcal{L}(X_t) = u(t,x)dx,\quad t>0.
        \end{dcases}
    \end{equation}
    Specifically, for $u=u^z$, using the self-similar form of the Barenblatt solution given in \eqref{eq:uz_self_similar_pLaplace}-\eqref{eq:g_for_Barenblatt_pLaplace}, we will show that the diffusion coefficient has the self-similar form
    \begin{equation}\label{eq:az_self_similar_pLaplace}
        a(x,u^z(t,\cdot)) = t^{-\alpha} h \big(t^{-\tfrac{k}{d}}|x-z|\big),
    \end{equation}
    where $ \alpha \coloneqq 1-\frac{2}{(p-2)d +p} \in (0,1)$, and the radial profile $h:\mathbb{R}_+\to \mathbb{R}_+$ is given by
    \begin{align}
        h(\xi)
        & \coloneqq \frac{1}{g(\xi)} \int_{\xi}^{R(1)} |g'(\rho)|^{p-1} \, d \rho 
        = \frac{k}{d} \frac{1}{g(\xi)} \int_{\xi}^{R(1)} \rho g (\rho)\, d \rho,\label{eq:h_for_az_pLaplace_another}
    \end{align}
    for $\xi\in [0,R(1))$, and $h (\xi) \coloneqq 0$ for $\xi\in [R(1),\infty)$. 
    The second equality above follows from \eqref{eq:gprime_for_Barenblatt_pLaplace}.
    In particular, $ a(x,u^z(t,\cdot)) \geq 0$ pointwise in $(t,x)$. 
    Although the diffusion coefficient appears more complicated here than in the PME case, its self-similar form is of the same type as the pure-diffusion coefficient for the PME \eqref{eq:az_porous_self-similar}–\eqref{eq:h_for_az_porous_self-similar}.

    Before proving the existence of a unique weak solution to the MV-SDE \eqref{p_lap_pure_diff_sde} and constructing the corresponding nonlinear Markov core, we present the following lemma at PDE level, where, in particular, all the above claims are proved.

    \begin{lemma}\label{lem:pLaplace-pure-diffusion_PDE}
        Let $p>2$, $z\in \R^d$ and $a: \R^d \times \Pscr_* \to \R$ and $\Pscr_*$ be as in \eqref{eq:a_pure_Ito_diffu_pLaplace_nl} and \eqref{eq:P_star_pure_Ito_diffu_pLaplace}, respectively. Then $u^z(t,\cdot) \in \Pscr_*$ for every $t>0$, $a(x,u^z(t,\cdot)) \geq 0$ pointwise in $(t,x)$, and $u^z$ is a distributional solution to the pure-diffusion FPE 
        \begin{equation}\label{eq:A-FPE}
            \partial_t u = \Delta \big(a(x,u) u \big).
        \end{equation}
        Moreover, the map $(t,x) \mapsto a (x,u^z(t,\cdot))$ has the self-similar form \eqref{eq:az_self_similar_pLaplace}-\eqref{eq:h_for_az_pLaplace_another}, belongs to $C((0,\infty)\times \mathbb{R}^d;\R_+)$, and satisfies
        \begin{equation}
            \int_{0}^{T} \int_{\mathbb{R}^d} |a (x,u^z(t,\cdot))| u^z (t,x) dx dt < +\infty,\quad \forall T>0.
            \label{p lap diffusion int}
        \end{equation}
    \end{lemma}

    \begin{proof}
        We split the proof into five steps.
        
        \smallskip
        \noindent
        \textbf{Step 1} 
        ($(u^z(t,\cdot))_{t>0}\subset \Pscr_*$ and $a(x,u^z(t,\cdot)) \geq 0$ pointwise). $u^z$ is radially symmetric around $z$, with 
        $u^z(t,x) = U(t,|x-z|)$, where $U(t,r) = t^{-k} g(t^{-\frac k d} r)$ with $g$ given in \eqref{eq:g_for_Barenblatt_pLaplace}. 
        Since $g$ is locally absolutely continuous on $(0,R(1))$ and vanishes on $[R(1),\infty)$, it follows that, for each fixed $t>0$, the map $r\mapsto U(t,r)$ is locally absolutely continuous on $(0,\infty)$ and hence belongs to $W^{1,1}_{\mathrm{loc}}(\R_+)$.
        Using \eqref{eq:uz_self_similar_pLaplace} and \eqref{eq:gprime_for_Barenblatt_pLaplace}, we have for $ r <R(t)$
        \begin{align} 
        \partial_r U(t,r) 
        &=\partial_r\bigl(t^{-k}g(t^{-\frac{k}{d}}r)\bigr) = t^{-k-\frac{k}{d}}g'(t^{-\frac{k}{d}} r) \label{eq:partial_r_U_plaplace_1} \\
        & =-t^{-k-\frac{k}{d}}\Bigl(\frac{k}{d}\,t^{-\frac{k}{d}}r\,g(t^{-\frac{k}{d}}r)\Bigr)^{\frac{1}{p-1}} 
         =-\Bigl(\frac{k}{d} \frac{1}{t} r\,U(t,r)\Bigr)^{\frac{1}{p-1}}. \label{eq:partial_r_U_plaplace_2}
        \end{align}
        This, together with the fact that $U(t,\cdot)$ is compactly supported in $[0,R(t)]$ and bounded, implies $[r \mapsto |\partial_r U (t,r)|^{p-1}] \in L^1(\R_+)$. Thus, $u^z(t,\cdot) \in \Pscr_*$.
        Using \eqref{eq:partial_r_U_plaplace_2}, we have for $|x-z| < R(t)$
        \begin{align}
        a(x,u^z(t,\cdot))
        & = \frac{1}{u^z(t,x)} \int_{|x-z|}^{R(t)} |\partial_\rho U(t,\rho)|^{p-1}  \, d \rho
         =  \frac{1}{u^z(t,x)} \frac{k}{d} \frac1t  \int^{R(t)}_{|x-z|} \rho U(t,\rho) \, d \rho, \label{eq:az_plaplace_another}
        \end{align}
        and, by definition, $a(x,u^z(t,\cdot)) = 0 $ for $|x-z| \geq R(t)$. In particular, $a(x,u^z(t,\cdot)) \geq 0$.
        
        \smallskip
        \noindent
        \textbf{Step 2} 
        (Self-similar form \eqref{eq:az_self_similar_pLaplace}-\eqref{eq:h_for_az_pLaplace_another}). Let $x\in \R^d$ such that $|x-z| < R(t)$.
        Using the self-similar form of $u^z$ \eqref{eq:uz_self_similar_pLaplace},  \eqref{eq:partial_r_U_plaplace_1}, and the change of variables $\eta\coloneqq t^{-\frac{k}{d}}\rho$, we obtain
        \begin{align}
        a(x,u^z(t,\cdot))
        &= \frac{1}{t^{-k}g(t^{-\frac{k}{d}}|x-z|)}
        \int_{R(t)}^{|x-z|}
        \Bigl|t^{-k-\frac{k}{d}}g'(t^{-\frac{k}{d}}\rho)\Bigr|^{p-2}
        \Bigl(t^{-k-\frac{k}{d}}g'(t^{-\frac{k}{d}}\rho)\Bigr)\,d\rho \nonumber\\
        &= t^{k-(p-1)\left(k+\frac{k}{d}\right)}
        \frac{1}{g(t^{-\frac{k}{d}}|x-z|)}
        \int_{R(t)}^{|x-z|}
        |g'(t^{-\frac{k}{d}}\rho)|^{p-2}g'(t^{-\frac{k}{d}}\rho)\,d\rho \\
        &= t^{k-(p-1)\left(k+\frac{k}{d}\right)+\frac{k}{d}}
        \frac{1}{g(t^{-\frac{k}{d}}|x-z|)}
        \int_{R(1)}^{t^{-\frac{k}{d}}|x-z|}
        |g'(\eta)|^{p-2}g'(\eta)\,d\eta \nonumber\\
        &= t^{-\alpha}\,h\bigl(t^{-\frac{k}{d}}|x-z|\bigr),
        \end{align}
        where $\alpha\coloneqq (p-2)(k+\frac{k}{d})=1-\frac{2}{(p-2)d+p} \in (0,1)$ for all $p>2, d\in\mathbb{N}$, and the radial profile $h$ is given by the middle expression in \eqref{eq:h_for_az_pLaplace_another}. 
        Using \eqref{eq:gprime_for_Barenblatt_pLaplace}, $h$ alternatively written as the right-hand expression in \eqref{eq:h_for_az_pLaplace_another}. This defines $h(\xi)$ for $\xi\in [0,R(1))$. We define $h(\xi)\coloneqq 0$ for  $\xi\in [R(1),+\infty)$.
    
        \smallskip
        \noindent
        \textbf{Step 3} (Integrability condition \eqref{p lap diffusion int}).
        Let $T>0$. Using \eqref{eq:uz_self_similar_pLaplace}, the self-similar form \eqref{eq:az_self_similar_pLaplace} proven in the previous step, polar coordinates, and the change of variables $\xi \coloneqq t^{-\frac{k}{d}}r$, we have
        \begin{align}
        \int_{0}^{T} \int_{\mathbb{R}^d} |a(x,u^z(t,\cdot))|\, u^z(t,x)\, d x\, d t
        & = S_{d-1} \int_{0}^{T} \int_0^{+\infty} t^{-\alpha-k}\, h \big(t^{-\tfrac{k}{d}} r\big)  g \big(t^{-\tfrac{k}{d}} r\big) r^{d-1}\, d r d t \\
        & =S_{d-1} \left(\int_0^{T} t^{-\alpha-k+k}\, d t\right) \left(\int_0^{R(1)} h(\xi) g(\xi)\, \xi^{d-1}\, d \xi \right) \\
        & = \frac{S_{d-1}}{d} \left(\int_0^{T} t^{-\alpha}\, d t\right) \left( - \int_0^{R(1)} (h(\xi) g(\xi))' \, \xi^{d}\, d \xi \right) \\
        & = \frac{S_{d-1}}{d} \frac{k}{d} \left(\int_0^{T} t^{-\alpha}\, d t\right) \left(  \int_0^{R(1)} \xi^{d+1} g(\xi) \, d \xi \right)  < + \infty,
        \end{align}
        where $S_{d-1}$ is the surface area of the unit sphere in $\mathbb{R}^d$.
        In the third step, we used integration by parts, and in the fourth step, we used the right-hand side of \eqref{eq:h_for_az_pLaplace_another}.
        The spatial integral is clearly finite, and note that the time integral is also finite since $\alpha \in (0,1)$. 
    
        \smallskip
        \noindent
        \textbf{Step 4} 
        ($u^z$ solves the FPE \eqref{eq:A-FPE} in the sense of \Cref{def:FPE-sol}). By Step 1, $a(x,u^z(t,\cdot)) \geq 0$. Step 3 implies $[(t,x)\mapsto a(x,u^z(t,\cdot))] \in L^1_{\textup{loc}}([0,\infty)\times \R^d; u^z (t,x) d x dt)$.
        It remains to verify the weak formulation of the FPE. 
        By definition of $a$, for any fixed $s>0$, we have 
        \begin{equation}
        a(x,u^z(s,\cdot))u^z(s,x)
        =
        - \int_{|x-z|}^{R(s)} 
        |\partial_\rho U(s,\rho)|^{p-2}\,
        \partial_\rho U(s,\rho)\, d\rho.
        \end{equation}
        Since $[r \mapsto |\partial_r U(s,r)|^{p-1}]\in L^1(\R_+)$ (Step 1), the function defined by the radial integral above is absolutely continuous in the radial variable. 
        Consequently, the map $x \mapsto a(x,u^z(s,\cdot))u^z(s,x)$ belongs to 
        $W^{1,1}_{\mathrm{loc}}(\R^d)$. 
        Taking gradient at a point $x$ with $0<|x-z|<R(s)$ and writing $r \coloneqq |x-z|$, we obtain, for a.e.\ such $x$,
        \begin{equation}\label{eq:ansatz_proof}
        \nabla \big(a(x,u^z(s,\cdot))u^z(s,x)\big)
        =
        |\partial_r U(s,r)|^{p-2}\,
        \partial_r U(s,r)\,
        \frac{x-z}{r}
        =
        |\nabla u^z(s,x)|^{p-2}\nabla u^z(s,x).
        \end{equation}
        Now let $\varphi \in C^2_c(\R^d)$ and fix $t >0$.
        Using integration by parts and the above, we have 
        \begin{align}
            &\int_0^t \int_{\R^d}a(x,u^z(s,\cdot)) \Delta \varphi(x) \,u^z(s,x) dx ds 
            = -\int_0^t \int_{\R^d} \nabla \big(a(x,u^z(s,\cdot))u^z(s,x)\big)\cdot\nabla\varphi(x) dx ds \\ 
            &= -\int_0^t \int_{\R^d}|\nabla u^z(s,x)|^{p-2}\nabla u^z(s,x) \cdot \nabla \varphi(x) dx ds = \int_{\R^d}\varphi(x) u^z(t,x)dx - \varphi(z).
        \end{align}
        where the final equality follows from the fact that $u^z$ solves \eqref{eq:pLaplace_eq}.    
    
        \smallskip
        \noindent
        \textbf{Step 5} ($(t,x) \mapsto a(x,u^z(t,\cdot))$ is continuous). From the self-similar form \eqref{eq:az_self_similar_pLaplace}, it is enough to show that the function $h: \R_+ \to \R_+$ is continuous. $h$ is clearly continuous on $[0,R(1))$, which can be seen by the right-hand side of \eqref{eq:h_for_az_pLaplace_another} and the fact that $g$ is continuous on $[0,R(1))$. Also recall that $h = 0$ on $(R(1),\infty)$. Thus, it remains to show the continuity of $h$ at $R(1)$. Indeed, using \eqref{eq:h_for_az_pLaplace_another}, l'H\^{o}pital rule, and \eqref{eq:gprime_for_Barenblatt_pLaplace}, we have
        \begin{align}
          \lim_{\xi\nearrow R(1)} h(\xi)
          & = \frac{k}{d} \,\lim_{r\nearrow R(1)}  \frac{1}{g(\xi)}\int_{\xi}^{R(1)} \rho\,g(\rho)\,d\rho = \left( \frac{k}{d} \right)^{\frac{p-2}{p-1}} \lim_{\xi\nearrow R(1)} \left( \xi g(\xi) \right)^{\frac{p-2}{p-1}} = 0.
        \end{align}
        Therefore, we conclude that
        $(t,x) \mapsto a(x,u^z(t,\cdot))$ is jointly continuous on $(0,+\infty)\times \mathbb{R}^d$.
    \end{proof} 

    Our main result of this section is the existence of solutions to the MV-SDE \eqref{p_lap_pure_diff_sde} with Barenblatt one-dimensional time marginal densities, as well as the construction of the corresponding nonlinear Markov core, which in this case, notably, consists of \emph{martingales}.
    \begin{proposition}\label{proposition:pure_diffusion_plaplace}
        Let $p>2$. For every $z\in \R^d$, there exists a unique weak solution $X^{1,z}$ to the MV-SDE \eqref{p_lap_pure_diff_sde} with one-dimensional time marginals given by the Barenblatt solution \eqref{eq:barenblatt_solution_plaplace}, i.e., $\mathcal{L}(X^{1,z}_t) = u^z(t,x)dx$, $t>0$, and $\mathcal{L}(X^{1,z}_0) = \delta_z$. 
        \\
        Furthermore, the corresponding family of solution path laws $\{\mathbb{V}^{1,z}\}_{z \in \R^d}$, $\mathbb{V}^{1,z} := \mathcal{L}(X^{1,z})$, is a nonlinear Markov core for $\{u^z\}_{z\in \R^d}$, which is uniquely determined by  $\{u^z\}_{z\in \R^d}$ and \eqref{p_lap_pure_diff_sde}.
    \end{proposition}

\begin{proof}
     \Cref{lem:pLaplace-pure-diffusion_PDE} allows us to apply the superposition principle (\Cref{thm:SP-pr}) for every $z\in \R^d$ to construct a weak solution $X^{1,z}$ to the MV-SDE \eqref{p_lap_pure_diff_sde} with $\mathcal{L}(X^{1,z}_t) = u^z(t,x)dx$, $t>0$, and initial condition $\mathcal{L}(X^{1,z}_0) = \delta_z$.
     \\
     For the remaining parts of the assertion, we apply \Cref{crl:nl_markov}. To this end, let $s>0$ and consider the linear FPE
    \begin{equation}
    \label{lin p lap diff}
        \partial_t v=\Delta (a^z(s+t, x) v),\quad v(s,x) = u^z(s,x),
    \end{equation}
    where $a^z(t,x) := a(x,u^z(t,\cdot))$.
    Let $v$ be such that $ [t\mapsto v(t,x)dx] \in C(\R_+;\Pscr)$ and $ 0\leq  v(t,x) \leq Cu^z(s+t,x)\,dxdt-\text{a.s. } \forall t \geq 0 \text{ for some }C\geq 1$ independent from $t$. Then $v \in (L^1 \cap L^\infty ) ([0,T]\times \R^d)$, and for all $T>0$, there is a compact set $K = K(T) \subseteq \R^d$ such that $v(t,x) = 0$ for all $(t,x) \in [0,T]\times K^c$. Thus, and since, by \Cref{lem:pLaplace-pure-diffusion_PDE}, $(t,x)\mapsto a^z(s+t,x)$ is continuous on $[0,T] \times \R^d$, we can apply \cite[Theorem 3.1]{Belaribi2012} to conclude that any such solution $v$ coincides with $u^z_{s+} \coloneqq (u^z({s+t},\cdot))_{t\geq 0} $ a.e. Thus, \Cref{crl:nl_markov} applies and yields all remaining assertions.
\end{proof}

\subsection{Pure-drift interpretation}\label{subsec:pLaplace_drift}
Here, we recast the $p$-Laplace quation \eqref{eq:pLaplace_eq} as 
\begin{equation}\label{eq:pLaplace_drift}
        \partial_t u   =  - \nabla \cdot \Big( \frac{-|\nabla u|^{p-2} \nabla u}{u}u \Big)
    \end{equation}
i.e. as the first-order nonlinear FPE with coefficients $(a,b): \R^d \times \Pscr_{*} \to \R_+ \times \R^d$,
\begin{equation}\label{B}
    a(x,u) \coloneqq 0, \qquad b(x,u) = -\frac{|\nabla u(x)|^{p-2} \nabla u(x)}{u(x)},
\end{equation}
and $b(x,u) := 0$ on $\{u=0\}$, defined on $\Pscr_* = \Pscr_{\textup{ac}} \cap W^{1,1}_{\textup{loc}}(\R^d)$. The associated MV-SDE is the distribution-dependent ODE
\begin{equation}\label{eq:DDSDE-pL}
    \begin{dcases}
        dX_t = -\frac{|\nabla u(t,X_t)|^{p-2} \nabla u(t,X_t)}{u(t,X_t)}dt, \\
        \mathcal{L}(X_t) = u(t,x)dx, \,\,\, \forall t>0.
    \end{dcases}
\end{equation}
Specifically, for $u = u^z$ given by the Barenblatt solution \eqref{eq:barenblatt_solution_plaplace}, a direct calculation (included in the proof below) gives
\begin{equation}\label{eqeqeq}
    b(x,u^z(t,\cdot)) = \frac{k}{d}\frac{x-z}{t} \mathds{1}_{B_{R(t)}(z)}(x),
\end{equation}
where $R(t)$ is given in \eqref{eq:support_plaplace}.
Then, \eqref{eq:DDSDE-pL} turns into
\begin{equation}\label{eq:pL-special}
    \begin{dcases}
    dX_t = \frac{k}{d}\frac{X_t-z}{t}\mathds{1}_{B_{R(t)}(z)}(X_t)dt, \\
    \mathcal{L}(X_t) = u^z(t,x)dx, \,\,\, \forall t>0.
    \end{dcases}
\end{equation}
An explicit solution to \eqref{eq:pL-special} with $\mathcal{L}(X_0) = \delta_z$ is given by $(X^{0,z}_t)_{t\geq 0}$,
\begin{equation}\label{eq:solutions_ODE_plap}
    X^{0,z}_t \coloneqq z+\eta \, t^{\frac{k}{d}}, 
\end{equation}
where $\eta$ is an $\R^d$-valued random variable  with $\mathcal{L}(\eta)=u^{0}(1,x) dx$.

\begin{proposition}\label{prop:pL-drift}
    Let $p>2$. For every $z\in \R^d$, $X^{0,z}$ from \eqref{eq:solutions_ODE_plap} is the unique weak solution to the MV-SDE \eqref{eq:DDSDE-pL} with one-dimensional time marginals given by the Barenblatt solution \eqref{eq:barenblatt_solution_plaplace}, i.e., $\mathcal{L}(X^{0,z}_t) = u^z(t,x)dx$, $t >0$, and $\mathcal{L}(X^{0,z}_0) = \delta_z$.
    \\
    Furthermore, the corresponding family of solution path laws $\{\mathbb{V}^{0,z}\}_{z\in \R^d}$, $\mathbb{V}^{0,z} := \mathcal{L}(X^{0,z})$, is a nonlinear Markov core for $\{u^z\}_{z\in \R^d}$, which is uniquely determined by $\{u^z\}_{z\in \R^d}$ and \eqref{eq:DDSDE-pL}.
\end{proposition}
\begin{remark}\label{rem:similar-ODEs}
    Comparing \eqref{eq:pL-special} with the ODEs associated to the pure-drift formulations of the heat equation and PME, i.e. \eqref{eq:HE-ODE_simplification} and \eqref{eq:PME-special}, respectively, we note that these ODEs coincide, up to their respective growth factor $\frac 1 2$ and $\frac k d$, and the cut off $\mathds{1}_{B(z)_{R(t)}}$ in \eqref{eq:pL-special} and \eqref{eq:PME-special}.
    The scaling factors stem from the self-similar forms of $u^z$ (\eqref{eq:uz_self_similar_heat}, \eqref{eq:uz_self_similar_porous}, \eqref{eq:uz_self_similar_pLaplace}), and the cut off corresponds to finite propagation speed in $p$-Laplace and PME cases, in contrast to the infinite propagation speed of the heat equation. The ODE solutions are shown in \Cref{fig:simulations-grid} (first row).
\end{remark}

\begin{proof}[Proof of \Cref{prop:pL-drift}]
The proof proceeds analogously to the proof of \Cref{prop2} for PME, i.e., all assertions will follow from \Cref{crl:nl_markov} once its conditions are verified.

First, we derive the explicit formula of the drift \eqref{eqeqeq}. By taking the gradient of $u^z(t,x)$ in \eqref{eq:uz_self_similar_pLaplace} at a point $x$ with $0<|x-z|<R(t)$, writing $r \coloneqq |x-z|$, and using \eqref{eq:gprime_for_Barenblatt_pLaplace}, we obtain 
    \begin{align}
        |\nabla u^z|^{p-2} \nabla u^z (t,x)
        & = - t^{-(p-1)(k+\frac{k}{d})}   |g'(t^{-\frac{k}{d}} r)|^{p-1} \frac{x-z}{r} \\
        & = -  \frac{k}{d}  t^{-(p-1)(k+\frac{k}{d}) - \frac{k}{d}}   g(t^{-\frac{k}{d}} r) (x-z)  = -  \frac{k}{d} u^z(t,x) \frac{ x-z}{t}, 
    \end{align}
    which, after dividing by $u(t,x)^z >0$, yields \eqref{eqeqeq}.

    Next, we check the integrability condition. For any $T > 0$, using \eqref{eqeqeq}, \eqref{eq:uz_self_similar_pLaplace}, and the change of variables $\xi \coloneqq t^{-\frac{k}{d}}r$, we have
    \begin{align}
        \int_{0}^{T} \int_{\mathbb{R}^d} |b(x,u^z (t,\cdot))| u^z (t,x) dx dt & = \frac{k}{d} S_{d-1} \int_{0}^{T}t^{-1-k}  \int_{0}^{+\infty} r g \big(t^{-\frac{k}{d}} r\big) r^{d-1} dr d t \\
        & = \frac{k}{d} S_{d-1} \left(\int_{0}^{T}t^{-(1-\frac{k}{d}) } d t \right) \left( \int_{0}^{R(1)} \xi^d  g (\xi ) d\xi   \right) < + \infty, \label{eq:proof_integ_b_plap}
     \end{align}
     where $S_{d-1}$ denotes the surface area of the unit sphere in $\R^d$.
     The time integral is finite since $(1-\frac{k}{d}) \in (0,1)$.
     The spatial integral is also finite, since $g$, as a continuous function, is bounded on $[0,R(1)]$. It is clear that solutions to \eqref{eq:pLaplace_eq} (in particular $u^z$) are also solutions to \eqref{eq:pLaplace_drift}. 

    It remains to prove that, for each $s>0$ and $z\in \R^d$, the linear FPE
\begin{equation}\label{eqg}
    \partial_t v(t,x) = -\nabla \cdot \big(b(x,u^z(s+t,\cdot)) v(t,x)\big),\quad v(0,x) = u^z(s,x),
\end{equation}
has a unique distributional solution in the class
 \begin{equation}\label{eqh}
      \{v: [t\mapsto v(t,x)dx] \in C(\R_+;\Pscr), v(t,x) \leq Cu^z(s+t,x) \text{ a.s. for some }C\geq 1\}.
 \end{equation}
 For this we argue as in the proof of \Cref{prop2}: By \eqref{eqeqeq} and since $v(t,x) = 0$ a.s. on $B_{R(s+t)}(z)^c$ for all $t\geq 0$ for every solution $v$ to \eqref{eqg} in the class \eqref{eqh}, it follows that every solution from \eqref{eqh} also solves the first-order linear FPE with coefficient $(t,x) \mapsto \frac k d \frac{x-z}{t+s}$ and initial condition $u^z(s,x)$. But since this coefficient is Lipschitz continuous in $x$ with Lipschitz constant bounded from above independently of $t\geq 0$ (here we use $s>0$), it follows that this equation has a unique distributional solution. Thus, uniqueness for \eqref{eqg} in the class \eqref{eqh} follows. Consequently, \Cref{crl:nl_markov} applies and yields all assertions.    
\end{proof}

Clearly, the nonlinear Markov cores $\{\mathbb{V}^{1,z}\}_{z\in \R^d}$ and $\{\mathbb{V}^{0,z}\}_{z\in \R^d}$ constructed in \Cref{proposition:pure_diffusion_plaplace} and \Cref{prop:pL-drift}, respectively, consist of mutually singular path laws, i.e., the path measures $\mathbb{V}^{1,z}$ and $\mathbb{V}^{0,z}$ are mutually singular for each $z\in \R^d$. This is readily seen, since by construction, $\mathbb{V}^{0,z}$ is concentrated on 
$$G:= \{\gamma \in \Cscr \,|\, \gamma: t \mapsto z + c t^{\frac k d}, c \in \R^d\} \subseteq \Cscr \cap C^1((0,\infty);\R^d),$$ while clearly $\mathbb{V}^{1,z}(G) = 0$. Indeed $\mathbb{V}^{1,z}(G) >0$ would contradict the fact that $\mathbb{V}^{1,z}$-a.e. path is at most $\frac 1 2$-H\"{o}lder continuous.

\begin{remark}[Interpolation between pure-diffusion and pure-drift interpretations of $p$-Laplace]\label{rem:interpolation-pL}
   Our approach to interpolating between the pure-drift and pure-diffusion interpretations (\Cref{subsect:step1}) can be used to solve, for every interpolation parameter $\beta \in  [0, \infty)$, the MV-SDE
    \begin{equation}\label{eq:pL-beta-SDE}
        \begin{dcases}
            dX_t = (1-\beta) b(X_t,u(t,\cdot) ) dt  + \sqrt{\beta} \sqrt{2 a(X_t,u(t,\cdot))} \d W_t, \\
            \mathcal{L}(X_t) = u(t,x)dx,\,\,\, \forall t >0,
        \end{dcases}
    \end{equation}
    where $a$ and $b$ are given in \eqref{eq:a_pure_Ito_diffu_pLaplace_nl} and \eqref{B}, respectively, defined on the intersection of the corresponding domains. Of course, $u^z$ solves the corresponding interpolated FPE, and since
    $$\int_0^T \int_{\R^d}\big( |b(x,u^z(t,\cdot))| + |a(x,u^z(t,\cdot))| \big) u^z(t,x) dx dt < \infty, \quad \forall T>0,$$
    which follows from \eqref{p lap diffusion int} and \eqref{eq:proof_integ_b_plap}, \Cref{thm:SP-pr} applies. Moreover, one can check that in this case (in analogy with \Cref{prop_porous_beta} for the PME), all conditions of \Cref{crl:nl_markov} hold and, therefore, one constructs for every $\beta \geq 0$, a nonlinear Markov core $\{\mathbb{V}^{\beta,z}\}_{z\in \R^d}$ for $\{u^z\}_{z\in \R^d}$, consisting of the unique solution path laws to \eqref{eq:pL-beta-SDE} with one-dimensional time marginal densities $u^z(t,\cdot)$. We leave the details to the reader. 
\end{remark}

\subsection{Further interpretations}\label{eq:further_interpretation_plaplace}
\subsubsection{$p$-Brownian motion}
In \cite{BarbuRehmeierRockner2024}, Barbu, R\"{o}ckner and the last named author of the present paper constructed a nonlinear Markov process with one-dimensional time marginal densities given by the Barenblatt solutions to the $p$-Laplace equation \eqref{eq:pLaplace_eq} by recasting it as the nonlinear FPE
\begin{equation}
    \partial_t u = \Delta \big(a (x,u) u\big) -  \nabla \cdot \big(\nabla a(x,u) u\big),
\end{equation}
where $a: \R^d \times \Pscr_{*} \to \R_+$ is given by 
$$a(x, u) =  |\nabla u |^{p-2} (x)$$
defined on a suitable domain $\Pscr_*$.
Observe that this interpretation gives rise to a special gradient structure on the coefficients $(a,b)$, namely $b = \nabla a$.
The corresponding nonlinear Markov process, which is called \emph{p-Brownian motion} in the aforementioned paper, consists of the unique solution path laws to the associated MV-SDE
 \begin{equation}
        \begin{dcases}
            dX_t = \nabla \big(|\nabla u(t,X_t)|^{p-2} \big)dt + \sqrt{2}|\nabla u(t,X_t)|^{\frac{p-2}{2}} \, dW_t, \\
            \mathcal{L}(X_t) = u(t,x)dx,\quad t>0,
        \end{dcases}
    \end{equation}
and coincides with standard Brownian motion for $p=2$. While $p$-Brownian motion was, at the time of \cite{BarbuRehmeierRockner2024}, the first probabilistic Markovian counterpart of the $p$-Laplace Barenblatt solutions, the results of the present paper show that there exists a large class of such counterparts (each associated with an MV-SDE), and thus raise the question of what singles out $p$-Brownian motion as the ``natural'' one.  

\subsubsection{Pure-Stratonovich-diffusion interpretation}\label{subsubsec:pL-Stratono}
It is a natural question whether there exists a pure Stratonovich diffusion interpretation of the $p$-Laplace equation (similarly to \Cref{subsect:PME-Stratono} for the PME). As pointed out in Appendix \ref{subsec:theta_interpretations}, this reduces to finding coefficients $(a,b)$ satisfying the gradient relation $b = \tfrac{1}{2} \nabla a $. This amounts to recasting \eqref{eq:pLaplace_eq} as a nonlinear FPE
\begin{equation}\label{eqt}
    \partial_t u = \Delta \big( a(x,u) u)\big) - \frac 1 2 \nabla \cdot \big(\nabla a  (x,u)u\big).
\end{equation}
Similarly to \Cref{subsec:pLaplace_diffusion}, one can show that a choice for $a:\R^d \times \Pscr_*\to \R$ is given by
\begin{equation}
    a(x,u) := - \frac{2}{u(x)^{2}}  \int_{|x-z|}^{+\infty} |U'(\rho)|^{p-2}U'(\rho) U(\rho) \, d \rho ,
 \end{equation}
   on $\{u>0\}$, and $a(x,u) \coloneqq 0$ on $\{u=0\}$, defined on the domain
   \begin{align}
       \Pscr_* \coloneqq \Big\{u \in \Pscr_{{\textup{ac}}}: \,\,\, \exists z \in \R^d, \, U:\R_+\to \R_+ \text{ such that } u(x) = U(|x-z|), \\
       U \in W^{1,1}_{\textup{loc}}(\R_+), \, |U'|^{p-1} U^{\frac{1}{\theta}-1}\in L^1(\R_+), \,  a(\cdot,u) \in W^{1,1}_{\textup{loc}}(\R^d) \Big\}. \label{eq:P_star_generalized_diffusion_pLaplace}
   \end{align}
   The corresponding MV-SDE in this case is given by
 \begin{equation}
\begin{dcases}
    dX_t = \frac 1 2 \nabla a(X_t,u(t))dt + \sqrt{2a(X_t,u(t)},\quad t>0
    \\
    \mathcal{L}(X_t) = u(t,x)dx,\,\,\, t >0.
\end{dcases}
\end{equation}  
This MV-SDE is equivalent, at least formally, to the pure-Stratonovich diffusion MV-SDE
\begin{equation}
\begin{dcases}
    dX_t = \sqrt{2a(X_t,u(t))}\circ dW_t,\quad t>0
    \\
    \mathcal{L}(X_t) = u(t,x)dx,\,\,\, t >0.
\end{dcases}
\end{equation}
However, the rigorous establishment of this as a Stratonovich SDE is delicate and not addressed here. We comment on this difficulty in Appendix \ref{subsec:theta_interpretations}.

\subsubsection{Generalized pure-diffusion interpretation}\label{subsec:pLaplace_gen_diffusion}
Finally, we mention a family of pure-diffusion interpretations of \eqref{eq:pLaplace_eq} comprising $p$-Brownian motion as well as the pure Stratonovich interpretation.
This family is given by nonlinear FPEs whose coefficients $(a,b)$ satisfy a generalized gradient structure of the form $b = (1-\theta) \nabla a $, depending on a parameter $\theta \in [0,1]$,
    \begin{equation}\label{eq:pLaplace_gen_diffusion}
        \partial_t u =  \Delta \big(a_\theta (x,u) u\big) -  \nabla \cdot \big((1-\theta)\nabla a_\theta(x,u) u\big),
    \end{equation}
For $\theta = 0$, $a_0(x,u) := |\nabla u(x)|^{p-2}$, while for $\theta \in (0,1]$, similarly to \Cref{subsec:pLaplace_diffusion}, one can show that a choice for $a_\theta:\R^d \times \Pscr_*\to \R$ is given by
    \begin{align}
         \label{eq:a_generalized_diffusion_pLaplace}
          a_\theta(x,u) &\coloneqq - \frac{1}{\theta u(x)^{\frac{1}{\theta}}}  \int_{|x-z|}^{+\infty} |U'(\rho)|^{p-2}U'(\rho) U(\rho)^{\frac{1}{\theta}-1} \, d \rho ,
    \end{align}
   on $\{u>0\}$, and $a_\theta(x,u) := 0$ on $\{u=0\}$, defined on the domain
   \begin{align}
       \Pscr_* \coloneqq \Big\{u \in \Pscr_{{\textup{ac}}}: \,\,\, \exists z \in \R^d, \, U:\R_+\to \R_+ \text{ such that } u(x) = U(|x-z|), \\
       U \in W^{1,1}_{\textup{loc}}(\R_+), \, |U'|^{p-1} U^{\frac{1}{\theta}-1}\in L^1(\R_+), \,  a_\theta(\cdot,u) \in W^{1,1}_{\textup{loc}}(\R^d) \Big\}. \label{eq:P_star_generalized_diffusion_pLaplace}
   \end{align}
    One retrieves the pure-diffusion case from \Cref{subsec:pLaplace_diffusion} via $\theta =1$, and the pure-Stratonovich diffusion case from \Cref{subsubsec:pL-Stratono} via $\theta = \frac 1 2$, and the $p$-Brownian motion from \cite{BarbuRehmeierRockner2024} via $\theta = 0$. The corresponding MV-SDEs to \eqref{eq:pLaplace_gen_diffusion} are
    \begin{equation}\label{eq:MVSDE_gen_diffus_pLaplace}
        \begin{dcases}
            dX_t = (1-\theta)\nabla a_\theta(X_t,u(t,\cdot)) dt + \sqrt{2a_\theta(X_t,u(t,\cdot))} \, dW_t, \\
            \mathcal{L}(X_t) = u(t,x)dx,\quad t>0,
        \end{dcases}
    \end{equation}
    provided $a_\theta(x,u(t,\cdot)) \geq 0$ pointwise along the solution.
    Formally, this MV-SDE is equivalent to the pure-diffusion equation
    \begin{equation}\label{eq:MVSDE_gen_diffus_pLaplace_circ}
    \begin{dcases}
        dX_t = \sqrt{2a_\theta(X_t,u(t,\cdot))} \circ^\theta  dW_t,\\
            \mathcal{L}(X_t) = u(t,x)dx,\quad t>0,
        \end{dcases}
    \end{equation}   
    where $\circ^\theta$ denotes $\theta$-stochastic integration (see Appendix \ref{subsec:theta_Integration}). In particular, $\theta= 0$ and $\theta = 1$ correspond to non-anticipating (It\^{o}) and ``fully anticipating'' stochastic integrals, respectively. 
    The (formal) equivalence between \eqref{eq:MVSDE_gen_diffus_pLaplace} and \eqref{eq:MVSDE_gen_diffus_pLaplace_circ}, as well as the reason why a rigorous justification of their equivalence seems difficult, is discussed in Appendix \ref{subsec:theta_Integration}. 
    
    Potentially interesting to us is that, at least formally, $p$-Brownian motion as constructed in \cite{BarbuRehmeierRockner2024} should solve the pure-diffusion MV-SDE \eqref{eq:MVSDE_gen_diffus_pLaplace_circ} for $\theta = 0$, i.e., with ``fully anticipating'' stochastic integral.

 \begin{remark}[$\theta$-pure-diffusion interpretation of PME]
    In the case of the porous medium equation, all $\theta$-pure-diffusion interpretations lie on the $\beta$-interpolation spectrum between pure-drift and pure-diffusion cases. We recover them by choosing $\beta = \frac{m}{1+\theta (m-1)}$ and thus
    \begin{equation}
    a=\frac{m}{1+\theta (m-1)}u^{m-1}, \qquad b = (1-\theta) \nabla a.
    \end{equation} 
    Again, $\theta=1/2$ coincides with the Stratonovich formulation, $\theta=1$ with the It\^{o} formulation. Hence, the process that corresponds to $p$-BM and thus a purely anticipatory pure diffusion interpretation would be $\theta=0$, meaning the one coming from the interpretation 
    \begin{equation}
    a=mu^{m-1}, \qquad b=\nabla a.
    \end{equation} 
\end{remark}

\section{Identifying McKean--Vlasov SDEs with geometries}\label{sec:SPDE}
In the previous sections, we constructed different MV-SDEs and their associated nonlinear Markov processes whose marginals coincide with a prescribed family of probability densities solving a PDE. Considering the MV-SDE as the mean field limit of a particle system, this property can be viewed as a law of large numbers for the latter. Taking into account fluctuations on top of the law of large numbers, one naturally arrives at a corresponding central limit theorem capturing mesoscopic scales of the particle system: (generalized) Dean--Kawasaki equations.

An alternative way of deriving the Dean--Kawasaki type equation is via Otto calculus: We interpret a PDE as a gradient flow with respect to an underlying metric on the formal infinite-dimensional manifold of probability measures. Perturbing the corresponding gradient flow structure with a space-time white noise that respects the geometry, we also arrive at a (generalized) Dean--Kawasaki equation.

\subsection{General strategy}
We aim to identify MV-SDEs associated with different interpretations of FPEs with a geometric interpretation of the latter as follows: For a MV-SDE, we first construct an associated particle system and derive its associated generalized Dean--Kawasaki equation. Then we provide geometric interpretations of the associated FPE and derive a generalized Dean--Kawasaki equation again. If the noise term of these equations (asymptotically) coincides, the mesoscopic fluctuations captured in the MV-SDE coincide with the ones suggested by a geometric interpretation. 

We perform this program first for the heat equation as an illustrative example and then discuss the general strategy. In \Cref{sec:DK-2}, we apply these considerations to the porous medium equation. 

We stress that all derivations in the present section are entirely formal, by which we hope to initiate a new perspective on the nonlinear Markov processes constructed earlier.

\subsubsection{The heat equation as an illustrative example}
To illustrate our strategy, we consider the heat equation 
\begin{equation}
    \label{heat spde}
    \partial_t\rho=\Delta \rho,
\end{equation}
and recall the following classical derivations
\begin{itemize}
    \item From a particle system of i.i.d. Brownian motion to the Dean--Kawasaki equation;
    \item From a Wasserstein gradient flow interpretation of \eqref{heat spde} to the Dean--Kawasaki equation.
\end{itemize}
In particular, we stress at this point that independent Brownian particles are consistent with a geometric interpretation of \eqref{heat spde} as a gradient flow in the sense that they imply the same mesoscopic SPDE. 

\begin{figure}
\centering
\scalebox{1}{
\begin{tikzpicture}

\node[draw, rounded corners, minimum height=1.6cm, align=center, font=\footnotesize]
(geo) at (0,0) {A geometry \\ on $\Pscr (\R^d)$ \\ + an energy \\ functional};

\node[draw, rounded corners, minimum height=1.6cm, align=center, font=\footnotesize] 
(dk) at (4.6,0) {Generalized\\ Dean--Kawasaki\\ equation};

\node[draw, dashed, rounded corners, minimum height=1.6cm, align=center, font=\footnotesize] 
(ps) at (8.9,0) {Particle\\system};

\node[draw, rounded corners, minimum height=1.6cm, align=center, font=\footnotesize] 
(mv) at (13.5,0) {MV-SDE};

\draw[->, thick] (geo) -- 
node[above, align=center, font=\footnotesize]
{Perturbation of\\gradient flow\\that respects\\ the geometry} 
(dk);

\draw[->, thick] (ps) -- 
node[above, align=center, font=\footnotesize]
{Sec.\,\ref{McKean to DK} Step (II)\\ derivation in the \\sense of Dean \\ \,} 
(dk);

\draw[->, thick] ([yshift=0.3cm]mv.west) -- 
node[above, align=center, font=\footnotesize]
{Sec.\,\ref{McKean to DK} Step (I)\\ a freezing procedure \\ \, } 
([yshift=0.3cm]ps.east);

\draw[->, thick] ([yshift=-0.3cm]ps.east) -- 
node[below, align=center, font=\footnotesize]
{\, \\ mean-field limit} 
([yshift=-0.3cm]mv.west);

\end{tikzpicture}
}
\captionsetup{font=footnotesize}
\caption{Framework of Section 8.}
\label{fig:DK}
\end{figure}
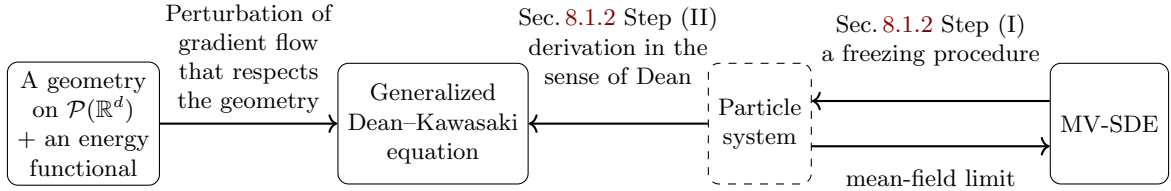

\subsubsection*{From particle system to Dean--Kawasaki equation}
A reference for this derivation is, for example, \cite[Section 1.1]{Cornalba2019}. See \cite{Dean1996} for the original paper by Dean. For independent Brownian motions $(W^k)_k$, consider independent particle motion modeled by
\[
dX^k_t=\sqrt{2}dW^k_t, \qquad X^k_0=x.
\]
We study the evolution of the empirical measure 
\[
\rho^N_t:=\frac{1}{N}\sum_{k=1}^N\delta_{X^k_t}.
\]
To this end, let $\phi\in C^\infty_c(\R^d)$ be a test function. Using It\^{o}'s formula and summation over particles, we obtain (using the short-hand notation $\langle \varphi, \mu\rangle :=\int_{\R^d} \varphi \,d\mu$)
\begin{equation}
    \langle \rho^N_t, \phi\rangle=\langle \rho^N_0, \phi\rangle+\int_0^t\langle \rho^N_s, \Delta \phi\rangle ds+\sqrt{2}M^N_t,
    \label{pde exact}
\end{equation}
where
\[
M^N_t=\frac{1}{N}\int_0^t\sum_{k=1}^N\nabla \phi(X^k_s)dW_s^k.
\]
Note that \eqref{pde exact} cannot be written in closed form as an equation for $\rho^N$ due to the noise term. To overcome this (that is, to treat \eqref{pde exact} as a closed SPDE for $\rho^N$), \cite{Dean1996} suggests to replace $M^N$ by another martingale preserving the quadratic variation. We observe 
\[
[ M^N]_t=\frac{1}{N^2}\sum_{k=1}^N\int_0^t |\nabla \phi(X^k_s)|^2ds=\frac{1}{N}\int_0^t \langle \rho^N_s, |\nabla \phi|^2\rangle ds.
\]
Formally, the process
\[
\tilde{M}^N_t \coloneqq \frac{1}{\sqrt{N}}\int_0^t \langle \sqrt{\rho^N}\xi_s, \nabla \phi\rangle ds
\]
has the same covariance structure as $M^N_t$,
where $\xi$ denotes vector valued space-time white noise. Hence, replacing $M^N$ in \eqref{pde exact} by $\tilde{M}^N$, we obtain the weak formulation of the Dean--Kawasaki equation 
\begin{equation}
    \label{Dean Kawasaki}
    \partial_t \rho=\Delta \rho +\sqrt{\frac{2}{N}}\nabla \cdot(\sqrt{\rho}\xi), \qquad \rho_0=\delta_x.
\end{equation}
It was observed in \cite{Konarovskyi2019} that \eqref{Dean Kawasaki} is either ill-posed or admits only trivial solutions. From this observation, several ``regularized versions of'' \eqref{Dean Kawasaki} have been considered in the literature, e.g \cite{Cornalba2019, Cornalba2021, Fehrman2024, Fehrman2025}.

\subsubsection*{From Wasserstein geometry to Dean--Kawasaki equation}
\label{Wasserstein geometry}
We recall how the Dean-Kawasaki equation can be derived from a geometry on Wasserstein space in the framework of Otto calculus \cite{Otto2001}. We refer to \cite[Chapter 4.2]{Figalli2021}, \cite[Chapter 18]{Ambrosio2024} for general introductions. We take as a formal manifold
\begin{equation}
\mathcal{M}\coloneqq \mathcal{P}_2 (\R^d) :=\Bigl\{ \mu \in \mathcal{P} \big| \int |x|^2\mu(dx)<+\infty \Bigr\},
\end{equation}
and formal tangent plane 
\begin{equation}\label{eq:TrhoM_1}
T_\rho\mathcal{M} \coloneqq \Bigl\{ s: \R^d \to \R \, \big| \, s+2\nabla \cdot(\rho v)=0 \ \text{for}\ \rho\in \mathcal{M}\  \text{and some velocity field}\ v\Bigr\}.
\end{equation}
For any $s\in T_\rho \mathcal{M}$, assume there is a unique (modulo constants) potential $\phi$ such that 
\begin{equation}
    \label{direction to potential}
    s=-2\nabla \cdot(\rho\nabla \phi).
\end{equation}
For $s_1, s_2\in T_\rho \mathcal{M}$, we can therefore define a scalar product via
\[
\langle s_1, s_2\rangle_\rho:=2\int \nabla \phi_1\cdot \nabla \phi_2 \rho dx,
\]
where $\phi_i$ is the potential retrieved from $s_i$ via \eqref{direction to potential}. We have therefore constructed a formal Riemannian manifold. It can be shown that the distance between two of its points $\rho_1, \rho_2$ computed by means of the induced geodesics coincides --- up to the constant $\sqrt2$ --- with the Wasserstein metric $\mathcal{W}_2(\rho_1, \rho_2)$ via the Benamou--Brenier formula \cite{BenamouBrenier2000}; see e.g. \cite[Chapter 17]{Ambrosio2024}.

Consider now the entropy functional
\[
E(\rho):=\frac{1}{2}\int \rho\log(\rho)dx.
\]
Using the formal Riemannian structure above, on can show (see e.g. \cite[Example 4.2.3]{Figalli2021})
\[
\mbox{grad}_{\mathcal{W}_2}E(\rho)=-\Delta \rho.
\]
Thus, the heat equation can be interpreted as the gradient flow
\[
\partial_t \rho =-\mbox{grad}_{\mathcal{W}_2}E(\rho).
\]
Based on this geometric interpretation of the heat equation, there is now a canonical way to introduce noise: If we want to move a probability measure along direction $s$, we do so by means of the continuity equation 
\[
s+2\nabla \cdot(\rho v)=0,
\]
and we choose the velocity field to be of the form $v=\nabla \phi$, which has minimal $L^2(\rho)$-norm among all such velocity fields (see \cite[Sections 8.3 and 8.4]{AGS2008GFs} for more details). Suppose now that at each point $\rho$ of the manifold, we perturb the optimal velocity field $v$ by a ``Brownian motion on the tangent space''. Since the tangent space at $\rho$ is a Hilbert space, we can think of ``Brownian motion on $T_\rho \mathcal{M}$'' as an isonormal Gaussian process \cite[Definition 1.1.1]{nualart2006malliavin}, i.e., a family $W=\{W(s): \ s\in T_\rho \mathcal{M}\}$ satisfying
\begin{equation}
    \label{isonormal}
    \mathbb E[W(s_1)W(s_2)]=\langle s_1, s_2\rangle_\rho.
\end{equation}
Suppose now that there exists $\eta\in T_\rho \mathcal{M}$ such that $W(s_i)=\langle \eta, s_i\rangle_\rho$  and denote by $\phi_\eta$ the associated potential. Then
\[
\mathbb E[W(s_1)W(s_2)]=4\mathbb E \left[\int_{\R^d}\nabla \phi_1 \cdot \nabla \phi_\eta \rho dx \int_{\R^d}\nabla \phi_2 \cdot \nabla \phi_\eta \rho dy \right],
\]
as well as 
\[
\langle s_1, s_2\rangle_\rho=2\int_{\R^d} \nabla \phi_1\cdot \nabla \phi_2 \rho dx.
\]
Hence, to recover \eqref{isonormal}, we need
$
\nabla \phi_\eta=\frac{1}{\sqrt{2\rho}}\xi,
$
where $\xi$ denotes vector valued space time white noise. Adding this additional stochastic perturbation to the velocity field, we obtain 
\[
v=\nabla \phi+\frac{1}{\sqrt{2\rho}}\xi.
\]
Recalling that the potential in the heat equation setting is given by $\phi=\frac{1}{2}\log(\rho)$ (e.g. as derived in \cite[Example 4.2.3]{Figalli2021}) and accounting for the scaling factor $N^{-1/2}$
yields
\[
\partial_t\rho=\Delta \rho +\sqrt{\frac{2}{N}}\nabla \cdot(\sqrt{\rho}\xi),
\]
which is again the Dean--Kawasaki equation.

\subsubsection{From MV-SDEs to generalized Dean--Kawasaki equations}
\label{McKean to DK}
In the previous section, we saw that the interpretation of the heat equation as a gradient flow on Wasserstein space is consistent with a particle system of i.i.d. Brownian motions in the sense that they yield the same mesoscopic SPDE. This point of view allows to identify a particle system with a geometry. Now we provide an arguably canonical possibility to retrieve a particle system from a McKean--Vlasov SDE. Using the arguments of Dean \cite{Dean1996}, this allows to derive an associated generalized Dean--Kawasaki equation.

\subsubsection*{Step (I). From MV-SDE to a particle system}
Consider a MV-SDE
\begin{equation}
    \label{McKean Vlasov}
    \begin{dcases}
        dX_t=b(t, X_t, \mu_t)dt+\sigma(t, X_t, \mu_t)dW_t \\
        \mathcal{L}(X_t) = \mu_t, \, t > 0, \quad X_0=x,
    \end{dcases}
\end{equation}
whose marginals coincide with a solution $\bar{\rho}$ to the corresponding nonlinear FPE. Moreover, assume that $\sigma$ can be written as a scalar function times $d \times d$ identity matrix. Let us construct a particle system with \eqref{McKean Vlasov} as its mean field limit. Towards this end, we fix a priori $\bar{\rho}$ in the measure-variables of $b$ and $\sigma$, which yields measure-independent coefficients
\begin{equation}
\label{frozen coeff}
    \bar{b}(t, x):=b(t, x, \bar{\rho}_t), \qquad \bar{\sigma}(t, x):=\sigma(t, x, \bar{\rho}_t)
\end{equation}
with associated (non-distribution dependent) SDE
\[
dX_t=\bar{b}(t, X_t)dt+\bar{\sigma}(t, X_t)dW_t.
\]
Let $(W^k)_k$ be i.i.d. Brownian motions and consider the particle system $(X^k)_k$ given by
\[
dX^k_t=\bar{b}(t, X^k_t)dt+\bar{\sigma}(t, X^k_t)dW^k_t.
\]
\begin{remark}
  Let us stress that this particle system is not the only one that recovers the given MV-SDE as a mean field limit. Considering it is therefore a choice we make, on which the further derivation of the associated generalized Dean--Kawasaki equation is contingent.
\end{remark}
\subsubsection*{Step (II). From particle system to generalized Dean--Kawasaki equation}
We aim to study the empirical measure 
\[
\rho^N_t=\frac{1}{N}\sum_{k=1}^N\delta_{X^k_t},
\]
associated to the above particle system $(X^k)_k$.  Using It\^o's formula on a test function $\phi\in C^\infty_c(\R^d)$, we obtain
\begin{equation}
    \begin{split}
\langle \rho^N_t, \phi\rangle=\langle \rho^N_0, \phi\rangle+\int_0^t \langle \rho^N_s, \bar{b}^T\nabla \phi\rangle ds+\frac{1}{2}\int_0^t\langle \rho^N_s, \mbox{Tr}(\bar{\sigma}^T \mbox{Hess}(\phi)\bar{\sigma})\rangle ds+M^N_t,
         \label{ito particle}
    \end{split}
\end{equation}
where
\begin{equation}
   M^N_t:=\frac{1}{N} \sum_{k=1}^N\int_0^t (\nabla \phi)^T(X^k_s) \bar{\sigma}(s, X^k_s)dW^k_s.
\end{equation}
Proceeding as in \cite{Dean1996}, we consider the following martingale with the same quadratic variation
\[
\tilde{M}^N_t=\frac{1}{\sqrt{N}}\int_0^t\langle \sqrt{\rho^N} \xi_s, \bar{\sigma}^T\nabla \phi\rangle ds.
\]
Again, formally replacing $M^N$ by $\tilde{M}^N$ and identifying the matrix $\bar{\sigma}$ with its corresponding scalar, we obtain the weak formulation of the generalized Dean--Kawasaki equation 
\begin{equation}
    \partial_t \rho=\frac{1}{2}\Delta (\bar\sigma^2\rho)-\nabla \cdot(\bar{b}\rho)dt+\frac{1}{\sqrt{N}}\nabla \cdot(\bar{\sigma}\sqrt{\rho} \xi).
    \label{generalized DK}
\end{equation}
Let us point out that the particular choice $(\sigma=\sqrt{2} \cdot  1_d, \, b = 0)$ yields the classical Dean--Kawasaki equation \eqref{Dean Kawasaki}. In particular, both the geometry and the MV-SDEs suggest the same mesoscopic fluctuations on top of the mean behavior as expressed by the noise term in \eqref{Dean Kawasaki}. 

\subsection{Application to porous medium equation}\label{sec:DK-2}
Here, we aim to apply the strategy of the previous section to the porous medium equation.
More precisely, we study different geometries associated with this equation and identify which MV-SDE from \Cref{sec:porous_medium} is consistent in the sense of (asymptotically) yielding the same noise in their respective generalized Dean--Kawasaki equations.

Let us stress again that all derivations in this section are entirely formal. We provide them to serve as a potential perspective on the nonlinear Markov processes constructed rigorously in Section \ref{sec:porous_medium}.

\subsubsection{Wasserstein geometry and additive noise MV-SDE}
Here, we consider the Wasserstein geometry. It is well-known that the porous medium equation can be interpreted as a gradient flow on the space of probability measures equipped with respect to this geometry. More precisely, consider the functional
\[
F(\rho):=\frac{1}{2}\frac{1}{(m-1)}\int \rho^mdx.
\]
As can be shown formally (see e.g. \cite[Example 4.2.3]{Figalli2021}), we have
\[
\mbox{grad}_{\mathcal{W}_2}F(\rho)=-\Delta \rho^m.
\]
Hence,
\[
\partial_t \rho=-\mbox{grad}_{\mathcal{W}_2}F(\rho).
\]
Similarly as before, note that the canonical noise consists in perturbing the velocity field by ``Brownian motion on the tangent plane.'' Since we remain in the same geometric setting, this means that the velocity field is again given by
\[
v=\nabla \phi+\frac{1}{\sqrt{2\rho}}\xi.
\]
This produces (upon adding the scaling factor $N^{-1/2}$) the generalized Dean--Kawasaki equation 
\begin{equation}
    \label{DK PME}
    \partial_t \rho =\Delta \rho^m+\sqrt{\frac{2}{N}}\nabla \cdot(\sqrt{\rho} \xi).
\end{equation}
Consider now the MV-SDE 
\begin{equation}
\begin{dcases}
    dX_t=\frac{\nabla (u(t, X_t)-u^m(t, X_t))}{u(t, X_t)} dt + \sqrt{2}dW_t \\
    \mathcal{L}(X_t)=u(t, x)dx, \,\, t>0,
\end{dcases}
\end{equation}
studied in \Cref{subsec:PME-Addit}.  Following \Cref{McKean to DK}, its associated generalized Dean--Kawasaki equation is
\begin{equation}
    \partial_t \rho= \Delta \rho-\nabla \cdot\left(\frac{\nabla(\bar\rho-\bar\rho^m)}{\bar\rho}\rho\right)+\sqrt{\frac{2}{N}}\nabla \cdot(\sqrt{\rho} \xi).
    \label{DK add noise}
\end{equation}
Note that the noise terms in \eqref{DK PME} and \eqref{DK add noise} coincide (essentially, since we have not changed the geometry from the heat equation setting, the diffusion term ($\sigma=\sqrt{2} \cdot 1_d$) needs to remain unchanged).
Concerning the drift term, observe that formally $\rho\to \bar \rho$  as $N\to\infty$. Taking this formal limit in \eqref{DK add noise}, we recover
\[
\partial_t \bar\rho=\Delta \bar\rho-\nabla \cdot\left(\frac{\nabla(\bar\rho-\bar\rho^m)}{\bar\rho}\bar\rho\right)=\Delta\bar\rho^m,
\]
i.e., the porous medium equation (again, simply a demonstration of the fact that the McKean--Vlasov solution has porous medium equation marginals). Thus, the additive noise interpretation of the porous medium equation is consistent with the Wasserstein geometry in the sense that their associated generalized Dean--Kawasaki equations have the same noise term, meaning that not only their mean behavior, but also their mesoscopic fluctuations coincide. 

\subsubsection{Thermodynamic geometry and interpolated MV-SDEs}
Instead of the Wasserstein geometry, here we consider the \emph{thermodynamic geometry} as studied in \cite{Dirr2016}. We again consider the manifold
\[
\mathcal{M} \coloneqq \mathcal{P}_2 (\R^d),
\]
but now endow it with a formal tangent plane different from \eqref{eq:TrhoM_1}:
\[
T_\rho\mathcal{M} \coloneqq \Bigl\{s: \R^d \to \R \, \big| \,  s+2\beta\nabla \cdot(\rho^mv)=0 \ \text{for}\ \rho \in \mathcal{M} \ \text{and some velocity field}\ v\Bigr\}.
\]
Suppose for any $\rho\in \mathcal{M}$ and $s\in T_\rho\mathcal{M}$, there is a unique (modulo constants) potential $\phi$ such that \begin{equation}
    \label{tangent to pot 2}
    s=-2\beta\nabla \cdot(\rho^m\nabla \phi).
\end{equation}
Again, \eqref{tangent to pot 2} can be used to define a scalar product
\begin{equation}\label{eq:thrm_scalar_product}
\langle s_1, s_2\rangle_\rho:=2\beta\int \nabla \phi_1 \cdot \nabla \phi_2 \rho^mdx,
\end{equation}
which implies a formal Riemannian manifold structure.
\begin{remark}
    Let us stress that this Riemannian structure is indeed purely formal. As shown in \cite{DNS2009}, the scalar product \eqref{eq:thrm_scalar_product} induces a metric only if $m\leq 1$. More generally, replacing the weight $\rho^m$ by a \emph{mobility function} $h(\rho)$, they show that it is required that $h$ be \emph{concave}. However, using this formal structure, one can still proceed with the derivation of an associated generalized Dean--Kawaski equation \cite{Dirr2016}. 
\end{remark}
Given the above formal scalar product, let us consider the functional
\[
E(\rho) \coloneq \frac{m}{2\beta}\int \rho \log(\rho)dx
\]
from before. To calculate the gradient of $E$ with respect to the scalar product defined above, we exploit the identity
\[
\langle \mbox{grad}_{\textup{thm}}\, E(\rho), s\rangle_\rho=D_sE(\rho),
\]
where $D_sE(\rho)$ denotes the directional derivative of $E$ in direction $s$ at point $ \rho$.
Using relation \eqref{tangent to pot 2}, we obtain
\[
D_sE(\rho)=\frac{m}{2\beta}\int s(1+\log(\rho))dx=m\int \rho^m \nabla \phi \cdot \nabla \log(\rho)dx,
\]
where $\phi$ is the potential associated with $s$ via \eqref{tangent to pot 2}. Denoting by $\psi$ the potential associated with $\mbox{grad}E(\rho)$, we again conclude $\psi=\frac{m}{2\beta}\log(\rho)$. Hence, we recover the gradient as 
\[
\mbox{grad}_{\textup{thm}}\,E(\rho)=-\nabla \cdot(m\rho^m \nabla \log(\rho))=-\nabla \cdot(m\rho^{m-1}\nabla \rho)=-\Delta \rho^m.
\]
Thus, we have shown that on this formal Riemannian manifold, the porous medium equation can be written as the gradient flow 
\[
\partial_t\rho =-\mbox{grad}_{\textup{thm}}\,E(\rho).
\]
Let us point out that this setting is consistent with the setting of the heat equation in the sense that for $\beta=m=1$, the geometry coincides with the one for the heat equation.

Similar to the Wasserstein geometry setting, we again think of a natural stochastic perturbation given by a ``Brownian motion on the tangent space,'' which is again an isonormal Gaussian process $W=\{W(s): \ s\in T_\rho \mathcal{M}\}$, thus satisfying
\[
\mathbb E[W(s_1)W(s_2)]=\langle s_1,  s_2\rangle_\rho.
\]
Similarly to before, we make the ansatz $W(s_i)=\langle \eta, s_i\rangle _\rho$ and, denoting by $\phi_\eta$ the associated potential, we obtain 
\[
\mathbb E[W(s_1)W(s_2)]=4\beta^2 \mathbb E\left[\int_{\R^d}\nabla \phi_1\cdot \nabla\phi_\eta \rho^mdx  \int_{\R^d}\nabla \phi_2\cdot \nabla \phi_\eta \rho^mdy \right],\]
as well as 
\[
\langle s_1, s_2\rangle_\rho=2\beta \int_{\R^d}\nabla \phi_1\cdot \nabla \phi_2 \rho^mdx.
\]
Again, to recover the identity for Gaussian isonormal processes, we need
\[
\nabla \phi_\eta=\frac{1}{\sqrt{2\beta \rho^m}}\xi,
\]
where again $\xi$ denotes vector-valued space-time white noise. Adding this stochastic perturbation to the continuity equation with nonlinear mobility, we obtain --- upon adding the scaling $N^{-1/2}$ --- the associated generalized Dean--Kawasaki equation 
\begin{equation}
    \partial_t \rho= \Delta \rho^m+\sqrt{\frac{2\beta}{N}}\nabla \cdot(\rho^{m/2}\xi).
    \label{genDK 2}
\end{equation}
Now consider the McKean--Vlasov SDE
\begin{equation}
    \begin{dcases}
    dX_t=-(1-\beta) \frac{\nabla u(t,X_t)^m}{u(t,X_t)}dt + \sqrt{2\beta} \, u(t,X_t)^{\frac{m-1}{2}}dW_t, \\ 
    \mathcal{L}(X_t) = u(t,x)dx, \,\, t >0,
    \end{dcases}
\end{equation}
studied in \Cref{subsec:PME-beta}. Again, following \Cref{McKean to DK}, we obtain the associated generalized Dean--Kawasaki equation 
\begin{equation}
     \partial_t \rho=\beta\Delta (\bar\rho^{m-1}\rho)-(1-\beta)\nabla \cdot(\frac{\nabla \bar\rho^m}{\bar\rho}\rho)dt+\sqrt{\frac{2\beta}{N}}\nabla \cdot(\bar\rho^{(m-1)/2}\sqrt{\rho} \xi).
     \label{DK beta}
\end{equation}
Using \eqref{genDK 2}, we make the formal ansatz
\[
\rho=\bar{\rho}+R_N,
\]
where $R_N=O(N^{-1/2})$. Plugging this ansatz into the noise term of \eqref{DK beta}, we obtain formally
\[
\sqrt{\frac{2\beta}{N}}\nabla \cdot(\bar\rho^{(m-1)/2} \sqrt{\rho}\xi)=\sqrt{\frac{2\beta}{N}}\nabla \cdot\left((\rho-R_N)^{(m-1)/2}\sqrt{\rho}\xi\right)=\sqrt{\frac{2\beta}{N}}\nabla \cdot(\rho^{m/2}\xi)+o(N^{-1/2}).
\]
This shows that, formally and asymptotically, the noise terms of \eqref{DK beta} and \eqref{genDK 2} coincide, meaning that the mesoscopic fluctuations induced by the geometry and the McKean--Vlasov SDE coincide. In this sense, we can identify the McKean--Vlasov SDEs studied in \Cref{subsec:PME-beta} with thermodynamic metrics as studied in \cite{Dirr2016}. 

\appendix
\refstepcounter{section}
\section*{Appendix}
\setcounter{subsection}{0}
\subsection{On pure Stratonovich-diffusion interpretation}\label{subsec:theta_interpretations}
In Sections \ref{sec:porous_medium} and \ref{sec:pLaplace}, we discussed a pure It\^{o} diffusion interpretation of the porous medium and the $p$-Laplace equation by deriving an interpretation (i.e. a diffusion coefficient $a$) for these equations as pure-diffusion FPEs such that the corresponding MV-SDE is of type
\begin{equation}
    \begin{dcases}
        dX_t=\sqrt{2a(X_t, u(t,\cdot))} \, dW_t, \\
        \mathcal{L}(X_t)=u(t,x)dx,\quad t>0.
    \end{dcases}
\end{equation}
An analogous pure Stratonovich setting can be formally identified by means of the following standard lemma (for a short formal proof, we refer to the more general Lemma \ref{correction lemma general}).
\begin{lemma}\label{ito strat corr}
    Let $a:(0,\infty)\times \R^d \to [0,\infty)$ satisfy $\sqrt{a}\in C^2_b((0,\infty)\times \R^d)$, and suppose that $X$ is a weak solution to
    \[
    dX_t=\frac{1}{2}\nabla a(t, X_t)dt+\sqrt{2a(t, X_t)}dW_t.
    \]
    Then $X$ is also a weak solution to
    \[
    dX_t= \sqrt{2a(t, X_t)}\circ dW_t.
    \]
\end{lemma}
Thus, formally, the pure Stratonovich diffusion case is obtained via the coefficient pair $(a,b)$,
\begin{equation}
    b=\frac{1}{2}\nabla a.
    \label{strat condition}
\end{equation}
\begin{remark}
    In \Cref{sec:porous_medium,sec:pLaplace}, we showed that such a choice of coefficients is possible for the porous medium as well as the $p$-Laplace equation, and write the corresponding MV-SDE.
    \\
    However, the regularity for $a$ required in Lemma \ref{ito strat corr} fails in these cases. For example, for the porous medium equation and its Barenblatt solutions $u^z$, the choice of $(a, b)$ with \eqref{strat condition} yields 
    \[
    \sqrt{a(x,u^z(t,\cdot))}=\sqrt{\frac{2m}{m+1}(u^z(t, x))^{m-1}}=\sqrt\frac{2}{m+1} t^{-k/2}\left( C-qt^{-2k/d}|x-z|^2\right)^{1/2}_+
    \]
    (see \Cref{subsect:PME-Stratono}).
    Even if we restrict ourselves to times bounded away from zero, this function is only $1/2$-H\"{o}lder continuous in space. Lemma \ref{ito strat corr} is therefore not applicable. In fact, for the same reason, it is even unclear whether the Stratonovich integral
    \begin{equation}
        \label{strat int}
        \int_0^t \sqrt{a(X_s, u^z(s,\cdot)}\circ dW_s
    \end{equation}
    is well-defined (due to the lack of regularity of $\sqrt{a}$, it is unclear to us if $t\mapsto \sqrt{a(X_t, u^z(t,\cdot))}$ is a semimartingale, since It\^{o}'s formula is not available).
    \\
    One approach to ensure that \eqref{strat int} is well-defined in our setting might be using Stochastic Sewing \cite{Le2020}. Indeed, using the elementary definition of the Stratonovich integral as the $L^2(\Omega)$ limit of Riemann sums, the Stochastic Sewing lemma gives concise conditions under which the existence of such limits is assured. Moreover, stochastic sewing arguments are typically sufficiently robust to deal with irregular coefficients. While this approach seems feasible, provided we integrate over times away from zero, let us point out that for small times, even more challenges appear: Since $\lim_{t\to 0}u^z(t, \cdot)= \delta_z$ weakly, we lose even any H\"{o}lder regularity for $t\to 0$.
    \\
    If one can prove that \eqref{strat int} is well defined, one may attempt to obtain the statement of Lemma \ref{ito strat corr} by means of a mollification argument, i.e., by mollifying $\sqrt{a}$, using Lemma \ref{ito strat corr} on the mollification level and passing to the respective limits.
    These tasks appear delicate and are left for future investigation. 
\end{remark}

\subsection{On further pure-diffusion interpretations}\label{subsec:theta_Integration}
Let us point out that the process constructed in \cite{BarbuRehmeierRockner2024}, called $p$-Brownian, is neither a pure It\^{o} diffusion, nor does it satisfy relation \eqref{strat condition} (i.e., it is not pure-Stratonovich-diffusion). However, we show that formally it does correspond to a certain pure diffusion setting. Towards this end, we recall the following notion of $\theta$-stochastic integration. 
    
    \begin{definition}[$\theta$-Integration]
       Let $\theta\in [0, 1]$, let $W$ be a Brownian motion and let $X$ be a semimartingale adapted to the filtration generated by $W$.  Define 
        \[
        \int_0^TX_s\circ^\theta dW_s:=\lim_{|\mathcal{P}|\to 0}\sum_{[u, v\in \mathcal{P}]}(\theta X_u+(1-\theta)X_v)(W_v-W_u),
        \]
        where the limit is taken in $L^2(\Omega)$ along any sequence of partitions of $[0,T]$ with mesh size converging to $0$ (see for instance \cite[Chapter 3.5]{Kloeden1992}).
    \end{definition}
    Note that in particular, $\theta=1$ and $\theta = \frac 1 2 $ correspond to It\^{o} and Stratonovich integration, respectively.
    
    \begin{lemma}
    \label{correction lemma general}
        Let $a:(0,\infty)\times \R^d \to [0,\infty)$ satisfy $\sqrt{a}\in C^2_b((0,\infty)\times \R^d)$ and $X$ be a weak solution to 
        \[
        dX_t=\sqrt{2a(t, X_t)}\circ^\theta  dW_t.
        \]
        Then $X$ is also a weak solution to         \[
         dX_t=(1-\theta)\nabla a(X_t)dt+\sqrt{2a(X_t)} dW_t
        \]
    \end{lemma}
    \begin{proof}
    We only provide a formal proof. Performing a Taylor expansion around $X_u$, we formally have 
    \begin{equation}
        \begin{split}
            &\int_0^T\sqrt{2a(X_s)}\circ^\theta dW_s\\&=\lim_{|\mathcal{P}|\to 0}\sum_{[u, v\in \mathcal{P}]}\sqrt{2a(\theta X_u+(1-\theta)X_v)}(W_v-W_u)\\
            &=\lim_{|\mathcal{P}|\to 0}\sum_{[u, v\in \mathcal{P}]}\left[\sqrt{2a(X_u)}+(1-\theta)\nabla \sqrt{2a}(X_u)\cdot (X_v-X_u)\right](W_v-W_u)\\
            &+\lim_{|\mathcal{P}|\to 0}\sum_{[u, v\in \mathcal{P}]}(1-\theta)^2(X_v-X_u)^T\mbox
            {Hess}(\sqrt{2a})(X_u)(X_v-X_u) (W_v-W_u)\\
            &=\int_0^T \sigma(X_s)dW_s+(1-\theta)\int_0^T \nabla \sqrt{2a(X_s)}d[X, W]_s,
        \end{split}
    \end{equation}
    where we assumed the second-order term to vanish in the limit. Thus, $X$ solves
    \[
    dX_t=(1-\theta) \nabla\sqrt{2a}(X_t) d[X, W]_t +\sqrt{2a(X_t)}dW_t,
    \]
    from which we infer
    \[
    [X, W]_t=\sqrt{2a(X_t)}
    \]
    and conclude the claim.
\end{proof}
To conclude, recall that $p$-Brownian motion from \cite{BarbuRehmeierRockner2024} solves a time-homogeneous MV-SDE whose coefficients satisfy the gradient relation $b=\nabla a $. 
From \Cref{correction lemma general}, we see that this relation is satisfied for $\theta=0$. Therefore, one can formally interpret $p$-Brownian motion as a pure diffusion interpretation with a $\theta$-stochastic integral with $\theta=0$. This choice corresponds to a purely anticipating stochastic integral.

\subsection{Numerical simulation: supplementary plots and details}\label{Apndx:num_sim}

\begin{figure}
    \centering
    \setlength{\tabcolsep}{4pt}
    \renewcommand{\arraystretch}{1.15}
    \begin{tabular}{>{\centering\arraybackslash}m{0.067\textwidth}
                    >{\centering\arraybackslash}m{0.288\textwidth}
                    >{\centering\arraybackslash}m{0.288\textwidth}
                    >{\centering\arraybackslash}m{0.288\textwidth}}         
        \rule{0pt}{5ex} 
        
        & {\small\textbf{$p$-Laplace equation}}
        & {\small\textbf{Porous medium equation}}
        & {\small\textbf{Heat equation}} \\
        {\footnotesize$\beta = 0.0$}
        & \includegraphics[width=\linewidth]{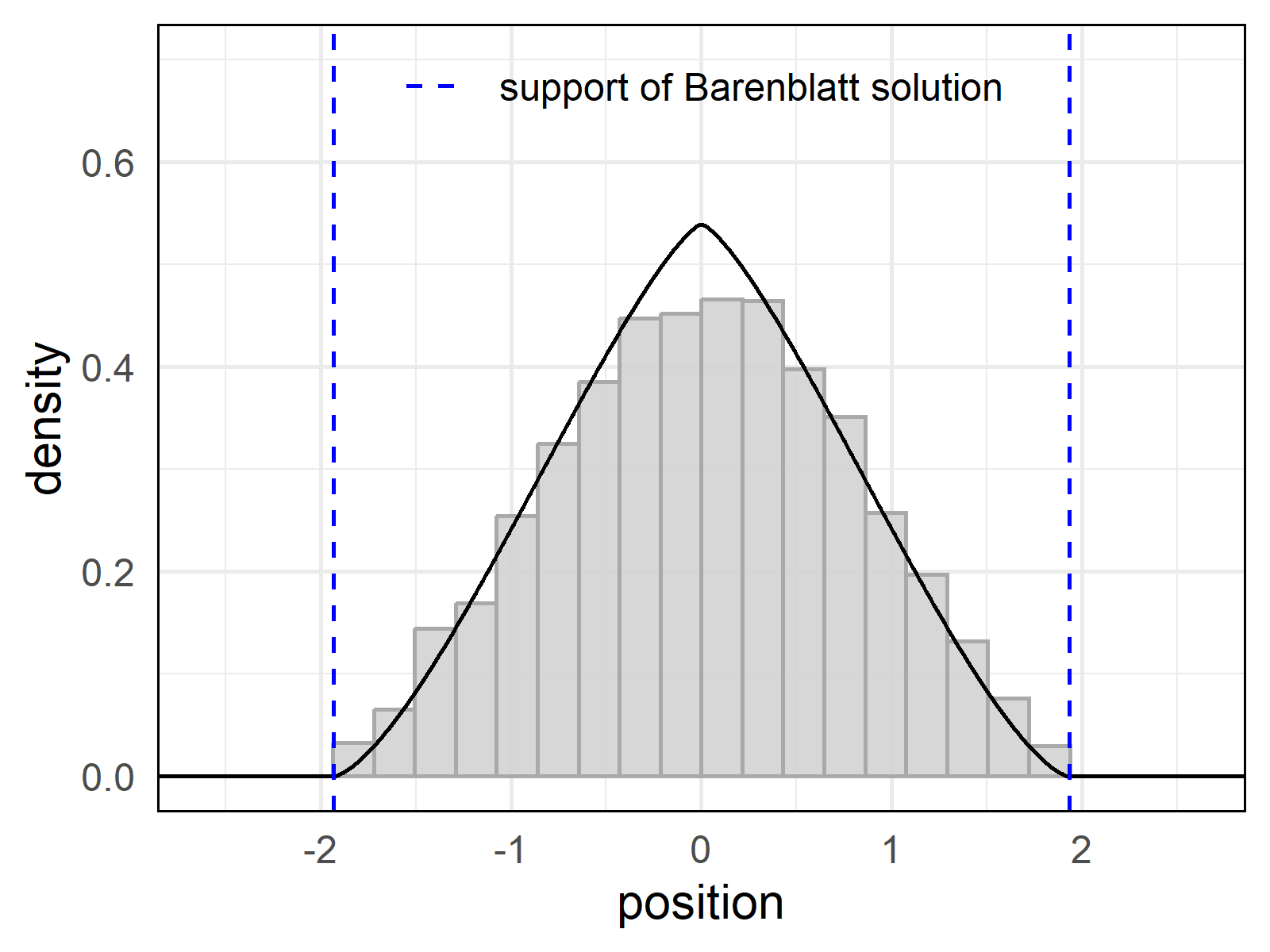}
        & \includegraphics[width=\linewidth]{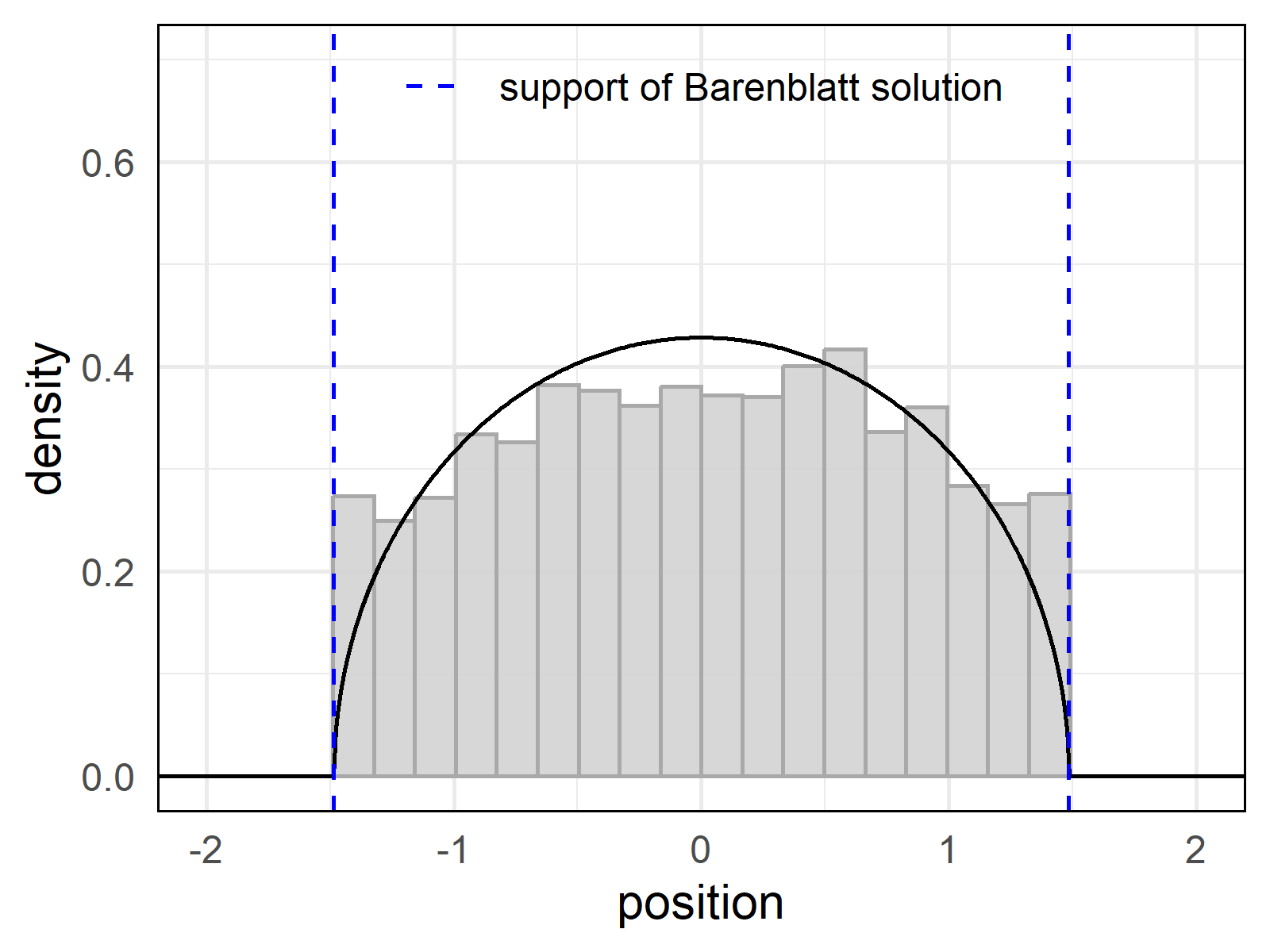}
        & \includegraphics[width=\linewidth]{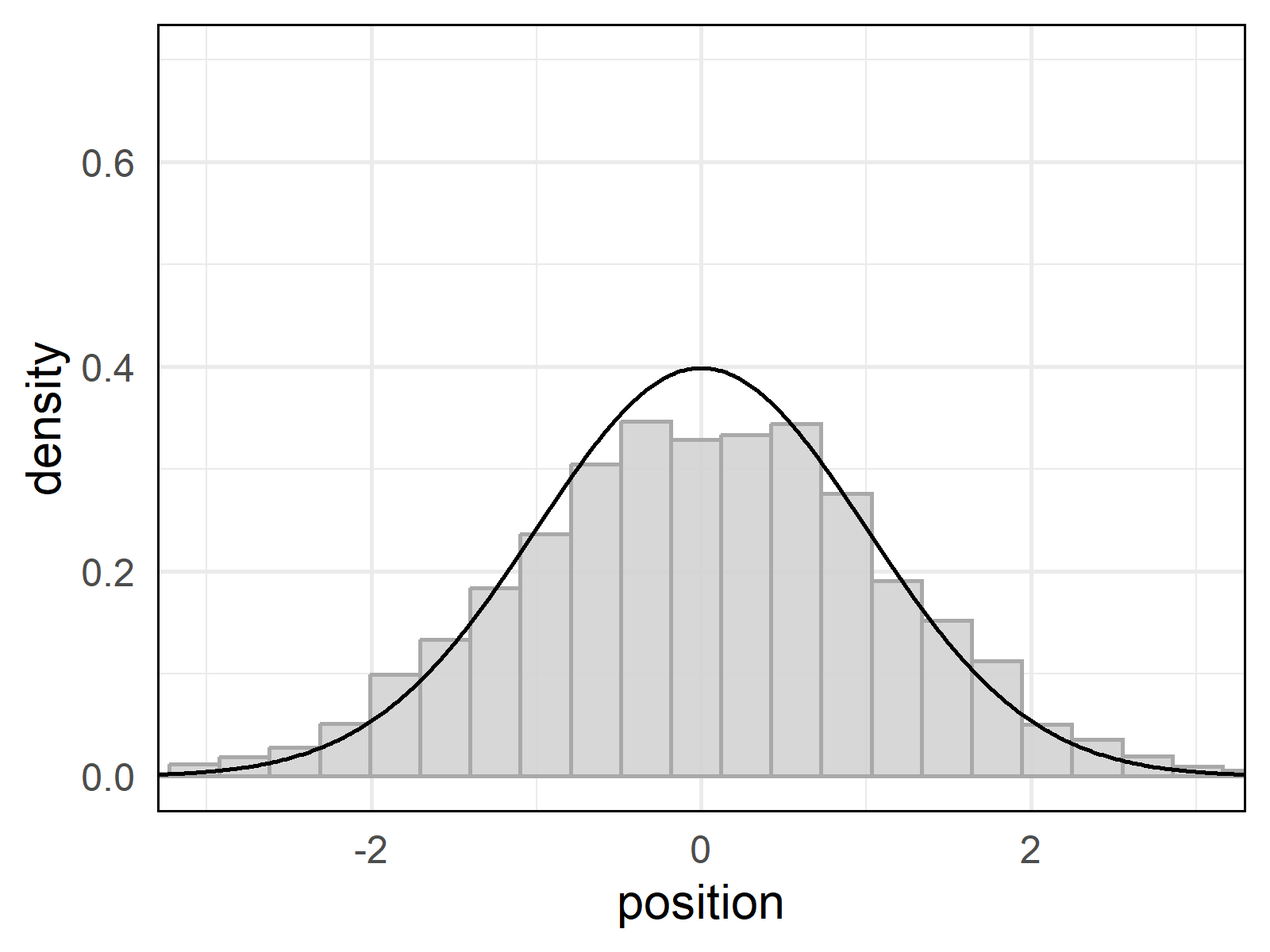} \\

        {\footnotesize$\beta = 0.1$}
        & \includegraphics[width=\linewidth]{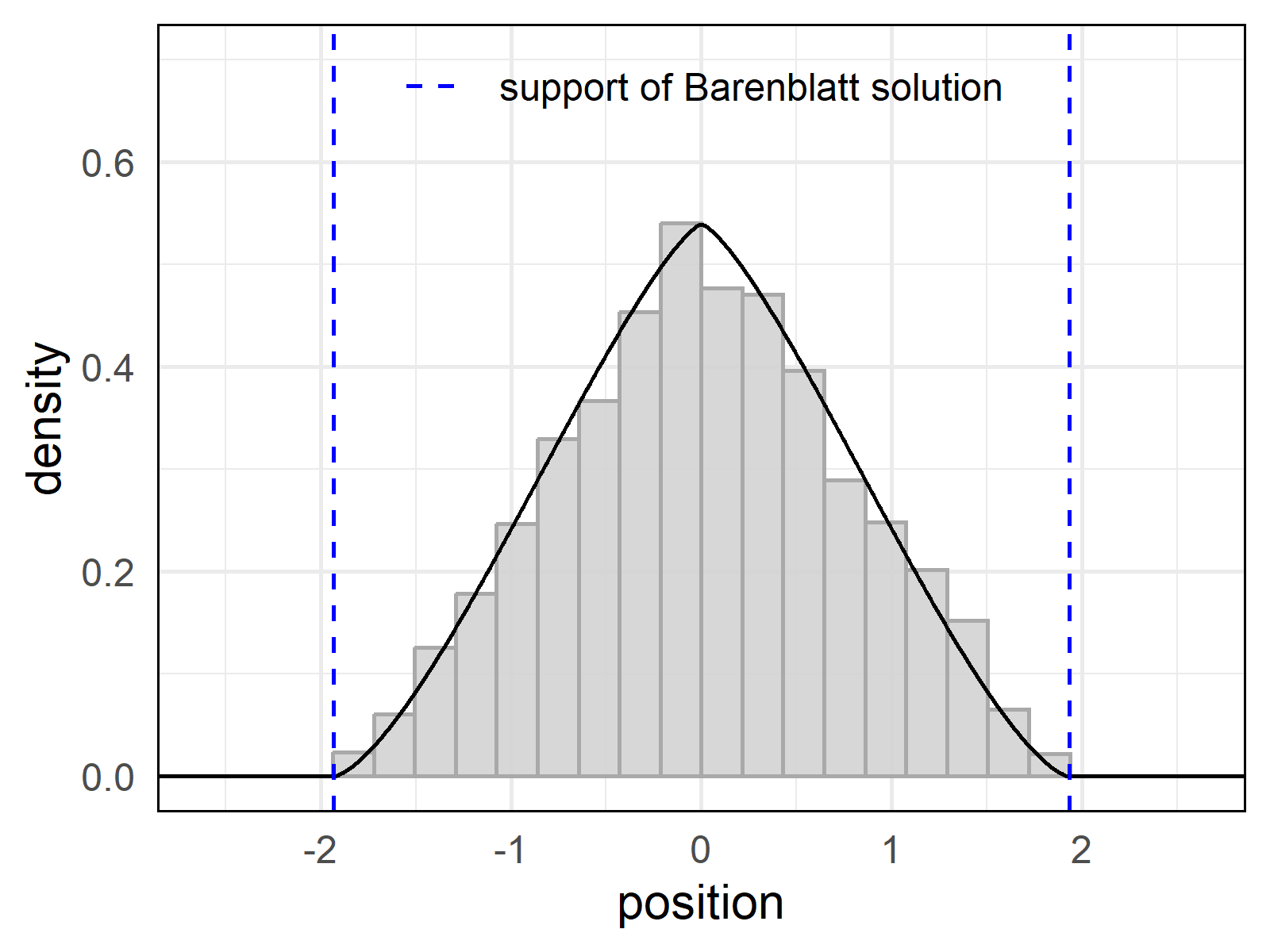}
        & \includegraphics[width=\linewidth]{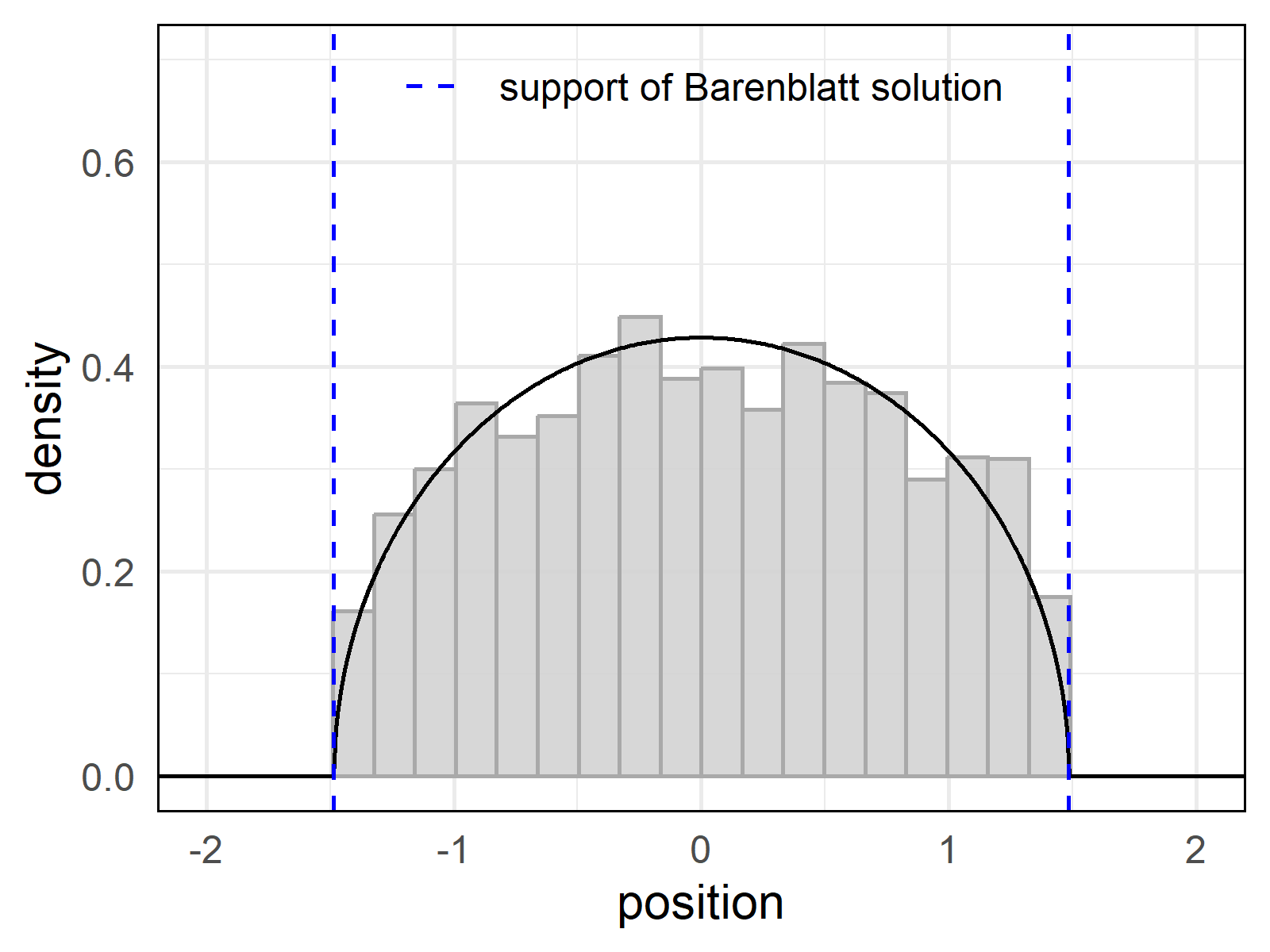}
        & \includegraphics[width=\linewidth]{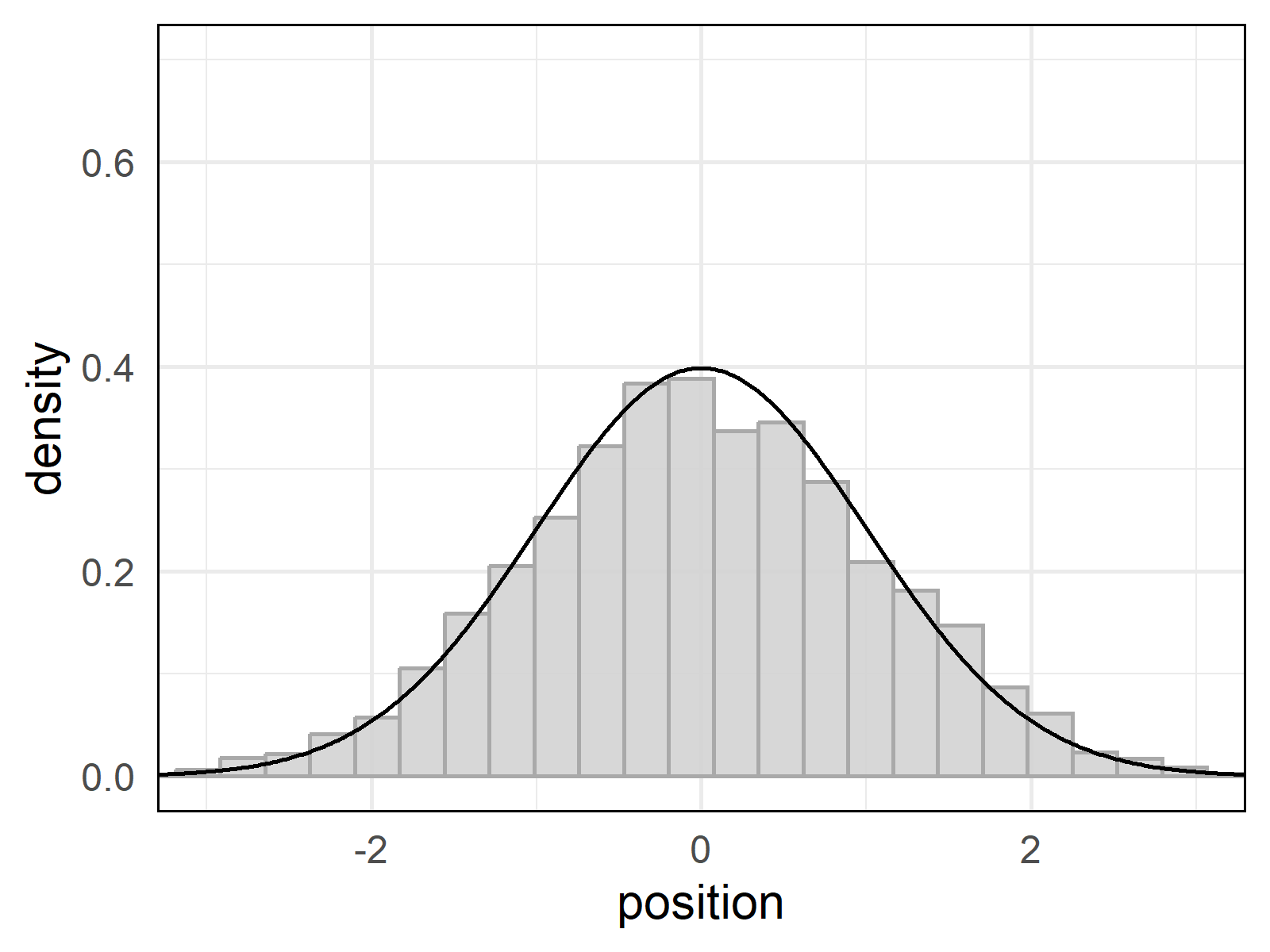} \\

        {\footnotesize$\beta = 1.0$}
        & \includegraphics[width=\linewidth]{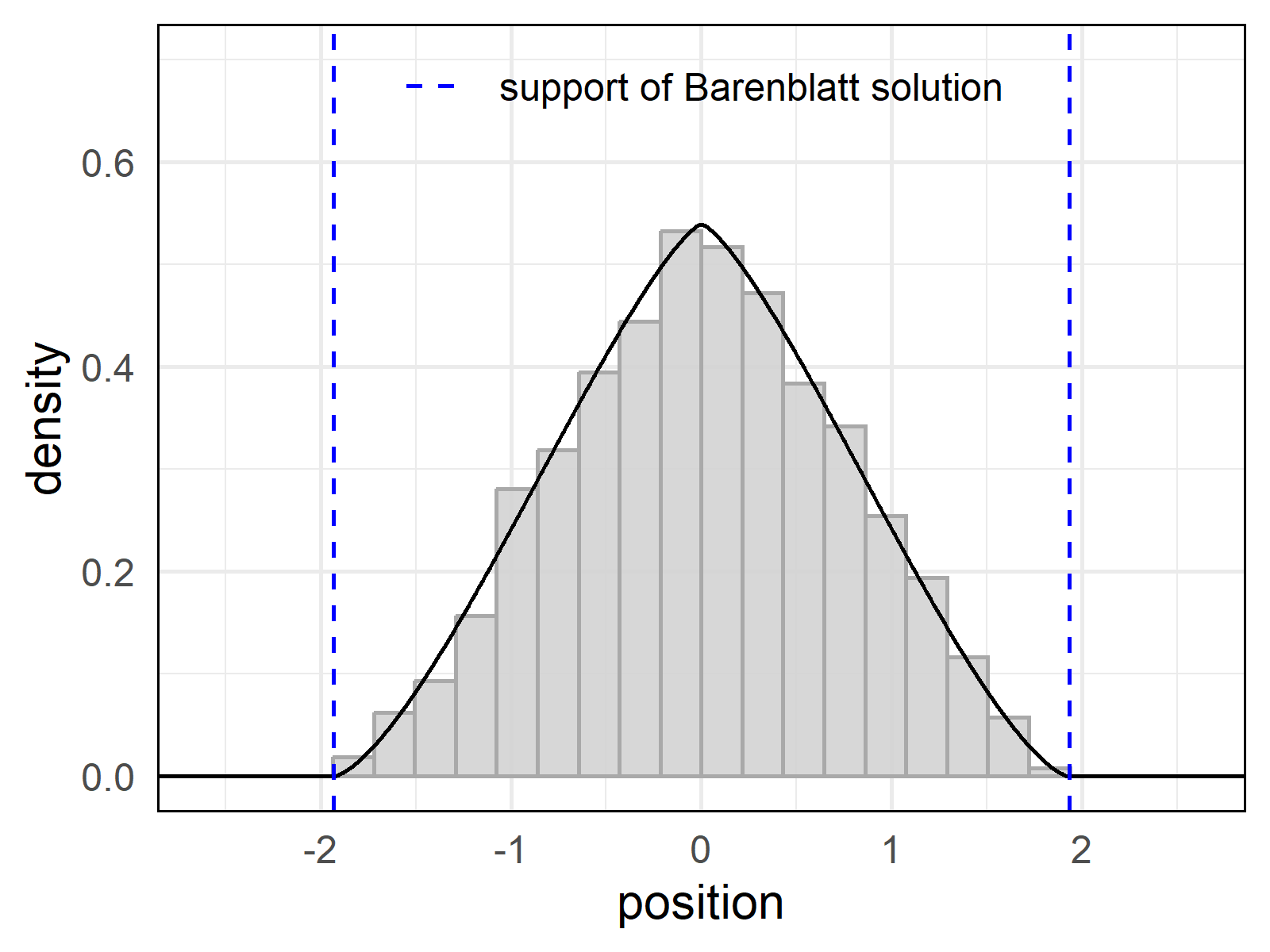}
        & \includegraphics[width=\linewidth]{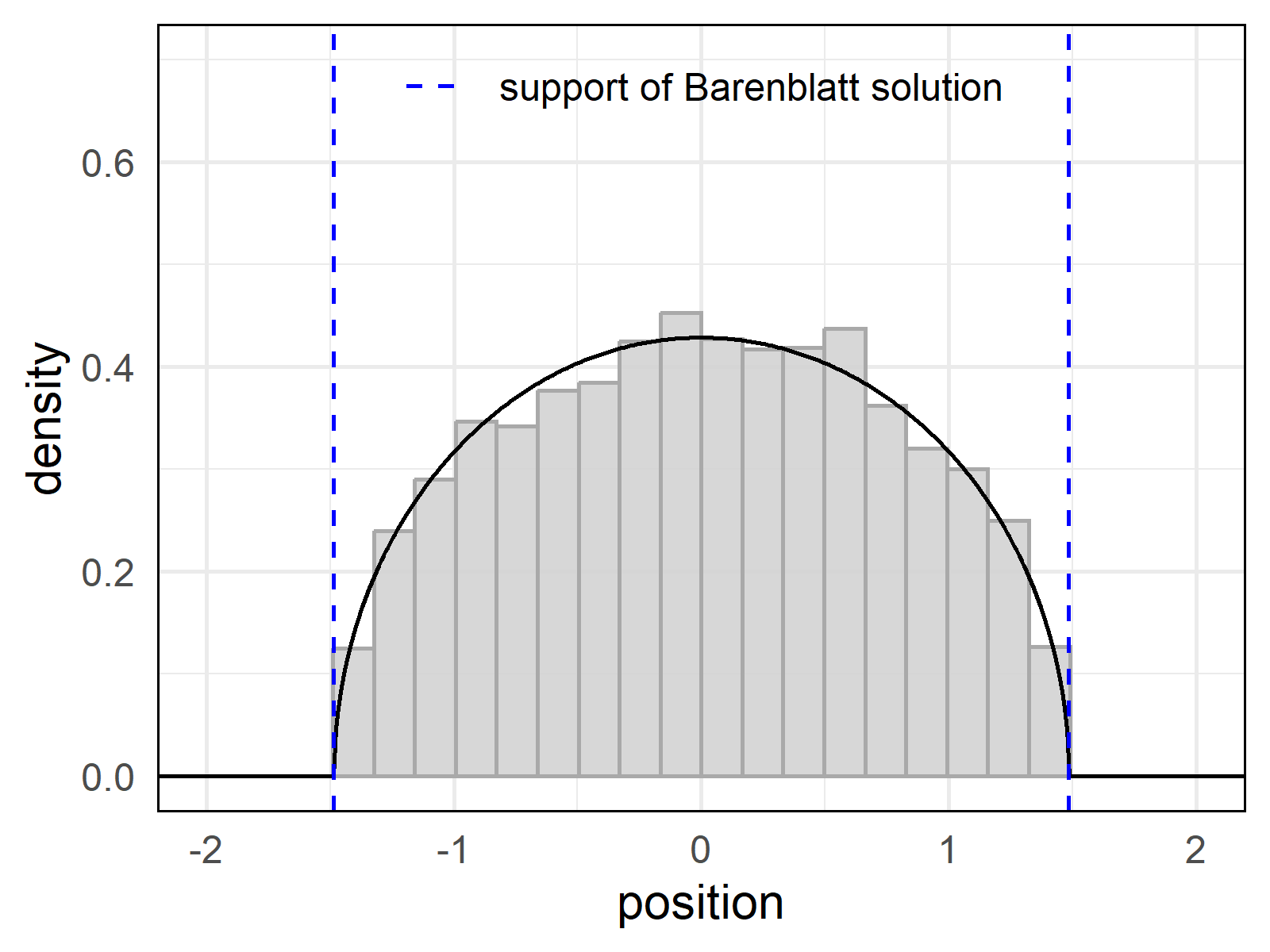}
        & \includegraphics[width=\linewidth]{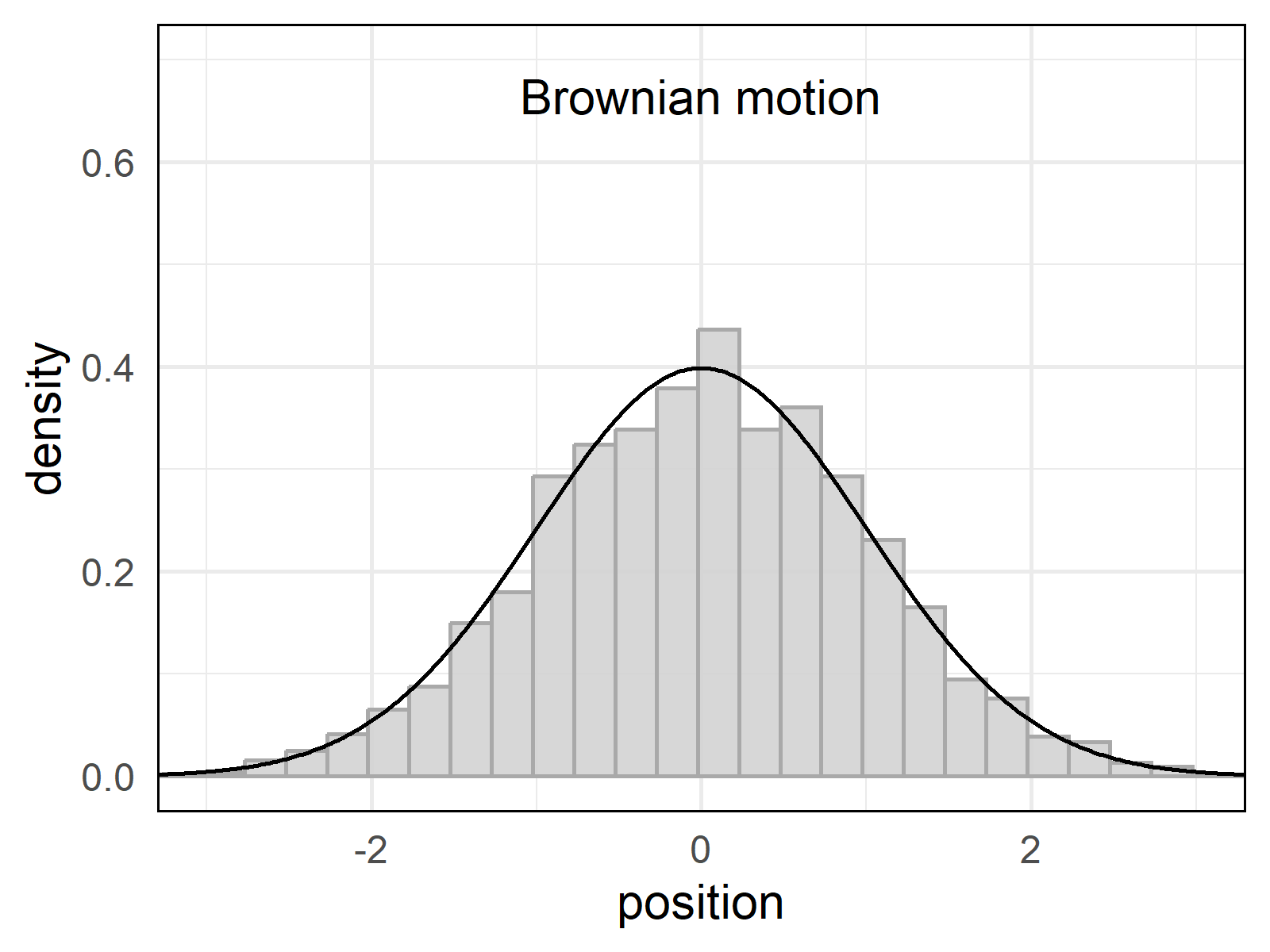} \\

        {\footnotesize$\beta = 1.5$}
        & \includegraphics[width=\linewidth]{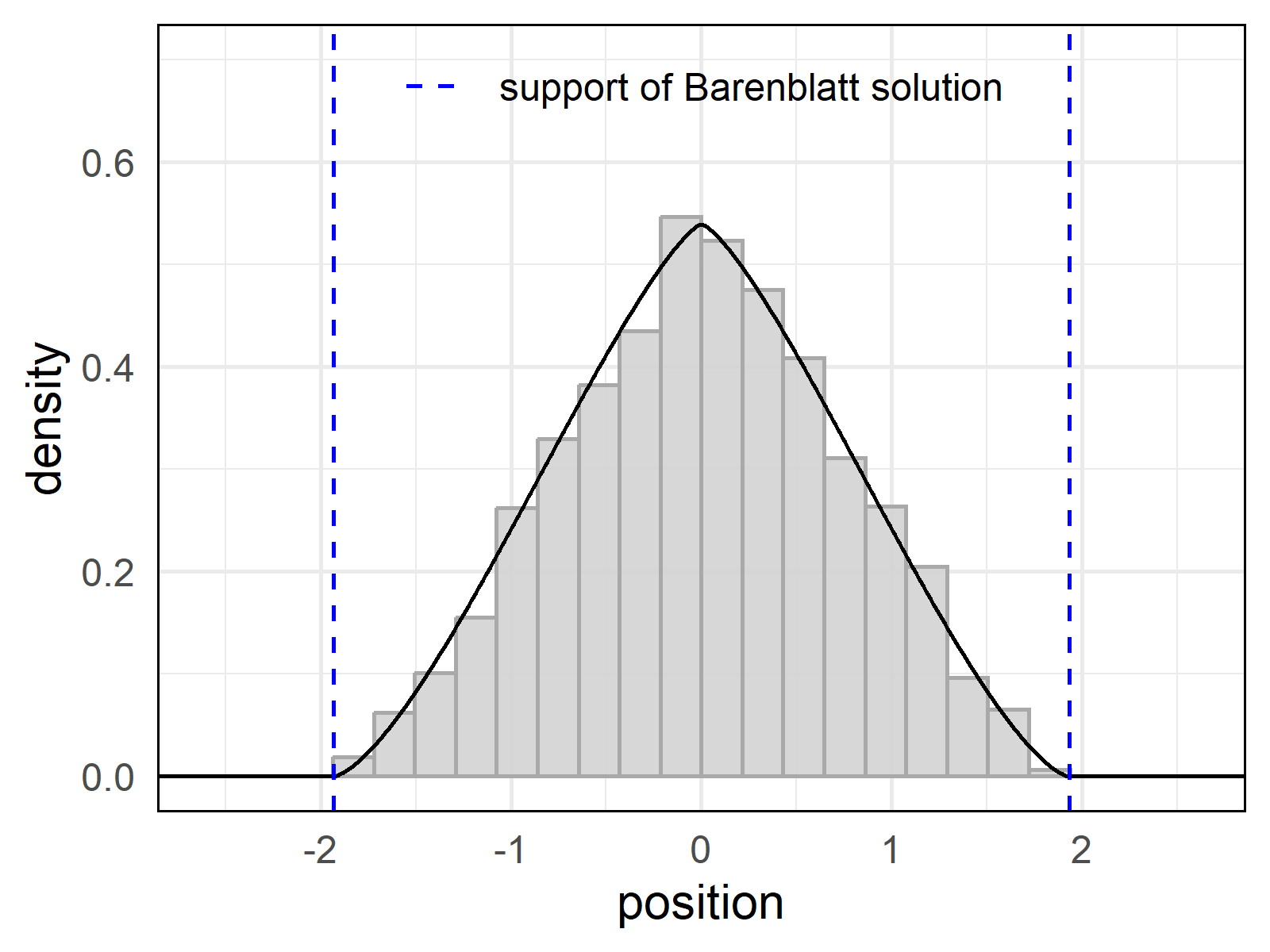}
        & \includegraphics[width=\linewidth]{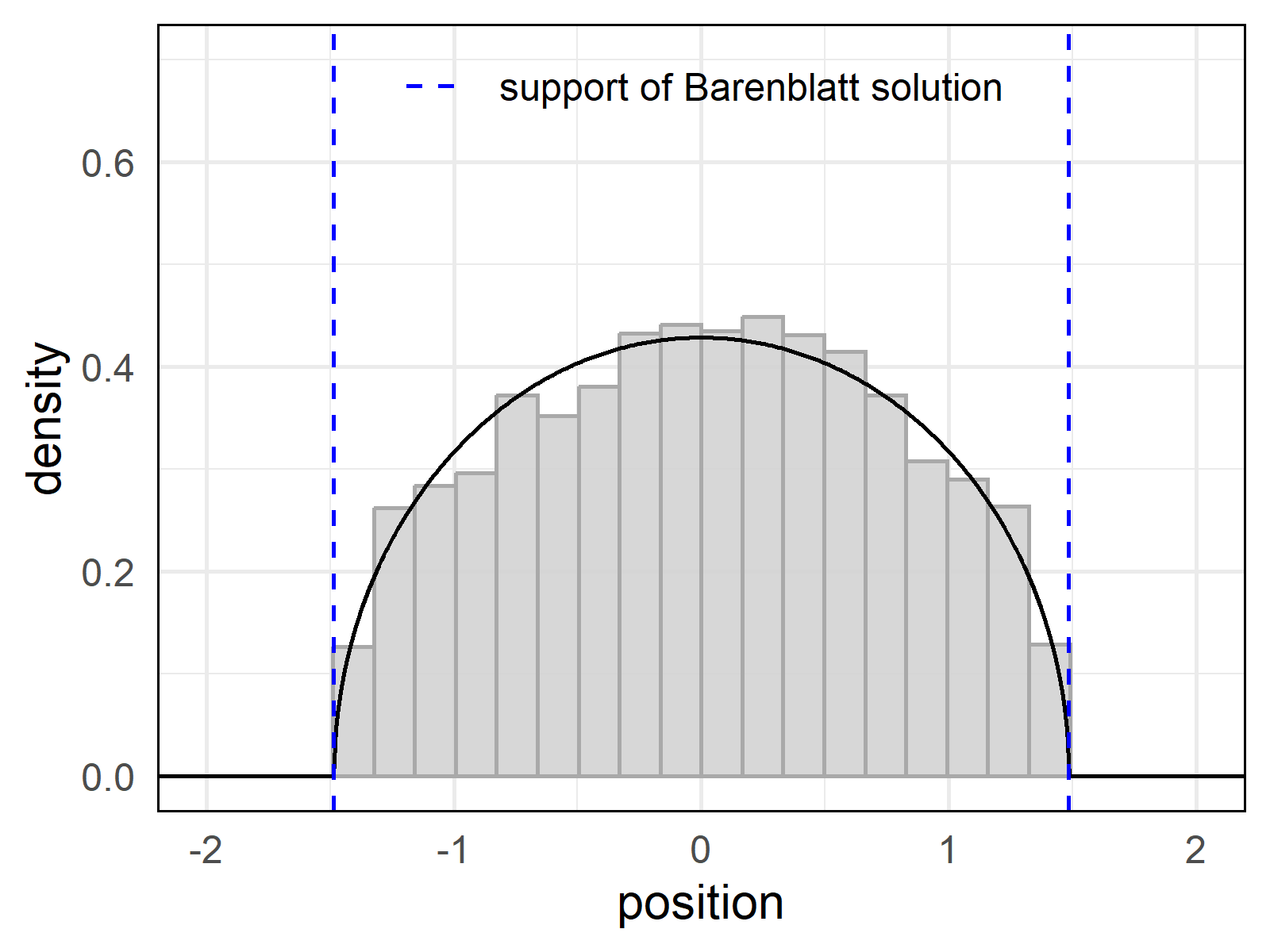}
        & \includegraphics[width=\linewidth]{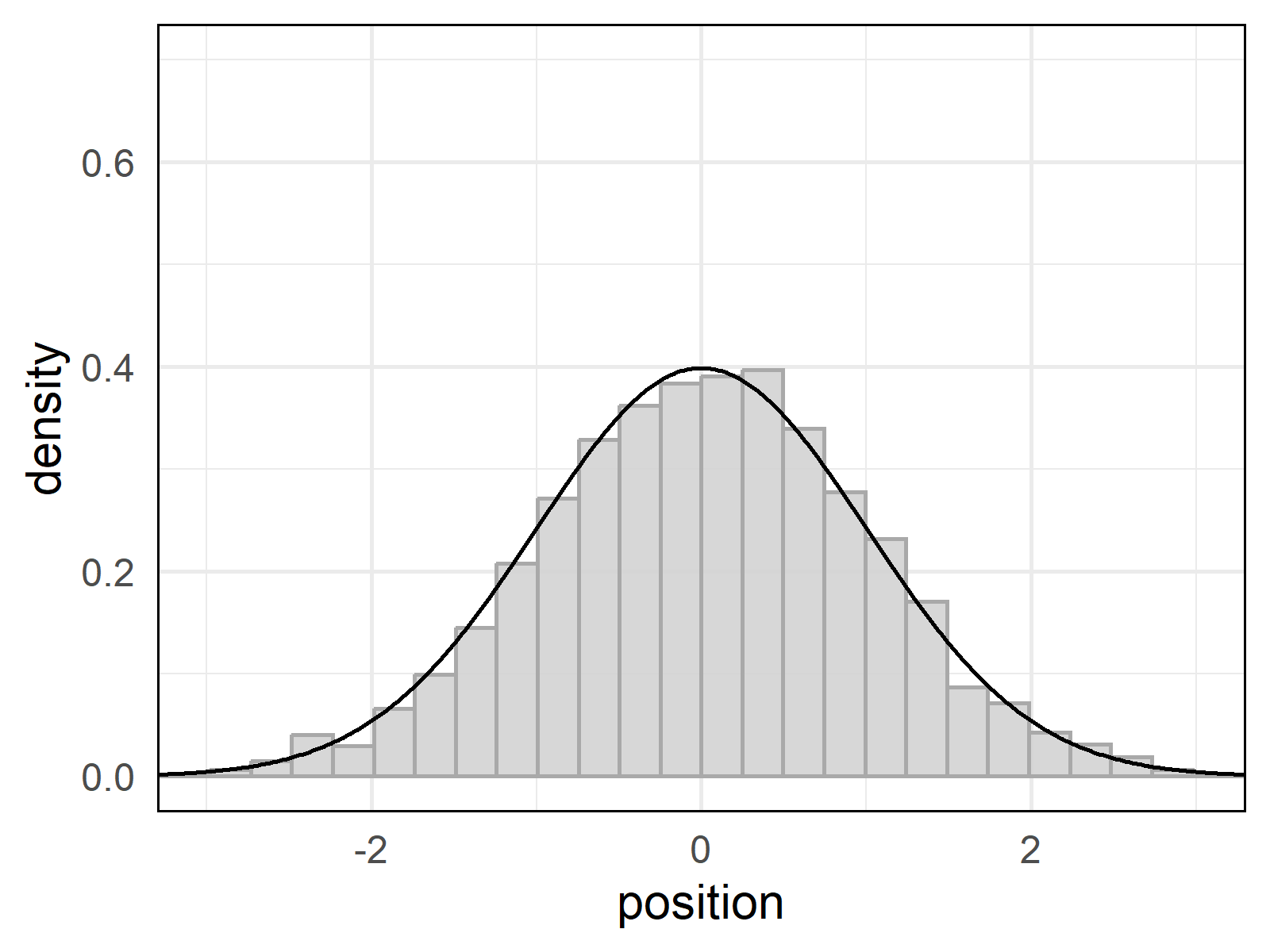} \\
    \end{tabular}
    \captionsetup{font=footnotesize}
    \caption{Empirical distributions vs.\ theoretical densities in the setting of \Cref{fig:simulations-grid}. Here, we run the simulation up to time $T=1$ with the same time step $\Delta t = 10^{-4}$ and compare the empirical distribution of the positions of $N=3000$ samples at time $T$ with the corresponding theoretical density $u^z(T,\cdot)$. The solid black curves represent the theoretical densities. The dashed blue lines indicate the support of the Barenblatt solutions in the porous medium and $p$-Laplace cases.}
    \label{fig:histogram-grid}
\end{figure}

\begin{figure}
    \centering
    \setlength{\tabcolsep}{4pt}
    \renewcommand{\arraystretch}{1.15}
    \begin{tabular}{>{\centering\arraybackslash}m{0.067\textwidth}
                    >{\centering\arraybackslash}m{0.288\textwidth}
                    >{\centering\arraybackslash}m{0.288\textwidth}
                    >{\centering\arraybackslash}m{0.288\textwidth}}
        \rule{0pt}{5ex} 
        
        & {\small\textbf{$p$-Laplace equation}}
        & {\small\textbf{Porous medium equation}}
        & {\small\textbf{Heat equation}} \\
        {\footnotesize$\beta = 0.0$}
        & \includegraphics[width=\linewidth]{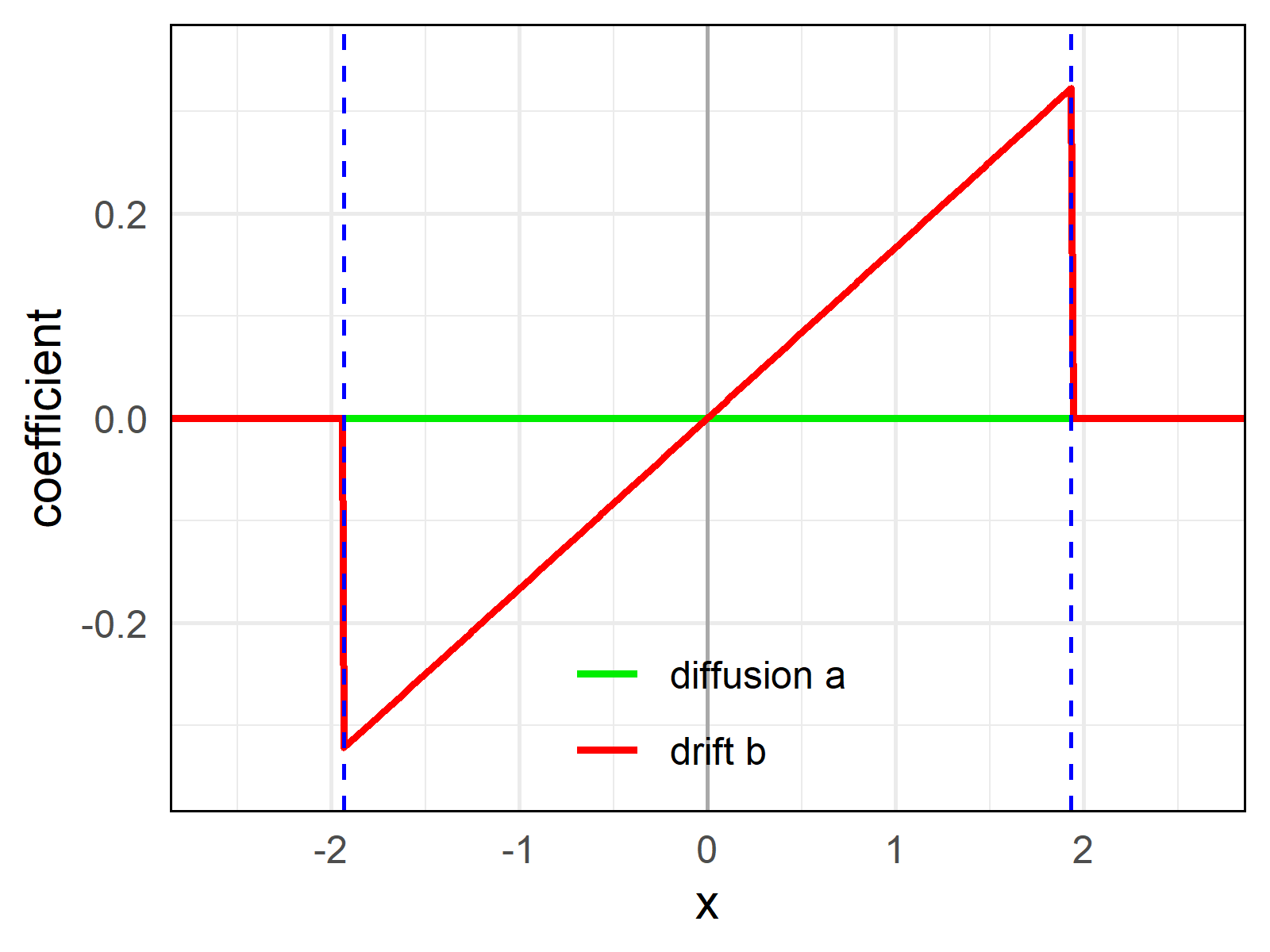}
        & \includegraphics[width=\linewidth]{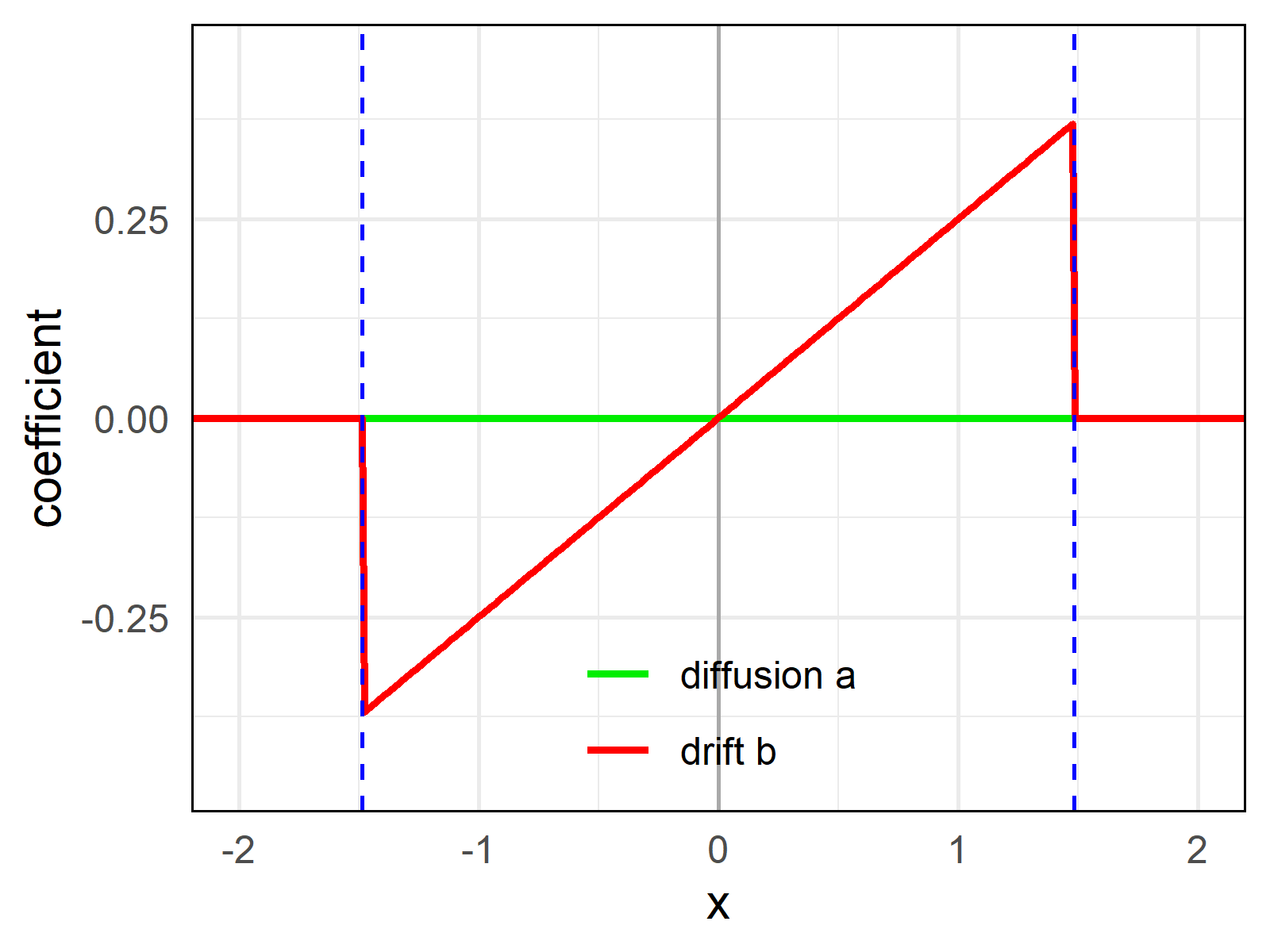}
        & \includegraphics[width=\linewidth]{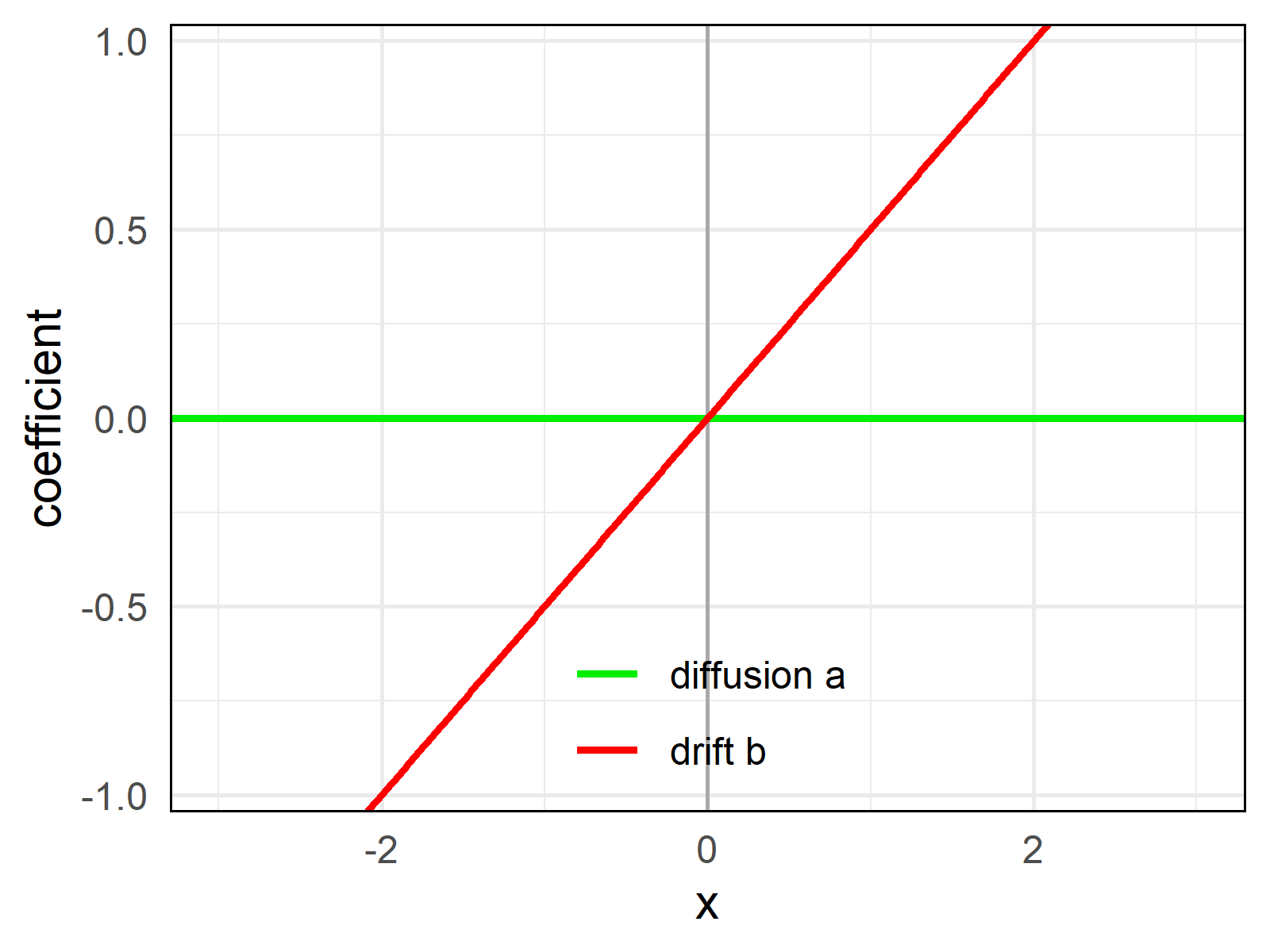} \\

        {\footnotesize$\beta = 0.1$}
        & \includegraphics[width=\linewidth]{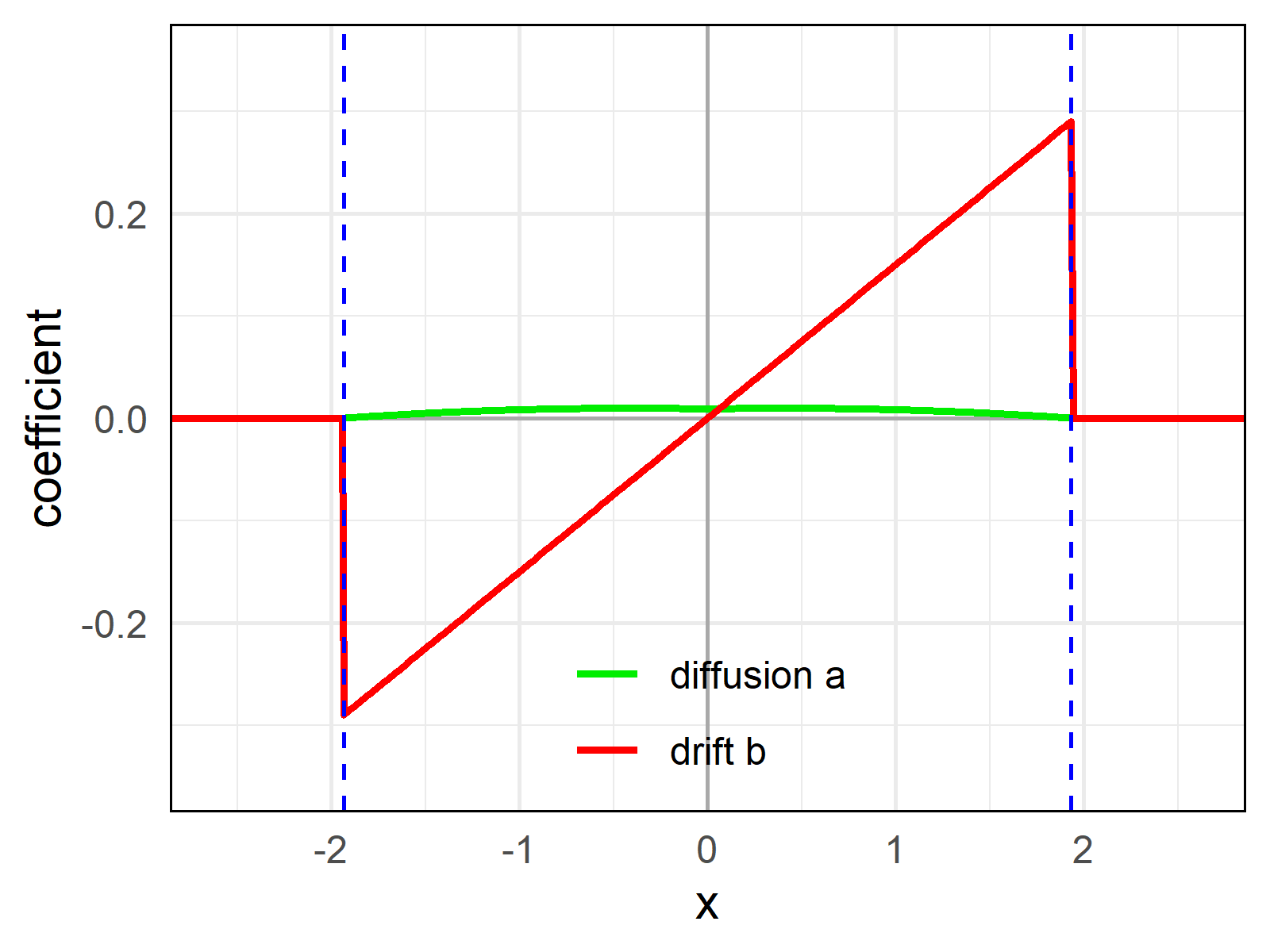}
        & \includegraphics[width=\linewidth]{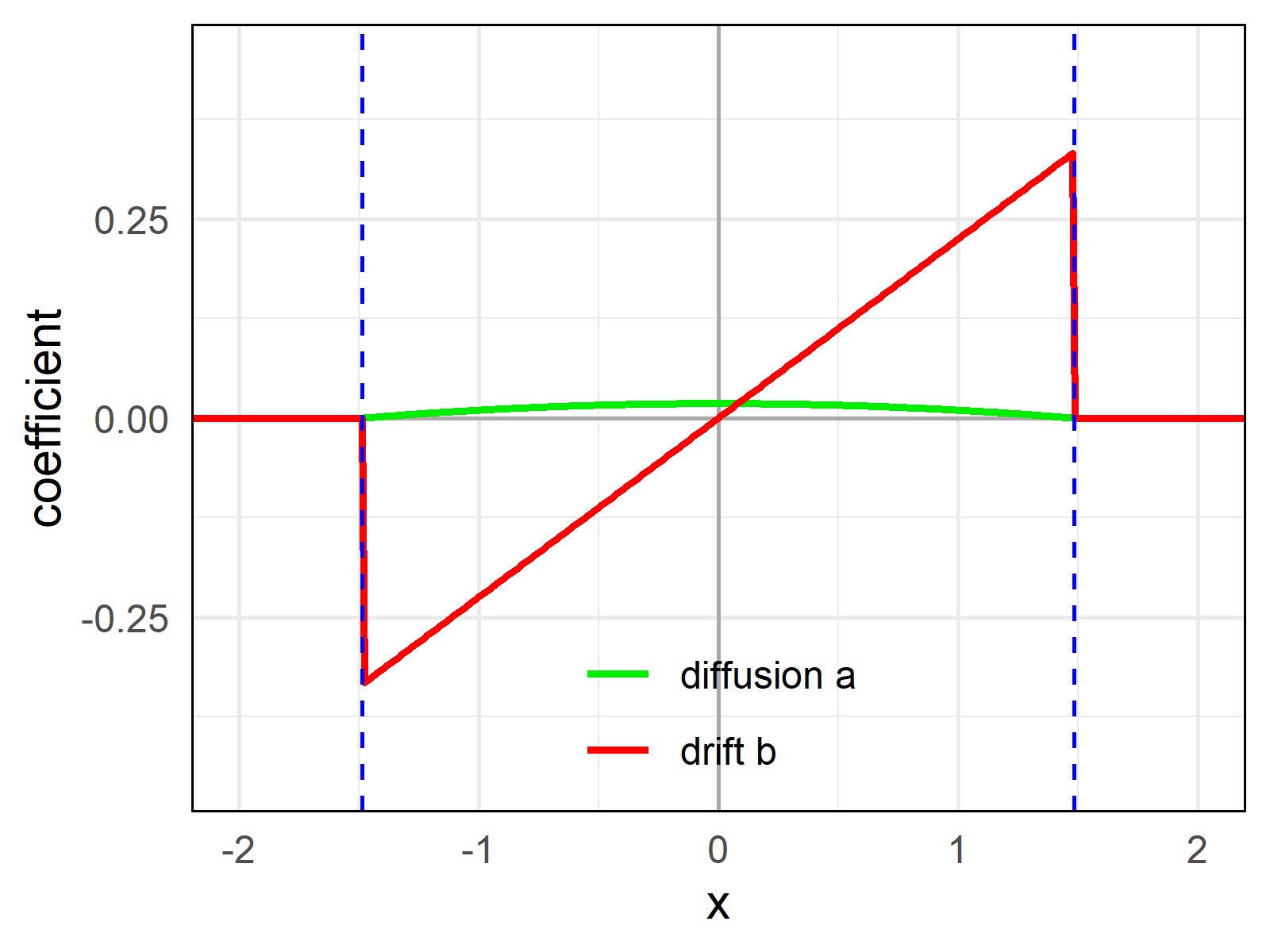}
        & \includegraphics[width=\linewidth]{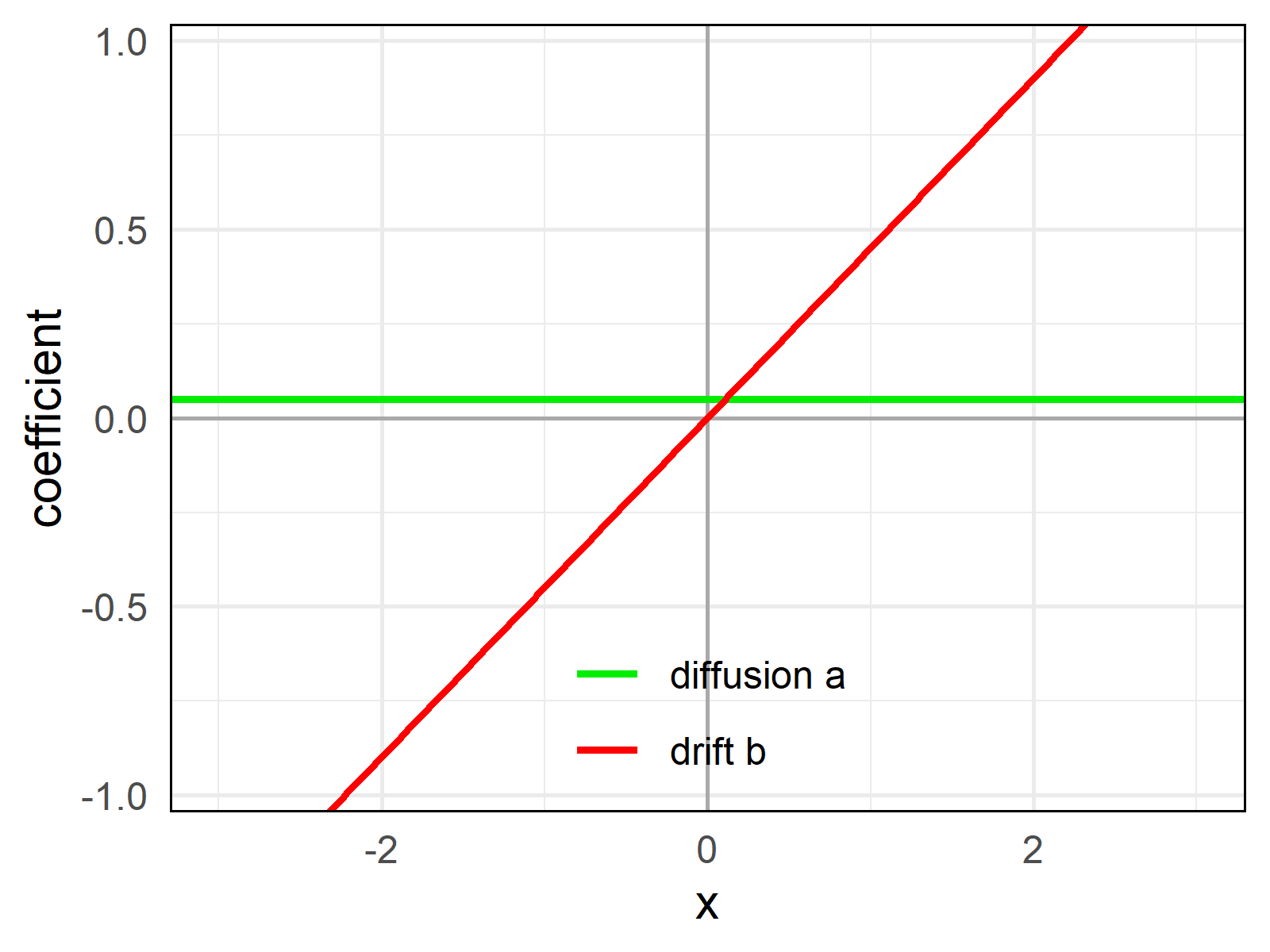} \\

        {\footnotesize$\beta = 1.0$}
        & \includegraphics[width=\linewidth]{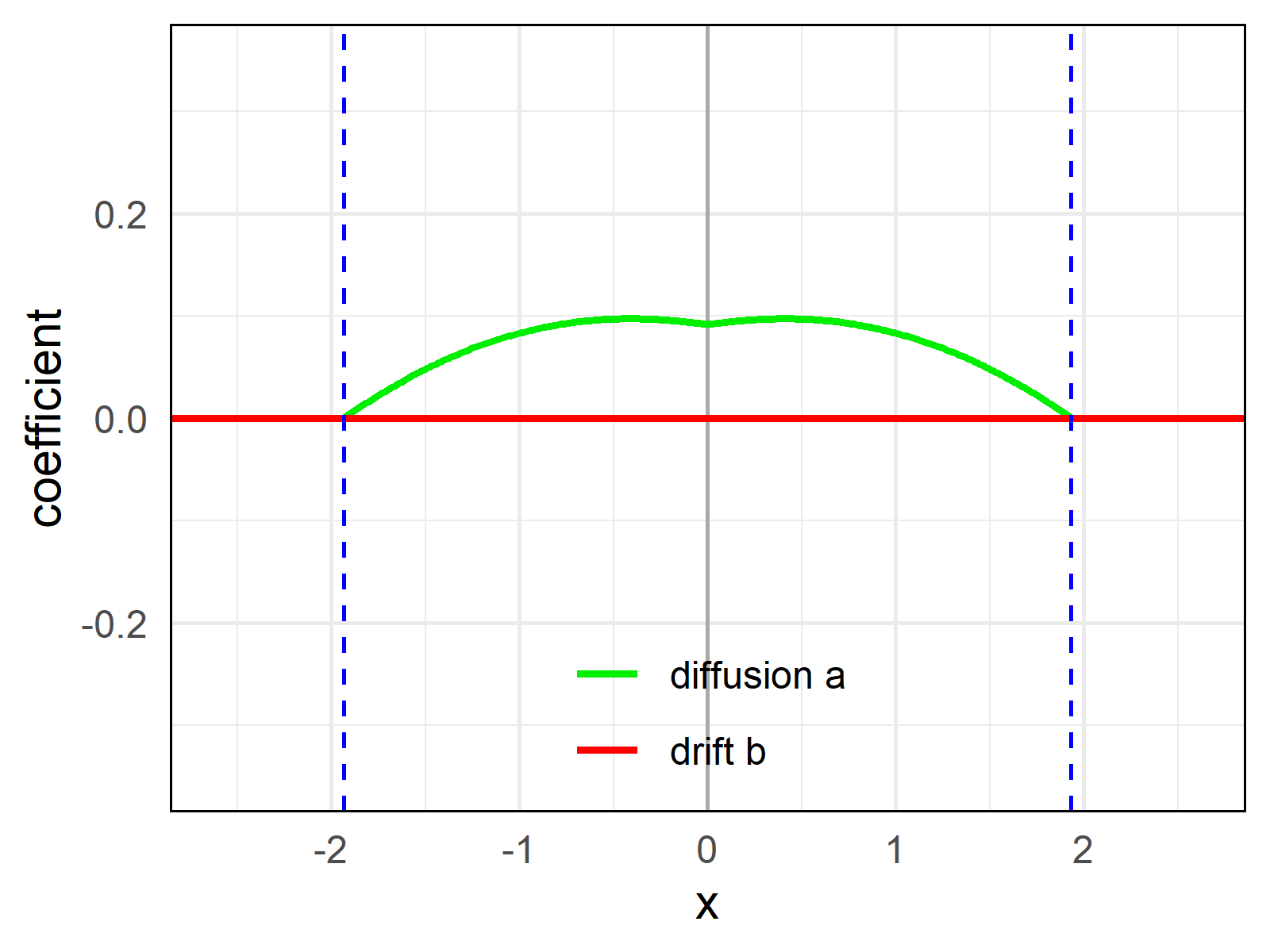}
        & \includegraphics[width=\linewidth]{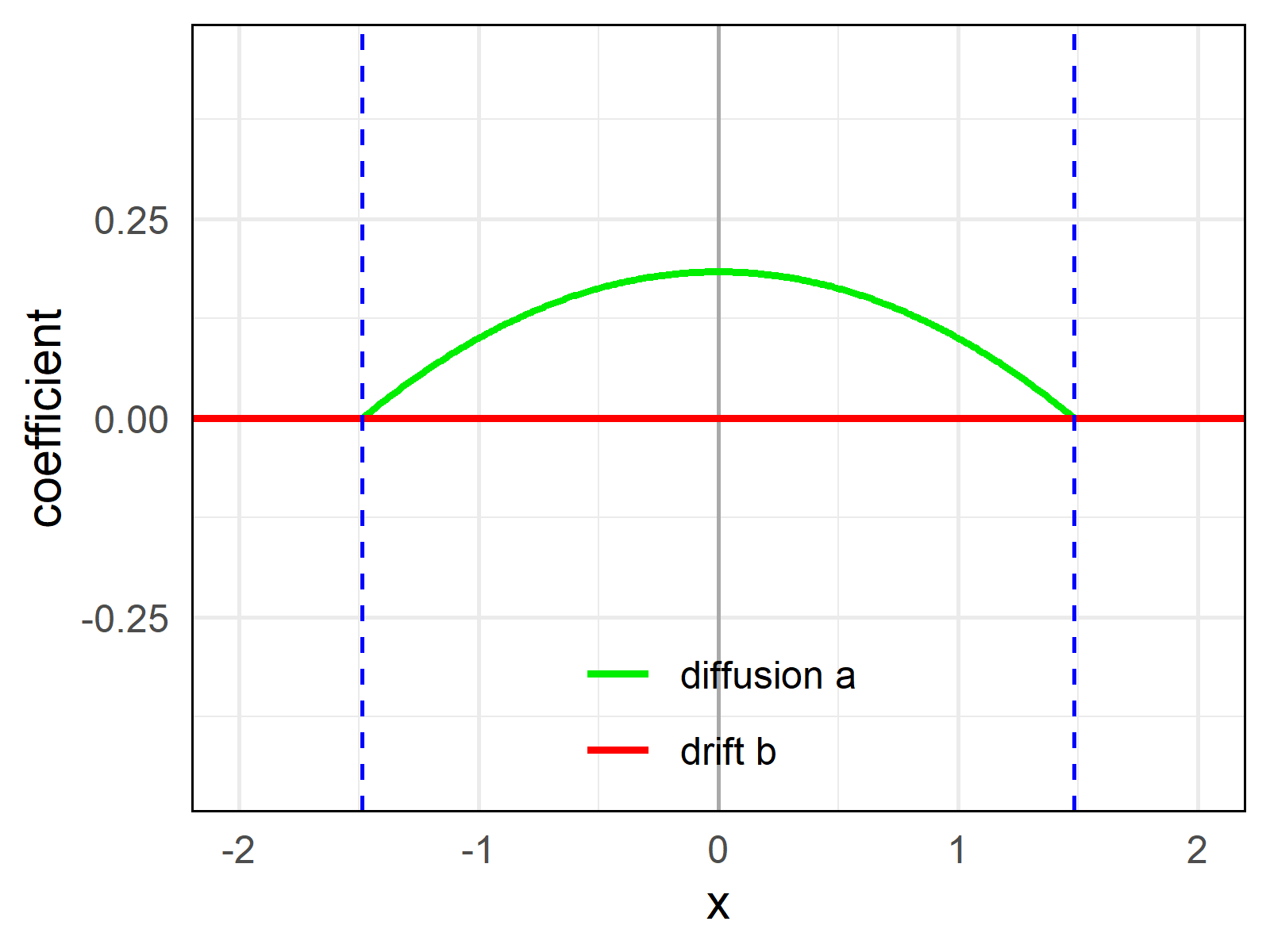}
        & \includegraphics[width=\linewidth]{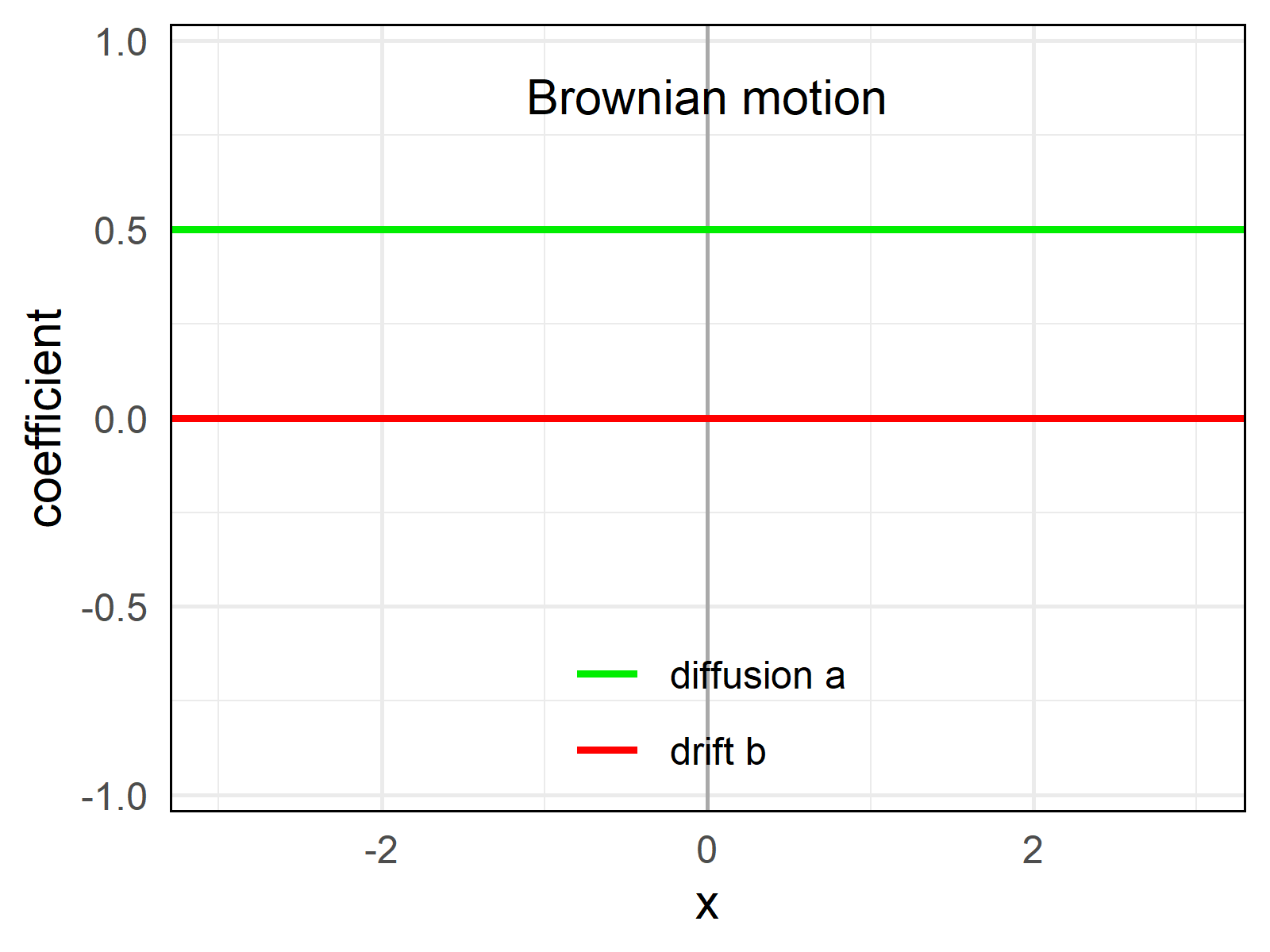} \\

        {\footnotesize$\beta = 1.5$}
        & \includegraphics[width=\linewidth]{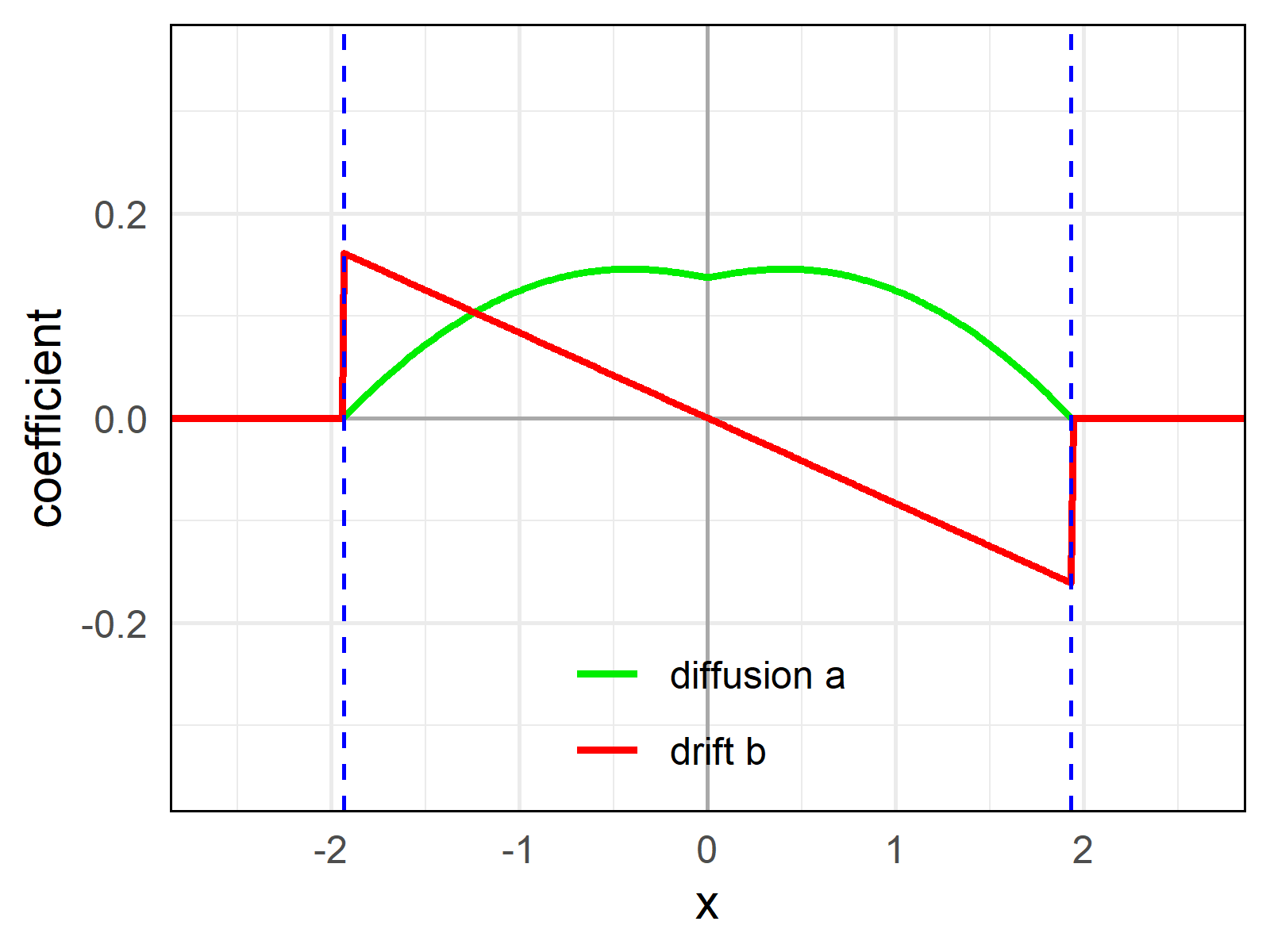}
        & \includegraphics[width=\linewidth]{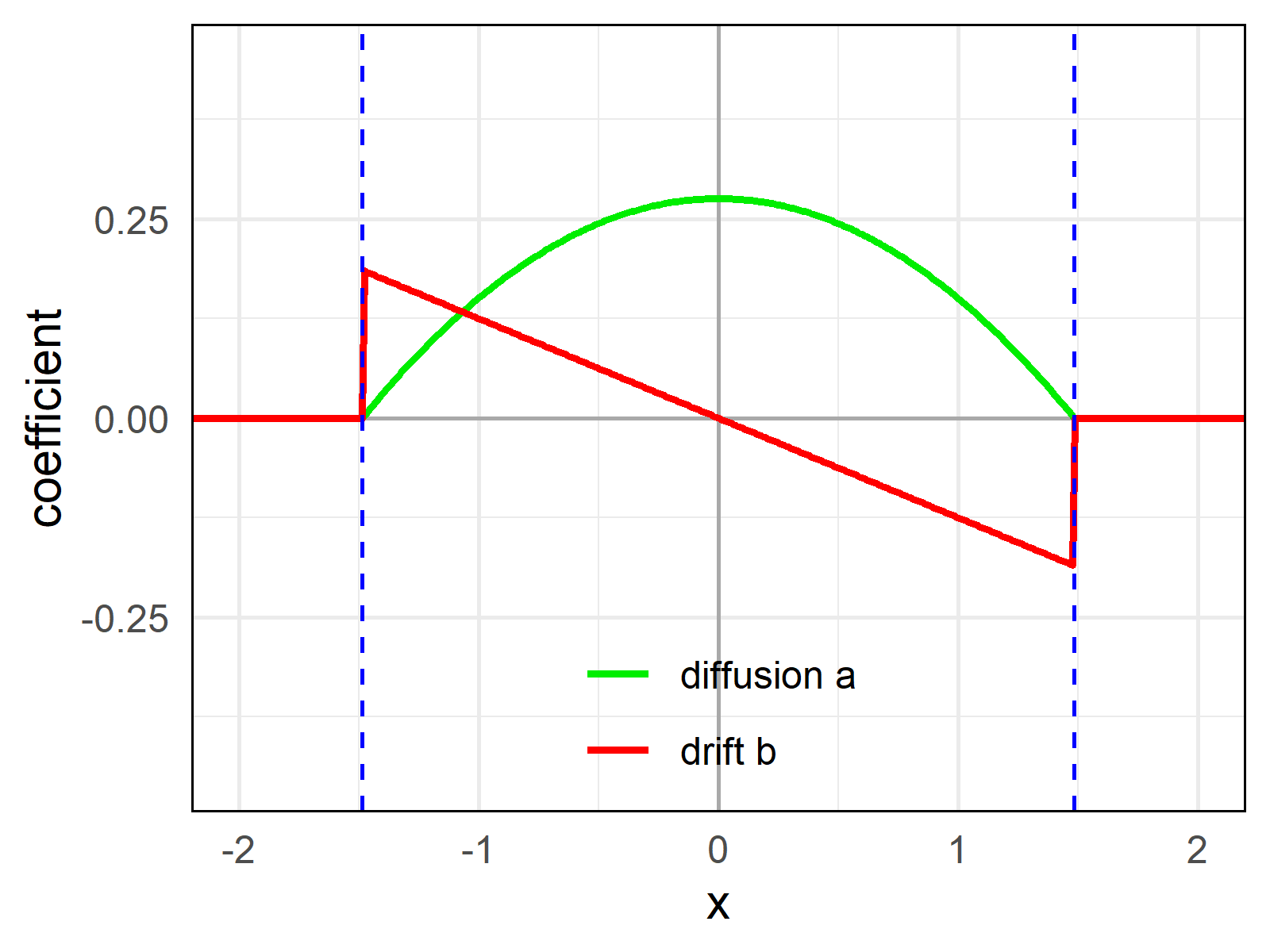}
        & \includegraphics[width=\linewidth]{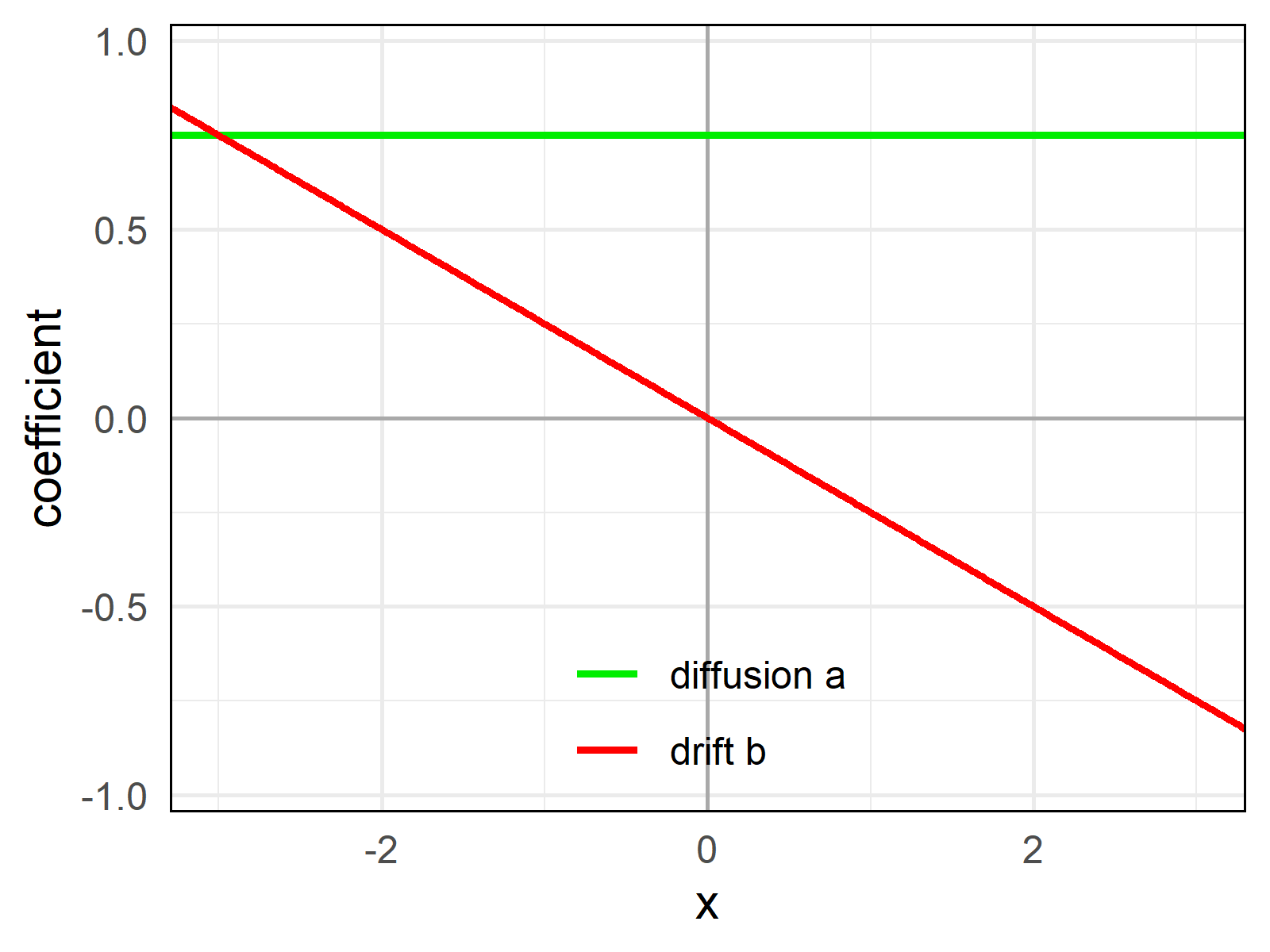} \\
    \end{tabular}
    \captionsetup{font=footnotesize}
    \caption{Diffusion coefficient $a^\beta(x,u^z(t,\cdot))$ (green) and drift coefficient $b^\beta(x,u^z(t,\cdot))$ (red) corresponding to the MV-SDEs simulated in \Cref{fig:simulations-grid}, plotted at time $t=1$ as functions of $x$. The dashed blue lines indicate the support of the Barenblatt solutions in the porous medium and $p$-Laplace cases.}
    \label{fig:coefficients-grid}
\end{figure}

\begin{figure}
    \centering
    \setlength{\tabcolsep}{4pt}
    \renewcommand{\arraystretch}{1.15}
    \begin{tabular}{>{\centering\arraybackslash}m{0.067\textwidth}
                    >{\centering\arraybackslash}m{0.288\textwidth}
                    >{\centering\arraybackslash}m{0.288\textwidth}
                    >{\centering\arraybackslash}m{0.288\textwidth}}          
        \rule{0pt}{5ex} 
        & {\small\textbf{$p$-Laplace equation}}
        & {\small\textbf{Porous medium equation}}
        & {\small\textbf{Heat equation}} \\
        {\footnotesize$\beta = 0.0$}
        & \includegraphics[width=\linewidth]{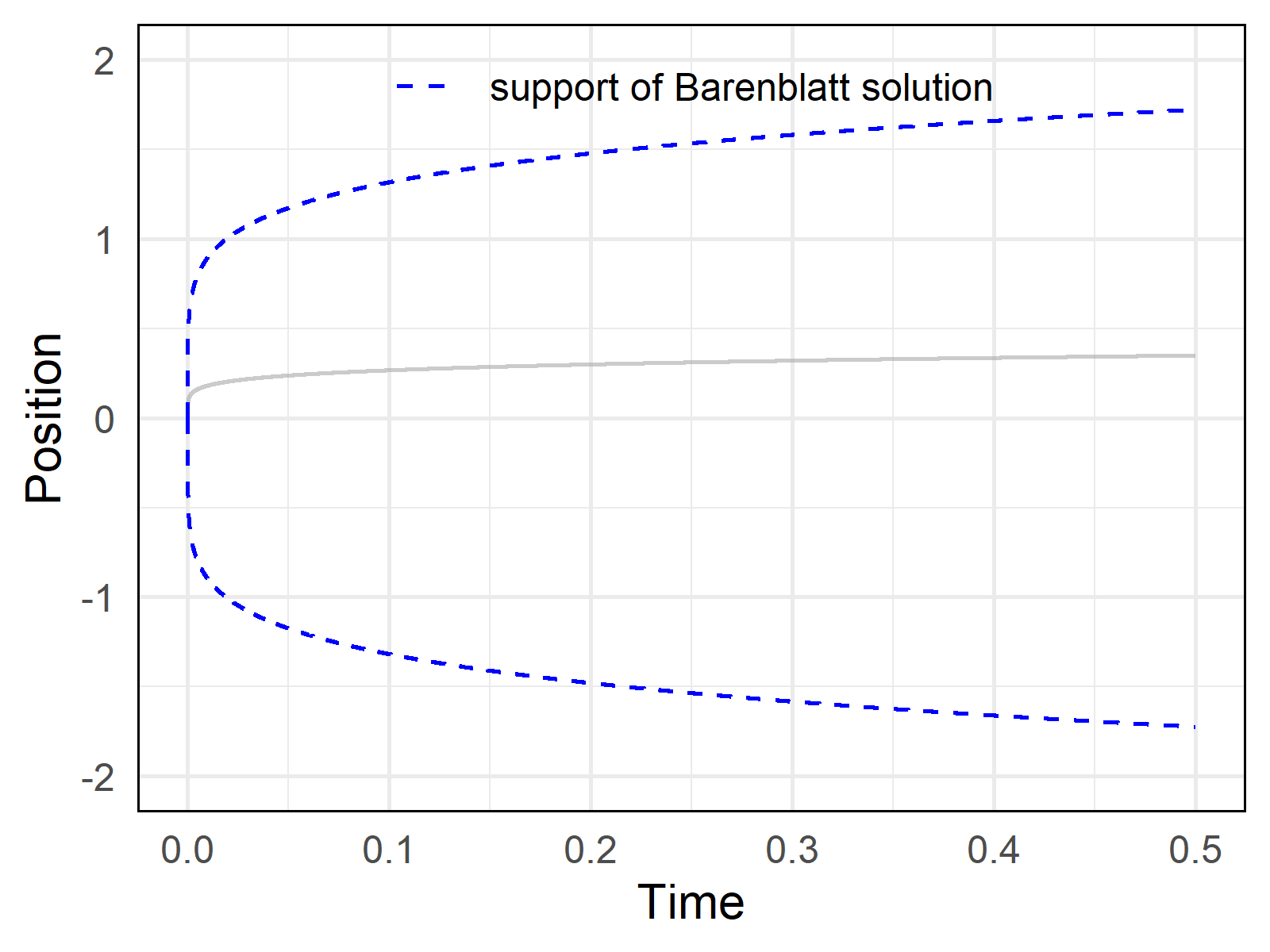}
        & \includegraphics[width=\linewidth]{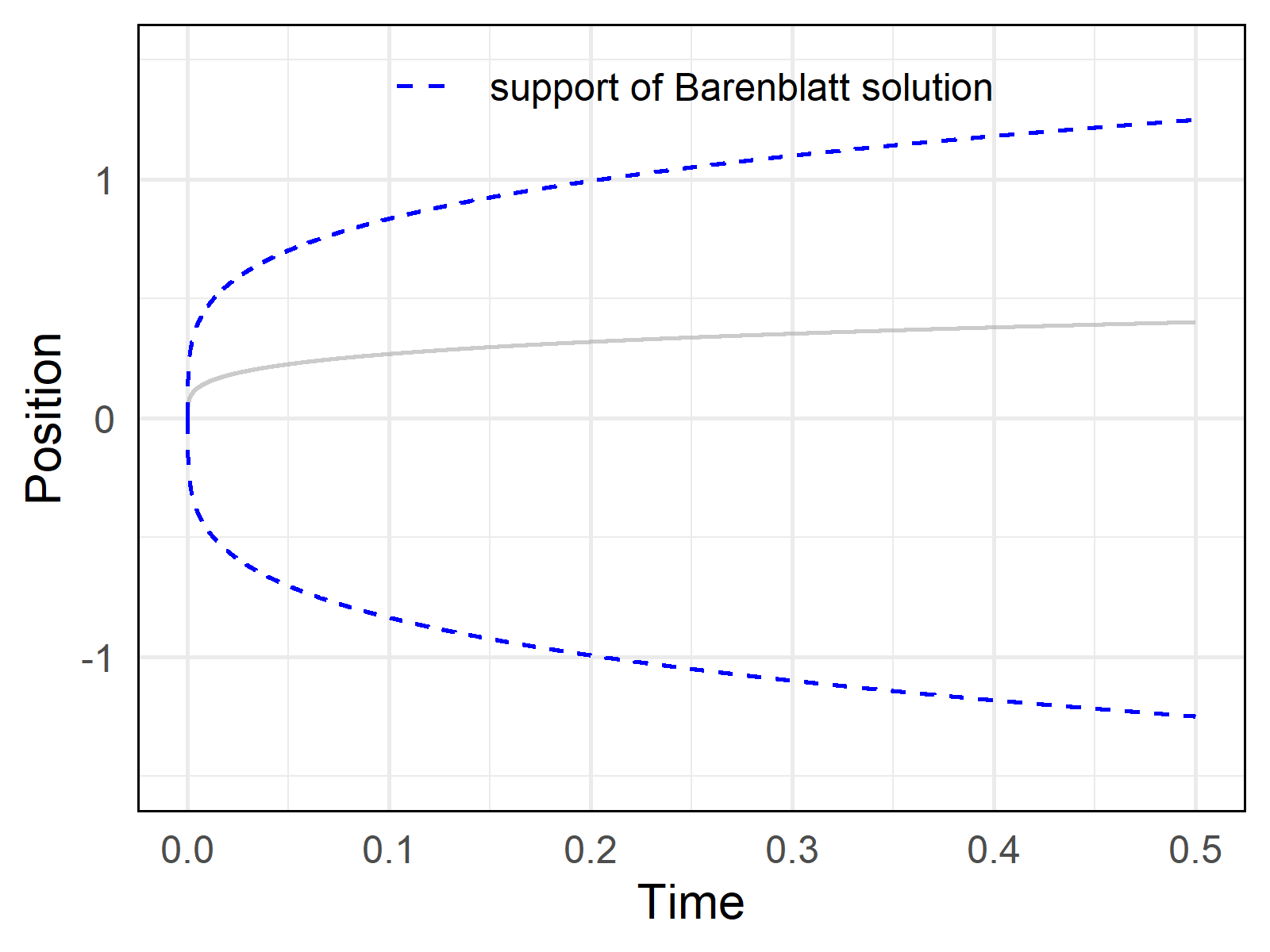}
        & \includegraphics[width=\linewidth]{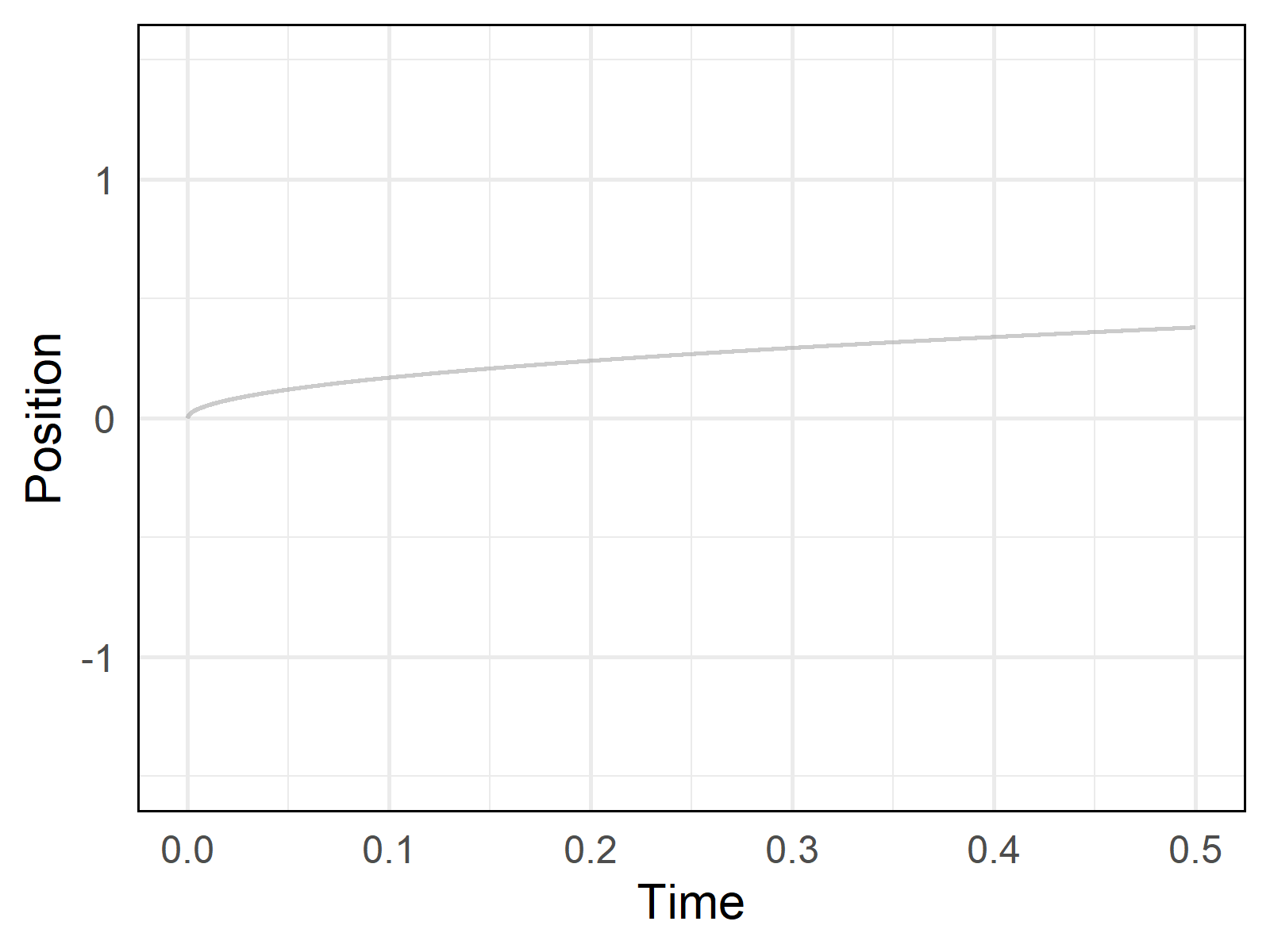} \\

        {\footnotesize$\beta = 0.1$}
        & \includegraphics[width=\linewidth]{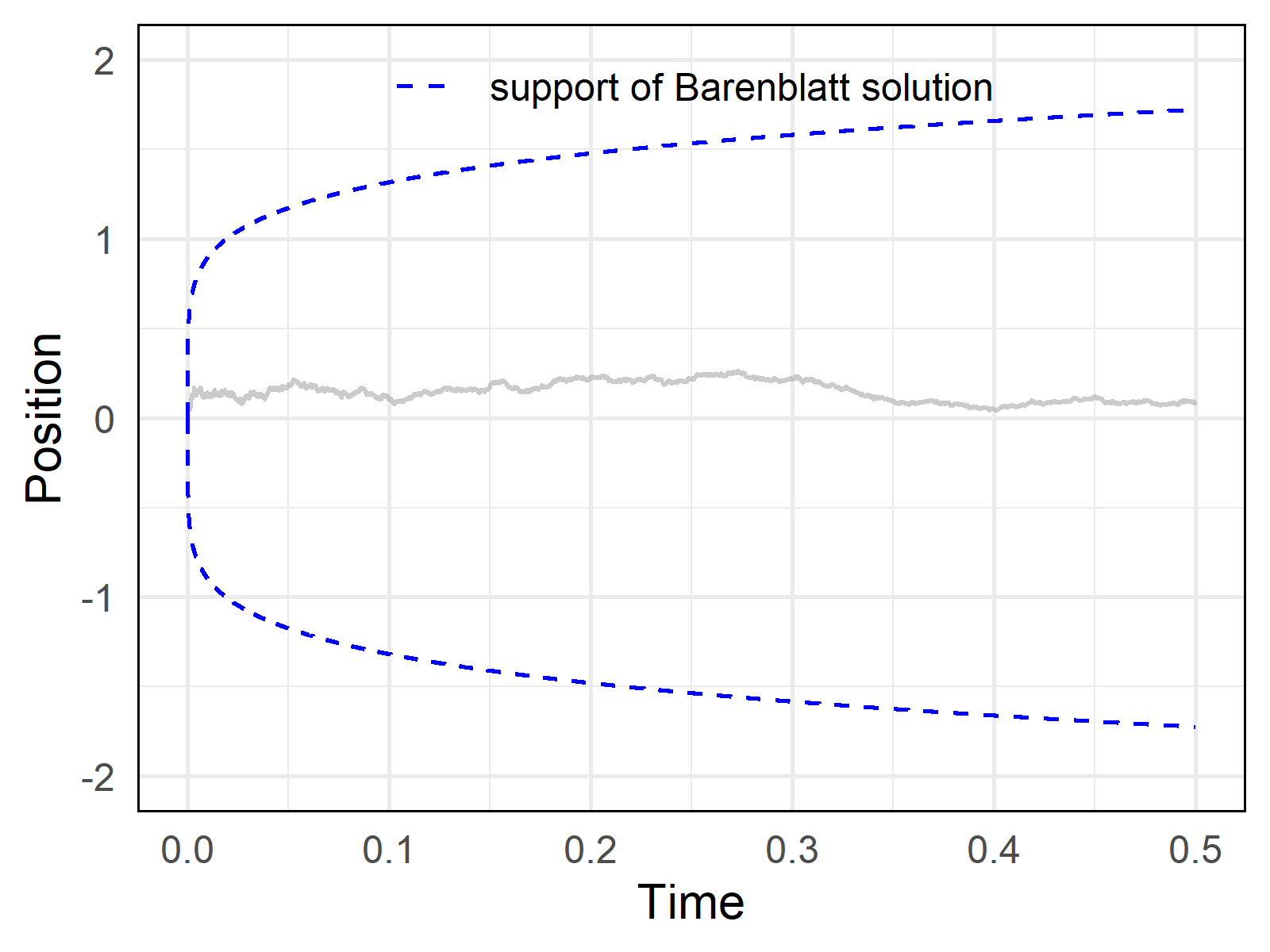}
        & \includegraphics[width=\linewidth]{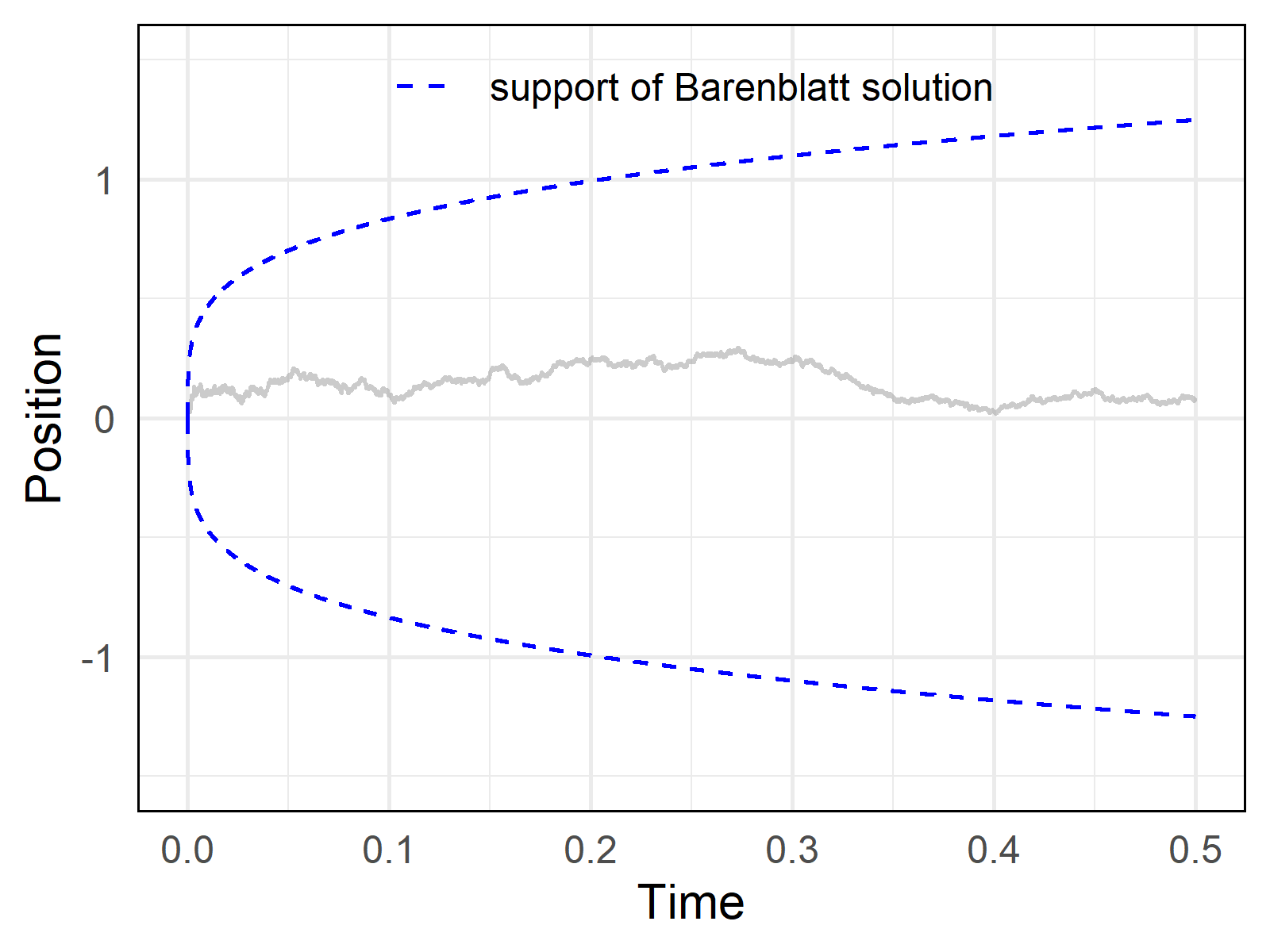}
        & \includegraphics[width=\linewidth]{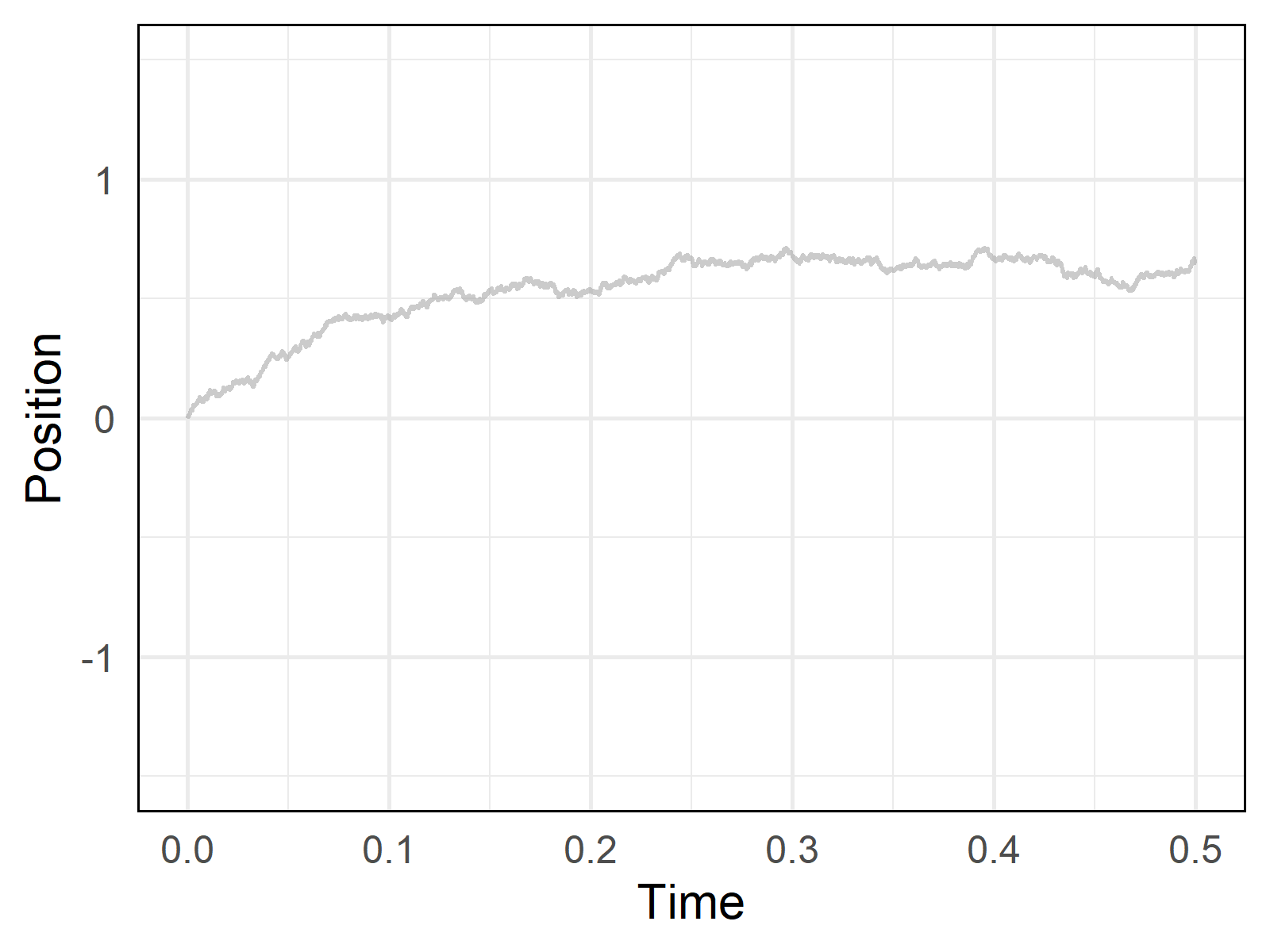} \\

        {\footnotesize$\beta = 1.0$}
        & \includegraphics[width=\linewidth]{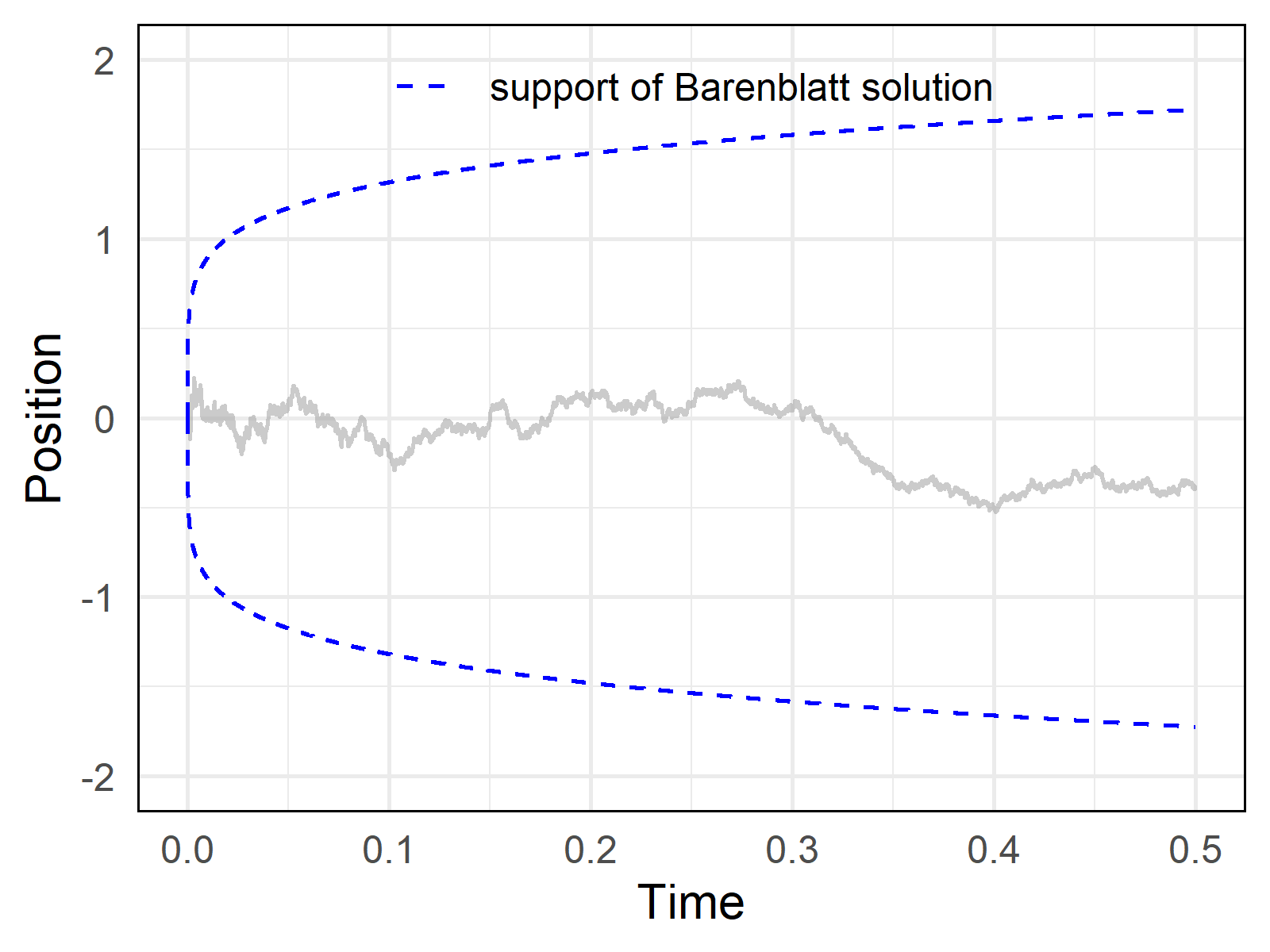}
        & \includegraphics[width=\linewidth]{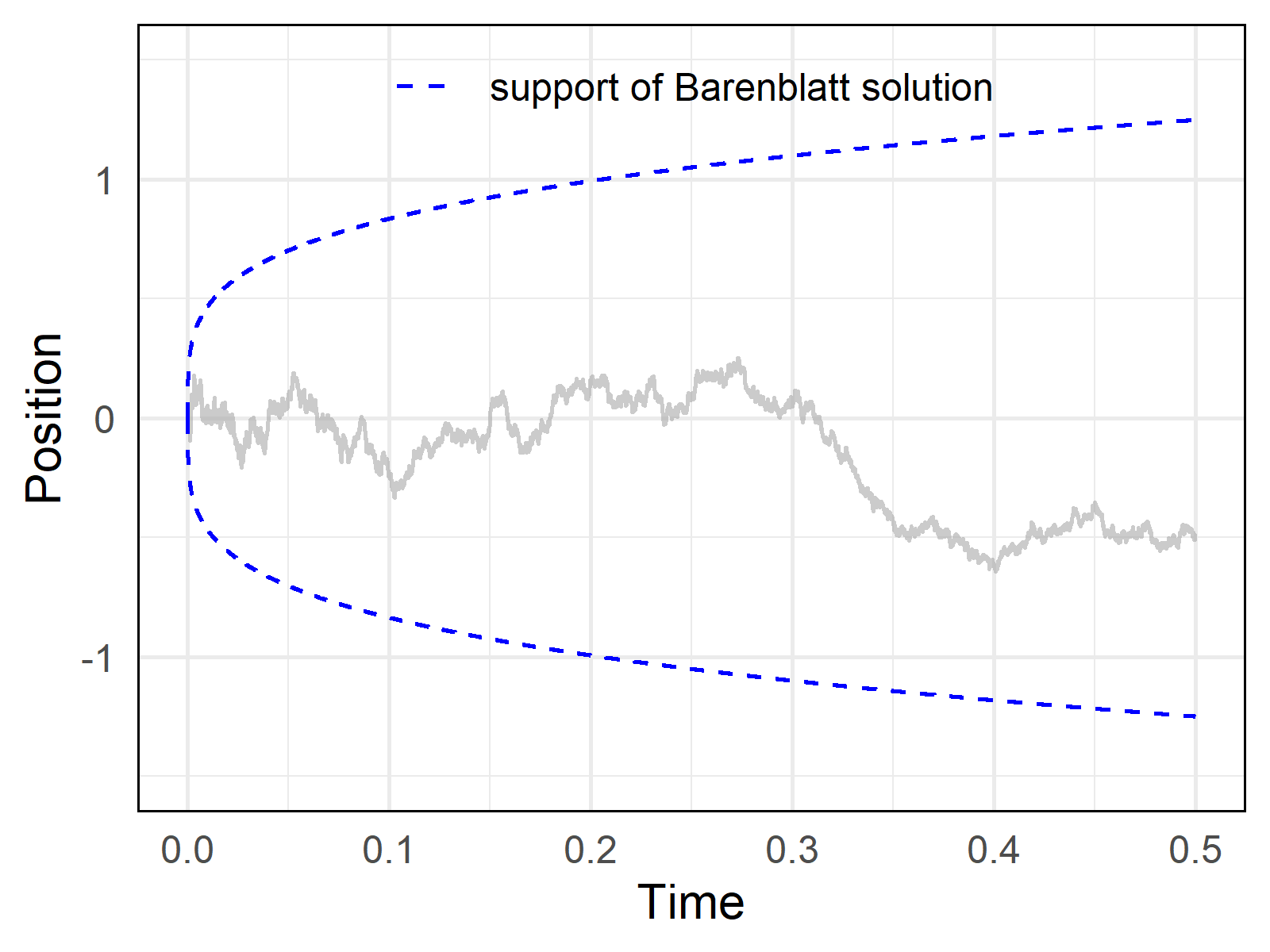}
        & \includegraphics[width=\linewidth]{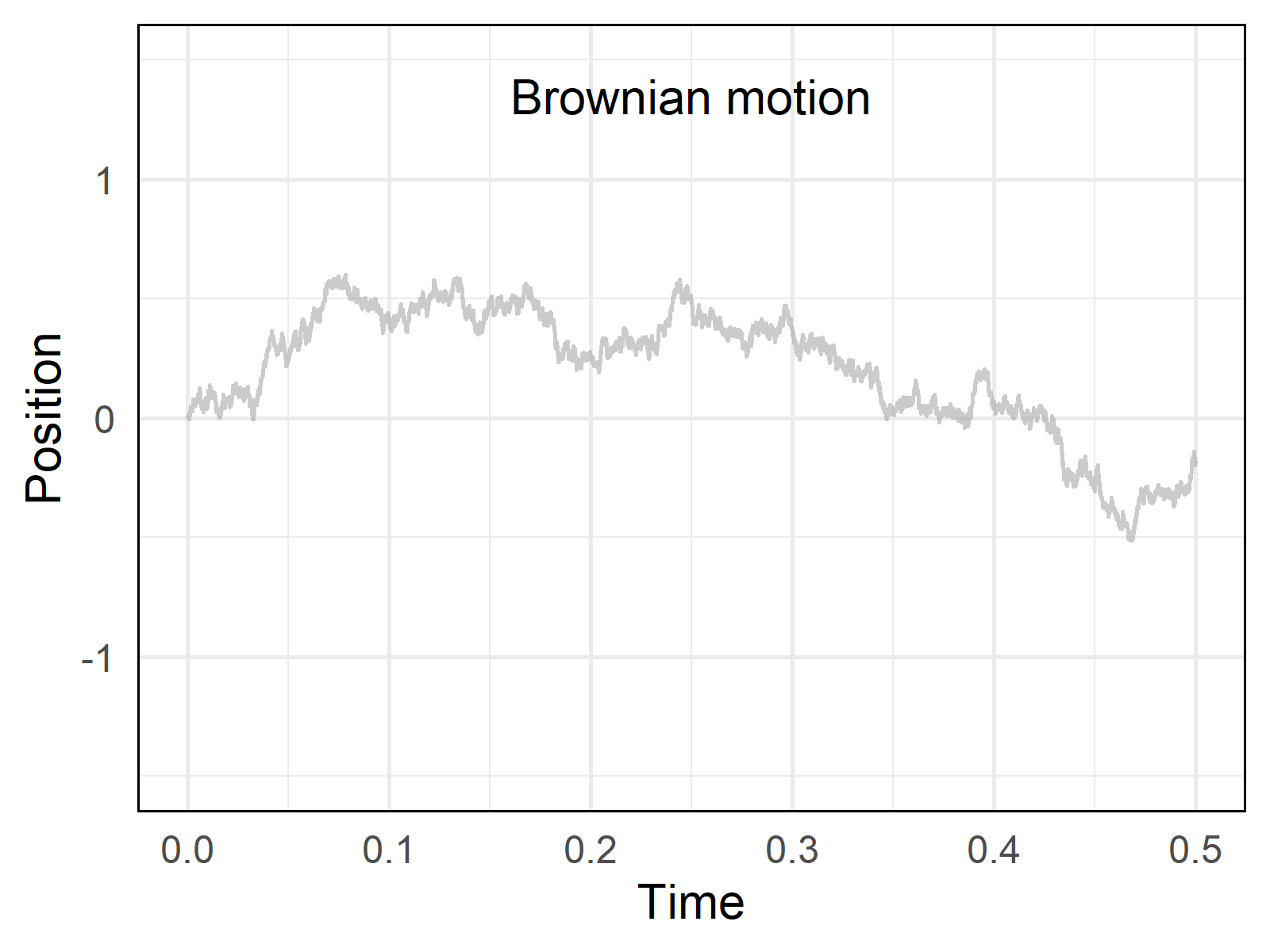} \\

        {\footnotesize$\beta = 1.5$}
        & \includegraphics[width=\linewidth]{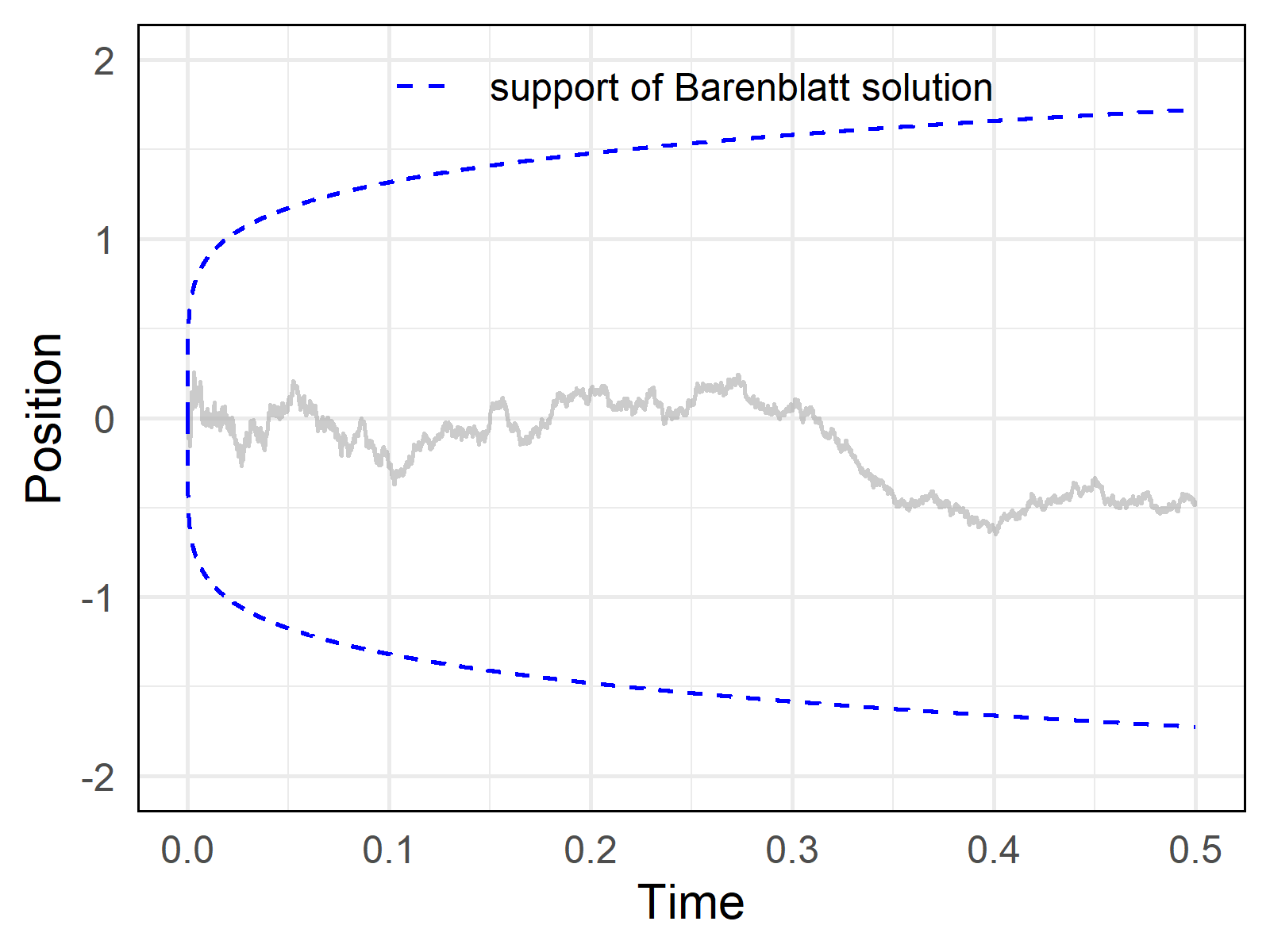}
        & \includegraphics[width=\linewidth]{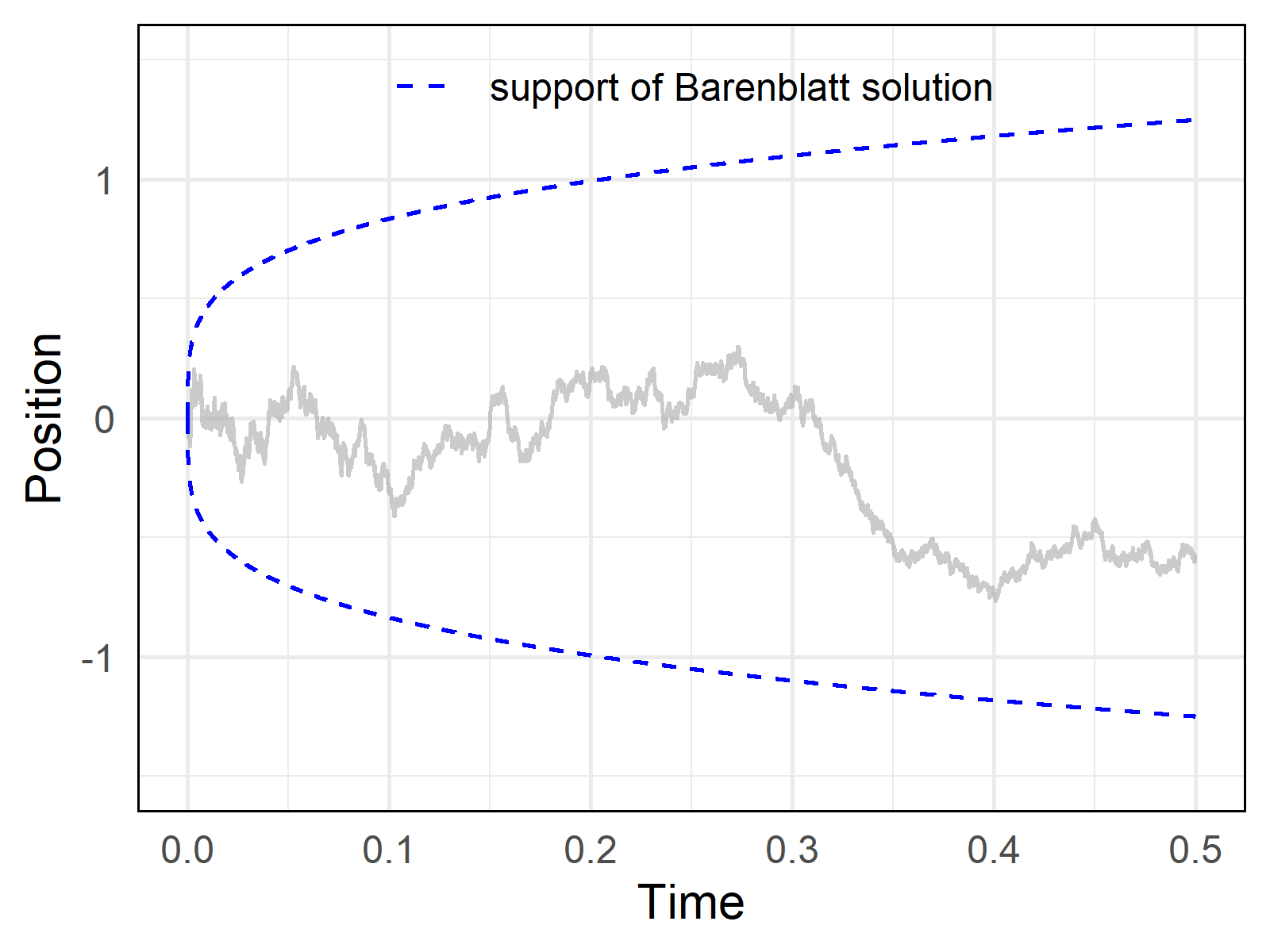}
        & \includegraphics[width=\linewidth]{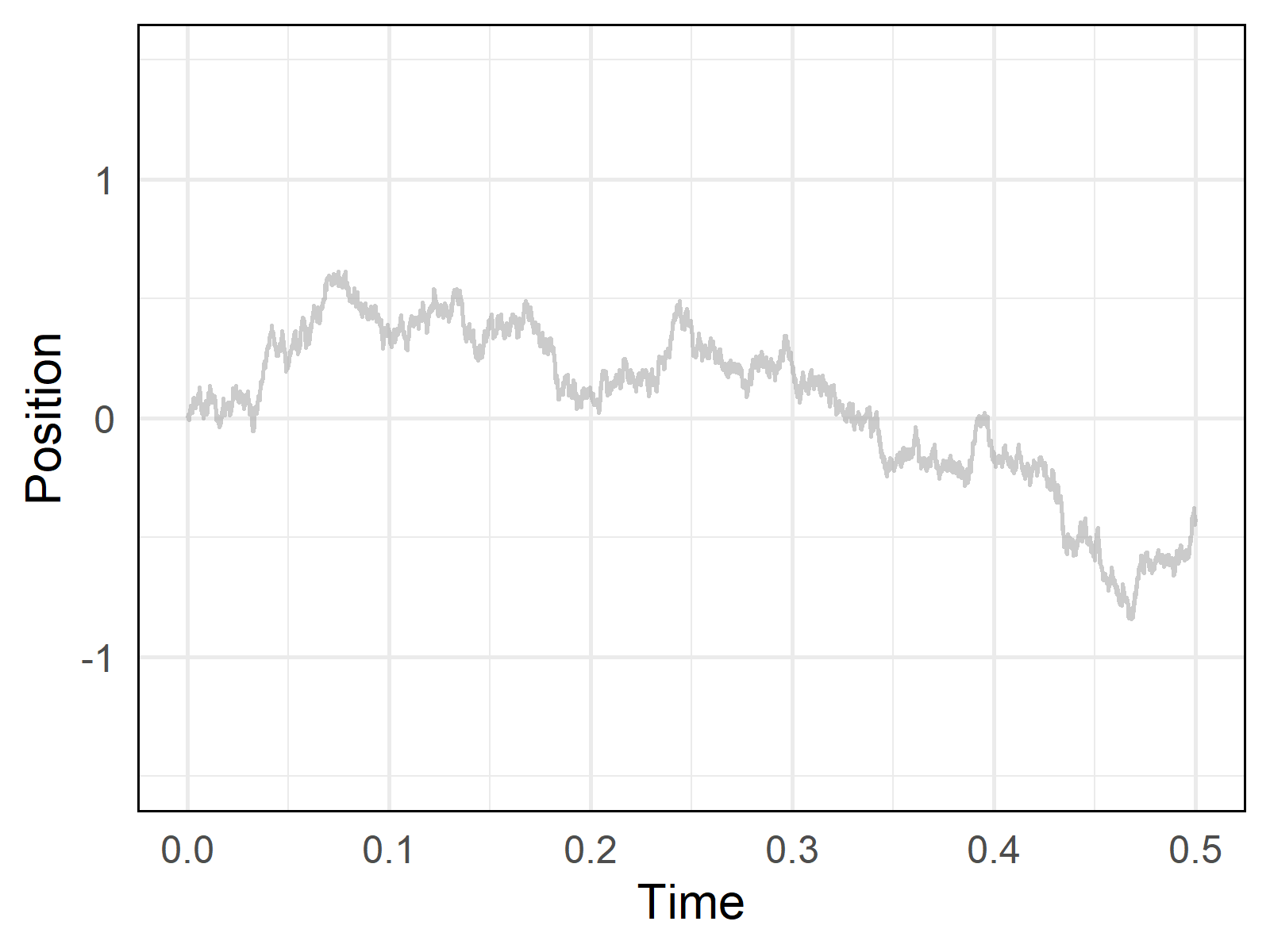} \\
    \end{tabular}
    \captionsetup{font=footnotesize}
    \caption{Single sample path simulation ($N=1$) in the same setting as \Cref{fig:simulations-grid}. The figure illustrates how the \emph{same realization} of Brownian increments in the scheme \eqref{eq:numerical_scheme} leads to different paths under the various nonlinear Fokker–Planck interpretations of a given PDE, interpolating between the pure-drift ($\beta=0$) and pure-diffusion ($\beta=1$) cases.}
    \label{fig:single-path-grid}
\end{figure}

Here we explain the numerical procedure used for the sample path simulations in \textbf{\Cref{fig:simulations-grid}}.
In \Cref{subsec:heat_interpolation,subsec:PME-beta} and \Cref{rem:interpolation-pL}, for the heat equation, porous medium equation, and the $p$-Laplace equation, respectively, we showed the existence of weak solutions to $\beta$-interpolated MV-SDEs with one-dimensional time marginals given by the PDE solution $u^z$. These equations are parameterized by $\beta \in (0,\infty)$ and interpolate between the pure-drift ($\beta=0$) and pure-diffusion ($\beta=1$) cases:
\begin{equation}
    \begin{dcases}
        dX_t = b^{\beta}(X_t,u^z(t,\cdot)) \, dt 
        + \sqrt{2}\,\big(a^{\beta}(X_t,u^z(t,\cdot))\big)^{1/2} \, dW_t, \\
        \mathcal{L}(X_t) = u^z(t,x)\,dx,\quad t>0, \\
        \mathcal{L}(X_0) = \delta_z,
    \end{dcases}
\end{equation}
where $(W_t)_{t\geq 0}$ is a standard $d$-dimensional Brownian motion. The interpolated coefficients are denoted here by $(a^{\beta},b^{\beta})$, which are given in \eqref{eq:HE_beta_coeff} for the heat equation, in \eqref{eq:PME_beta_coeff} for the porous medium equation, and in \Cref{rem:interpolation-pL} for the $p$-Laplace equation. The family $\{u^z\}_{z \in \R^d }$ is the heat kernel \eqref{eq:fundamental_solution_heat} for the heat equation, the Barenblatt solution \eqref{eq:barenblatt_solution_porous} for the porous medium equation, and the Barenblatt solution \eqref{eq:barenblatt_solution_plaplace} for the $p$-Laplace equation.
Plugging $u^z(t,\cdot)$ into these coefficients, we get coefficients depending on $(t,x)$, here denoted by $ a^{\beta,z}(t,x) \coloneqq a^{\beta}(x,u^z(t,\cdot)) $ and $ b^{\beta,z}(t,x) \coloneqq b^{\beta}(x,u^z(t,\cdot)) $ for $(t,x) \in (0,\infty) \times \R^d$.
The corresponding explicit formulas for $(a^{\beta,z},b^{\beta,z})$ are already computed in the aforementioned sections. 
Our goal is then to simulate the linear SDE with the time-marginal constraint:
\begin{equation}
    \begin{dcases}
        dX_t = b^{\beta,z}(t,X_t)\,dt 
        + \sqrt{2} \big(a^{\beta,z}(t,X_t) \big)^{1/2} \, dW_t, \\
        \mathcal{L}(X_t) =   u^z(t,x)\,dx,\quad t>0, \\
        X_0 = z.
    \end{dcases}
\end{equation}
We consider the initial point $z=0$ and dimension $d=1$. We fix a time horizon $T>0$ and discretize the interval $[0,T]$ with time step $\Delta t$ into a grid $(t_i)_{i=0}^M$, where $t_i \coloneqq i\,\Delta t$ and $M \coloneqq T/\Delta t$. We implement the \emph{Euler--Maruyama scheme}.
Since the coefficients $(a^{\beta,z},b^{\beta,z})$ are singular at $t=0$ (except for $\beta=1$ of the heat equation, which corresponds to standard Brownian motion), to avoid evaluating them at $t=0$, we sample the position at the first time step $\Delta t$ from the distribution $u^z(\Delta t,x)\,dx$ using the \emph{inverse CDF method}. The resulting approximation $(\tilde X_i)_{i=0}^M$ is given by
\begin{equation}\label{eq:numerical_scheme}
    \begin{dcases}
        \tilde X_{i+1} = \tilde X_i 
        + b^{\beta,z}(t_i,\tilde X_i)\,\Delta t
        + \sqrt{2} \big( a^{\beta,z}(t_i,\tilde X_i) \big)^{1/2}\,\Delta W_i, \qquad i \in \{ 1,\cdots M-1\}, \\
        \mathcal{L}(\tilde X_1) =  u^z(\Delta t,x)\,dx, \\
        \tilde X_0 = z,
        \end{dcases}
\end{equation}
with i.i.d. increments $\Delta W_i $ with distribution $ \mathcal{N}(0,\Delta t)$. We simulate $N$ independent sample paths according to this scheme. The implementation is done in the statistical programming language \texttt{R}.

Complementing \textbf{\Cref{fig:simulations-grid}} in the same setting, we provide the following supplementary plots:
\begin{itemize}
    \item \textbf{\Cref{fig:histogram-grid}.} Empirical distributions vs.\ theoretical densities. This serves as a consistency check for our numerical scheme.
    
    \item \textbf{\Cref{fig:coefficients-grid}.} Drift and diffusion coefficients. These plots display the coefficients $a^\beta(x,u^z(t,\cdot))$ and $b^\beta(x,u^z(t,\cdot))$ as functions of $x$ at the fixed time $t=1$.
    
    \item \textbf{\Cref{fig:single-path-grid}.} Single sample path simulations. Here, we fix a single realization of the Brownian increments in the scheme \eqref{eq:numerical_scheme} and illustrate the resulting pathwise behavior.
\end{itemize}

\clearpage

\bibliographystyle{alpha}
\bibliography{main}

@article{Abedi2025paths,
      title={{Fractional Sobolev paths on Wasserstein spaces and their energy-minimizing particle representations}}, 
      author={E. Abedi},
      year={2025},
      eprint={2502.12068},
      archivePrefix={arXiv},
      primaryClass={math.MG},
      journal={arXiv preprint 2502.12068}
}

@article{Abedi2025processes,
      title={{Fractional Sobolev processes on Wasserstein spaces and their energy-minimizing particle representations with applications}}, 
      author={E. Abedi},
      year={2025},
      eprint={2503.10859},
      archivePrefix={arXiv},
      primaryClass={math.PR},
      journal={arXiv preprint 2503.10859}
}

@book{Ambrosio2024,
  title = {Lectures on Optimal Transport},
  ISBN = {9783031768347},
  ISSN = {2532-3318},
  url = {http://dx.doi.org/10.1007/978-3-031-76834-7},
  DOI = {10.1007/978-3-031-76834-7},
  journal = {UNITEXT},
  publisher = {Springer Nature Switzerland},
  author = {Ambrosio,  L. and Brué,  E. and Semola,  D.},
  year = {2024}
}

@book{AGS2008GFs,
  author       = {L. Ambrosio and N. Gigli and G. Savar{\'e}},
  title        = {Gradient Flows in Metric Spaces and in the Space of Probability Measures},
  series       = {Lectures in Mathematics. ETH Z{\"u}rich},
  edition      = {2nd},
  year         = {2008},
  publisher    = {Birkh{\"a}user Basel},
  address      = {Basel},
  isbn         = {978-3-7643-8721-1},
  eisbn        = {978-3-7643-8722-8},
  pages        = {IX, 334},
  doi          = {10.1007/978-3-7643-8722-8}
}

@article{Barenblatt1952,
  author    = {G. I. Barenblatt},
  title     = {On self-similar motions of compressible fluids in porous media},
  journal   = {Prikl. Mat. Mekh.},
  volume    = {16},
  pages     = {679--698},
  year      = {1952},
  note      = {(in Russian)}
}

@article{Barenblatt1952porous,
  author    = {G. I. Barenblatt},
  title     = {On some unsteady motions of a liquid or a gas in a porous medium},
  journal   = {Prikl. Mat. Mekh.},
  volume    = {16},
  pages     = {67--78},
  year      = {1952},
  note      = {(in Russian)}
}

@article{BenamouBrenier2000,
  author = {Benamou, J.-D. and Brenier, Y.},
  title = {{A computational fluid mechanics solution to the Monge-Kantorovich mass transfer problem}},
  journal = {Numerische Mathematik},
  year = {2000},
  volume = {84},
  number = {3},
  pages = {375--393},
  doi = {10.1007/s002110050002},
  issn = {0945-3245}
}

@book{BlumenthalGetoor1968,
  author    = {R. M. Blumenthal and R. K. Getoor},
  title     = {Markov Processes and Potential Theory},
  series    = {Pure and Applied Mathematics},
  volume    = {29},
  publisher = {Academic Press},
  address   = {New York and London},
  year      = {1968}
}

@article{R.BGR25,
	title={The {L}eibenson process},
	author={V. Barbu and S. Grube and M. Rehmeier and M. R{\"o}ckner},
	year={2025},
	journal={arXiv preprint 2508.12979},
}

@book{BliedtnerHansen1986,
  author    = {J. Bliedtner and W. Hansen},
  title     = {Potential Theory: An Analytic and Probabilistic Approach to Balayage},
  series    = {Universitext},
  publisher = {Springer-Verlag},
  address   = {Berlin},
  year      = {1986}
}

@book {BogachevKrylovRoecknerShaposhnikov2015,
    AUTHOR = {Bogachev, V. I. and Krylov, N. V. and R\"ockner, M. and Shaposhnikov, S. V.},
     TITLE = {Fokker-{P}lanck-{K}olmogorov equations},
    SERIES = {Mathematical Surveys and Monographs},
    VOLUME = {207},
 PUBLISHER = {American Mathematical Society, Providence, RI},
      YEAR = {2015},
     PAGES = {xii+479},
      ISBN = {978-1-4704-2558-6},
   MRCLASS = {35-02 (60J35 60J60)},
  MRNUMBER = {3443169},
MRREVIEWER = {Zhen-Qing\ Chen},
       DOI = {10.1090/surv/207},
}

@article{BogachevRoecknerShaposhnikov2021,
    author = {Bogachev, V. I. and R{\"o}ckner, M. and Shaposhnikov, S. V.},
    title = {{On the Ambrosio–Figalli–Trevisan Superposition Principle for Probability Solutions to Fokker–Planck–Kolmogorov Equations}},
    volume = {33},
    journal = {Journal of Dynamics and Differential Equations},
    pages = {715 -- 739},
    year = {2021},
    doi = {10.1007/s10884-020-09828-5},
    URL = {https://doi.org/10.1007/s10884-020-09828-5}
}

@article{Belaribi2012,
  title = {Uniqueness for {F}okker--{P}lanck equations with measurable coefficients and applications to the fast diffusion equation},
  volume = {17},
  ISSN = {1083-6489},
  url = {http://dx.doi.org/10.1214/EJP.v17-2349},
  DOI = {10.1214/ejp.v17-2349},
  number = {none},
  journal = {Electronic Journal of Probability},
  publisher = {Institute of Mathematical Statistics},
  author = {Belaribi,  N. and Russo,  F.},
  year = {2012},
  month = jan 
}

@article{NLFPK-DDSDE5,
	Author = {Barbu, V. and R{\"o}ckner, M.},
	Da = {2021/04/01/},
	Doi = {https://doi.org/10.1016/j.jfa.2021.108926},
	Isbn = {0022-1236},
	Journal = {Journal of Functional Analysis},
	Number = {7},
	Pages = {108926},
	Title = {Solutions for nonlinear {F}okker--{P}lanck equations with measures as initial data and {M}c{K}ean-{V}lasov equations},
	Ty = {JOUR},
	Url = {https://www.sciencedirect.com/science/article/pii/S0022123621000082},
	Volume = {280},
	Year = {2021}}

@article{BarbuRoeckner2020nonlinearsuperpositionprinciple,
author = {V. Barbu and M. R{\"o}ckner},
title = {{From nonlinear Fokker–Planck equations to solutions of distribution dependent SDE}},
volume = {48},
journal = {The Annals of Probability},
number = {4},
publisher = {Institute of Mathematical Statistics},
pages = {1902 -- 1920},
year = {2020},
doi = {10.1214/19-AOP1410},
}

@article{BarbuRehmeierRockner2024,
  author        = {V. Barbu and M. Rehmeier and M. R{\"o}ckner},
  title         = {$p$-{B}rownian Motion and the $p$-{L}aplacian},
  journal       = {The Annals of Probability},
  year          = {2024},
  note          = {to appear},
  eprint        = {2409.18744},
  archivePrefix = {arXiv},
}

@article{BarbuRocknerZhang2025,
  author  = {V. Barbu and M. R{\"o}ckner and D. Zhang},
  title   = {Uniqueness of Distributional Solutions to the 2D Vorticity Navier--Stokes Equation and Its Associated Nonlinear Markov Process},
  journal = {Journal of the European Mathematical Society},
  year    = {2025}
}

@article{Cornalba2021,
  title = {Well-posedness for a regularised inertial Dean–Kawasaki model for slender particles in several space dimensions},
  volume = {284},
  ISSN = {0022-0396},
  url = {http://dx.doi.org/10.1016/j.jde.2021.02.048},
  DOI = {10.1016/j.jde.2021.02.048},
  journal = {Journal of Differential Equations},
  publisher = {Elsevier BV},
  author = {Cornalba,  F. and Shardlow,  T. and Zimmer,  J.},
  year = {2021},
  month = may,
  pages = {253–283}
}

@article{Cornalba2019,
  title = {A Regularized Dean--Kawasaki Model: Derivation and Analysis},
  volume = {51},
  ISSN = {1095-7154},
  url = {http://dx.doi.org/10.1137/18M1172697},
  DOI = {10.1137/18m1172697},
  number = {2},
  journal = {SIAM Journal on Mathematical Analysis},
  publisher = {Society for Industrial & Applied Mathematics (SIAM)},
  author = {Cornalba,  F. and Shardlow,  T. and Zimmer,  J.},
  year = {2019},
  month = jan,
  pages = {1137–1187}
}

@article{Dean1996,
  title = {Langevin equation for the density of a system of interacting Langevin processes},
  volume = {29},
  ISSN = {1361-6447},
  url = {http://dx.doi.org/10.1088/0305-4470/29/24/001},
  DOI = {10.1088/0305-4470/29/24/001},
  number = {24},
  journal = {Journal of Physics A: Mathematical and General},
  publisher = {IOP Publishing},
  author = {Dean,  D. S.},
  year = {1996},
  month = dec,
  pages = {L613–L617}
}

@article {DNS2009,
    AUTHOR = {Dolbeault, J. and Nazaret, B. and Savar\'{e}, G.},
     TITLE = {A new class of transport distances between measures},
   JOURNAL = {Calc. Var. Partial Differential Equations},
  FJOURNAL = {Calculus of Variations and Partial Differential Equations},
    VOLUME = {34},
      YEAR = {2009},
    NUMBER = {2},
     PAGES = {193--231},
      ISSN = {0944-2669,1432-0835},
   MRCLASS = {49J40 (49Q20)},
  MRNUMBER = {2448650},
MRREVIEWER = {Luca\ Granieri},
       DOI = {10.1007/s00526-008-0182-5},
}

@book{Doob2001,
  author    = {J. L. Doob},
  title     = {Classical Potential Theory and Its Probabilistic Counterpart},
  series    = {Classics in Mathematics},
  publisher = {Springer-Verlag},
  address   = {Berlin},
  year      = {2001},
  note      = {Reprint of the 1984 edition}
}

@article{Dirr2016,
  title = {Entropic and gradient flow formulations for nonlinear diffusion},
  volume = {57},
  ISSN = {1089-7658},
  url = {http://dx.doi.org/10.1063/1.4960748},
  DOI = {10.1063/1.4960748},
  number = {8},
  journal = {Journal of Mathematical Physics},
  publisher = {AIP Publishing},
  author = {Dirr,  N. and Stamatakis,  M. and Zimmer,  J.},
  year = {2016},
  month = aug 
}

@book{Dynkin1965,
  author    = {E. B. Dynkin},
  title     = {Markov Processes},
  volume    = {1 and 2},
  series    = {Die Grundlehren der Mathematischen Wissenschaften},
  publisher = {Academic Press and Springer-Verlag},
  address   = {New York and Berlin-G{\"o}ttingen-Heidelberg},
  year      = {1965},
  note      = {Translated by J. Fabius, V. Greenberg, A. Maitra, G. Majone}
}

@book{Figalli2021,
  title = {An Invitation to Optimal Transport,  Wasserstein Distances,  and Gradient Flows},
  ISBN = {9783985475100},
  ISSN = {2943-4955},
  url = {http://dx.doi.org/10.4171/ETB/22},
  DOI = {10.4171/etb/22},
  journal = {EMS Textbooks in Mathematics},
  publisher = {EMS Press},
  author = {Figalli, A. and Glaudo,  F.},
  year = {2021},
  month = aug 
}

@article{Fehrman2024,
  title = {Well-Posedness of the Dean–Kawasaki and the Nonlinear Dawson–Watanabe Equation with Correlated Noise},
  volume = {248},
  ISSN = {1432-0673},
  url = {http://dx.doi.org/10.1007/s00205-024-01963-3},
  DOI = {10.1007/s00205-024-01963-3},
  number = {2},
  journal = {Archive for Rational Mechanics and Analysis},
  publisher = {Springer Science and Business Media LLC},
  author = {Fehrman, B. and Gess,  B.},
  year = {2024},
  month = mar 
}

@article{Fehrman2025,
  title = {Conservative stochastic PDEs on the whole space},
  ISSN = {2194-041X},
  url = {http://dx.doi.org/10.1007/s40072-025-00369-w},
  DOI = {10.1007/s40072-025-00369-w},
  journal = {Stochastics and Partial Differential Equations: Analysis and Computations},
  publisher = {Springer Science and Business Media LLC},
  author = {Fehrman, B. and Gess,  B.},
  year = {2025},
  month = jul 
}

@book{FukushimaOshimaTakeda2011,
  author    = {M. Fukushima and Y. Oshima and M. Takeda},
  title     = {Dirichlet Forms and Symmetric Markov Processes},
  series    = {De Gruyter Studies in Mathematics},
  volume    = {19},
  publisher = {Walter de Gruyter},
  address   = {Berlin},
  year      = {2011},
  note      = {Extended edition}
}

@book{Freidlin1996,
  author    = {M. Freidlin},
  title     = {Markov Processes and Differential Equations: Asymptotic Problems},
  series    = {Lectures in Mathematics ETH Z{\"u}rich},
  publisher = {Birkh{\"a}user},
  address   = {Basel},
  year      = {1996}
}

@book{K21,
	author = {Kallenberg, O. },
	isbn = {9783030618735},
	publisher = {Springer},
	title = {Foundations of Modern Probability},
	title1 = {Probability theory and stochastic modelling},
	ty = {BOOK},
	url = {https://books.google.de/books?id=mAVfzgEACAAJ},
	year = {2021}
}

@article{Konarovskyi2019,
  title = {Dean-Kawasaki dynamics: ill-posedness vs. triviality},
  volume = {24},
  ISSN = {1083-589X},
  url = {http://dx.doi.org/10.1214/19-ECP208},
  DOI = {10.1214/19-ecp208},
  number = {none},
  journal = {Electronic Communications in Probability},
  publisher = {Institute of Mathematical Statistics},
  author = {Konarovskyi,  V. and Lehmann,  T. and von Renesse,  M.-K.},
  year = {2019},
  month = jan 
}

@book{Kloeden1992,
  title = {Numerical Solution of Stochastic Differential Equations},
  ISBN = {9783662126165},
  url = {http://dx.doi.org/10.1007/978-3-662-12616-5},
  DOI = {10.1007/978-3-662-12616-5},
  publisher = {Springer Berlin Heidelberg},
  author = {Kloeden,  P. E. and Platen,  E.},
  year = {1992}
}

@article{KaminVazquez1988,
author = {S. Kamin and J. L. V\'azquez},
journal = {Revista Matemática Iberoamericana},
number = {2},
pages = {339-354},
title = {Fundamental solutions and asymptotic behaviour for the p-Laplacian equation.},
volume = {4},
year = {1988},
}

@book{Liggett2010,
  author    = {T. M. Liggett},
  title     = {Continuous Time Markov Processes: An Introduction},
  series    = {Graduate Studies in Mathematics},
  volume    = {113},
  publisher = {American Mathematical Society},
  address   = {Providence, RI},
  year      = {2010}
}

@article {Lisini2007,
    AUTHOR = {S. Lisini},
     TITLE = {Characterization of absolutely continuous curves in
              {W}asserstein spaces},
   JOURNAL = {Calc. Var. Partial Differential Equations},
  FJOURNAL = {Calculus of Variations and Partial Differential Equations},
    VOLUME = {28},
      YEAR = {2007},
    NUMBER = {1},
     PAGES = {85--120},
      ISSN = {0944-2669},
   MRCLASS = {49J10 (28A33 35Q35 49Q20)},
  MRNUMBER = {2267755},
MRREVIEWER = {Paul Raynaud de Fitte},
       DOI = {10.1007/s00526-006-0032-2},
       URL = {https://doi.org/10.1007/s00526-006-0032-2},
}

@article{Le2020,
  title = {A stochastic sewing lemma and applications},
  volume = {25},
  ISSN = {1083-6489},
  url = {http://dx.doi.org/10.1214/20-EJP442},
  DOI = {10.1214/20-ejp442},
  number = {none},
  journal = {Electronic Journal of Probability},
  publisher = {Institute of Mathematical Statistics},
  author = {Lê,  K.},
  year = {2020},
  month = jan 
}

@article{McKean1966,
 ISSN = {00278424},
 author = {H. P. McKean},
 journal = {Proceedings of the National Academy of Sciences of the United States of America},
 number = {6},
 pages = {1907--1911},
 publisher = {National Academy of Sciences},
 title = {{A Class of Markov Processes Associated with Nonlinear Parabolic Equations}},
 volume = {56},
 year = {1966}
}

@Article{MaoutsaReichManfred2020,
    AUTHOR = {D. Maoutsa and S. Reich and M. Opper},
    TITLE = {{Interacting Particle Solutions of Fokker–Planck Equations Through Gradient–Log–Density Estimation}},
    JOURNAL = {Entropy},
    VOLUME = {22},
    YEAR = {2020},
    NUMBER = {8},
    ARTICLE-NUMBER = {802},
    URL = {https://www.mdpi.com/1099-4300/22/8/802},
    PubMedID = {33286573},
    ISSN = {1099-4300},
    DOI = {10.3390/e22080802}
}

@book{nualart2006malliavin,
  title     = {The Malliavin Calculus and Related Topics},
  author    = {Nualart, D.},
  edition   = {2nd},
  series    = {Probability and Its Applications},
  year      = {2006},
  publisher = {Springer-Verlag},
  address   = {Berlin Heidelberg},
  isbn      = {978-3-540-28328-7},
  doi       = {10.1007/3-540-28329-3}
}

@article{Otto2001,
  title = {THE GEOMETRY OF DISSIPATIVE EVOLUTION EQUATIONS: THE POROUS MEDIUM EQUATION},
  volume = {26},
  ISSN = {1532-4133},
  url = {http://dx.doi.org/10.1081/PDE-100002243},
  DOI = {10.1081/pde-100002243},
  number = {1–2},
  journal = {Communications in Partial Differential Equations},
  publisher = {Informa UK Limited},
  author = {Otto,  F.},
  year = {2001},
  month = jan,
  pages = {101–174}
}

@article{RehmeierRoeckner2025NonlinearMarkov,
  author    = {M. Rehmeier and M. R{\"o}ckner},
  title     = {{On Nonlinear Markov Processes in the Sense of McKean}},
  journal   = {Journal of Theoretical Probability},
  volume    = {38},
  number    = {3},
  pages     = {60},
  year      = {2025},
  doi       = {10.1007/s10959-025-01428-7},
  issn      = {1572-9230}
}

@article{RehmeierRomito2025,
  author        = {M. Rehmeier and M. Romito},
  title         = {2{D} Vorticity {E}uler Equations: Superposition Solutions and Nonlinear {M}arkov Processes},
  journal       = {Stochastics and Partial Differential Equations: Analysis and Computations},
  year          = {2025},
  doi           = {10.1007/s40072-025-00387-8},
  eprint        = {2407.16609},
  archivePrefix = {arXiv},
}

@book{RogersWilliams2000,
  author    = {L. C. G. Rogers and D. Williams},
  title     = {Diffusions, Markov Processes, and Martingales, Vol. 2: It{\^o} Calculus},
  series    = {Cambridge Mathematical Library},
  publisher = {Cambridge University Press},
  address   = {Cambridge},
  year      = {2000},
  note      = {Reprint of the second (1994) edition}
}

@book{Sharpe1988,
  author    = {M. Sharpe},
  title     = {General Theory of Markov Processes},
  series    = {Pure and Applied Mathematics},
  volume    = {133},
  publisher = {Academic Press},
  address   = {Boston, MA},
  year      = {1988}
}

@book{Stroock87,
    place={Cambridge},
    series={London Mathematical Society Student Texts},
    title={Lectures on Stochastic Analysis: Diffusion Theory}, DOI={10.1017/CBO9780511623752},
    publisher={Cambridge University Press},
    author={Stroock, D. W.},
    year={1987},
    collection={London Mathematical Society Student Texts}
}

@Book{StroockVaradh2007,
	title     = {Multidimensional Diffusion Processes},
	publisher = {Springer Berlin Heidelberg},
	year      = {2007},
	author    = {Stroock, D. W. and Varadhan, S. R. S.},
	series    = {Classics in Mathematics},
	isbn      = {9783540289999},
	url       = {https://books.google.de/books?id=vKC1BwAAQBAJ},
}

@book{Stroock2014,
  author    = {D. W. Stroock},
  title     = {An Introduction to Markov Processes},
  series    = {Graduate Texts in Mathematics},
  volume    = {230},
  publisher = {Springer},
  address   = {Heidelberg},
  year      = {2014},
  edition   = {2}
}

@inbook{Sznitman1991,
  title = {Topics in propagation of chaos},
  ISBN = {9783540463191},
  ISSN = {1617-9692},
  url = {http://dx.doi.org/10.1007/BFb0085169},
  DOI = {10.1007/bfb0085169},
  booktitle = {Ecole d’Eté de Probabilités de Saint-Flour XIX — 1989},
  publisher = {Springer Berlin Heidelberg},
  author = {Sznitman,  A.-S.},
  year = {1991},
  pages = {165–251}
}

@INPROCEEDINGS{Taghvaei2016,
    author={A. Taghvaei and P. G. Mehta},
    booktitle={2016 American Control Conference (ACC)},
    title={An optimal transport formulation of the linear feedback particle filter},
    year={2016},
    volume={},
    number={},
    pages={3614-3619},
    doi={10.1109/ACC.2016.7525474}
}

@article{Trevisan16,
	author = {D. Trevisan},
	title = {{Well-posedness of multidimensional diffusion processes with weakly differentiable coefficients}},
	volume = {21},
	journal = {Electronic Journal of Probability},
	number = {none},
	publisher = {Institute of Mathematical Statistics and Bernoulli Society},
	pages = {1 -- 41},
	year = {2016},
	doi = {10.1214/16-EJP4453},
	URL = {https://doi.org/10.1214/16-EJP4453}
}

@incollection{ZeldovichKompaneets1950,
  author    = {Y. B. Zel'dovich and A. S. Kompaneets},
  title     = {Towards a theory of heat conduction with thermal conductivity depending on the temperature},
  booktitle = {Collection of Papers Dedicated to the 70th Birthday of Academician A. F. Ioffe},
  publisher = {Izd. Akad. Nauk SSSR},
  address   = {Moscow},
  pages     = {61--71},
  year      = {1950}
}

\end{document}